\documentclass[cm,dissertation]{uhthesis}
\usepackage{enumerate,amsmath,amssymb,fancyhdr,mathrsfs,amsthm,url,stmaryrd}
\usepackage[colorlinks=true,urlcolor=black,linkcolor=black,citecolor=black]{hyperref}
\usepackage{dcolumn}
\usepackage{booktabs}
\usepackage{multicol}
\usepackage{marvosym}
\usepackage{makeidx}
\usepackage{graphicx}
\usepackage{enumerate} 
\usepackage{tikz}
\usetikzlibrary{calc}
\usepackage{scalefnt}
\usepackage{verbatim}
\usepackage{color}
\usepackage{ifthen}
\theoremstyle{plain}
\newtheorem{theorem}{Theorem}[section]
\newtheorem{corollary}[theorem]{Corollary}
\newtheorem{lemma}[theorem]{Lemma}
\newtheorem{prop}[theorem]{Proposition}

\theoremstyle{definition}

\newtheorem{example}[theorem]{Example}

\newcounter{claim}
\newtheorem{claim}[claim]{Claim}
\newcounter{conjecture}
\newtheorem{conjecture}[conjecture]{Conjecture}

\theoremstyle{remark}
\newtheorem*{remark}{Remark}
\newtheorem*{remarks}{Remarks}

\numberwithin{theorem}{section}
\numberwithin{claim}{chapter}
\numberwithin{equation}{section}
\numberwithin{conjecture}{chapter}

\newcommand{\todo}[1]{
\ifthenelse{\boolean{todos}}{~\\{\bf TODO(wjd):} \emph{#1}\\}{}
}

\newcommand{\indexTuma}{\index{T\r{u}ma, Ji\v{r}\'i}}
\newcommand{\indexPudlak}{\index{Pudl\'ak, Pavel}}

\newcommand{\Jiri}{Ji\v{r}\'i}
\newcommand{\Tuma}{T\r{u}ma}
\newcommand{\Palfy}{P\'alfy}
\newcommand{\Pudlak}{Pudl\'ak}
\newcommand{\PP}{P\'alfy-Pudl\'ak}

\newcommand{\PAP}{P\'alfy\ and Pudl\'ak}
\newcommand{\Gratzer}{Gr\"{a}tzer}
\newcommand{\<}{\ensuremath{\langle}}
\renewcommand{\>}{\ensuremath{\rangle}}

\newcommand{\arity}{\ensuremath{\mathfrak{a}}}

\renewcommand{\leq}{\ensuremath{\leqslant}}
\renewcommand{\nleq}{\ensuremath{\nleqslant}}
\renewcommand{\geq}{\ensuremath{\geqslant}}

\renewcommand{\ngeq}{\ensuremath{\ngeqslant}}
\newcommand{\ssubnormal}{\ensuremath{\vartriangleleft}}
\newcommand{\subnormal}{\ensuremath{\trianglelefteqslant}}

\newcommand{\notsubnormal}{\ensuremath{\ntrianglelefteqslant}}
\newcommand{\meet}{\ensuremath{\wedge}}
\newcommand{\join}{\ensuremath{\vee}}
\newcommand{\Meet}{\ensuremath{\bigwedge}}
\renewcommand{\Join}{\ensuremath{\bigvee}}
\newcommand{\F}{\ensuremath{\mathbb{F}}}   
\newcommand{\Z}{\ensuremath{\mathbb{Z}}}   
\newcommand{\Q}{\ensuremath{\mathbb{Q}}}   
\newcommand{\N}{\ensuremath{\mathbb{N}}}   
\newcommand{\Hom}{\ensuremath{\mathrm{Hom}}}

\newcommand{\End}{\ensuremath{\mathrm{End}}}
\newcommand{\Aut}{\ensuremath{\mathrm{Aut}}}
\newcommand{\ID}[1]{\ensuremath{\mathcal{ID}(#1)}}
\newcommand{\idemdec}{\ID{X}}
\newcommand{\EqX}{\ensuremath{\mbox{Eq}(X)}}
\newcommand{\Eq}{\ensuremath{\mathrm{Eq}}}
\newcommand{\bEq}{\ensuremath{\mathbf{Eq}}}
\newcommand{\bEqX}{\ensuremath{\mathbf{Eq}(X)}}
\newcommand{\Cg}{\ensuremath{\mathrm{Cg}}}
\newcommand{\Sg}{\ensuremath{\mathrm{Sg}}}
\newcommand{\SD}{\ensuremath{\mathrm{SD}}}
\newcommand{\Stab}{\ensuremath{\mathrm{Stab}}}
\newcommand{\bCon}{\ensuremath{\mathbf{Con\,}}}
\newcommand{\Con}{\ensuremath{\mathrm{Con\,}}}
\newcommand{\Sub}{\ensuremath{\mathrm{Sub}}}

\newcommand{\Pol}{\ensuremath{\mathrm{Pol}}}
\newcommand{\Clo}{\ensuremath{\mathrm{Clo}}}
\newcommand{\Sym}{\ensuremath{\mathrm{Sym}}}
\newcommand{\core}{\ensuremath{\mathrm{core}}}

\newcommand{\htheta}{\ensuremath{\hat{\theta}}}
\newcommand{\supi}{\ensuremath{^{i}}}
\newcommand{\supj}{\ensuremath{^{j}}}
\newcommand{\FLRP}{{\small FLRP}}

\newcommand{\GAP}{\textsf{GAP}}

\newcommand{\uacalc}{\textsf{UACalc}}

\newcommand{\power}[1]{\ensuremath{\mathscr{P}(#1)}}

\newcommand{\2}{\ensuremath{\mathbf{2}}}
\newcommand{\3}{\ensuremath{\mathbf{3}}}

\newcommand{\5}{\ensuremath{\mathbf{5}}}
\newcommand{\bA}{\ensuremath{\mathbf{A}}}

\newcommand{\bB}{\ensuremath{\mathbf{B}}}
\newcommand{\bBi}{\ensuremath{\mathbf{B}_i}}
\newcommand{\sB}{\ensuremath{\mathscr{B}}}

\newcommand{\bC}{\ensuremath{\mathbf{C}}}

\newcommand{\sC}{\ensuremath{\mathscr{C}}}
\newcommand{\bd}{\ensuremath{\mathbf{d}}}

\newcommand{\sE}{\ensuremath{\mathscr{E}}}
\newcommand{\bF}{\ensuremath{\mathbf{F}}}
\newcommand{\sF}{\ensuremath{\mathscr{F}}}
\newcommand{\bG}{\ensuremath{\mathbf{G}}}
\newcommand{\sG}{\ensuremath{\mathscr{G}}}
\newcommand{\barG}{\ensuremath{\bar{G}}}
\newcommand{\barg}{\ensuremath{\bar{g}}}
\newcommand{\G}{\ensuremath{\mathfrak{G}}}
\newcommand{\bH}{\ensuremath{\mathbf{H}}}
\newcommand{\sH}{\ensuremath{\mathscr{H}}}
\newcommand{\barH}{\ensuremath{\overline{H}}}
\newcommand{\barh}{\ensuremath{\overline{h}}}

\newcommand{\sI}{\ensuremath{\mathscr{I}}}

\newcommand{\sK}{\ensuremath{\mathscr{K}}}
\newcommand{\bL}{\ensuremath{\mathbf{L}}}
\newcommand{\sL}{\ensuremath{\mathscr{L}}}

\newcommand{\bM}{\ensuremath{\mathbf{M}}}

\newcommand{\bn}{\ensuremath{\mathbf{n}}}
\newcommand{\bO}{\ensuremath{\mathbf{O}}}
\newcommand{\sO}{\ensuremath{\mathscr{O}}}

\newcommand{\bOmega}{\ensuremath{\alg \Omega}}
\newcommand{\ConO}{\ensuremath{\Con \bOmega}}
\newcommand{\bP}{\ensuremath{\mathbf{P}}}
\newcommand{\sP}{\ensuremath{\mathscr{P}}}
\newcommand{\cP}{\ensuremath{\mathcal{P}}}

\newcommand{\bS}{\ensuremath{\mathbf{S}}}
\newcommand{\bs}{\ensuremath{\mathbf{s}}}

\newcommand{\V}{\ensuremath{\mathrm{V}}}
\newcommand{\sV}{\ensuremath{\mathscr{V}}}
\newcommand{\bX}{\ensuremath{\mathbf{X}}}

\newcommand{\barx}{\ensuremath{\overline{x}}}

\newcommand{\bary}{\ensuremath{\overline{y}}}
\newcommand{\hlambda}{\ensuremath{\hat{\lambda}}}
\newcommand{\dotsize}{.8pt}
\newcommand{\giant}{\ensuremath{\mathfrak{Gi}}}
\newcommand{\solvable}{\ensuremath{\mathfrak{S}}}

\renewcommand{\iff}{\ensuremath{\quad \Leftrightarrow \quad}}
\newcommand{\ISLE}{{\small ISLE}}
\newcommand{\Soc}{\ensuremath{\mathrm{Soc}}}
\newcommand{\AGL}{\ensuremath{\mathrm{AGL}}}
\newcommand{\Out}{\ensuremath{\mathrm{Out}}}
\newcommand{\Inn}{\ensuremath{\mathrm{Inn}}}
\newcommand{\stab}[1]{\ensuremath{G_{#1}}}
\newcommand{\Gset}{\ensuremath{G\text{-set}}}
\newcommand{\Gsets}{\ensuremath{G\text{-sets}}}
\newcommand{\Mset}{\ensuremath{M\text{-set}}}

\newcommand{\op}{\operatorname}			
\newcommand{\alg}[1]{\mathbf{#1}}
\newcommand{\la}{\langle}	     
\newcommand{\ra}{\rangle}	     
\newcommand{\id}{\ensuremath{\mathrm{id}}}
\renewcommand{\phi}{\ensuremath{\varphi}}

\newcommand{\upalpha}{\ensuremath{\alpha^{\uparrow}}}
\newcommand{\downalpha}{\ensuremath{\alpha^{\downarrow}}}
\newcommand{\upbeta}{\ensuremath{\beta^{\uparrow}}}
\newcommand{\downbeta}{\ensuremath{\beta^{\downarrow}}}
\newcommand{\resB}{\ensuremath{|_{_B}}}
\newcommand{\resBi}{\ensuremath{|_{_{B_i}}}}
\newcommand{\eps}{\ensuremath{\varepsilon}}
\newcommand{\hatmap}{\ensuremath{\widehat{\phantom{x}}}}
\newcommand{\one}{\ensuremath{\mathbf{1}}}
\newcommand{\two}{\ensuremath{\mathbf{2}}}

\newcommand{\tbeta}{\ensuremath{\widetilde{\beta}}}
\newcommand{\hbeta}{\ensuremath{\widehat{\beta}}}
\newcommand{\cick}{\ensuremath{\sC_i}}
\newcommand{\CE}{\ensuremath{\sC_i^{\sE}}}
\newcommand{\CO}{\ensuremath{\sC_i^{\sO}}}
\newcommand{\CICK}{\ensuremath{c_i/\beta \cup c^{K+1}_i/\beta^{\bB_{K+1}}}}

\newboolean{todos}
\setboolean{todos}{false}  

\newcommand{\indexit}[1]{\index{#1|textit}}
\def\defn#1{\gdef\defnstring{#1}%
  \xdef\dodefnii{{\noexpand\em
       \defnstring}\noexpand\indexit{\defnstring}\noexpand\makeatother}%
  \futurelet\nextthing\dodefn}
\def\dodefn{%
  \ifx\nextthing[\let\next=\dodefni
    \else\let\next=\dodefnii\fi
  \makeatletter
  \next}

\def\dodefni[#1]{%
  {\em\defnstring}%
  \indexit{#1}%
  \makeatother}

\makeindex

\begin{document}

    \title{{\bf Congruence Lattices of Finite Algebras}}
    \author{William J.~DeMeo}

\degreemonth{May}
\degreeyear{2012}
\degree{Doctor of Philosophy}
\field{Mathematics}
\chair{Ralph Freese}
\othermembers{William Lampe\\
J.B.~Nation\\
Peter Jipsen\\
Nick Kaiser}
\numberofmembers{5}
\versionnum{2.2.0}

\maketitle

\begin{frontmatter}

\copyrightpage


\begin{acknowledgments}
First, I would like to thank my advisor, Ralph Freese, for his patience, support, and expert 
guidance, without which I could not have completed this dissertation.
Next, I thank the members of my dissertation
committee, Peter Jipsen, Bill Lampe, and J.B. Nation.  All have made
significant contributions to this work.
Bill Lampe, in particular, is responsible for introducing me to the
beautiful subject of universal algebra.

I thank Nick Kaiser for agreeing to act as the University Representative on my
dissertation committee, and for enduring long meetings about topics unrelated to his
area of expertise (though I suspect he understands far more than he lets on).

A number of other professors played a significant role in my mathematical
training. Among them, I would especially like to thank Ron Brown, Tom
Craven, Erik Guentner, Bj{\o}rn Kjos-Hanssen, Tom Ramsey, and Wayne Smith.  Mike
Hilden was kind enough 
to administer my French language exam, and I thank him for his help
with this minor hurdle, and for not setting the bar too high.

The Mathematics Department at the University of Hawai`i has generously supported
me through the doctoral program, and for that I am grateful.
I would also like to thank other members of the department who have played vital
roles in my progress through the program; in particular, I thank Susan Hasegawa,
Shirley Kikiloi, and Troy Ludwick. 

I thank the ARCS Foundation of Honolulu for generously supporting me with the
Sarah Ann Martin award for outstanding research in mathematics, 
as well as the Graduate Student Organization of the
University of Hawai`i for supporting me with a travel grant.  

My deepest appreciation goes to Hyeyoung Shin, my greatest source of inspiration, 
to my sister, B.J.~Casey, and to my parents, Bill and Benita DeMeo, and Barbara
and Ted Terry, whose contributions to this dissertation are
immeasurable.  Their moral support and encouragement seem unbounded and
independent of their understanding or appreciation of my work.

Finally, I dedicate this dissertation to my mother, Barbara Anderson Terry,
for her unconditional love and support, for her patience, and for inspiring me
to do good work. I owe her everything.  
\end{acknowledgments}


\begin{abstract}
An important and long-standing open problem in universal algebra asks whether
every finite lattice is isomorphic to the congruence lattice of a finite
algebra. Until this problem is resolved, our understanding of finite algebras is 
incomplete, since, given an arbitrary finite algebra, we cannot say whether
there are any restrictions on the shape of its congruence lattice. If we find a
finite lattice that does not occur as the congruence lattice of a finite algebra
(as many suspect we will), then we can finally declare that such restrictions do
exist.   

By a well known result of \Palfy\ and \Pudlak, the problem would be solved if we
could prove the existence of a finite lattice that is not the congruence lattice
of a transitive group action or, equivalently, is not an interval in the lattice
of subgroups of a finite group.  Thus the problem of characterizing
congruence lattices of finite algebras is closely related to the problem of
characterizing intervals in subgroup lattices.

In this work, we review a number of methods for finding a finite algebra with a
given congruence lattice, including searching for intervals in
subgroup lattices.  We also consider methods for proving that algebras with a given
congruence lattice exist without actually constructing them. By combining these
well known methods with a new method we have developed, and with much help from
computer software like the \uacalc\ and \GAP, we prove that with one possible
exception every lattice with at most seven elements is isomorphic to the
congruence lattice of a finite algebra.  As such, we have identified the unique
smallest lattice for which there is no known representation. 
We examine this exceptional lattice in detail, and prove results that
characterize the class of algebras that could possibly represent this lattice.

We conclude with what we feel are the most interesting open questions
surrounding this problem and discuss possibilities for future work. 

\end{abstract}

\tableofcontents

\newcommand{\skipamt}{0mm}
\listofnewresults{
\vskip\skipamt \noindent
{\bf Proposition~\ref{Concrete-prop-1}} \dotfill \pageref{Concrete-prop-1}
\vskip\skipamt \noindent
{\bf Proposition~\ref{Concrete-prop-1'}} \dotfill  \pageref{Concrete-prop-1'}
\vskip\skipamt \noindent
{\bf Lemma~\ref{Concrete-lemma-1}} \dotfill  \pageref{Concrete-lemma-1}
\vskip\skipamt \noindent
{\bf Lemma~\ref{Concrete-lemma-2}} \dotfill  \pageref{Concrete-lemma-2}
\vskip\skipamt \noindent
{\bf Theorem~\ref{Concrete-thm-1}} \dotfill  \pageref{Concrete-thm-1}
\vskip\skipamt \noindent
{\bf Corollary~\ref{Concrete-cor-2}} \dotfill  \pageref{Concrete-cor-2}
\vskip\skipamt \noindent
{\bf Corollary~\ref{Concrete-cor-nondensity-1}} \dotfill  \pageref{Concrete-cor-nondensity-1}
\vskip\skipamt \noindent
{\bf Lemma~\ref{lemma-intransGsets}}  \dotfill \pageref{lemma-intransGsets}
\vskip\skipamt \noindent
{\bf Proposition~\ref{prop:parachute}} \dotfill \pageref{prop:parachute}
\vskip\skipamt \noindent
{\bf Lemma~\ref{lemma-wjd-1}} \dotfill  \pageref{lemma-wjd-1}
\vskip\skipamt \noindent
{\bf Lemma~\ref{lemma-wjd-2}} \dotfill  \pageref{lemma-wjd-2}
\vskip\skipamt \noindent
{\bf Lemma~\ref{lemma-wjd-2}${}'$} \dotfill  \pageref{lemma-wjd-2}
\vskip\skipamt \noindent
{\bf Lemma~\ref{lemma-wjd-3}} \dotfill  \pageref{lemma-wjd-3}
\vskip\skipamt \noindent
{\bf Corollary~\ref{cor:isle-prop-groups-1}} \dotfill  \pageref{cor:isle-prop-groups-1}
\vskip\skipamt \noindent
{\bf Lemma~\ref{lem:ISLE-must-have-wreaths}} \dotfill  \pageref{lem:ISLE-must-have-wreaths}
\vskip\skipamt \noindent
{\bf Lemma~\ref{lemma-wjd-4}} \dotfill  \pageref{lemma-wjd-4}
\vskip\skipamt \noindent
{\bf Lemma~\ref{lemma-wjd-5}} \dotfill  \pageref{lemma-wjd-5}
\vskip\skipamt \noindent
{\bf Theorem~\ref{thm:sevenelementlattices}} \dotfill  \pageref{thm:sevenelementlattices}
\vskip\skipamt \noindent
{\bf Theorem~\ref{thm:except-seven-elem}} \dotfill  \pageref{thm:except-seven-elem}
\vskip\skipamt \noindent
{\bf Lemma~\ref{lem:residuation-lemma}} \dotfill  \pageref{lem:residuation-lemma}
\vskip\skipamt \noindent
{\bf Theorem~\ref{OAthm1}} \dotfill  \pageref{OAthm1}
\vskip\skipamt \noindent
{\bf Theorem~\ref{OAthm2}} \dotfill  \pageref{OAthm2}
\vskip\skipamt \noindent
{\bf Proposition~\ref{prop:expansion}} \dotfill  \pageref{prop:expansion}
\vskip\skipamt \noindent
{\bf Theorem~\ref{OAthm3}} \dotfill  \pageref{OAthm3}
\vskip\skipamt \noindent
{\bf Lemma~\ref{lem3.1}} \dotfill  \pageref{lem3.1}
\vskip\skipamt \noindent
{\bf Theorem~\ref{OAthm4}} \dotfill  \pageref{OAthm4}
\vskip\skipamt \noindent
{\bf Theorem~\ref{thm-overalgebras-iii}} \dotfill  \pageref{thm-overalgebras-iii}
\vskip\skipamt
}

\listoffigures



\listofsymbols{
\begin{table}[h!]
    \begin{tabular}{ll}
      $\2$ & $\{0, 1\}$, or the two element lattice \\
      $\3$ & $\{0, 1, 2\}$, or the three element lattice \\
      $\bn$ & the set $\{0, 1, \dots, n-1\}$, or the $n$ element chain\\
      $\omega$ & the natural numbers, $\{0, 1, 2, \dots \}$\\
      \Z & the integers, $\{\dots, -1, 0, 1, \dots\}$\\
      \Q & the rational numbers\\
      \F & an arbitrary field \\
      $\bA, \bB, \bC, \dots$ & universal algebras\\
      $\bA = \<A, F\>$ & an algebra with universe $A$ and operations $F$\\
      $\Clo (\bA)$ & the clone of term operations of $\bA$\\
      $\Pol(\bA)$ & the clone of polynomial operations of $\bA$\\
      $\Pol_n(\bA)$ & the set of $n$-ary members of $\Pol(\bA)$\\
      $\Aut(\bA)$ & the group of automorphisms of $\bA$\\
      $\Inn(\bA)$ & the inner automorphisms of $\bA$\\
      $\Out(\bA)$ & the outer of automorphisms of $\bA$\\
      $\End(\bA)$ & the monoid of endomorphisms of $\bA$\\
      $\Hom(\bA, \bB)$& the set of homomorphisms from $\bA$ into $\bB$\\
      $\Con(\bA)$& the lattice of congruence relations of $\bA$\\
      $\Sub(\bA)$& the lattice of subalgebras of $\bA$\\
      $\Sg^{\bA}(X)$& the subuniverse of $\bA$ generated by the set $X\subseteq A$\\
      $\Cg^{\bA}(X)$& the congruence of $\bA$ generated by the set $X\subseteq A\times A$\\
      $\Eq(X)$& the lattice of equivalence relations on the set $X$\\
      $X^X$ & the set of unary maps from a set $X$ into itself\\
      $\ker f$ & the kernel of $f$, $\{(x,y) \mid f(x) = f(y)\}$\\
      $\ID{X}$ & the idempotent decreasing functions in $X^X$\\
      $\sqsubseteq$ & the partial order defined on $\ID{X}$ by $f\sqsubseteq g \;
      \Leftrightarrow \; \ker f \leq \ker g$\\
      $\sK$ & a class of algebras\\
      $\bH (\sK)$ & the class of homomorphic images of algebras in $\sK$\\
      $\bS (\sK)$ & the class of subalgebras of algebras in $\sK$\\
      $\bP (\sK)$ & the class of direct products of algebras in $\sK$\\
      $\bP_{\mathrm{fi}} (\sK)$ & the class of finite direct products of
      algebras in $\sK$\\ 
      $\sV$ & a variety, or equational class, of algebras\\
      $\V(\bA)$ & the variety generated by $\bA$ 
      (thus $\V(\bA) = \bH \bS \bP(\bA)$\\
      $\V(\sK)$ & the variety generated by the class $\sK$\\
      $\bF_{\sV}(X)$& the free algebra in the variety $\sV$ over the
      generating set $X$\\
      $\sL_0$ & the class of finite lattices\\
      $\sL_1$ & the class of lattices isomorphic to sublattices of finite partition lattices\\
      $\sL_2$ & the class of lattices isomorphic to strong congruence lattices of
      finite partial algebras\\
      $\sL_3$ & the class of lattices isomorphic to congruence lattices of finite algebras\\
      $\sL_4$ & the class of lattices isomorphic to intervals in subgroup lattices of finite groups\\
      $\sL_5$ & the class of lattices isomorphic to subgroup lattices of finite groups\\
    \end{tabular}
\end{table}
}

\end{frontmatter}

\part{Background}

\chapter{Introduction}
\label{cha:introduction}

We begin with an informal overview of some of the basic objects
of study.
This will help to fix notation and motivate our discussion.
(Italicized terms are defined more formally in later sections or
in the appendix.)  Then we 
introduce the problem that is the main focus of this
\index{FLRP}%
dissertation, \emph{the finite lattice representation problem} (\FLRP).  
In subsequent sections, we give further
notational and algebraic prerequisites and summarize the well known results
surrounding the FLRP.  In the final section of this chapter we provide a list of
the new results of this thesis.

\section{Motivation and problem statement}
Among the most basic objects of study in all of mathematics are algebras.
An 
\index{algebra}%
\emph{algebra} 
$\bA = \<A, F\>$ consists of a nonempty set $A$ and a
collection $F$ of operations;  
the most important examples are lattices, groups, rings, and modules.  To
understand a particular algebra, $\bA$, we often study its representations,
which are 
\emph{homomorphisms} from $\bA$ into some other algebra $\bB$.
A very important feature of such a homomorphism $\varphi$ is its 
\defn{kernel}, 
which we define as the set $\{(x,y)\in A^2 \mid \varphi(x) = \varphi(y)\}$.
This is a 
\index{congruence relation}%
\emph{congruence relation}
 of the algebra $\bA$ which tells us how
$\bA$ is ``reduced'' when represented by its image under $\varphi$ in $\bB$.   

Thus, every homomorphism gives rise to a congruence relation, and the set 
$\Con \bA$ of all congruence relations of the algebra $\bA$ forms a 
\index{lattice}%
\emph{lattice}.  
For example, if $\bA$ happens to be a group, $\Con \bA$ is isomorphic
to the lattice of normal subgroups  of $\bA$.\footnote{In this context, by
  ``kernel'' of a homomorphism $\varphi$ one typically means the normal subgroup
  $\{a \in A \mid \varphi(a) = e\}$, whereas 
  this is a single  congruence class of the kernel as we have defined it.}
To each congruence $\theta \in \Con \bA$ there corresponds the natural
homomorphism of $\bA$ onto $\bA/\theta$ which has $\theta$ as its kernel.   
Thus, there is a one-to-one correspondence between
$\Con\bA$ and the natural homomorphisms, and the shape of $\Con\bA$
provides useful information about the algebra and its representations.  
For instance, $\Con \bA$ tells us whether and how $\bA$ can be decomposed as, or
embedded in, a product of simpler algebras. 

Given an arbitrary algebra, then, we ought to know whether there are, 
{\it a priori}, any restrictions on the possible shape of its congruence
lattice.  A celebrated result of 
\index{Gr\"{a}tzer, George}%
Gr\"{a}tzer and
\index{Schmidt, E.T.}%
Schmidt says that there are (essentially) no such
restrictions. Indeed, in~\cite{GratzerSchmidt:1963} it is
proved that every (algebraic) lattice is the congruence lattice of some algebra.   
Moreover, as \Jiri\ \Tuma\ proves in~\cite{Tuma:1986},
the Gr\"{a}tzer-Schmidt Theorem still holds if we restrict ourselves to
intervals in subgroup lattices.  That is,  every algebraic lattice is isomorphic
to an interval in the subgroup lattice of an (infinite) group.  

Now, suppose we restrict our attention to \emph{finite} algebras.  
Given an arbitrary finite algebra, it is natural to ask whether there are any
restrictions (besides finiteness) on the shape of its congruence lattice. 
If it turns out that, given an arbitrary finite lattice $\bL$, we can always find
a finite algebra $\bA$ that has $\bL$ as its congruence lattice, then apparently
there are no such restrictions.

We call a lattice 
\index{representable lattice}%
\defn{finitely representable}, 
or simply \emph{representable}, if it is isomorphic to the congruence lattice of a finite
algebra, and deciding whether every finite lattice is
representable is known as the 
\defn{finite lattice representation problem} (\FLRP).  For 
the reasons mentioned above, this is a fundamental question of modern
algebra, and the fact that it remains unanswered is quite remarkable.

\section{Universal algebra preliminaries}
\label{sec:univ-algebra-prel}
We now describe in greater detail some of the algebraic objects that are central
to our work.  A more complete introduction to this material can be found in
the books and articles listed in the bibliography.  In particular,
the following are the main references for this work: \cite{alvi:1987}, \cite{Palfy:1980},
\cite{Dixon:1996}, \cite{Rose:1978}, and \cite{Isaacs:2008}.  Two excellent
survey articles on the finite lattice representation problem
are~\cite{Palfy:1995} and~\cite{Palfy:2001}. 

First, a few words about notation.  When discussing universal algebras,
such as $\bA = \<A, F\>$, we denote the algebras using bold symbols, as in $\bA,
\bB, \dots$, 
and reserve the symbols $A, B, \dots$ for the universes of these algebras.
However, this convention becomes tiresome and inconvenient if strictly adhered to
for all algebras, and we often find ourselves referring to an algebra by its universe.
For example, we frequently use $L$ when referring to the lattice
$\bL = \<L,\join, \meet\>$, and we usually refer to ``the lattice of congruence
relations $\Con \la A, F \ra$,'' even though it would be more precise to 
call $\Con \la A, F \ra$ the universe (a set) and use
$\bCon \bA = \<\Con \la A, F \ra, \meet, \join\>$ to denote the 
lattice (an algebra).  Certainly we will feel free to commit this sort of abuse when
speaking about groups, preferring to use $G$ when referring to the group 
$\bG = \<G, \cdot, ^{-1}, 1\>$.
Sometimes we use the more precise notation $\bEqX$ to denote the lattice of
equivalence relations on the set $X$, but more frequently we will refer to this
lattice by its universe, $\EqX$.  This has never been a source of confusion.

An \defn{operation symbol} $f$ is an object that
has an associated \defn{arity}, which we denote by $\arity(f)$.  A set of operation
symbols $F$ is called a  \defn{similarity type}.  
An \defn{algebra} of similarity type $F$ is a pair $\bA = \<A, F^\bA \>$ consisting of
a set $A$, which we call the \defn{universe} of $\bA$, and a set
$F^\bA = \{f^\bA : f\in F\}$  of \defn{operations} on $A$, which are functions
$f^\bA : A^{\arity(f)} \rightarrow A$ of arity $\arity(f)$.
Occasionally the set of operations only enters the discussion abstractly,
and it becomes unnecessary to refer to specific operation symbols.  In such
instances, we often denote the algebra by $\<A, \dots \>$.

Note that the symbol $f$ -- like the operation symbol $+$ that is
used to denote addition in \emph{some} algebras -- is an abstract operation
symbol which, apart from its arity, has no specific meaning attached to it.  We use the
notation $f^\bA$ to signify that we have given the operation symbol a specific
interpretation as an operation in the algebra $\bA$. 
Having said that, when there is only one algebra under consideration, it
seems pedantic to attach the superscript $\bA$ to every operation.  In such cases,
when no confusion can arise, we allow the operation symbol $f$ to
denote a specific operation interpreted in the algebra.  Also, if $F$ is the set
\index{$F_n$, the $n$-ary operations in $F$}%
of operations (or operation symbols) of $\bA$, we let $F_n\subseteq F$ denote the $n$-ary
operations (or operation symbols) of $\bA$.

Let $A$ and $B$ be sets and let $\varphi : A\rightarrow B$ be any 
mapping.
We say that a pair $(a_0, a_1)\in A^2$ belongs to the \defn{kernel} of $\varphi$, and we
write $(a_0, a_1) \in \ker \varphi$, provided
$\varphi(a_0)=\varphi(a_1)$. 
It is easily verified that $\ker \varphi$ is an equivalence relation on the set $A$.
If $\theta$ is an equivalence relation on a set $A$, then $a/\theta$ denotes the
equivalence class containing $a$; that is, 
$a/\theta := \{ a' \in A \mid (a,a')\in \theta \}$. The set of all
equivalence classes of $\theta$ in $A$ is denoted $A/\theta$. That is, 
$A/\theta = \{a/\theta \mid a\in A\}$. 

Let $\bA = \< A, F^\bA \>$ and $\bB = \<B, F^\bB\>$ be algebras of the same
similarity type.  A 
\defn{homomorphism} from $\bA$ to $\bB$ is a function
$\varphi : A \rightarrow B$ that respects the interpretation of the operation
symbols.  That is, if $f\in F$ with, say, $n = \arity(f)$, and if 
$a_1, \dots, a_n \in A$, then
$\phi(f^\bA(a_1, \dots, a_n)) = f^{\bB}(\phi(a_1), \dots, \phi(a_n))$.
A \defn{congruence relation} of $\bA$ is the
kernel of a homomorphism defined on $\bA$.
We denote the set of all congruence relations
of $\bA$ by $\Con \bA$. 
Thus, $\theta \in \Con \bA$ if and only if $\theta = \ker \varphi$ for some 
homomorphism 
$\varphi : \bA \rightarrow \bB$.
It is easy to check that this is
equivalent to the following:
$\theta \in \Con \bA$ if and only if $\theta \in \Eq(A)$ and for all $n$
\begin{equation}
\label{eq:cong-re}
(a_i, a_i') \in \theta \quad (0\leq i < n) \quad \Rightarrow \quad 
(f(a_0, \dots, a_{n-1}), f(a_0', \dots, a_{n-1}')) \in \theta,
\end{equation}
for all $f\in F_n$ and all $a_0, \dots, a_{n-1}, a_0', \dots, a_{n-1}' \in A$. 
Equivalently, $\Con \bA = \Eq(A) \cap \Sub(\bA\times \bA)$.

Given a congruence relation $\theta\in \Con \bA$, the 
\defn{quotient algebra} $\bA/\theta$ is the algebra with universe
$A/\theta = \{a/\theta \mid a\in A\}$ and operations $\{f^{\bA/\theta} \mid f\in
F\}$ defined as follows: 
\[
f^{\bA/\theta}(a_1/\theta, \dots, a_n/\theta) = f^\bA(a_1, \dots, a_n)/\theta,
\text{ where $n=\arity(f)$.}
\]

A \defn{partial algebra} is a set $A$ (the universe) along with
a set of \defn{partial operations}, that is, operations which may be defined
on only part of the universe.  A \defn{strong congruence relation} of a partial algebra $\bA$ is
an equivalence relation $\theta \in \Eq(A)$ with the following property: for each (partial)
operation $f$ of $A$, if $f$ is $k$-ary, if
$(x_i, y_i)\in \theta$ $\,(1\leq i\leq k)$, and if $f(x_1, \dots, x_k)$ exists, then
$f(y_1, \dots, y_k)$ exists, and 
$(f(x_1, \dots, x_k), f(y_1, \dots, y_k)) \in \theta$.
We will have very little to say about partial algebras, but
they appear below in our overview of significant results related to the
\FLRP. 

Let $\bA = \<A, \dots \>$ be an algebra with congruence lattice $\Con\<A, \dots \>$.
Recall that a 
\defn{clone} 
on a non-void set $A$ is a set of operations on $A$
that contains the projection operations and is closed under compositions. 
The 
\index{clone!of term operations}%
\emph{clone of term operations}
of the algebra $\bA$, denoted by 
\index{$\Clo (\bA)$|see{clone of term operations}}%
$\Clo (\bA)$,
is the smallest clone on $A$ containing the basic operations of $\bA$.
The 
\index{clone!of polynomial operations}%
\emph{clone of polynomial operations} of $\bA$, 
denoted by 
\index{$\Pol(\bA)$|see{clone of polynomial operations}}%
$\Pol(\bA)$, 
is the clone generated by the basic operations
of $\bA$ and the constant unary maps on $A$. 
The set of $n$-ary members of $\Pol(\bA)$ is denoted by 
\index{$\Pol_n(\bA)$}%
$\Pol_n(\bA)$.

By a \defn{unary algebra} we mean an algebra with any number of unary
operations.\footnote{Note that some authors reserve this
term for algebras with a single unary operation, and use the term
  \defn{multi-unary algebra} when referring to what we call unary algebra.}
In our work, as we are primarily concerned with congruence lattices, we may
restrict our attention to unary algebras whenever helpful or convenient,
as the next result shows (cf.~Theorem~4.18 of~\cite{alvi:1987}).
\begin{lemma}
\label{sec:unarycongruences}
If $F$ is a set of operations on $A$, then
\[
\Con \la A, F \ra = \Con \la A, F' \ra,
\]
where $F'$ is any of $\op{Pol}(\alg A)$, $\op{Pol_1}(\alg A)$,
or the set of basic translations (operations in $\op{Pol_1}(\alg A)$
obtained from $F$ by fixing all but one coordinate).
\end{lemma}

The lattice formed by all subgroups of a group $G$, denoted $\Sub(G)$, is
called the \defn{subgroup lattice} of $G$. It is a 
\defn{complete lattice}: any number of 
subgroups $H_i$ have a \defn{meet} (greatest lower bound) $\Meet H_i$, namely their
intersection
$\bigcap H_i$, and a \defn{join} (least upper bound) $\Join H_i$, namely the subgroup
generated by the union of them.  We denote the group generated by the
subgroups $\{H_i : i \in I\}$ 
by $\<H_i : i \in I\>$ when
$I$ is infinite, and by $\<H_0, H_1, \dots, H_{n-1}\>$, otherwise. 
Since a complete lattice is algebraic if and only if
every element is a join of compact elements, we see that subgroup lattices are
always algebraic.  We mention these facts because of their general importance,
but we remind the reader that all groups in this work are finite.

\section{Overview of well known results}
\label{sec:overv-known-results}
Major inroads toward a solution to the \FLRP\ have been made by many prominent
researchers, including 
\index{Aschbacher, Michael}%
Michael Aschbacher, 
\index{Feit, Walter}%
Walter Feit, 
\index{Kurzweil, Hans}%
Hans Kurzweil, 
\index{Lucchini, Adrea}%
Adrea Lucchini, 
\index{McKenzie, Ralph}%
Ralph McKenzie, 
\index{Netter, Raimund}%
Raimund Netter,
\index{P\'alfy, P\'eter}%
P\'eter \Palfy,
\index{Pudl\'ak, Pavel}%
Pavel \Pudlak, 
\index{Snow, John}%
John Snow, and 
\indexTuma%
\Jiri\ \Tuma, to name a few. 
We will have occasion to discuss and apply a number of their results in the
sequel.  Here we merely mention some of the highlights, in roughly chronological
order.  

\index{Gr\"{a}tzer, George}%
In his 1968 book {\it Universal Algebra}~\cite{Gratzer:1968}, George \Gratzer\ defines the
following classes of lattices:
\begin{itemize}
\index{$\mathscr{L}_0$}%
\item $\sL_0 =$ the class of finite lattices;
\index{$\mathscr{L}_1$}%
\item $\sL_1 =$ the class of lattices isomorphic to sublattices of finite partition lattices;
\index{$\mathscr{L}_2$}%
\item $\sL_2 =$ the class of lattices isomorphic to strong congruence lattices of
  finite partial algebras;
\index{$\mathscr{L}_3$}%
\item $\sL_3 =$ the class of lattices isomorphic to congruence lattices of finite algebras.
\end{itemize}
Clearly $\sL_0 \supseteq \sL_1 \supseteq \sL_2 \supseteq \sL_3$.
\Gratzer\ asks (\cite{Gratzer:1968} prob.~13, p.~116) whether equality holds
in each case. 
Whether $\sL_0 = \sL_1$ is the finite version of a question 
\index{Birkhoff, Garrett}%
Garrett Birkhoff had asked by 1935.  
In~\cite{Birkhoff:1935} Birkhoff asks whether every lattice is isomorphic 
\index{Whitman, P.M.}%
to a sublattice of some partition lattice.  Whitman~\cite{Whitman:1946} answered
this affirmatively in 1946, but his proof embeds every finite lattice in a countably
infinite partition lattice.  Still, the result of Whitman also proves that 
there is no non-trivial law that holds in the subgroup lattice of every group.
That is,
\begin{theorem}[Whitman~\cite{Whitman:1946}] Every lattice is isomorphic to a
  sublattice of the subgroup lattice of some group.
\end{theorem}
Confirmation that $\sL_0=\sL_1$ did not come until the late 1970's, when 
\index{Pudl\'ak, Pavel}%
\indexTuma%
Pavel \Pudlak\ and \Jiri\ \Tuma\ published~\cite{Pudlak:1980}, in which they
prove that every finite lattice can be embedded in a finite partition lattice,
thus settling this important and long-standing open question.
This result also yields the following finite analogue of Whitman's result:
\begin{theorem}[\Pudlak-\Tuma~\cite{Pudlak:1980}]
Every finite lattice is isomorphic to a sublattice of the subgroup lattice of
some finite group.
\end{theorem}

If we confine ourselves to distributive lattices, the analogue of the
\FLRP\ is relatively easy.  By the 1930's it was already known to 
\index{Dilworth, Robert}%
Robert Dilworth that every finite distributive lattice is the congruence lattice
of a finite lattice.\footnote{This is mentioned in~\cite{Birkhoff:1940} without proof.}
(In fact, if we allow representations by infinite algebras -- which, as
a rule in this work, we do not -- then the congruence lattices of
modular lattices already account for all distributive lattices.  This is shown by
\index{Schmidt, E.T.}%
E.T.~Schmidt in~\cite{Schmidt:1982}, and extended by
\index{Freese, Ralph}%
Ralph Freese who shows in~\cite{Freese:1975} that finitely generated modular
lattices suffice.)\footnote{It turns out that the finite distributive lattices
  are representable as congruence lattices of other restricted classes of
  algebras.  We will say a bit more about this below, but we refer the
  reader to~\cite{Palfy:1987} for more details.}

A lattice $L$ is called \defn{strongly representable} 
if, whenever $L$ is isomorphic to a \defn{spanning sublattice}\footnote{By a
  \defn{spanning sublattice} of a bounded lattice $L_0$, we mean a sublattice
  $L\leq L_0$ that has the same top and bottom as $L_0$.  That is  $1_L =
  1_{L_0}$ and $0_L = 0_{L_0}$.}  
$L_0 \leq \Eq(X)$ for some $X$, then there is an algebra $\<X, \dots\>$ whose
congruence lattice is $L_0$. 
\index{Berman, Joel} \index{Quackenbush, R.} \index{Wolk, B.}%
\begin{theorem}[Berman~\cite{Berman:1970}, Quackenbush and Wolk~\cite{Quack:1971}]
\label{thm:distr-lattices}
Every finite distributive lattice is strongly representable.
\end{theorem}
\noindent (We give a short proof of this result in Section~\ref{sec:distr-latt}.)
Berman also proves  that if $\bA_p$ is a finite partial unary algebra with
strong congruence lattice $\mathrm{Con}_s \bA_p$, then there is a finite unary
algebra $\bA$ with $\Con \bA \cong \mathrm{Con}_s \bA_p$. Therefore, by
Lemma~\ref{sec:unarycongruences}, $\sL_2 = 
\sL_3$.  
As our focus
is mainly on whether $\sL_0 = \sL_3$, we will not say more about partial
algebras except to note
\index{Pudl\'ak, Pavel}%
\indexTuma%
\index{Berman, Joel}%
that the results of \Pudlak, \Tuma, and Berman imply that $\sL_0 = \sL_3$
holds if and only if $\sL_1=\sL_2$ holds.

Next, we mention another deep result of 
\index{Pudl\'ak, Pavel}%
\indexTuma%
\Pudlak\ and \Tuma, which proves the existence of congruence lattice
representations for a large class of lattices.
\begin{theorem}[\Pudlak\ and \Tuma~\cite{Pudlak:1976}]
\label{thm:fermentable}
Let $L$ be a finite lattice such that
both $L$ and its congruence lattice 
have the same number of join irreducible elements. 
Then $L$ is representable. 
\end{theorem}
\noindent Notice that finite distributive lattices satisfy the
assumption of Theorem~\ref{thm:fermentable}, so this provides yet another proof
that such lattices are representable.

We now turn to subgroup lattices of finite groups and their connection with the \FLRP.
The study of subgroup lattices has a long history, starting with Richard
\index{Dedekind, Richard}%
Dedekind's work~\cite{Dedekind:1877} in 1877, including 
\index{Rottlaender, Ada}%
Ada Rottlaender's paper~\cite{Rottlaender:1928} from 1928, and later numerous
important contributions by 
\index{Baer, Reinhold}%
Reinhold Baer, 
\index{Ore, {\O}ystein}%
{\O}ystein Ore, 
\index{Iwasawa, Kenkichi}%
Kenkichi Iwasawa, 
Leonid Efimovich Sadovskii, 
\index{Suzuki, Michio}%
Michio Suzuki, 
Giovanni Zacher, Mario Curzio, Federico Menegazzo, 
\index{Schmidt, Roland}%
Roland Schmidt, Stewart Stonehewer, Giorgio Busetto, and many
others.  The book~\cite{Schmidt:1994} by Roland Schmidt gives a comprehensive
account of this work.

Suppose $H$ is a subgroup of $G$  (denoted $H\leq G$).  
By the \defn{interval sublattice} $[H, G]$ we mean the sublattice of $\Sub(G)$
given by: 
    \[
      [H,G] := \{K \mid H\leq K \leq G\},
      \]
That is $[H,G]$ is the lattice of subgroups of $G$ that contain 
$H$.\footnote{The reader may anticipate confusion arising from the
  conflict between our notation and the well-established notation for the
  \defn{commutator subgroup}, $[H,G] := \<\{hgh^{-1}g^{-1} \mid h\in H, g\in G\}\>$,
  which we will also have occasion to use.  However, we have found that context always
  makes clear which meaning is intended.  In any case, we often refer to ``the interval
  $[H,G]$'' or ``the commutator $[H,G]$.''} 

We define the following classes of lattices:
\begin{itemize}
\index{$\mathscr{L}_4$}%
\item $\sL_4 = $ the class of lattices isomorphic to intervals
in subgroup lattices of finite groups;
\index{$\mathscr{L}_5$}%
\item  $\sL_5 = $  the class of lattices isomorphic to subgroup lattices of finite groups.
\end{itemize}
Recall that $\sL_3$, the class of all lattices isomorphic to congruence lattices
of finite algebras, is known as the class of \emph{representable}
lattices. We adhere to this convention throughout and, moreover, we will call a
lattice \defn{group representable} if it belongs to $\sL_4$.

Clearly,  $\sL_4 \supseteq \sL_5$, since $\Sub(G)$ is itself the interval 
$[1, G]$.  Moreover, it's easy to find a lattice that is in $\sL_4$ but not it
$\sL_5$, so the inclusion is strict.  For example, there is no group $G$ for which
$\Sub(G)$ is isomorphic to the lattice shown below.  
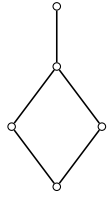
\begin{center}
\begin{tikzpicture}[scale=0.8]
\node (0) at (0,0) [draw, circle,inner sep=1pt] {};
\node (1) at (0.75,1) [draw, circle, inner sep=1pt] {};
\node (2) at (-0.75,1) [draw, circle, inner sep=1pt] {};
\node (3) at (-0,2) [draw, circle, inner sep=1pt] {};
\node (4) at (-0,3) [draw, circle, inner sep=1pt] {};
\draw[semithick]
(0) to (1)
(0) to (2)
(1) to (3)
(2) to (3)
(3) to (4);
\end{tikzpicture}

\end{center}
To see this, note that if
$G$ has a unique maximal subgroup $H$, then there exists $g\in G\setminus
H$ and we must have $\< g\> = G$.  Thus, if $\Sub(G)$ has a unique
coatom, then $G$ is cyclic, and subgroup lattices of cyclic groups are
self-dual, unlike the lattice shown above.
However, this lattice belongs to $\sL_4$.  For example, it is the 
filter above $H = C_3$ in the subgroup lattice of $G = C_3 \times (C_3 \rtimes
C_4)$. 

We will have a lot more to say about intervals in subgroup lattices throughout
this thesis.
Perhaps the most useful fact for our work is the following:
\begin{equation}
  \label{eq:Intro1}
\text{\emph{Every interval in a subgroup lattice is the congruence lattice of
  a finite algebra.}}
\end{equation}
In particular, as we explain below in Chapter~\ref{cha:congr-latt-group}, if 
$\<G/H, G\>$ is the algebra consisting of the group $G$
acting on the left (right) cosets of a subgroup $H \leq G$ by left (right)
multiplication, then $\Con\<G/H, G\> \cong [H, G]$.
Thus, we see that $\sL_3 \supseteq \sL_4$.

Whether the converse of (\ref{eq:Intro1}) holds -- and thus whether 
$\sL_3 = \sL_4$ -- is an open question.  In other words, it is not known whether
every congruence lattice of a finite algebra is isomorphic to an interval in the
subgroup lattice of a finite group.
However, a surprising and deep result related to this question was proved in
1980 by 
\index{P\'alfy, P\'eter}%
\index{Pudl\'ak, Pavel}%
P\'eter \Palfy\ and Pavel \Pudlak.
In~\cite{Palfy:1980}, they prove 
\begin{theorem}
\label{thm:IntroP5}
The following statements are equivalent:
\begin{enumerate}[(i)]
\item Every finite lattice is isomorphic to
  the congruence lattice of a finite algebra.
\item Every finite lattice is isomorphic to the congruence lattice of a finite transitive G-set.
\end{enumerate}
\end{theorem}
As we will see later (Theorem~\ref{thm:g-set-isomorphism2}), statement (ii) is equivalent to
\\[4pt]
{\it (ii)${}'$ Every finite lattice is isomorphic to an interval in the subgroup lattice of a finite group.}

It is important to note that Theorem~\ref{thm:IntroP5} does \emph{not} say
$\sL_3 = \sL_4$.  Rather, it says that $\sL_0 = \sL_3$ if and only if
$\sL_0=\sL_4$.  Moreover, this result implies that if we prove the existence of a lattice
which is not isomorphic to an interval in a subgroup lattice of a finite group,
then we have solved the \FLRP. 

It is surprising that a problem about general algebras can be reduced to
a problem about such a special class of algebras -- finite transitive
$G$-sets.  Also surprising, in view of all that we know about 
finite groups and their actions, is that we have
yet to determine whether these statements are true or false.
To put it another way, given an arbitrary finite lattice $L$, 
it is unknown whether there must be a finite group having this lattice as an
interval in its lattice of subgroups.  

\index{P\'alfy, P\'eter}%
We pause for a moment to consider the $\sL_3 = \sL_4$ question in the restricted 
case of finite distributive lattices (which we know are strongly
representable).  
Silcock~\cite{Silcock:1977} and \Palfy~\cite{Palfy:1987} prove that 
every finite distributive lattice is an interval in the subgroup lattice of some
finite solvable group.  The main result is stated below as
Theorem~\ref{thm:diag-normals}, and this can be combined with the
following easy lemma to establish the claim. 
\begin{lemma}
\label{lem:diag-normals}
If $D = \{(g,g) \in G \times G \mid g\in G\}$ then the interval $[D, G \times G]$ is isomorphic to the
lattice of normal subgroups of $G$.
\end{lemma}
\begin{theorem}
\label{thm:diag-normals}
Every finite distributive lattice is isomorphic to the lattice of normal subgroups
of a finite solvable group.
\end{theorem}

Beyond those mentioned in this brief introduction, many other results
surrounding the \FLRP\ have been proven.  Some of these are not as relevant to
our work, and others will be discussed in detail in
Chapter~\ref{cha:an-overview-finite}. A more complete overview of
the \FLRP\ with an emphasis on group theory can be found in the articles by
\Palfy, \cite{Palfy:1995} and~\cite{Palfy:2001}.

\chapter{An Overview of Finite Lattice Representations}
\label{cha:an-overview-finite}
In this chapter we give a brief overview of various known methods for
representing a given lattice as the congruence lattice of a finite algebra or
proving that such a representation exists.
In later chapters we describe these methods in greater detail and show how to
apply them.  In particular, in Section~\ref{sec:seven-elem-latt}, we use them
along with some new methods to show that, with one possible exception, every
lattice with no more than seven elements is isomorphic to the congruence
lattice of a finite algebra. Throughout this chapter, we continue to use $\sL_3$
to denote the class of finite lattices that are isomorphic to congruence lattices of
finite algebras.  Again, we call the lattices that belong to $\sL_3$ \emph{representable}
lattices. 

\section{Closure properties of the class of representable lattices}
\label{sec:clos-prop-class}
This section concerns 
\defn{closure properties} 
of the class $\sL_3$. 
More precisely, if $\bO$ is an operation that can be applied to a lattice or
collection of lattices, we say that $\sL_3$ is \emph{closed under $\bO$} provided
$\bO(\sK) \subseteq \sL_3$ for all 
$\sK\subseteq \sL_3$. For example, if $\bS(\sK) = \{\text{all sublattices of
  lattices in $\sK$}\}$, then it is clearly unknown whether $\sL_3$ is closed under
$\bS$, for otherwise the \FLRP\ would be solved.
(Clearly, $\Eq(X)\in \sL_3$ for every finite set $X$ --
take the algebra to be the set $X$ with no operations. 
Then $\Con\<X, \emptyset\> = \Eq(X)$.  So, if $\sL_3$ were
closed under $\bS$, then $\sL_3$ would contain all finite lattices, by the
\index{Pudl\'ak, Pavel}%
\indexTuma%
result of \Pudlak\ and \Tuma\ mentioned above; that is, $\sL_0=\sL_1$.) 

The following is a list of known closure properties of $\sL_3$ and the names of those
who first (or independently) proved them.  We discuss some of these results in greater
detail later in this section.
The class $\sL_3$ of lattices isomorphic to congruence lattices of finite
algebras is closed under 
      \begin{enumerate} 
      \item lattice duals\footnote{Recall, the \defn{dual of a lattice} is simply the
            lattice turned on its head, that is, the lattice obtained by
            reversing the partial order of the original lattice.}  (Hans Kurzweil~\cite{Kurzweil:1985} and
        Raimund Netter~\cite{Netter:1986}, 1986),
        \index{Kurzweil, Hans}        \index{Netter, Raimund}
      \item interval sublattices (follows from Kurzweil-Netter),
      \item  direct products (\Jiri\ \Tuma~\cite{Tuma:1986}, 1986), \indexTuma%
      \item  ordinal sums \index{ordinal sum}
        \index{McKenzie, Ralph}
        (Ralph McKenzie~\cite{McKenzie:1984}, 1984; John Snow~\cite{Snow:2000}, 2000),
      \item  parallel sums (John Snow~\cite{Snow:2000}, 2000), \index{parallel sum}
      \item certain sublattices of lattices in $\sL_3$ -- namely, those which
        are obtained as a union of a filter and an ideal of a lattice in
       $\sL_3$ (John Snow~\cite{Snow:2000}, 2000). \index{Snow, John}
      \end{enumerate}
\begin{center}
        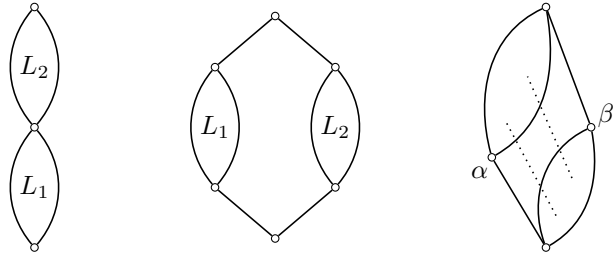
\begin{figure}[h!]
          \label{fig:ordinal-and-parallel}
\centering
  \begin{tikzpicture}[scale=0.4]

    \node (botL1) at (-8,-4) [draw, circle,inner sep=1pt] {};
    \node (m80) at (-8,0) [draw, circle,inner sep=1pt] {};
    \node (topL2) at (-8,4) [draw,circle,inner sep=1pt] {};
    \node (bot) at (0,-3.7) [draw, circle,inner sep=1pt] {};
    \node (top) at (0,3.7) [draw,circle,inner sep=1pt] {};
    \node (botL) at (-2,-2) [draw,circle,inner sep=1pt] {};
    \node (topL) at (-2,2) [draw,circle,inner sep=1pt] {};
    \node (botN) at (2,-2) [draw,circle,inner sep=1pt] {};
    \node (topN) at (2,2) [draw,circle,inner sep=1pt] {};

    \draw (-8,-2) node {$L_1$};
    \draw (-8,2) node {$L_2$};
    \draw (-2,0) node {$L_1$};
    \draw (2,0) node {$L_2$};

    \draw[semithick] 
    (bot) to (botL)
    (bot) to (botN)
    (top) to (topL)
    (top) to (topN);

    \draw [semithick]  
    (botL) to [out=50,in=-50] (topL)
    (botL) to [out=130,in=-130] (topL)
    (botN) to [out=50,in=-50] (topN)
    (botN) to [out=130,in=-130] (topN);

    \draw [semithick]  
    (botL1) to [out=50,in=-50] (m80)
    (botL1) to [out=130,in=-130] (m80)
    (m80) to [out=50,in=-50] (topL2)
    (m80) to [out=130,in=-130] (topL2);

          \node (c) at (9.5,-3.25)   {};
          \node (d) at (7.5,0.5)   {};
          \node (e) at (10,-2)   {};
          \node (f) at (8.25,2)   {};
          \node (bottom) at (9,-4)  [draw, circle, inner sep=1pt] {};
          \node (topfi) at (9,4)  [draw, circle, inner sep=1pt] {};
          \node (alpha) at (7.2,-1)  [draw, circle, inner sep=1pt] {};
          \node (beta) at (10.5,0)  [draw, circle, inner sep=1pt] {};
          \draw[semithick, dotted] (c) to (d) (e) to (f);
          \draw[semithick] (bottom) to (alpha)  (beta) to (topfi);
          \draw[semithick] 
          (alpha) to [out=30, in=-80] (topfi)
          (topfi) to [out=205, in=105] (alpha)
          (bottom) to [out=30, in=-80] (beta)
          (beta) to [out=205, in=110] (bottom);
          \draw (6.8,-1.5) node {$\alpha$};
          \draw (11,0.4) node {$\beta$};

  \end{tikzpicture}
          \caption{The ordinal (left) and parallel (middle) sum of the lattices
            $L_1$ and $L_2$; a sublattice obtained as a union of a filter $\alpha^\uparrow $ and an ideal
            $\beta^\downarrow$ (right).}
        \end{figure}
\end{center}

      \begin{remarks}
        ~
      \begin{enumerate} 
      \item  The first result says that if $L$ is representable then so is the
        dual of $L$. \index{dual}
        \item It follows from item 1.~that any interval sublattice of a representable lattice is
        representable.  For, let $[\alpha, \beta] := \{\theta \in L \mid \alpha \leq \theta \leq \beta\}$ be an
        interval in the representable lattice $L= \Con\bA$.  Then $[\alpha, 1_A] \cong \Con
        \bA/\alpha$. By 1., the dual of $\ell := [\alpha, 1_A]$ is
        representable. 
        Now take the filter above $\beta'$ in $\ell'$ (where $\beta'$ is the
        image of $\beta$ under dualization) and we obtain a representation of a
        lattice isomorphic to the dual of $[\alpha, \beta]$.  Apply 1.
        again and we have the desired representation of $[\alpha, \beta]$.

      \item Of course, by direct products we mean \emph{finite} direct products.
      \item[4.-5.] By the
        \index{ordinal sum}  \index{parallel sum}
        ordinal (parallel) sum of two lattices $L_1, L_2$, we mean the lattice
        on the left (middle) of Figure~\ref{fig:ordinal-and-parallel}.

\item[6.] The property in item 6.~is very useful and we discuss it further in
  Section~\ref{sec:union-filter-ideal} below, where we present a very short proof of
  this result.  It will come up again in
  Section~\ref{cha:lattices-with-at} when we prove the existence of
  representations of small lattices.
      \end{enumerate}
      Whether the class $\sL_3$ is closed under homomorphic images
      seems to be an open question. 
      \end{remarks}

\section{Lattice duals: the theorem of Kurzweil and Netter}
\label{sec:duals-interv-subl-detail}
As mentioned above, 
the class $\sL_3$ -- the lattices isomorphic to congruence lattices of finite
algebras -- is closed under
dualization.
That is, if $L$ is representable, then so is the dual of $L$. This was proved in
\index{Netter, Raimund}%
1986 by Raimund Netter~\cite{Netter:1986}, generalizing the idea of his advisor,
\index{Kurzweil, Hans}%
Hans Kurzweil~\cite{Kurzweil:1985}. 
Though Kurzweil's article did appear (in German), it is unclear whether Netter's
article was ever published.
In this section we present a proof of their result.
The argument requires a fair bit of machinery, but it is a nice idea and
well worth the effort.\footnote{We learned 
  of the main argument used in the proof from slides of a series of three
\index{P\'alfy, P\'eter}%
  lectures given by P{\'e}ter \Palfy\ in 2009~\cite{Palfy:2009}.
  \Palfy\ gives credit for the argument to Kurzweil and Netter.} 

If $G$ is a group and $X$ a set, then the set $\{f \mid X\rightarrow G\}$ of 
functions from $X$ into $G$ is denoted by $G^X$.  This is a group with binary
operation $(f,g) \mapsto f\cdot g$, where,  
for each $x\in X$, $(f\cdot g)(x)= f(x)g(x)$ is simply multiplication
in the group $G$.  The identity of the group $G^X$ is of course the constant map $f(x) =
1_G$ for all $x\in X$.

Let $X$ be a finite totally ordered set, with order relation $\leq$,
and consider the set $X^X$ of functions mapping $X$ into itself.  
The subset of $X^X$ consisting of functions that are both idempotent and
decreasing\footnote{When we say that the map $f$ is \emph{decreasing} we mean
  $f(x)\leq x$ for all $x$. (We do not mean $x\leq y$ implies $f(y) \leq x$.)}
 will be denoted by $\ID{X}$.  That is,
\[
\ID{X} = \{f\in X^X \mid f^2 = f \text{ and }\; \forall x\; f(x) \leq x\}.
\]
Define a partial order $\sqsubseteq$ on the set $\ID{X}$ by
\begin{equation}
  \label{eq:MID111}
 f\sqsubseteq g \quad \Leftrightarrow \quad \ker f \leq \ker g,
\end{equation}
where $\ker f = \{(x,y) \mid f(x) = f(y)\}$.
It is easy to see that $f\sqsubseteq g$ holds if and only if $gf = g$.  
Moreover, under this partial ordering $\ID{X}$ is a lattice which is 
isomorphic to $\bEqX$ (viz.~the map $\Theta : \EqX \rightarrow
\ID{X}$ given by $\Theta(\alpha) = f_\alpha$, where
$f_\alpha(x) = \min\{y\in X \mid (x,y)\in \alpha\}$.) 

\renewcommand{\bn}{\ensuremath{n}}  
\renewcommand{\5}{\ensuremath{5}}  

Suppose $S$ is a finite nonabelian simple
group, and consider $S^\bn$, the direct power of $n$ copies of $S$.
An element of $S^\bn$ may be viewed as a map from the set 
$\bn = \{0, 1, \dots, n-1\}$ into $S$.  Thus, if 
$x = (x_0, x_1, \dots, x_{n-1})\in S^n$, then by 
$\ker x$ we mean the relation $(i,j) \in \ker x$ if and only if $x_i = x_j$.
The set of constant maps is a subgroup $D < S^\bn$, sometimes called the
\defn{diagonal subgroup}; that is,
$D = \{(s, s, \dots, s) \mid s\in S\} \leq S^\bn$.

For each $f \in \ID{\bn}$, define
\[
K_f = \{(x_{f(0)}, x_{f(1)}, \dots, x_{f(n-1)}) \mid x_{f(i)}\in S, \; i = 0, 1,
\dots, n-1\}.
\]
Then $D \leq K_f\leq S^\bn$, and $K_f$ is the set of maps
$K_f = \{x f \in S^\bn \mid x \in S^\bn \}$; i.e., compositions of the given
map $f\in n^n$, followed by  any $x\in S^n$.  Thus, 
$K_f = \{ y\in S^n \mid \ker f \leq \ker y \}$.
For example, 
if $f = (0, 0, 2, 3, 2)\in \ID{\5}$, then 
$\ker f = |0,1|2,4|3|$ and 
$K_f$ is the subgroup 
of all $(y_0, y_1, \dots, y_4)\in S^\5$ having $y_0 = y_1$ and $y_2 = y_4$. That is,
$K_f = \{(x_{0}, x_{0}, x_2, x_3, x_2) \mid x\in S^5\}$.

\begin{lemma}
\label{lem:latt-duals}
  The map $f \mapsto K_f$ is a dual lattice isomorphism from $\bEq(\bn)$ onto the
  interval sublattice $[D, S^\bn] \leq \Sub(S^\bn)$.
\end{lemma}
\begin{proof}
This is clear since $\ID{\bn}$ is ordered by (\ref{eq:MID111}), and 
we have 
$f\sqsubseteq h$ if and only if
$K_h = \{y \in S^\bn \mid \ker h \leq \ker y \}
\leq \{y \in S^\bn \mid \ker f \leq \ker y \} =  K_f$.
\end{proof}

\index{Kurzweil-Netter Theorem}%
\begin{theorem}[Kurzweil~\cite{Kurzweil:1985}, Netter~\cite{Netter:1986}]
\label{thm:duals-interv-subl}
  If the finite lattice $L$ is representable (as the congruence lattice of a
  finite algebra), then so is the dual lattice $L'$.
\end{theorem}
\begin{proof}
  Without loss of generality, we assume that $L$ is concretely represented
  as $L = \Con \<\bn, F\>$.
  By Lemma~\ref{sec:unarycongruences},  we can further
  assume that $F$ consists of unary operations: $F \subseteq \bn^\bn$.
  As above, let $S$ be a nonabelian simple group
  and let $D$ be the diagonal subgroup of $S^\bn$.
  Then the unary algebra $\<S^\bn/D, S^\bn\>$  is a transitive $S^\bn$-set which (by
  Theorem~\ref{thm:g-set-isomorphism2} below) has congruence lattice isomorphic
  to the interval $[D, S^\bn]$.  By Lemma~\ref{lem:latt-duals}, this is the dual
  of the lattice $\bEq(\bn)$.  That is, 
  $\Con \<S^\bn/D, S^\bn\> \cong (\bEq(\bn))'$.
  
  Now, each operation $\phi \in F$ gives rise to an operation on $S^\bn$
  by composition:
  \[
  \hat{\phi}(\bs) = \hat{\phi}(s_{0}, s_1  \dots, s_{n-1}) = (s_{\phi(0)},
  s_{\phi(1)}\dots, s_{\phi(n-1)}). 
  \]
  Thus, $\phi$ induces  an operation on $S^\bn/D$ since, for 
  $\bd = (d, d, \dots, d) \in D$ and $\bs \in S^\bn$ we have 
  $\bs \bd = (s_{0}d, s_{1}d, \dots , s_{n-1}d)$ and 
  $\hat{\phi}(\bs \bd) = (s_{\phi(0)}d, s_{\phi(1)}d, \dots , s_{\phi(n-1)}d) = \hat{\phi}(\bs) \bd$,
  so $\hat{\phi}(\bs D)  = \hat{\phi}(\bs) D$.  Finally, add the set of operations 
  $\hat{F} = \{\hat{\phi} \mid \phi \in F\}$ to $\<S^\bn/D, S^\bn\>$, yielding the
  new algebra  $\<S^\bn/D, S^\bn \cup \hat{F}\>$, and observe
  that a congruence $\theta \in \Con\<S^\bn/D, S^\bn\>$ remains a congruence of
  $\<S^\bn/D, S^\bn \cup \hat{F}\>$ if and only if it correponds to a partition on
  $\bn$ that is invariant under $F$.
\end{proof}

\todo{Perhaps we should give more details in the last sentence of the proof.
  Some notes are below, but they need to be cleaned up.}


\renewcommand{\bn}{\ensuremath{\mathbf{n}}}
\renewcommand{\5}{\ensuremath{\mathbf{5}}}

\section{Union of a filter and ideal}
\label{sec:union-filter-ideal}
The lemma in this section was originally proved by John Snow using primitive positive
formulas.  Since it provides such a useful tool for proving that certain finite lattices 
are representable as congruence lattices, we give our own direct
proof of the result below.  In Chapter~\ref{cha:lattices-with-at} we use this lemma to prove
the existence of representations of a number of small lattices.

Before stating the lemma, we need a couple of definitions.  (These will be
discussed in greater detail in Section~\ref{sec:closure-method}.)
Given a relation $\theta \subseteq X\times X$, we say that the map 
$f: X^n\rightarrow X$ \emph{respects} $\theta$ and we write 
$f(\theta) \subseteq \theta$ provided $(x_i, y_i)\in \theta$ implies
$(f(x_1, \dots, x_n), f(y_1, \dots, y_n))\in \theta$.
For a set $L\subseteq \Eq(X)$ of equivalence relations we define
      \[
      \lambda(L) = \{f\in X^X: (\forall \theta \in L) \; f(\theta) \subseteq \theta \},
      \]
which is the set of all unary maps on $X$ which respect all relations in $L$.
\begin{lemma} 
\label{lemma:union-filter-ideal}
Let $X$ be a finite set.
  If $\bL \leq \bEqX$ is representable and $\bL_0\leq \bL$ is a sublattice with universe
  $\upalpha\cup \downbeta$ where $\upalpha=\{x\in L \mid \alpha \leq x\}$ and 
$\downbeta=\{x\in L \mid x\leq \beta\}$ for some $\alpha, \beta \in L$, then $\bL_0$ is representable.
\end{lemma}

\vskip3mm

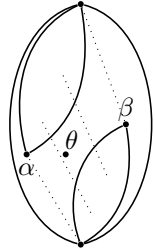
      \begin{center}
        \begin{tikzpicture}[scale=.4]
          \node (c) at (.5,.75)   {};
          \node (d) at (-1.5,4.5)   {};
          \node (e) at (1,2)   {};
          \node (f) at (-.75,6)   {};
          \node (bottom) at (0,0)  [fill, circle, inner sep=.8pt] {};
          \node (top) at (0,8)  [fill, circle, inner sep=.8pt] {};
          \node (alpha) at (-1.8,3)  [fill, circle, inner sep=.8pt] {};
          \node (beta) at (1.5,4)  [fill, circle, inner sep=.8pt] {};
          \node (theta) at (-.5,3)  [fill, circle, inner sep=.8pt] {};
          \draw (-.3,3.5) node {$\theta$};
          \draw[semithick] 
          (bottom) to [out=15, in=-15] (top) 
          (top) to [out=195, in=165] (bottom);
          \draw[dotted] (c) to (d) (e) to (f);
          \draw[dotted] (bottom) to (alpha)  (beta) to (top);
          \draw[semithick] 
          (alpha) to [out=30, in=-80] (top)
          (top) to [out=205, in=105] (alpha)
          (bottom) to [out=30, in=-80] (beta)
          (beta) to [out=205, in=110] (bottom);
          \draw (0,-1.5) node {$L_0 \leq  L$};
          \draw (-1.8,2.5) node {$\alpha$};
          \draw (1.5,4.5) node {$\beta$};
        \end{tikzpicture}
      \end{center}

\vskip3mm

\begin{proof}
Assume $\bL_0 \ncong \two$, otherwise the result holds trivially. 
Since $\bL\leq \bEqX$ is representable, we have $\bL = \bCon
\<X, \lambda(L)\>$ (cf.~Section~\ref{sec:closure-method}).  Take an arbitrary
$\theta \in L \setminus L_0$. Since $\theta \notin \upalpha$, 
there is a pair 
$(a,b) \in \alpha \setminus \theta$.  Since $\theta \notin \downbeta$, there is
a pair $(u,v)\in \theta\setminus \beta$. Define $h\in X^X$ as follows:
\begin{equation*}
h(x) = \begin{cases}
a,& \quad x\in u/\beta,\\
b,& \quad \text{ otherwise.}
\end{cases}
\end{equation*}
Then, $\beta \leq \ker h = (u/\beta)^2 \cup ((u/\beta)^c)^2$, where $(u/\beta)^c$ denotes the
complement of the $\beta$ class containing $u$.  Therefore, $h$ respects every
$\gamma \leq \beta$.  Furthermore, $(a, b) \in \gamma$ for all $\gamma \geq \alpha$,
so $h$ respects every $\gamma$ above $\alpha$.  This proves that $h\in \lambda(L_0)$.
Now, $\theta$ was arbitrary, so we have proved that for every $\theta \in L
\setminus L_0$ there exists a function in $\lambda(L_0)$ which respects every
$\gamma \in \upalpha\cup \downbeta = L_0$, but violates $\theta$.  Finally,
since 
$\bL_0 \leq \bL$, we have $\lambda(L)\subseteq \lambda(L_0)$.  Combining these
observations, we see that every $\theta \in \Eq(X) \setminus L_0$ is
violated by some function in $\lambda(L_0)$. Therefore, $\bL_0 = \bCon \< X, \lambda(L_0)\>$.
\end{proof}

\section{Ordinal sums}
\label{sec:ordinal-sums}
The following theorem is a consequence of 
\index{McKenzie, Ralph}%
McKenzie's shift product construction~\cite{McKenzie:1984}. 
\todo{Possibly add a short description of the shift product.}
\index{adjoined ordinal sum}%
\index{ordinal sum}%
\begin{theorem}
\label{thm:ordinal-sums}
If $L_1, \dots, L_n \in \sL_3$ is a collection of representable lattices, then
the ordinal sum and the adjoined ordinal sum, shown in
Figure~\ref{fig:adjordinal}, are representable.
\end{theorem}
A more direct proof of Theorem~\ref{thm:ordinal-sums} follows the argument given
\index{Snow, John}%
by John Snow in~\cite{Snow:2000}.  As noted above, 
\indexTuma%
\Jiri\ \Tuma\ proved that
the class of finite representable lattices is closed under direct products.
Thus, if $L_1$ and 
$L_2$ are representable, then so is $L_1 \times L_2$.  Now note that the
adjoined ordinal sum of $L_1$ and $L_2$ is the union, $\alpha^\uparrow \cup
\beta^\downarrow$,  of a filter and ideal  
 in the lattice $L_1 \times L_2$, where
$\alpha = \beta = 1_{L_1} \times 0_{L_2}$.  
Therefore, by Lemma~\ref{lemma:union-filter-ideal},
the adjoined ordinal sum is representable.  A trivial induction argument proves the
result for adjoined ordinal sums of $n$ lattices.  The same result for ordinal
sums (Figure~\ref{fig:adjordinal} left) follows since the two element lattice is
obviously representable. 

\begin{center}
  \begin{figure}[h!]
    \label{fig:adjordinal}
    \centering
{\scalefont{.8}
    \begin{tikzpicture}[scale=0.3]
      \node (botL1) at (8,-4) [fill, circle,inner sep=.6pt] {};
      \node (00) at (8,0) [fill, circle,inner sep=.6pt] {};
      \node (topL2) at (8,4) [fill,circle,inner sep=.6pt] {};
      \node (topLn) at (8,10) [fill, circle,inner sep=.6pt] {};
      \node (botLn) at (8,6) [fill, circle,inner sep=.6pt] {};

      \draw (8,-2) node {$L_1$};
      \draw (8,8) node {$L_n$};
      \draw (8,2) node {$L_2$};
      \draw (8,5.5) node {$\vdots$};

    \draw 
    (botL1) to [out=50,in=-50] (00)
    (botL1) to [out=130,in=-130] (00)
    (00) to [out=50,in=-50] (topL2)
    (00) to [out=130,in=-130] (topL2)
    (botLn) to [out=50,in=-50] (topLn)
    (botLn) to [out=130,in=-130] (topLn);

      \node (botL1) at (0,-6) [fill, circle,inner sep=.6pt] {};
      \node (topL1) at (0,-2) [fill, circle,inner sep=.6pt] {};
      \node (botL2) at (0,-1) [fill, circle,inner sep=.6pt] {};
      \node (topL2) at (0,3) [fill,circle,inner sep=.6pt] {};
      \node (04) at (0,4) [fill, circle,inner sep=.6pt] {};
      \node (06) at (0,6) [fill, circle,inner sep=.6pt] {};
      \node (botLn) at (0,7) [fill, circle,inner sep=.6pt] {};
      \node (topLn) at (0,11) [fill, circle,inner sep=.6pt] {};

      \draw (0,-4) node {$L_1$};
      \draw (0,9) node {$L_n$};
      \draw (0,1) node {$L_2$};
      \draw (0,5.5) node {$\vdots$};

    \draw 
    (botL1) to [out=50,in=-50] (topL1)
    (botL1) to [out=130,in=-130] (topL1) 
    to (botL2) to [out=50,in=-50] (topL2)
    (botL2) to [out=130,in=-130] (topL2)
    (topL2) to (04)
    (06) to (botLn) to [out=50,in=-50] (topLn)
    (botLn) to [out=130,in=-130] (topLn);
  \end{tikzpicture}
}
          \caption{The ordinal sum (left) and the adjoined ordinal sum (right) of the lattices
            $L_1, \dots, L_n$.}
        \end{figure}
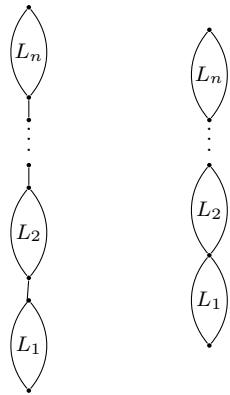
\end{center}


\part{Finite Lattice Representations}

\chapter{Concrete Representations}
\label{cha:concr-repr}


In this chapter we introduce a strategy that has proven very useful for showing
that a given lattice is representable as a congruence lattice of a finite
algebra. We call it the \defn{closure method}, and it has become especially
useful with the advent of powerful computers which can search for such
representations.  Here, as above, $\Eq(X)$ denotes the lattice of equivalence
relations on $X$. Sometimes we abuse notation and take $\Eq(X)$ to mean the
lattice of partitions of the set $X$. This has never caused problems because these
two lattices are isomorphic.  

\section{Concrete versus abstract representations}
\index{J\'onsson, Bjarni}
As Bjarni J\'onsson explains in~\cite{Jonsson:1972}, there are two types of
representation problems for congruence lattices, the concrete and the
abstract.  The \emph{concrete representation problem} asks whether a specific family of
equivalence relations on a set $A$ is equal to $\Con \bA$ for some
algebra $\bA$ with universe $A$.  The \defn{abstract representation problem}
asks whether a given lattice is isomorphic to $\Con \bA$ for some algebra $\bA$.

These two problems are closely related, and have become even more so since the
publication in 1980 
\indexPudlak
of~\cite{Pudlak:1980}, in which Pavel \Pudlak\ and
\indexTuma
\Jiri\ \Tuma\ prove that every finite
lattice can be embedded as a 
spanning sublattice\footnote{Recall, by a 
\defn{spanning sublattice}
of a bounded lattice $L_0$, we mean a sublattice $L\leq L_0$ that has the same top and
bottom as $L_0$.  That is  $1_L = 1_{L_0}$ and $0_L = 0_{L_0}$.}
of the lattice $\Eq(X)$ of equivalence relations on a finite set $X$.   
Given this result, we see that even if our goal is to solve the abstract
representation problem for some (abstract) lattice $L$, then
we can embed $L$ into $\Eq(X)$ as $L\cong L_0\leq \Eq(X)$, for some finite set
$X$, and then try to solve the concrete representation problem for $L_0$.  

A point of clarification is in order here.  The term 
\defn{representation} 
has become a bit overused in the literature about the finite lattice
representation problem.  On the one hand, given a finite lattice $L$, if there
is a finite algebra $\bA$ such that $L \cong \Con \bA$, then $L$ is called a
\defn{representable lattice}.  On the other hand, given a sublattice $L_0\leq \Eq(X)$, 
if $L_0\cong L$, then $L_0$ is sometimes called a 
\defn{concrete representation}
of the lattice $L$ (whether or not it is the congruence lattice of an algebra).    
Below we will define the notion of a \defn{closed concrete representation}, and if we
have this special kind of concrete representation of a give lattice, then that
lattice is indeed representable in the first sense.

As we will see below, there are many examples in which a particular concrete
representation $L_0\leq \Eq(X)$ of $L$ is not a congruence lattice of a
finite algebra.  (In fact, we will describe general situations in which we can
guarantee that there are no non-trivial\footnote{By a 
\defn{non-trivial function} we mean a function that is
  not constant and not the identity.} operations which respect the equivalence
relations of $L_0$.)  This does not imply that $L \notin \sL_3$.  It may
simply mean that $L_0$ is not the ``right'' concrete representation of $L$, and
perhaps we can find some other $L \cong L_1\leq \Eq(X)$ such that $L_1 = \Con
\<X, \lambda(L_1)\>$.

\section{The closure method}
\label{sec:closure-method}
The idea described in this section
first appeared in \emph{Topics in Universal Algebra}~\cite{Jonsson:1972}, pages
174--175, where J\'onsson states, ``these or related results were discovered
independently by at least three different parties during the summer and fall of
1970: by Stanley Burris, Henry Crapo, Alan Day, Dennis Higgs and Warren Nickols
at the University of Waterloo, by R.~Quackenbush and B.~Wolk at the University
of Manitoba, and by B.~J\'{o}nsson at Vanderbilt University.''

Let $X^X$ denote the set of all (unary) maps from the set $X$ to itself, and let 
$\Eq(X)$ denote the lattice of equivalence relations on the set $X$.  If $\theta
\in \Eq(X)$ and $h\in X^X$, we write $h(\theta) \subseteq \theta$ and say
that ``$h$ respects $\theta$'' if and only if for all $(x,y)\in X^2$ $(x,y)\in
\theta$ implies 
$(h(x),h(y)) \in \theta$.  If $h(\theta) \nsubseteq \theta$, we sometimes say
that ``$h$ violates $\theta$.''

For $L\subseteq \Eq(X)$ define
      \[
      \lambda(L) = \{h\in X^X: (\forall \theta \in L) \; h(\theta) \subseteq \theta \}.
      \]

For $H\subseteq X^X$ define
      \[
      \rho(H) = \{\theta \in \Eq(X) \mid   (\forall h\in H) \; h(\theta) \subseteq \theta\}.
      \]
The map $\rho \lambda$ is a \defn{closure operator} on $\Sub[\Eq(X)]$.
That is, $\rho \lambda$ is
\begin{itemize}
\item \emph{idempotent:}\footnote{In fact, $\rho \lambda \rho = \rho$ and 
$\lambda \rho \lambda = \lambda$.} $\rho \lambda \rho \lambda = \rho \lambda$;
\item \emph{extensive:} $L \subseteq \rho \lambda (L)$ for every $L \leq \Eq(X)$;
\item \emph{order preserving:} $\rho \lambda (L) \leq \rho \lambda (L_0)$ if $L \leq L_0$.
\end{itemize}
Given $L\leq \Eq(X)$, if $\rho\lambda(L) = L$, then we say $L$ is a 
\index{closed sublattice}%
\emph{closed} sublattice of $\Eq(X)$, in which case we clearly have
      \[L = \Con \<X, \lambda(L)\>.\]
This suggests the following strategy for solving the representation problem for a
given abstract finite lattice $L$: search for a concrete representation $L \cong
L_0\leq \Eq(X)$,
compute $\lambda(L_0)$, compute $\rho\lambda(L_0)$, and determine whether 
$\rho\lambda(L_0) = L_0$.  If so, then we have solved the abstract representation
problem for $L$, by finding a \defn{closed concrete representation}, or simply
\emph{closed representation}, of $L_0$.  We call this strategy the \defn{closure method}.

We now state without proof a well known theorem which shows that the finite lattice
  representation problem can be formulated in terms of closed concrete
representations (cf.~\cite{Jonsson:1972}).
\begin{theorem}\label{Concrete-thm-3}
If $\bL \leq \bEqX$, then $\bL= \bCon\bA$ for some algebra 
$\mathbf{A} = \langle X, F\rangle$ if and only if $\bL$ is closed.
\end{theorem}

In the remaining sections of this chapter, we consider various aspects of the
closure method and prove some results about it.  Later, in
Section~\ref{sec:seven-elem-latt}, we apply it to the problem of finding
closed representations of all lattices of small order.  
Before proceeding, however, we introduce a slightly different set-up than the
one introduced above that we have found particularly useful
for implementing the closure method on a computer. Instead of considering the
set of equivalence relations on a finite set, we work with the set of idempotent
decreasing maps.  These were introduced above in
Section~\ref{sec:duals-interv-subl-detail}, but we briefly review the definitions here
for convenience.

Given a totally ordered set $X$, 
let the set $\idemdec = \{f\in X^X: f^2 = f \text{ and } f(x) \leq x\}$ be partially
ordered by $\sqsubseteq$ as follows:
\[
 f\sqsubseteq g \quad \Leftrightarrow \quad \ker f \leq \ker g.  
\]
As noted above, 
this makes \idemdec\ into a lattice that is isomorphic to $\bEqX$.   
Define a relation $R$ on $X^X \times \idemdec$ as follows: 
\[
(h,f) \in R \quad \Leftrightarrow \quad (\forall (x,y) \in \ker f)\; (h(x),h(y))
\in \ker f.
\]
If $h\, R\, f$, we say that  $h$ \emph{respects} $f$.

Let $\sF = \power{\idemdec}$ and $\sH = \power{X^X}$ be partially ordered by set
inclusion, and define the maps 
$\lambda: \sF \rightarrow \sH$ and $\rho: \sH \rightarrow \sF$ as follows:
\[
\lambda(F) = \{h\in X^X: \forall f\in F,\, h\, R \,f\} \quad (F \in \sF)
\]
\[
\rho(H) = \{f\in \idemdec: \forall h\in H,\, h\, R\, f\} \quad (H \in \sH)
\]
The pair $(\lambda, \rho)$ defines a \defn{Galois correspondence} between
$\idemdec$ and $X^X$.  That is, $\lambda$ and $\rho$ are 
antitone 
maps such that $\lambda \rho \geq \id_{\sH}$ and $\rho
\lambda \geq \id_{\sF}$.  In particular, for any set $F \in \sF$ we have 
$F \subseteq \rho \lambda (F)$.  These statements are all trivial verifications, and
a couple of easy consequences are:
\begin{enumerate}
\item  $\rho\lambda\rho = \rho$ and $\lambda\rho \lambda= \lambda$,
\item $\rho \lambda$ and $\lambda \rho$ are idempotent.
\end{enumerate}

Since the map $\rho \lambda$ from $\sF$ to itself is idempotent, extensive, 
and order preserving, it is a 
\defn{closure operator} 
on $\sF$, and we say a set $F\in \sF$ is
\emph{closed} if and only if $\rho\lambda(F) = F$. Equivalently,
$F$ is closed if and only if $F = \rho(H)$ for some $H\in \sH$.

\section{Superbad representations}
\label{sec:superbad}
In this section we describe what is in some sense the worst kind of concrete
representation. Given an abstract finite lattice $\bL$, it may happen that, upon
computing the closure of a particular representation $\bL\cong \bL_0 \leq
\bEqX$,  we find that $\rho\lambda (L_0)$ is all of $\Eq(X)$.  We call such an
$\bL_0$ a  \defn{dense sublattice} of $\bEqX$, or more colloquially, a
\defn{superbad representation} of $\bL$.    

More generally, if $A$ and $B$ are subsets of $\idemdec$, we say that $A$ is 
\defn{dense} in $B$ if and only if $\rho\lambda(A) \supseteq B$.  If $\bL$ is a
finite lattice and there exists an embedding $\bL\cong \bL_0\leq \bEqX$ such that
$\rho\lambda(L_0) = \EqX$,  we say 
 that $\bL$ can be \defn{densely embedded} in $\bEqX$.

\subsection{Density}
One of the first questions we asked concerned the 5-element modular lattice,
denoted $\bM_3$ (sometimes called the 
\defn{diamond}; see Figure~\ref{fig:diamond}).  
We asked for which sets $X$ does the lattice of equivalence relations on $X$ contain
a dense $\bM_3$ sublattice. 
The answer is given by
\begin{prop}\label{Concrete-prop-1}
The lattice \bEqX\ contains a proper dense $\bM_3$ sublattice if and only if $|X|\geq 5$.
\end{prop}

\begin{figure}[centering,h!]
\label{fig:diamond}
\begin{center}
\begin{tikzpicture}[scale=0.8]
\draw (0,-1) node {$\bM_3$};
  \node (m3) at (0,0) [fill,circle,inner sep=1.2pt] {};
  \node (m31) at (0,1.5) [fill,circle,inner sep=1.2pt] {};
  \node (m32) at (-1.5,1.5) [fill,circle,inner sep=1.2pt] {};
  \node (m33) at (1.5,1.5) [fill,circle,inner sep=1.2pt] {};
  \node (m34) at (0,3) [fill,circle,inner sep=1.2pt] {};
  \draw [semithick]  (m3) to (m31) to (m34) to (m33) to (m3) to (m32) to (m34);
\end{tikzpicture}
\end{center}
\caption{The 5-element non-distributive lattice, $\bM_3$.}
\end{figure}

This basically says that, when $|X|\geq 5$, the lattice of equivalences on $X$ contains a
spanning diamond $\bL$ with the property that 
every non-trivial operation in $X^X$ violates some equivalence relation in the
universe $L$ of $\bL$. Thus, the closure $\rho \lambda (L)$ is all of $\Eq(X)$.
John Snow proved this for $|X|$ odd. Using the same technique (and some rather
tedious calculations), we verified that the result holds for $|X|$ even as well.

Before moving on to the next result, we note that the necessity part of the
proposition above is obvious.  For, if $|X|\leq 2$, then $\bEqX$ has no $\bM_3$
sublattice.  If $|X|=3$, then $\bEqX$ is itself $\bM_3$.  It can be checked directly
(by computing all possibilities) that, when $|X| = 4$, $\bEqX$ has one closed $\bM_3$
sublattice and five $\bM_3$ sublattices that are neither closed nor dense. 

For ease of notation, let $\Eq(n)$ denote the set of equivalence relations on an
$n$-element set, and let $\bM_n$ denote the $(n+2)$-element lattice of height two (Figure~\ref{fig:mn}).
\begin{figure}[centering,h]
\caption{The $(n+2)$-element lattice of height 2, $\bM_n$.}
\label{fig:mn}
\begin{center}
\begin{tikzpicture}[scale=0.8]
\draw (0,-1) node {$\bM_n$};
  \node (mt) at (0,3) [fill,circle,inner sep=1.2pt] {};
  \node (m1) at (-1.5,1.5) [fill,circle,inner sep=1.2pt] {};
  \node (m2) at (-1,1.5) [fill,circle,inner sep=1.2pt] {};
  \node (m3) at (-.5,1.5) [fill,circle,inner sep=1.2pt] {};
  \node (m4) at (1.5,1.5) [fill,circle,inner sep=1.2pt] {};
  \node (mb) at (0,0) [fill,circle,inner sep=1.2pt] {};
  \draw [semithick]  (mb) to (m1) to (mt) to (m2) to (mb) to (m3) to (mt) to (m4) to (mb);
  \draw [semithick] (0,1.5) node {$\cdots$};
\end{tikzpicture}
\end{center}
\end{figure}
\begin{prop}\label{Concrete-prop-1'}
  For $n\geq 1$, $\bEq(2n+1)$ contains a dense $\bM_{n+2}$.
\end{prop}
\noindent Thus, every $\bM_n$ can be densely embedded in $\bEqX$ for some finite
set $X$. 
\begin{proof} (sketch)
We begin with Snow's example of a dense $\bM_3$ sublattice of $\bEq(X)$, where 
$X =  \{0,1,2,3,4\}$.  
Define three partitions of $X$,
\[
\alpha_1 = |0,1|2,3|4|, \quad 
\alpha_2 = |0|1,2|3,4|, \quad 
\alpha_3 = |0,2,4|1,3|, \quad 
\]
let $L =\{0_X, \alpha_1, \alpha_2, \alpha_3, 1_X\}$ and let $\bL  = \<L, \meet, \join\>$ denote the 
sublattice of $\bEq(X)$ generated by the three equivalences $\alpha_1, \alpha_2,
\alpha_3$ (Figure~\ref{fig:diamondeq}).

\begin{figure}[centering,h]
\caption{The lattice $\bL  = \<\{ 0_X, \alpha_1, \alpha_2, \alpha_3, 1_X \};
  \meet, \join\>$.}
\label{fig:diamondeq}
\begin{center}
\begin{tikzpicture}[scale=0.8]
  \node (top) at (0,3) [fill,circle,inner sep=1.2pt] {};
  \node (a1) at (-1.5,1.5) [fill,circle,inner sep=1.2pt] {};
  \node (a2) at (0,1.5) [fill,circle,inner sep=1.2pt] {};
  \node (a3) at (1.5,1.5) [fill,circle,inner sep=1.2pt] {};
  \node (bot) at (0,0) [fill,circle,inner sep=1.2pt] {};
  \draw [semithick]  (bot) to (a1) to (top) to (a2) to (bot) to (a3) to (top);
\draw (.4,3.2) node {$1_X$};
\draw (-1.8,1.5) node {$\alpha_1$};
\draw (.4,1.5) node {$\alpha_2$};
\draw (1.9,1.5) node {$\alpha_3$};
\draw (.4,-.2) node {$0_X$};
\end{tikzpicture}
\end{center}
\end{figure}
\noindent Obviously $\bL \cong \bM_3$, and it is not hard to show that the only
unary maps which respect all equivalences in $\bL$ are the constants and the
identity.  In other words, the set $\lambda(L)\subseteq X^X$ consists of
the six trivial maps in $X^X$.   Therefore, $\rho \lambda(L) = \Eq(X)$.

Now notice that if we adjoin the equivalence $\alpha_4 = |0,3|1,4|2|$ to $\bL$ we get
an $\bM_4$, which we denote by $\bL(\alpha_4)$.  Obviously,
$\lambda(L) \supseteq \lambda(L(\alpha_4))$, as adding more equivalences
only shrinks the set of functions respecting all equivalences.
Therefore, $\Eq(X) = \rho \lambda(L) \subseteq \rho \lambda(L(\alpha_4))$, so
$\bL(\alpha_4)$ is a dense $\bM_4$ sublattice of $\bEq(5)$.

Similarly, letting
$X =  \{0,1,\dots, 6\}$ and 
\[
\alpha_1 = |0,1|2,3|4,5|6|, \quad 
\alpha_2 = |0|1,2|3,4|5,6|, \quad 
\alpha_3 = |0,2,4,6|1,3,5|, \quad 
\]
the sublattice 
$\bL  = \<\{0_X, \alpha_1, \alpha_2, \alpha_3, 1_X\}, \meet, \join\>$ is a dense $\bM_3$ in
\bEqX. Adjoining the partitions
\[
\alpha_4 = |0,3|2,5|1,6|4| \quad \text{ and } \quad \alpha_5 = |0,5|1,4|3,6|2|
\]
results in a dense $\bM_5$ in \bEqX.
Proceeding inductively, when $|X| = 2n+1$ there are $n+1$ partitions of
the form $\alpha_i = |x_{i_0}|x_{i_1}, x_{i_2}| \cdots | x_{i_{2n-1}},x_{i_{2n}}|$,
and one of the form 
$\alpha_{n+2} = |\text{evens} | \text{odds}|$, with the following properties:
\begin{enumerate}
\item $\alpha_i \meet \alpha_j = 0_X$,
\item $\alpha_i \join \alpha_j = 1_X$,
\item the lattice generated by $\alpha_{n+2}$ and at least two other $\alpha_i$ is dense in $\bEqX$.
\end{enumerate}
\end{proof}

\subsection{Non-density}
\label{sec:non-density}
The results in this section give sufficient conditions under which a lattice
cannot be densely embedded in a lattice of equivalence relations. 
These results require some standard terminology that we have not yet introduced,
so we begin the section with these preliminaries.  As always, we will only deal
with finite lattices $\bL = \<L, \meet, \join\>$, and we use $0_L = \Meet L$ to
denote the bottom of $\bL$ and $1_L = \Join L$ to denote the top.

If $\bL = \<L, \meet, \join\>$ is a lattice, a non-empty subset $I\subseteq L$
is called an \defn{ideal} of $\bL$ if 
\begin{enumerate}[(i)]
\item $I$ is a \defn{down-set}: if $\alpha \in I$ and $\beta\leq \alpha$, then $\beta\in I$;
\item $I$ is closed under finite joins: $\alpha, \beta \in I$ implies $\alpha \join \beta \in I$.
\end{enumerate}
A \defn{filter} of a lattice is defined dually as a non-empty \defn{up-set} that is closed
under finite meets. 
An ideal or filter is said to be \emph{proper} if it is not equal to all of $L$.
The smallest ideal that contains a given element $\alpha$ is a 
\defn{principal ideal} and $\alpha$ is said to be a 
\defn{principal element} or \emph{generator} 
of the ideal in this situation.  The 
\emph{principal ideal generated by $\alpha$} 
is defined and denoted by 
$\downalpha = \{\theta \in L \mid \theta \leq \alpha\}$. 
Similarly, $\upalpha = \{\theta \in L \mid \theta \geq \alpha\}$ is the 
\emph{principal filter generated by $\alpha$}.
An ideal $I$ called a \defn{prime ideal} provided
$\alpha \meet \beta \in I$ implies $\alpha \in I$ or $\beta \in I$ 
for all $\alpha, \beta \in L$.
Equivalently, a \defn{prime ideal} is an ideal whose set-theoretic complement is a filter.
Since we require ideals (filters) to be non-empty, every prime filter (ideal) is
necessarily proper.  
An element is called \defn{meet prime} if it is
the generator of a principal prime ideal. Equivalently, 
$\alpha \in L \setminus \{1_L\}$ is meet prime if for all $\beta, \gamma \in L$
we have $\beta \meet \gamma  \leq \alpha$  implies $\beta\leq \alpha$ or
$\gamma\leq \alpha$.  \defn{Join prime} is defined dually.  

\begin{lemma}
\label{Concrete-lemma-1}
Suppose $\bL = \langle L, \meet, \join\rangle$ is a complete $0,1$-lattice. Then the following
are equivalent:
\begin{enumerate}[(i)]
\item There is an element 
$\alpha \in L \setminus \{0_L\}$
such that $\bigvee\{\gamma\in L: \gamma \ngeq \alpha \} < 1_L$.
\item There is an element $\alpha \in L \setminus \{1_L\}$ such that $\bigwedge\{\gamma\in L:
  \gamma \nleq \alpha \} > 0_L$.
\item $\bL$ is the union of a proper principal ideal and a proper principal filter.
\end{enumerate}
\end{lemma}
\begin{proof}
(i) $\Rightarrow$ (ii): Suppose $\alpha \in L \setminus \{0_L\}$ is such that the element
$\alpha' = \bigvee\{\gamma: \gamma \ngeq \alpha \}$ is strictly below $1_L$, and 
consider $\bigwedge\{\gamma: \gamma \nleq \alpha'\}$.  If $\beta \nleq \alpha'$, then
$\beta \notin \{\gamma: \gamma \ngeq \alpha \}$ so $\beta \geq \alpha$.  Therefore,
$\bigwedge\{\gamma: \gamma \nleq \alpha'\}\geq \alpha > 0_L$.  Thus $\alpha'\in
L\setminus \{1_L\}$ is such that 
$\bigwedge\{\gamma: \gamma \nleq \alpha'\}> 0_L$ so (ii) holds.\\[4pt]
(ii) $\Rightarrow$ (iii): Let $\alpha < 1_L$ be such that $\beta = \bigwedge
\{\gamma: \gamma\nleq  \alpha\} > 0_L$.  
Then, $\bL = \downalpha \cup \upbeta$ satisfies (iii).\\[4pt]
(iii) $\Rightarrow$ (i): Suppose 
 $\bL = \upalpha \cup \downbeta$ for some $\alpha > 0_L$,  $\beta < 1_L$.
 Then $\{\gamma\in L: \gamma \ngeq \alpha\} \subseteq \downbeta$; 
i.e.~$\gamma \ngeq \alpha \Rightarrow \gamma \leq \beta$.  Therefore,
 $\bigvee \{\gamma: \gamma\ngeq \alpha \} \leq \beta < 1_L$, so (i) holds.
\end{proof}

\begin{lemma} 
\label{Concrete-lemma-2} 
If $\bL\ncong \mathbf{2}$ is a sublattice of $\bEqX$ satisfying the
conditions of Lemma~\ref{Concrete-lemma-1}, then $\lambda(L)$ contains a non-trivial unary function.
\end{lemma}
\begin{proof}
Suppose $\bL\ncong \mathbf{2}$ is a sublattice of $\bEqX$ which satisfies condition (i)
of the lemma.  We must show that there is a non-trivial (i.e.~non-constant, non-identity)
$h\in X^X$ which respects every $\theta \in L$.  
By condition (i), there is an element $\alpha \in L \setminus \{0_L\}$ such that $\beta =
\bigvee\{\gamma\in L: \gamma \ngeq \alpha \}$ is strictly below $1_L$. Since 
$\alpha > 0_L$, there is a pair $(u,v)$ of distinct elements of $X$ that are
$\alpha$ related.  Since $\beta < 1_L$, there is a $\beta$ equivalence class 
$B \subsetneqq X$.  Define $h\in X^X$ as follows:
\begin{equation}
  \label{eq:h}
h(x) = \begin{cases}
u,& \quad x\in B,\\
v,& \quad x\notin B.
\end{cases}
\end{equation}
Then $h$ is not constant, since $\emptyset \neq B \neq X$;
$h$ is not the identity, since $\bL\ncong \mathbf{2}$; $h$
respects everything above $\alpha$ and everything below $\beta$, and therefore, $h\in
\lambda(\upalpha \cup \downbeta) = \lambda(L)$.
\end{proof}

\begin{theorem}\label{Concrete-thm-1} If $\bL\ncong \mathbf{2}$ is a lattice satisfying the conditions of Lemma~\ref{Concrete-lemma-1}
  and $X$ is any set, then $\bL$ cannot be densely embedded in $\bEqX$.
\end{theorem}
\begin{proof}
The theorem says that, for any embedding $\bL\cong \bL_0 \leq \bEqX$ of
such a lattice, $\bL_0$ is not dense in $\bEqX$; 
i.e.~$\rho \lambda(L_0) \lneqq  \bEqX$.  
To prove that this follows from Lemma~\ref{Concrete-lemma-2}, we must verify 
the following statement: If $\mathbf{2} \ncong \bL \leq \bEqX$
and if there is a non-trivial unary function $h\in \lambda(L)$, then 
$\rho \lambda(L) \lneqq \EqX$.

If $h\in X^X$ is any non-trivial unary function, then there are elements 
$\{x, y, u, v\}$ of $X$ such that $x\neq y$ and $h(x) = u \neq v = h(y)$.
We can assume $X$ has at least three distinct elements since $\bL \ncong \mathbf{2}$.
There are two cases to consider.  In the first, $h$ simply permutes $x$ and $y$.
In this case, $x=v$ and $y=u$, and $h(v) = u$, $h(u)=v$. There must be a third
element of $X$, say, $w\notin \{u,v\}$.  
If $h(w) \neq u$, then $h$ violates any equivalence that puts $v, w$ in the same
block and puts $u$ and $h(w)$ in separate blocks. 
If $h(w) \neq v$, then $h$ violates any equivalence that puts $u, w$ in the same
block and $v$ and $h(w)$ in separate blocks. 
In the second case to consider, $\{x, u, v\}$ are three distinct elements. In
this case, $h$ violates every relation that puts $x, y$ in the same block and puts
$u$ and $v$ in separate blocks. 

We have thus proved that $\rho \lambda(L) \lneqq \EqX$ whenever $\lambda(L)$ contains
a non-trivial unary function.
\end{proof}

\begin{corollary}\label{Concrete-cor-2}
If $\bL\ncong \mathbf{2}$ is a finite lattice with a meet prime element and $X$ is any set, then $\bL$ cannot
be densely embedded in $\bEqX$.
\end{corollary}
\begin{remark}
  The same result holds if we assume the lattice has a join prime element.
\end{remark}
\begin{proof}
It is clear by the definition of meet prime that a lattice 
satisfying the hypotheses of the corollary also satisfies the conditions
of Lemma~\ref{Concrete-lemma-1}, so the result follows from
Theorem~\ref{Concrete-thm-1}. 
\end{proof}

A lattice is called 
\defn{meet-semidistributive} 
if it satisfies the
\emph{meet-semidistributive law},
\[
\SD_\wedge: \quad
\alpha \meet \beta = \alpha \meet \gamma \quad \Rightarrow \quad \alpha
\meet (\beta \join \gamma) = \alpha \meet \beta.
\]

\begin{corollary}
\label{Concrete-cor-nondensity-1}
If $\bL\ncong \mathbf{2}$  is a finite meet-semidistributive lattice
and $X$ is any set, then $\bL$ cannot be densely embedded in $\bEqX$.
\end{corollary}
\begin{proof}
  We prove that every finite meet-semidistributive lattice $\bL$ contains a
  meet prime element. The result will then follow by Corollary~\ref{Concrete-cor-2}.
  Since $\bL$ is finite, there exists an atom $\alpha\in L$.  If $\alpha$ is the
  only atom, then $\upalpha$ is trivially prime.  Suppose $\beta \join \gamma
  \in \upalpha$.  Then $(\beta \join \gamma)\meet \alpha = \alpha$, and $\beta
  \meet \alpha \leq \alpha$ implies $\beta \meet \alpha \in \{0_L, \alpha\}$.
  Similarly for $\gamma$.  If both $\beta\meet \alpha = 0_L = \gamma \meet
  \alpha$ then $\SD_\meet$ implies $(\beta \join \gamma) \meet \alpha = 0_L$,
  which is a contradiction.
\end{proof}

The converse of Corollary~\ref{Concrete-cor-2}
is false.  That is, there exists a finite lattice
$\bL\ncong \mathbf{2}$ with no meet prime element that cannot be densely 
  embedded in some $\bEqX$. 
The lattice $\bM_{3,3}$ shown below is an example.  It has no meet prime element
but it does satisfy the conditions of Lemma~\ref{Concrete-lemma-1}. Thus, by
Theorem~\ref{Concrete-thm-1}, $\bM_{3,3}$ is not densely embeddable.  

\begin{figure}[!h]
\begin{center}
\begin{tikzpicture}[scale=0.6]
\draw (1,-1) node {$\bM_{3,3}$};
  \node (m3) at (0,0) [fill,circle,inner sep=1.2pt] {};
  \node (m31) at (0,1.5) [fill,circle,inner sep=1.2pt] {};
  \node (m32) at (-1.5,1.5) [fill,circle,inner sep=1.2pt] {};
  \node (m33) at (1.5,1.5) [fill,circle,inner sep=1.2pt] {};
  \node (m34) at (0,3) [fill,circle,inner sep=1.2pt] {};

  \node (top) at (1.5,4.5) [fill,circle,inner sep=1.2pt] {};
  \node (mid) at (1.5,3) [fill,circle,inner sep=1.2pt] {};
  \node (right) at (3,3) [fill,circle,inner sep=1.2pt] {};

  \draw [semithick]  (m3) to (m31) to (m34) to (m33) to (m3) to (m32) to (m34);
  \draw [semithick]  (m33) to (mid) to (top) to (right) to (m33) (top) to (m34);

\end{tikzpicture}
\end{center}
  \caption{The lattice $\bM_{3,3}$.}
  \label{fig:M33}
\end{figure}

\subsection{Distributive lattices}
\label{sec:distr-latt}
A lattice $\bL$ is called 
\defn{strongly representable} 
as a congruence lattice if
whenever $\bL \cong \bL_0 \leq \bEqX$ for some $X$ then there is an algebra based on $X$
whose congruence lattice is $\bL_0$.

\index{Berman, Joel} \index{Quackenbush, R.} \index{Wolk, B.}
\begin{theorem}[Berman~\cite{Berman:1970}, Quackenbush and Wolk~\cite{Quack:1971}]
Every finite distributive lattice is strongly representable.
\end{theorem}
{\bf Remark:} By Theorem \ref{Concrete-thm-3} above, the result of Berman,
Quackenbush and Wolk says, if $\bL$ is a finite distributive lattice then every
embedding $\bL\cong \bL_0\leq \bEqX$ is closed. 
The following proof is only slightly shorter than
to the original in~\cite{Quack:1971}, and the methods are similar.
\begin{proof}
Without loss of generality, suppose $\bL\leq \bEqX$. 
Fix $\theta\in \EqX \setminus L$ and define 
$\theta^* = \Meet \{ \gamma \in L \mid \gamma \geq \theta \}$ and 
$\theta_* = \Join \{ \gamma \in L \mid \gamma \leq \theta \}$.
Let $\alpha$ be a join irreducible in $L$ below $\theta^*$ and not below
$\theta_*$. 
Note that $\alpha$ is not below $\theta$.
Let $\beta = \Join \{ \gamma \in L \mid \gamma \ngeq \alpha \}$.
If $\beta$ were above $\theta$, then $\beta$ would be above $\theta^*$,
and so $\beta$ would be above $\alpha$. But $\alpha$ is join prime, so $\beta$ is not
above $\theta$. 

Choose $(u, v) \in \alpha \setminus \theta$ and note that $u \neq v$.
Choose $(x, y) \in \theta \setminus \beta$ and note that $x \neq y$.
Let $B$ be the $\beta$ block of $y$ and define $h\in X^X$ as in~(\ref{eq:h}). Then it
is clear that $h$ violates $\theta$, $h$ respects all elements in the sets
$\upalpha = \{\gamma \in L: \alpha \leq \gamma\}$ and 
$\downbeta = \{\gamma \in L: \gamma \leq \beta\}$, and $L = \upalpha \cup
\downbeta$. Since $\theta$ was an arbitrary element of $\EqX \setminus L$, we can
construct such an $h = h_\theta$ for each $\theta \in \EqX \setminus L$.  Let $\sH = \{h_\theta:
\theta \in \EqX \setminus L\}$ and let $\mathbf{A}$ be the algebra 
$\langle X, \sH\rangle$.  Then, $\bL =\bCon(\mathbf{A})$. 
\end{proof}

\section{Conclusions and open questions}

J.B.~Nation has found examples of densely embedded double-winged pentagons
  none of whose sublattices are densely embedded.  John Snow then asked if any
  of the sublattices are closed embeddings.  In general, we might ask the
  following: Are there closed sublattices of dense embeddings?

Another question we have not answered is whether the converse of
Theorem~\ref{Concrete-thm-1} is true, but this seems unlikely. Rather, we expect
there exists a finite lattice that is neither densely embeddable nor the union of
a proper principal ideal and a proper principal filter. 

Finally, we mention that even if we restrict ourselves to one of the smaller
classes of finite lattices mentioned above -- those satisfying the conditions of
Lemma~\ref{Concrete-lemma-1} or Corollary~\ref{Concrete-cor-2}, or the 
finite meet-semidistributive lattices -- it is still unknown whether every
lattice is this class is representable as the congruence lattice of a finite
algebra.  

\chapter{Congruence Lattices of Group Actions}
\label{cha:congr-latt-group}
Let $X$ be a finite set and consider the set $X^X$ of all maps from $X$ to
itself, which, when endowed with composition of maps and the identity mapping,
forms a monoid, $\<X^X, \circ, \id_X\>$.  The submonoid $S_X$ of all bijective
maps in $X^X$ is a group, the \defn{symmetric group on} $X$.  When the
underlying set is more complicated, or for emphasis, we denote the symmetric
group on $X$ by $\Sym(X)$.  When the  
underlying set isn't important, we usually write $S_n$ to denote the
symmetric group on an $n$-element set. 

If we have defined some set $F$ of basic operations on $X$, so that
$\bX = \<X, F\>$ is an algebra, then two other important submonoids of
$X^X$ are $\End(\bX)$, the set of maps in $X^X$ which respect all 
operations in $F$, and $\Aut(\bX)$, the set of bijective maps in  $X^X$ which
respect all operations in $F$.  It is apparent from the definition that
$\Aut(\bX)= S_X \cap \End(\bX)$, and  $\Aut(\bX)$ is a submonoid of $\End(\bX)$
and a subgroup of $S_X$.  These four fundamental monoids associated with the
algebra $\bX$, and their relative ordering under inclusion, are shown in the diagram
below. 

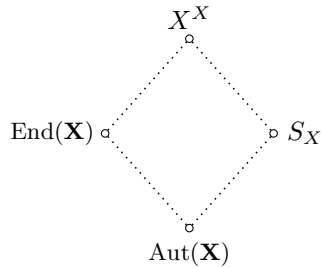
\begin{center}
  \begin{tikzpicture}[scale=.7]
    \node (Aut) at (0,0.2) [draw,circle,inner sep=1pt] {};
    \draw[font=\small] (0,-.30) node {$\Aut(\bX)$};

    \node (End) at (-1.6,2) [draw,circle,inner sep=1pt] {};
    \draw[font=\small] (-2.6,2) node {$\End(\bX)$};

    \node (Sx) at (1.6,2) [draw,circle,inner sep=1pt] {};
    \draw (2.2,2) node {$S_X$};

    \node (XX) at (0,3.8) [draw,circle,inner sep=1pt] {};
    \draw (0,4.2) node {$X^X$};
      \draw[semithick,dotted]    (Aut) to (End) to (XX) to (Sx) to (Aut);
  \end{tikzpicture}
\end{center}

Given a finite group $G$, and an algebra $\bX = \<X, F\>$, a
\index{representation!of a finite group}%
\emph{representation} of $G$ on $\bX$ is a group homomorphism
from $G$ into $\Aut(\bX)$.  That is, a representation of $G$ is a mapping
$\varphi : G \rightarrow \Aut(\bX)$ which satisfies $\varphi(g_1 g_2) =
\varphi(g_1) \circ \varphi(g_2)$, where (as above) $\circ$ denotes composition
of maps in $\Aut(\bX)$.

\section{Transitive $G$-sets}
From the foregoing, we see that a representation defines an action by $G$ on the
set $X$, as follows: $\bar{g} x = \varphi(g)(x)$.  If $\bar{G} = \varphi[G] \leq
\Aut(\bX)$
denotes the image of $G$ under $\varphi$, we call the algebra $\< X, \bar{G}\>$
a \defn{G-set}.\footnote{More 
  generally, a \Gset\ is sometimes defined to be a pair $(X, \varphi)$, where
  $\varphi$ is a homomorphism from a group into the symmetric group $S_X$, see
  e.g.~\cite{Suzuki:1982}.}   
The action is called
\index{transitive!action}%
\emph{transitive} if for each pair $x, y \in X$ there is some $g\in
G$ such that $\bar{g} x = y$. The representation $\varphi$ is called 
\emph{faithful}
\index{faithful!representation}%
if it is a monomorphism, in which case $G$ is isomorphic to its image under
$\varphi$, which is a subgroup of $\Aut(\bX)$.  We also say, in this case, that
the group acts faithfully, and call it a 
\defn{permutation group}.
A group which acts transitively on some set is called a 
\index{transitive!group}%
\emph{transitive group}.
Without specifying the set, however, this term is meaningless, since
every group acts transitively on some sets and intransitively on others.  
A representation $\varphi$ is called \emph{transitive} if the resulting action
is transitive. 
Finally, we define \defn{degree} of a group action on a set $X$ to be the
cardinality of $X$.

Two special cases are almost always what one means when one speaks of a
representation of a finite group.  These are the so called
\begin{itemize}
\item \defn{linear representations}, where $\bX = \<X, +, \circ, -, 0, 1, \F\>$
  is a finite dimensional vector space over a field $\F$, so $\Aut(\bX)$ is the
  set of invertible matrices with entries from $\F$; 
\item \defn{permutation representations}, where $\bX = X$ is just a set, so
  $\Aut(\bX) = S_X$. 
\end{itemize}

For us the most important representation of a group $G$ is its action 
on a set of cosets of a subgroup.  That is, for any subgroup $H\leq G$,
we define a transitive permutation representation of $G$, which we
will denote by $\hlambda_H$.  Specifically, $\hlambda_H$ is a group homomorphism
from $G$ into the symmetric group $\Sym(G/H)$ of permutations on the set $G/H =
\{H, x_1H, x_2H, \dots \}$ of \emph{left} cosets of $H$ in $G$.
The action is simply left multiplication by elements of $G$. That is,
$\hlambda_H(g)(xH)= gxH$.
Clearly, $\hlambda_H(g_1 g_2) = \hlambda_H(g_1)\hlambda_H(g_2)$ for all $g_1,
g_2 \in G$, so $\hlambda_H$ is a homomorphism.
Each $xH$ is a point in the set $G/H$, and the
\defn{point stabilizer} of $xH$ in $G$ is defined by
$G_{xH} = \{g\in G \mid gxH = xH \}$.  Notice that
\[
G_{xH} =\{g\in G \mid x^{-1}gxH  = H \} = 
x G_H x^{-1}  = x H x^{-1} = H^x,
\]
where $G_H = \{g\in G \mid g H = H \}$ is the point stabilizer of $H$ in $G$.  
Thus, the kernel of the homomorphism $\hlambda_H$ is 
\[
\ker \hlambda_H = \{g\in G \mid \forall x \in G,\; gxH = xH \} = 
\bigcap_{x\in G}G_{xH} = \bigcap_{x\in G} x H x^{-1}  = \bigcap_{x\in G} H^x.
\]
Note that $\ker \hlambda_H$ is the largest normal subgroup of $G$ 
contained in $H$, also known as the \defn{core} of $H$ in $G$, which we denote
by 
\[
\core_G(H) = \bigcap_{x\in G} H^x.
\]
If the subgroup $H$ happens to be \defn{core-free}, that is,
$\core_G(H)=1$, 
then $\hlambda_H : G \hookrightarrow \Sym(G/H)$ is an embedding, so 
$\hlambda_H$ is a 
\index{faithful!representation}%
faithful representation; 
\index{faithful!action}%
$G$ acts faithfully on $G/H$.
Hence the group $G$, being isomorphic to a subgroup of $\Sym(G/H)$, is itself a
permutation group.

Other definitions relating to \Gsets\ will be introduced as needed and in
the appendix, and we assume the reader is already familiar with these.
However, we mention one more important concept before proceeding, as it is a
potential source of confusion.  By a \defn{primitive group} we mean a group that
contains a core-free maximal subgroup.  This definition is not the typical one
found in group theory textbooks, but we feel it is better. (See the appendix
Section~\ref{sec:group-acti-perm} for justification.)

\subsection{$G$-set isomorphism theorems}
\label{subsec:g-set-isomorphism}
We have seen above that the action of a group on cosets of a subgroup $H$ is a
transitive permutation representation, and the representation is faithful when
$H$ is core-free. 
The first theorem in this section states that every
transitive permutation representation is of this form.
(In fact, as we will see in Lemma~\ref{lem:intransitive-gsets} below, every permutation
representation, whether transitive or not, can be viewed as an action on cosets.)

First, we need some more notation. Given a \Gset\ $\bA = \<A, G\>$ and any
element $a\in A$,  the set  
$G_a = \{g\in G \mid ga = a\}$ 
of all elements of $G$ which fix $a$ is a
\index{stabilizer subgroup}
subgroup of $G$, called the \emph{stabilizer of $a$ in $G$}.

\begin{theorem}[1st \Gset\ Isomorphism Theorem]
\label{thm:g-set-isomorphism1}
  If $\bA = \<A, \barG\>$ is a transitive \Gset, then $\bA$ is
  isomorphic to the \Gset 
  \[
  \Gamma :=  \<G/\stab{a}, \{\hat{\lambda}_g : g\in G\}\>
  \]
  for any $a\in A$.
\end{theorem}
\begin{proof}
  Suppose $\bA= \<A, \barG\>$ is a transitive \Gset, so 
  $A = \{\barg a \mid  g\in G\}$ for any $a\in A$.  The operations of the \Gset\ $\Gamma$ are defined, for each $g\in G$ and each
  coset $x \stab{a}\in G/\stab{a}$, by $\hat{\lambda}_g(x \stab{a}) = gx\stab{a}$. 

  Let 
  $\bG_\Lambda$ denote the \Gset\ $\<G, \{\lambda_g : g\in G\}\>$, that is,
  the group $G$ acting on itself by left multiplication. 
  Fix $a\in A$, and define $\varphi_a:G \rightarrow A$ by $\varphi_a(x) =
  \barx(a)$ for 
  each $x\in G$.  
  Then $\varphi_a$ is a homomorphism
  from 
  $\bG_\Lambda$ 
  into $\bA$ --
  that is, $\varphi_a$ respects operations:\footnote{In general, if $\bA
    =\< A, F\>$ and $\bB = \<B, F\>$ are two algebras of the same
    similarity type, then   
    $\varphi: \bA \rightarrow \bB$ is a homomorphism provided
    \[
    \varphi(f^\bA(a_1,\dots, a_n)) = f^\bB(\varphi(a_1),\dots,\varphi(a_n))
    \]
    whenever $f^\bA$ is an $n$-ary operation of $\bA$, $f^\bB$ is the
    corresponding $n$-ary operation of $\bB$, and $a_1,\dots, a_n$ are arbitrary 
    elements of $A$.  (Note that a one-to-one correspondence between the
    operations of 
    two algebras of the same similarity type is assumed, and required for the
    definition 
    of homomorphism to make sense.)
  }
  \[
  \varphi_a(\lambda_g(x)) = \varphi_a(gx)= \overline{gx}(a)
  = \barg \cdot \barx(a)
  = \barg \varphi_a(x).
  \]
  Moreover, since $\bA$ is transitive, $\varphi_a(G) = \{\barg a \mid  g\in G\}
  = A$, 
  so $\varphi_a$ is an epimorphism.  Therefore,
  $\bG_\Lambda/\ker \varphi_a \cong \bA$.  
  To complete the proof, one simply checks that the two algebras $\bG_\Lambda
  /\ker \varphi_a $ and $\Gamma$ are identical.\footnote{
    Indeed, 
    $\ker \varphi_a = \{(x,y) \in G^2 \mid  \varphi_a(x) = \varphi_a(y)\}$
    and the  universe of $\bG_\Lambda /\ker \varphi_a$ is 
    $G/\ker \varphi_a = \{x/\ker \varphi_a  \mid x\in G \}$. 
     where for each $x\in G$
  \begin{align*}
    x/\ker \varphi_a &= \{y\in G  \mid  (x,y) \in \ker \varphi_a\}
    = \{y\in G  \mid  \varphi_a(x) = \varphi_a(y)\}
    = \{y\in G  \mid  \barx(a) = \bary(a)\}\\
    &= \{y\in G  \mid  \id_{A}(a) = \overline{x^{-1}y}(a)\}
    = \{y\in G  \mid  x^{-1}y \in \stab{a}\}
    = x \stab{a}.
  \end{align*}
  These are precisely the elements of $G/\stab{a}$,
  so the universes of $\bG_\Lambda /\ker \varphi_a$ and $\Gamma$ are
  the same, as are their operations (left multiplication by $g\in G$).}
\end{proof}

The next theorem shows why intervals of subgroup lattices are so important for
our work.
\begin{theorem}[2nd \Gset\ Isomorphism Theorem]
\label{thm:g-set-isomorphism2}
  Let $\bA = \<A, G\>$ be a transitive \Gset\ and fix $a\in A$.  
  Then the lattice $\Con \bA$ is isomorphic to the
  interval $[G_a, G]$ in the subgroup lattice of $G$.
\end{theorem}
\begin{proof}
    For each $\theta  \in \Con \bA$, let $H_\theta =\{g\in G \mid (g(a),a) \in
    \theta\}$,  
    and for each $H \in [G_a, G]$, let 
    $(b, c) \in \theta_H$ mean there exist $g \in G$ and $h \in H$ 
    such that $gh(a) = b$ and $g(a) =c$. 
    If $g_1, g_2 \in H_\theta$, then 
    \[
    (g_2(a), a) \in \theta \quad \Rightarrow \quad (g_2^{-1} g_2(a), g_2^{-1}(a)) =
    (a, g_2^{-1}(a)) \in \theta,
    \]
    so $(g_2^{-1}(a), a) \in \theta$, by symmetry.  Therefore,
    $(g_1g_2^{-1} (a), g_1 (a)) \in \theta$, so 
    $(g_1g_2^{-1} (a), (a)) \in \theta$, by transitivity. Thus $H_\theta$ is a
    subgroup of 
    $G$, and clearly  $G_a \leq H_\theta$. 
    It is also easy to see that $\theta_H$ is a congruence of $\bA$.
    The equality $H_{\theta_H}=H$ trivially follows from the definitions. 
    On the other hand $(b, c) \in \theta_{H_\theta}$ if and only if there exist 
    $g, h\in G$ for which $(h(a),a) \in \theta$ and
    $b = gh(a)$, and $c = g(a)$. Since $G$ is transitive, it is equivalent 
    to $(b, c) \in \theta$.
    Therefore, $\theta_{H_\theta} = \theta$. Finally, $H_\theta \leq H_\phi$ if
    and only if
    $\theta \leq \phi$, so $\theta \mapsto H_\theta$ is an isomorphism
    between $\Con \bA$ and $[G_a, G]$.
\end{proof}

Since the foregoing theorem is so central to our work, we provide an alternative 
statement of it. This is the version typically found in group theory 
textbooks (e.g., \cite{Dixon:1996}).  Keeping these two alternative perspectives in
mind can be useful.

  \begin{theorem}[2nd \Gset\ Isomorphism Theorem, version 2]
    Let $\bA = \<A, \barG\>$ be a transitive \Gset\
    and let $a \in A$. Let $\sB$ be the set of all blocks $B$ with $a\in B$.
    Let $[\stab{a},G] \subseteq \Sub(G)$ denote the set of all subgroups of 
    $G$ containing $\stab{a}$.  Then there is a
    bijection $\Psi :\sB \rightarrow [\stab{a},G]$ given by $\Psi(B)= G(B)$,
    with inverse mapping $\Phi: [\stab{a},G] \rightarrow \sB $ 
    given by $\Phi(H) = \barH a = \{\barh a  \mid  h\in H\}$. 
    The mapping $\Psi$ is order-preserving in the sense
    that if $B_1, B_2 \in  \sB$ then 
    $B_1\subseteq B_2 \Leftrightarrow \Psi(B_1) \leq \Psi(B_2)$.
  \end{theorem}

  Briefly, the poset $\<\sB, \subseteq\>$ is order-isomorphic to the 
  poset $\<[\stab{a},G], \leq\>$. 

  \begin{corollary}
    Let $G$ act transitively on a set with at least two
    points. 
    Then $G$ is primitive if and only if each stabilizer $\stab{a}$ is a
    maximal subgroup of $G$.
  \end{corollary}

  Since the point stabilizers of a transitive group are all conjugate, 
  one stabilizer is maximal only when all of the stabilizers are maximal. 
  In particular, a regular permutation group is primitive if and only if it has
  prime degree. 

Next we describe (up to equivalence) all transitive permutation
representations of a given group $G$.  
We call two representations (or actions) 
\index{equivalent representations}%
\emph{equivalent}
provided the associated $G$-sets are isomorphic. 
The foregoing implies that every transitive permutation representation of $G$ is
equivalent to $\hlambda_H$ for some subgroup $H \leq G$.  The following
lemma\footnote{Lemma 1.6B of \cite{Dixon:1996}.} 
shows that we need only consider a single representative $H$ from each of the
conjugacy classes of subgroups.  

\begin{lemma}
  Suppose $G$ acts transitively on two sets,
  $A$ and $B$.  Fix $a\in A$ and let $G_a$ be the stabilizer of $a$ (under the first
  action).  Then the two actions are equivalent
  if and only if the subgroup $G_a$ is also a stabilizer under the second action
  of some point $b\in B$. 
\end{lemma}

The point stabilizers of the action $\hlambda_H$ described above are the
conjugates of $H$ in $G$.  Therefore, the lemma implies that, for any two
subgroups $H, K \leq G$, the representations $\hlambda_H$ and $\hlambda_K$ are
equivalent precisely when $K = x Hx^{-1}$ for some $x\in G$. 
Hence, the transitive permutation representations of $G$ are given, up to
equivalence, by $\hlambda_{K_i}$ as $K_i$ runs over a set of representatives of
conjugacy classes of subgroups of $G$.

  \subsection{An \Mset\ isomorphism theorem}
  It is natural to ask whether the two theorems of the previous subsection hold
  more generally for a unary algebra $\<X, M\>$, where $M$ is a monoid (rather
  than a permutation group).  We call such an algebra $\<X, M\>$ an \Mset, and
  although we will see that there is no analogue to the 2nd \Gset\ Isomorphism
  Theorem, we do have 
  \begin{theorem}[1st \Mset\ Isomorphism Theorem]
    If $\<X, M\>$ is a transitive \Mset, then for any fixed
    $x\in X$, the map $\varphi_x : M \rightarrow X$ defined by $\varphi_x(m) = mx$
    is an \Mset\ epimorphism. 
    Moreover, the (transitive) \Mset\ $\<M/\ker \varphi_x, M\>$ is isomorphic to 
    $\<X, M\>$.
  \end{theorem}
  \begin{proof}
    By transitivity, for each $y\in X$, there is an $m\in M$ such
    that $\varphi_x(m) = mx = y$, so $\varphi_x$ is onto.  Also, $\varphi_x$ is a
    homomorphism of the \Mset\ $\<M, M\>$ onto
    the \Mset\ $\<X, M\>$, since for all $m, m_1\in M$,
    \[
    \varphi_x(m\circ m_1) = m(m_1 x) = m \varphi_x(m_1).
    \]
    By the usual isomorphism theorem,
    \begin{equation}
      \label{eq:msetcong}
      \<M/\ker \varphi_x, M\> \cong \<X, M\>
    \end{equation}
    where 
    \[
    \ker \varphi_x = \{(m_1, m_2) \in M^2 \mid \varphi_x(m_1) = \varphi_x(m_2)\} =
    \{(m_1, m_2) \in M^2 \mid m_1 x = m_2 x \}.
    \]
    Note that, since $\<X, M\>$ is a transitive \Mset, the \Mset\ 
    $\<M/\ker \varphi_x, M\>$ must also be transitive, otherwise~(\ref{eq:msetcong}) would
    fail. 

    Just to be sure, let's verify that $\<M/\ker \varphi_x, M\>$ is indeed transitive.
    Let $m_1/\ker\varphi_x$, $m_2/\ker\varphi_x$ be any two $\ker\varphi_x$-classes
    of $M$.  We must show there exists $m_3\in M$ such that 
    $m_3[m_1/\ker\varphi_x]= m_2/\ker\varphi_x$.
    Let $\varphi_x(m_1) = y_1$ and 
    $\varphi_x(m_2) = y_2$.  Let $m_3\in M$ be a map which takes  $y_1$ to $y_2$,
    (guaranteed to exist by transitivity of $\<X, M\>$).  Then for all 
    $m\in m_1/\ker\varphi_x$, we have 
    $m_3mx = m_3y_1 = y_2$, so $m_3m \in m_2/\ker\varphi_x$.  Therefore, 
    \[
    m_3[m_1/\ker\varphi_x]\subseteq  m_2/\ker\varphi_x.
    \]
    By the same argument, there is $m_3'\in M$ such that 
    \[
    m_3'[m_2/\ker\varphi_x]\subseteq  m_1/\ker\varphi_x.
    \]
    By cardinality, $m_3[m_1/\ker\varphi_x]= m_2/\ker\varphi_x$.
  \end{proof}

  An analogue to the 2nd \Gset\ Isomorphism Theorem for monoids would be that
  $[M_x, M] \cong \Con\<X, M\>$ should hold for a transitive \Mset\ $\<X, M\>$.
  By the following counter-example, we see that this is false:
  Consider the monoid $M$ consisting of the identity and constant maps.  Of
  course, $\<X, M\>$ is a transitive \Mset, and $\Con\<X, M\> =
  \Eq(X)$. However, for $x\in X$, the stabilizer is  
  $M_x = \{m \in M: mx = x\}$ which is the set containing the identity map on $X$
  and the constant function that maps all points to $x$.
  So the lattice $[M_x, M]$ of submonoids of $M$ above $M_x$ is just the lattice
  of subsets of $M$ which contain the identity and the constant map $x$.  This is a
  distributive lattice, so it cannot be isomorphic to $\Con\<X, M\> = \Eq(X)$.

\section{Intransitive \Gsets}
\label{sec:congr-latt-intr}
The problem of characterizing congruence lattices of 
intransitive \Gsets\ seems open. 
In this section we prove a couple of results
which help determine the shape of 
congruence lattices of intransitive $G$-sets. 
In~\cite{gsets} we use these and other
results to show that for many lattices a minimal representation as the
congruence lattice of an intransitive \Gset\ is not possible.\footnote{In other
  words, if there exists a representation of such a lattice as the congruence
  lattice of an algebra (of minimal cardinality), then the
  algebra must be a \emph{transitive} \Gset.}

In the previous section we considered transitive, or one-generated, \Gsets.
In Theorem~\ref{thm:g-set-isomorphism1}, we presented the well known result that
a transitive \Gset\ $\<\Omega, G\>$, with universe $\Omega$, is isomorphic to the
\Gset\ $\<G/H, G\>$, where the universe is now the collection of cosets of
a subgroup $H= G_\omega$ -- the stabilizer of a point $\omega\in \Omega$.
Then, Theorem~\ref{thm:g-set-isomorphism2} gave us a precise 
description of the shape of the congruence lattice: $\Con\<G/H, G\> \cong [H,G]$. 
It is natural to ask whether results analogous to these 
hold for intransitive \Gsets.

In this section, we first prove 
that an arbitrary (intransitive)
\Gset\ $\<\Omega, G\>$ is isomorphic to a \Gset\ of the form
$\<G_1/H_1 \cup \cdots \cup G_r/H_r, G\>$, where
$H_i \leq G_i\cong G$. 
This result is well known, and appears as
Theorem 3.4 in~\cite{alvi:1987}.  Nonetheless we present a short proof and
describe the \Gset\ isomorphism explicitly.\footnote{Such an explicit
  description is useful when we are working with such algebras on the computer,
  using the Universal Algebra Calculator or \GAP, for example.}
Thereafter, we prove lemma which, along
with the first, gives a characterization of the congruence lattice of an
arbitrary \Gset.   
It is almost certain that this simple result is also well known,
but to my knowledge it does not appear in print elsewhere.\footnote{I thank Alexander
Hulpke for alerting me to the special case, described below, of the second lemma.}

Throughout this section, we adhere to the convention that \emph{groups act on the
  left}, so we will denote the action of $g\in G$ on an element 
$\omega\in \Omega$ 
by $g: \omega \mapsto g \omega$, and we use $G\omega$ to denote the orbit of
$\omega$ under this action, that is, $G\omega  = \{g\omega  \mid g\in G\}$.
Finally, we remind the reader that
\text{\emph{all groups under consideration are finite}.}

Our first lemma shows that, even in the intransitive case, we can take the
universe of an arbitrary \Gset\ to be a collection cosets of the group $G$.
\begin{lemma}
\label{lem:intransitive-gsets}
Every \Gset\ $\la \Omega,G\ra$ is isomorphic to a $G$-set on a universe of
the form $G_1/H_1 \cup \cdots \cup G_r/H_r$, where
$H_i \leq G_i\cong G$ and $G_i/H_i$ is the set of left cosets of $H_i$ in $G_i$,
for each $1\leq i\leq r$, 
\end{lemma}
\begin{proof}
Suppose $\alg \Omega = \la \Omega, G\ra$ is an arbitrary $G$-set, and let $\<\Omega_i, G\>$, 
$1 \leq i \leq r$,
be the minimal subalgebras of $\alg \Omega$.  That is, each $\Omega_i$ is an
orbit, say, $\Omega_i = G\omega_i$,  
and $\Omega = G\omega_1 \cup \cdots \cup G\omega_r$ is a disjoint union.
For each $1\leq i \leq r$, let $G_i$ be an isomorphic copy of $G$, with, say,
$\phi_i: G_i\cong G$ as the isomorphism. Clearly,
\[
H_i := \{x \in G_i \mid  \phi_i(x)\omega_i = \omega_i\} \cong \{g\in G \mid
g \omega_i = \omega_i\} = G_{\omega_i}.
\]
Note that $\<G_i/H_i, G\> \cong \<G \omega_i, G\>$, where $G$ acts on $G_i/H_i$ as one
expects: for $g\in G$ and $xH_i\in G_i/H_i$, the action is
$g: xH_i \mapsto \phi_i^{-1}(g)x H_i$. 

Define 
$\psi: G_1/H_1 \cup \cdots \cup G_r/H_r \rightarrow \Omega$
by $\psi(xH_i)  = \phi_i(x)\omega_i$.  This map is well-defined.  For, if 
$xH_i = x'H_j$, then $i=j$ and $x^{-1}x'\in H_i$, and 
it is easy to verify that $x^{-1}x'\in H_i$ holds if and only if $\phi_i(x')\omega_i = \phi_i(x) \omega_i$.
Thus, $\psi(xH_i) = \psi(x'H_j)$.

Now consider the \Gset\ 
$\<G_1/H_1 \cup \cdots \cup G_r/H_r, G\>$ with the same action as above:
$g(xH_i) =\phi_i^{-1}(g)(x H_i)$.   We claim that $\psi$ is a \Gset\ isomorphism of
$\<G_1/H_1 \cup \cdots \cup G_r/H_r, G\>$ onto
$\<\Omega, G\>$.  It is clearly a bijection.\footnote{Define
$\zeta: \Omega \rightarrow 
G_1/H_1 \cup \cdots \cup G_r/H_r$ by 
$\zeta(g\omega_i) = \phi_i^{-1}(g)H_i$, 
check that this map is well-defined, and
note that $\psi \zeta  = \id_\Omega$, and 
$\zeta \psi$ is the identity on $G_1/H_1 \cup \cdots \cup G_r/H_r$.}
We check that $\psi$ respects the
interpretation of the action of $G$:  Fix $g \in G$ and
$x\in G_i$.  Then, since $\phi_i$ is a homomorphism,
\[
\psi (\phi_i^{-1}(g)(x H_i)) = 
\phi_i (\phi_i^{-1}(g)x) \omega_i
= \phi_i (\phi_i^{-1}(g)) \phi_i (x) \omega_i 
= g \psi (x H_i).
\]
\end{proof}
The foregoing lemma shows that we can always take the universe of an
intransitive \Gset\ to be a disjoint union of sets of cosets of stabilizer
subgroups.  We now use this fact to describe the structure of the congruence
lattice of an arbitrary  \Gset.

As above, let $\alg \Omega = \la \Omega, G\ra$ be a $G$-set with universe
$\Omega = G\omega_1 \cup \cdots \cup G\omega_r$, where each 
$\<G\omega_i, G\>$ is a minimal subalgebra.
Consider the partition $\tau \in \Eq(\Omega)$, given by 
$\tau = |G\omega_1| G\omega_2|\cdots| G\omega_r|$.
Clearly, this is a congruence relation, since the action of every $g\in G$ fixes
each block.  We call $\tau$ the 
\defn{intransitivity congruence}.
It's clear that we can join two or more blocks of $\tau$
and the new larger block will still be preserved by every $g\in G$.
Thus, the interval above $\tau$ in the congruence lattice $\alg \Omega$ is
isomorphic to the lattice of partitions of a set of size $r$.  That is, 
\begin{equation}
  \label{eq:8}
[\tau, 1_\Omega] := \{\theta \in \Con \bOmega \mid \tau \leq \theta \leq 1_\Omega\} \cong \Eq(r).
\end{equation}
Another obvious fact is that the interval below $\tau$ in $\Con \bOmega$ is 
\begin{equation}
\label{eq:7}
[0_\Omega, \tau] \cong \prod_{i=1}^{r} \Con(\<G\omega_i, G\>).
\end{equation}
Since each minimal algebra $\<G\omega_i, G\>\cong
\<G_i/H_i, G\>$ is transitive, we have
$\Con(\<G\omega_i, G\>) \cong [H_i, G_i]$.
Thus, the structure of that part of $\Con\bOmega$ that is comparable with the 
intransitivity congruence is explicitly described by~(\ref{eq:8}) and~(\ref{eq:7}).  

Our next result describes the
congruences that are incomparable with the intransitivity congruence.
The description is in terms of the blocks of congruences below the intransitivity
congruence. Thus, the lemma does not give a nice
abstract characterization of the shape of the $\ConO$ in terms of the
shape of $\Sub(G)$, as we had in the transitive case.
However, besides being useful for computing the congruences, this result
can be used in certain situations to draw conclusions about the general shape of
$\ConO$, based on the subgroup structure of $G$ (for example, using combinatorial
arguments involving the index of subgroups of $G$).  We will say more about
this below.

Though the proof of Lemma~\ref{lemma-intransGsets} is elementary, it gets a bit 
complicated when presented in full generality.  Therefore, we begin by discussing the
simplest special case of an intransitive \Gset, that is, one which has just two
minimal subalgebras.  
Suppose $\bOmega = \<\Omega, G\>  = \<\Omega_1 \cup \Omega_2, G\>$ is a $G$-set with 
$\Omega_i = G\omega_i$ for some $\omega_i\in \Omega_i$, $i=1, 2$.
For each subset $\Lambda\subseteq \Omega$, for each $g\in G$, let 
$g\Lambda:=\{g\omega \mid \omega \in \Lambda\}$, and define the 
\defn{set-wise stabilizer} 
of $\Lambda$ in $G$ to be the subgroup
\[
\Stab_G(\Lambda) := \{ g\in G \mid g\omega \in \Lambda \text{ for all } \omega\in \Lambda\}.
\]
As above, 
we call the congruence $\tau = |\Omega_1 | \Omega_2|$ the intransitivity
congruence. 
Fix a congruence $\tau_0$ strictly below $\tau$, and for each $i=1,2$ 
let $\Lambda_i = \omega_i/\tau_0$ denote the block of $\tau_0$ containing
$\omega_i$.
Then there is a congruence $\theta$ above $\tau_0$ with a block $\Lambda_1 \cup
\Lambda_2$ if and only if $\Stab_G(\Lambda_1) = \Stab_G(\Lambda_2)$.
(We will verify this claim below when we prove it more generally in
Lemma~\ref{lemma-intransGsets}.)
This characterizes all congruences in $\ConO$ that are incomparable with
the intransitivity congruence, $\tau$, in terms of the congruences below
$\tau$. 

Let $\bOmega = \<\Omega_1 \cup \cdots \cup \Omega_r, G\>$ be a $G$-set with
minimal subalgebras $\Omega_i = G\omega_i$, for some $\omega_i \in \Omega_i$,
$1\leq i \leq r$.
Let $\tau = |\Omega_1 | \Omega_2 | \cdots | \Omega_r|$ be the intransitivity
congruence and fix $\tau_0 < \tau$ in $\Con \bOmega$.  For each $1\leq i \leq
r$, let $\Lambda_i = \omega_i/\tau_0$ denote the block of $\tau_0$ containing $\omega_i$,
and let $T_i = \{g_{i,0}{=}1, g_{i,1}, \dots, g_{i,n_i}\}$ be a transversal
of $G/\Stab_G(\Lambda_i)$.\footnote{Here $G/\Stab_G(\Lambda_i)$ denotes the set
  of right cosets of $\Stab_G(\Lambda_i)$ in $G$, and a \defn{transversal} is a
  set containing one element from each coset.}

It is important to note that the blocks of $\tau_0$ are $g_{i,k}\Lambda_i $, where 
$1\leq i \leq r$ and $0\leq k \leq n_i$.
This is illustrated in the following diagram, where the blocks of $\tau_0$
appear below the blocks of $\tau$ to which they belong.  

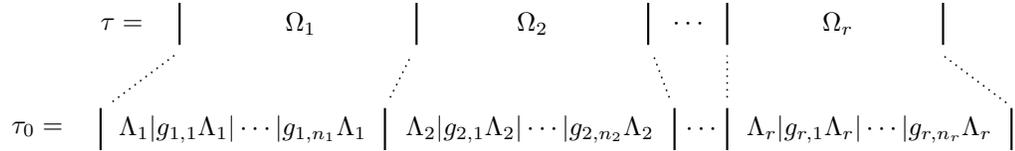
\begin{center}
    \begin{tikzpicture}[scale=.7]
      \draw (1.4,2) node {$\tau = $};
      \draw[thick] (2.5,1.6) -- (2.5,2.4);
      \draw (4.8,2) node {$\Omega_1$};
      \draw[thick] (7,1.6) -- (7,2.4);
      \draw (9.2,2) node {$\Omega_2$};
       \draw[thick] (11.4,1.6) -- (11.4,2.4);
       \draw (12.2,2) node {$\cdots$};
       \draw[thick] (12.9,1.6) -- (12.9,2.4);
      \draw (15,2) node {$\Omega_r$};
      \draw[thick] (17,1.6) -- (17,2.4);
      \draw[semithick, dotted] (2.4,1.4) -- (1.2,.5);
      \draw[semithick,dotted] (6.5,.6) -- (6.9,1.4);
      \draw[semithick,dotted] (11.8,.6) -- (11.5,1.4);
      \draw[semithick,dotted] (12.9,.6) -- (12.9,1.4);
      \draw[semithick,dotted] (18.2,.5) -- (17,1.4);
      \draw (-.20,0) node {$\tau_0 = $};
      \draw[thick] (1,-.4) -- (1,.4);
      \draw (3.7,0) node {$\Lambda_1| g_{1,1} \Lambda_1 | \cdots | g_{1,n_1} \Lambda_1 $};
       \draw[thick] (6.4,-.4) -- (6.4,.4);
       \draw (9.15,0) node {$\Lambda_2| g_{2,1}\Lambda_2 | \cdots | g_{2,n_2}\Lambda_2 $};
       \draw[thick] (11.9,-.4) -- (11.9,.4);
       \draw (12.45,0) node {$\cdots$};
       \draw[thick] (12.9,-.4) -- (12.9,.4);
      \draw (15.6,0) node {$\Lambda_r|  g_{r,1} \Lambda_r | \cdots |g_{r,n_r}  \Lambda_r $};
      \draw[thick] (18.3,-.4) -- (18.3,.4);

    \end{tikzpicture}

\end{center}

It should be obvious that the blocks of $\tau_0$ are as given above, but since this plays
such an important role in the lemma below, we check it explicitly:
If $\Lambda_i\subseteq \Omega_i$ is a block of $\tau_0$, then so is $g \Lambda_i$ for all $g \in G$, and either
$g\Lambda_i \cap \Lambda_i = \emptyset$ or $g\Lambda_i = \Lambda_i$.  If $\Lambda'
\subseteq \Omega_i$ is also a  block of $\tau_0$, then $\Lambda' = g' \Lambda_i$ for some
$g'\in G = \Stab_G(\Lambda_i) \cup g_{i,1} \Stab_G(\Lambda_i) \cup g_{i,n_i} \Stab_G(\Lambda_i)$, say $g'\in g_{i,j}
\Stab_G(\Lambda_i)$.  Then, $g_{i,j}^{-1}g'\in 
\Stab_G(\Lambda_i)$, so $g_{i,j}^{-1}g'\Lambda_i=\Lambda_i$.  Therefore, 
$g'\Lambda_i = g_{i,j}\Lambda_i$.

Another obvious but important consequence:
If $T_1 = \{g_{1,0}{=}1, g_{1,1}, \dots, g_{1,n_1}\}$ is a transversal of 
$G/\Stab(\Lambda_1)$, and if $\Stab(\Lambda_1) = \Stab(\Lambda_j)$, then 
$T_1$ is also a transversal of $G/\Stab(\Lambda_j)$, so the blocks of $\tau_0$ in $\Omega_j$
may be written as $g_{1,k}\Lambda_j$, where $0\leq k \leq n_1$.

\begin{lemma}
\label{lemma-intransGsets}
Given a subset $\{i_1, \dots, i_m\} \subseteq \{1,\dots, r\}$, 
there exists $\theta \in \ConO$ with block $\Lambda_{i_1} \cup \dots \cup \Lambda_{i_m}$ if and only if
$\Stab_G(\Lambda_{i_1}) = \cdots = \Stab_G(\Lambda_{i_m})$.  For example,
\begin{equation}
  \label{eq:th}
\theta = \tau_0 \cup \bigcup_{k=0}^{n_{i_1}} 
\left(g_{{i_1}k}\Lambda_{i_1} \cup \dots \cup g_{{i_1}k} \Lambda_{i_m}\right)^2. 
\end{equation}
\end{lemma}
\begin{remarks}
  The index set $\{i_1, \dots, i_m\}$ identifies the subalgebras from which to
  choose blocks that will be joined in the new congruence $\theta$.  
  The number of blocks of $\tau_0$ which
  intersect the subalgebra $\Omega_{i_j}$ is $n_{i_j}$, which is the length of the
  transversal of $G/\Stab_G(\Lambda_{i_j})$. Therefore, $n_{i_j} = |G:\Stab_G(\Lambda_{i_j})|$.  

  As noted above, if
  $\Stab_G(\Lambda_{i_1}) = \Stab_G(\Lambda_{i_m})$, then we can assume the transversals
  $T_1 = \{g_{i_11}, \dots, g_{i_1n_{i_1}}\}$ and  
  $T_m = \{g_{i_m1}, \dots,g_{i_mn_{i_m}}\}$ are the same. 
In the proof below, we will use $T$ to denote this common transversal.
\end{remarks}
\begin{proof}
$(\Rightarrow)$  Assume there is a congruence $\theta \in \ConO$ with block 
$\Lambda_{i_1} \cup \dots  \cup \Lambda_{i_m}$.  
Suppose there exists $1\leq j < k \leq m$ such that 
$\Stab_G(\Lambda_{i_j}) \neq \Stab_G(\Lambda_{i_k})$.  Without loss of generality, assume 
$g\in \Stab_G(\Lambda_{i_j}) \setminus \Stab_G(\Lambda_{i_k})$, so $g \Lambda_{i_j} = \Lambda_{i_j}$ and
there is an $x\in \Lambda_{i_k}$ such that $g x \notin \Lambda_{i_k}$.  Of course, 
$g \Omega_{i_k} = \Omega_{i_k}$, so we must have
$g x \notin \Lambda_{i_1} \cup \dots  \cup \Lambda_{i_m}$.  Thus, choosing any $y\in
\Lambda_{i_j}$, we have $(x,y)\in \theta$ while 
$(g x, g y)\notin \theta$, contradicting $\theta \in \ConO$.  Therefore, it must be
the case that 
$\Stab_G(\Lambda_{i_1}) = \cdots = \Stab_G(\Lambda_{i_m})$.

\medskip

\noindent $(\Leftarrow)$  
Suppose $\Stab_G(\Lambda_{i_1}) = \cdots = \Stab_G(\Lambda_{i_m})$.
Let $\theta$ be the relation defined in~(\ref{eq:th}).  We will prove $\theta \in
\ConO$.  It is easy to see that $\theta$ is an equivalence relation, so we
just need to check $g \theta \subseteq \theta$; that is, we prove
$(\forall \, (x,y)\in \theta)\, (\forall \, g\in G) \,
(gx ,gy)\in \theta$.

Fix $(x,y)\in \theta$, say, 
$x\in g_{i_1 k}\Lambda_{i_j}$ and 
$y\in g_{i_1 k}\Lambda_{i_\ell}$, 
for some $0 \leq k \leq n_{i_1}$,
$1\leq j < \ell \leq m$.  
For each $g\in G$ we have 
$g\, g_{i_1 k}\Lambda_{i_j} = g_{i_1 s}\Lambda_{i_j}$
for some $g_{i_1 s}\in T$.
Thus, $g_{i_1 s}^{-1} \, g\, g_{i_1 k} \in \Stab_G(\Lambda_{i_j})$.
Similarly, $g\, g_{i_1 k}\Lambda_{i_\ell} = g_{i_1 t}\Lambda_{i_\ell}$
for some $g_{i_1 t}\in T$, so
$g_{i_1 t}^{-1}\, g \, g_{i_1 k} \in
\Stab_G(\Lambda_{i_\ell})$.
This and the hypothesis $\Stab_G(\Lambda_{i_j}) = \Stab_G(\Lambda_{i_\ell})$ together imply
$g_{i_1 s}\Stab_G(\Lambda_{i_j}) = g_{i_1 t}\Stab_G(\Lambda_{i_j})$, 
so $g_{i_1 s} = g_{i_1 t}$, since they are both elements of the transversal of
$\Stab_G(\Lambda_{i_j})$.   We have thus shown that the action of $g\in G$ 
maps pairs of blocks with equal stabilizers to the same block of $\theta$; that is, 
$g \, g_{i_1 k}\Lambda_{i_j} = g_{i_1 s}\Lambda_{i_j}  \; \theta \;
g_{i_1 t} \Lambda_{i_\ell} = g \, g_{i_1 k}\Lambda_{i_\ell}$.
\end{proof}

\chapter{Interval Sublattice Enforceable Properties}
\label{cha:subl-interv-enforc}
\section{Introduction}
Given a finite lattice $L$, the
expression $L \cong [H, G]$ means ``there exist finite groups $H < G$ such that 
$L$ is isomorphic to the interval $\{K \mid H\leq K \leq G\}$ in
the subgroup lattice of $G$.''  
A group $G$ is called \emph{almost simple} if $G$ has a normal subgroup $S \subnormal
G$ which is nonabelian, simple, and has trivial centralizer, $C_G(S) = 1$.
If $H \leq G$, then the
\emph{core} of $H$ in $G$, denoted $\core_G(H)$, is the largest normal subgroup of $G$
contained in $H$; it is given by  $\core_G(H) = \bigcap\limits_{g\in G} gHg^{-1}$.
A subgroup $H\leq G$ for which $\core_G(H)=1$ is called \emph{core-free in $G$}.
If every finite lattice can be represented as the congruence lattice of a finite
algebra, we say that the \FLRP\ has a positive answer.

If we assume that the \FLRP\ has a positive answer, then for every finite
lattice $L$ there is a finite group $G$ having $L$ as an upper interval in
$\Sub(G)$. 
In this chapter we consider the following question: Given a finite lattice $L$,
what can we say about a finite group $G$ that has  $L$
as an upper interval in its subgroup lattice?
Taking this a step further, we consider certain finite collections of finite 
lattices ask what sort of properties we can prove about a
group $G$ if we assume it has all of these lattices as 
upper intervals in its subgroup lattice.  In this and the next section, we
address these questions somewhat informally in order to motivate this
approach. In Section~\ref{sec:isle-prop-groups} we introduce a 
new formalism for \emph{interval sublattice enforceable} properties of
groups. 

One easy consequence that comes out of this investigation is the following observation:
\begin{prop}
\label{prop:parachute}
  Let $\sL$ be a finite collection of finite lattices.
  If the \FLRP\ has a positive answer, then there exists a finite group $G$ such
  that each lattice $L_i \in \sL$ is an upper interval $L_i\cong [H_i, G] \leq
  \Sub(G)$, with $H_i$ core-free in $G$. 
\end{prop}
By the ``parachute'' construction described in the next section,
we will see that the only non-trivial part of this proposition is the conclusion
that all the $H_i$ be core-free in $G$.  However, this will follow easily from
Lemma~\ref{lemma-wjd-3} below.  

Before proceeding, it might be worth pausing to consider what seems like a
striking consequence of the proposition above: 
If the \FLRP\ has a positive answer, then no matter 
what we take as our finite collection $\sL$ -- for example, we
might take $\sL$ to be \emph{all} finite lattices with
at most $N$ elements for some large $N< \omega$ -- we can always find a \emph{single}
finite group $G$ such that every lattice in $\sL$ is an upper interval in
$\Sub(G)$; moreover, (by Lemma~\ref{lemma-wjd-3}) we can assume the subgroup
$H_i$ at the bottom of each interval is core-free.  As a result, the single
finite group $G$ must have so many faithful representations,  $G\hookrightarrow \Sym(G/H_i)$
with  $\Con\<G/H_i, G\> \cong L_i$,  one such representation for each distinct $L_i\in \sL$. 

\section{Parachute lattices}
\label{sec:parachute-lattices}
\index{P\'alfy, P\'eter}%
\index{Pudl\'ak, Pavel}%
As mentioned above, in 1980 \Palfy\ and \Pudlak\ published the following
striking result:
\begin{theorem}[\Palfy-\Pudlak~\cite{Palfy:1980}]
\label{thm:P5}
The following statements are equivalent:
\begin{enumerate}[(A)]
\item Every finite lattice is isomorphic to
  the congruence lattice of a finite algebra.
\item Every finite lattice is isomorphic to
  an interval in the subgroup lattice of a finite group.
\end{enumerate}
\end{theorem}
\noindent Also noted in~\cite{Palfy:1980} is the important fact that (B) is equivalent to:
\\[4pt]
{\it (B') Every finite lattice is isomorphic to
  the congruence lattice of a finite transitive G-set.}

\vskip3mm

There are a number of examples in the literature of the following situation: a specific
finite lattice is considered, and it is shown that if
such a lattice is an interval in the subgroup lattice of a finite group, then this
group must be of a certain form or have certain properties.
As the number of such results grows, it becomes increasingly useful to keep in
mind the following simple observation: 
\begin{lemma}
\label{lemma-wjd-1}
Let $\sG_1, \dots, \sG_n$ be classes of groups and  
suppose that for each $i\in \{1, \dots, n\}$ there exists a finite lattice $L_i$
such that
$L_i \cong [H, G]$ only if $G\in \sG_i$.
Then (B) is equivalent to
\begin{enumerate}[(C)]
\item For each finite lattice $L$, there is a finite group $G \in
  \bigcap\limits_{i=1}^n \sG_i$ such that $L \cong [H,G]$.
\end{enumerate}
\end{lemma}
\begin{proof}
Obviously, (C) implies (B).  Assume (B) holds and let $L$ be any finite lattice.  Suppose 
$\sG_1, \dots, \sG_n$ and $L_1, \dots, L_n$ satisfy the hypothesis of the lemma.
Construct a new lattice $\sP = \sP(L, L_1, \dots, L_n)$ as shown in
Figure~\ref{fig:parachute} (a).
By (B), there exist finite groups $H \leq G$ with $\sP \cong [H,G]$.  
Let $K, K_1, \dots, K_n$ be the subgroups of $G$ 
which cover $H$ and satisfy $L \cong [K, G]$, and 
$L_i \cong [K_i, G],\; i=1, \dots, n$ (Figure~\ref{fig:parachute} (b)).
Thus, $L$ is an interval in the subgroup lattice of $G$, and,
since $L_i \cong [K_i, G]$, we must have $G\in \sG_i$, by hypothesis.  This is true
for all $1\leq i \leq n$, so $G \in \bigcap\limits_{i=1}^n \sG_i$, which proves that
(B) implies (C).
\end{proof}
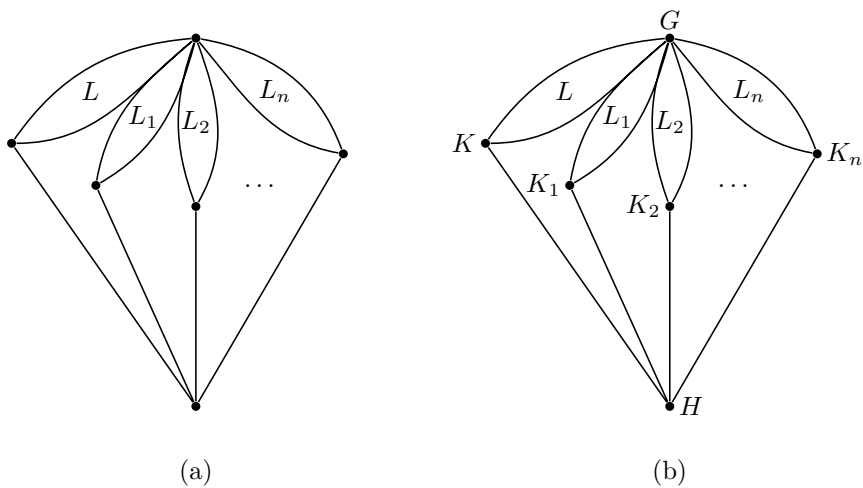
\begin{figure}[centering]
  \caption{The parachute construction.}
  \label{fig:parachute}
\begin{center}
\begin{tikzpicture}[scale=0.7]

  \node (G) at (-8,0) [fill,circle,inner sep=1.2pt] {};
  \node (K) at (-11.5,-2) [fill,circle,inner sep=1.2pt] {};
  \node (K1) at (-9.9,-2.8) [fill,circle,inner sep=1.2pt] {};
  \node (K2) at (-8,-3.2) [fill,circle,inner sep=1.2pt] {};
  \node (Kn) at (-5.2,-2.2) [fill,circle,inner sep=1.2pt] {};
  \node (H) at (-8,-7) [fill,circle,inner sep=1.2pt] {};

\draw (-10,-1) node {$L$};
\draw (-9,-1.5) node {$L_1$};
\draw (-8,-1.6) node {$L_2$};
\draw (-6.5,-1) node {$L_n$};
\draw (-6.75,-2.8) node {$\dots$};

\draw (-8,-8.25) node {(a)};

\draw[semithick] 
   (K) to (H) to (K1)
   (K2) to (H) to (Kn);

\draw [semithick]  
   (G) to [out=-140,in=0] (K)
   (K)  to [out=55,in=185] (G)
   (G) to [out=-105,in=30] (K1)
   (K1) to [out=80,in=-140] (G)
   (G) to [out=-70,in=60] (K2)
   (K2)  to [out=110,in=-110] (G)
   (G) to [out=-10,in=110] (Kn)
   (Kn)  to [out=170,in=-50] (G);


  \node (Gr) at (1,0) [fill,circle,inner sep=1.2pt] {};
  \node (Kr) at (-2.5,-2) [fill,circle,inner sep=1.2pt] {};
  \node (K1r) at (-0.9,-2.8) [fill,circle,inner sep=1.2pt] {};
  \node (K2r) at (1,-3.2) [fill,circle,inner sep=1.2pt] {};
  \node (Knr) at (3.8,-2.2) [fill,circle,inner sep=1.2pt] {};
  \node (Hr) at (1,-7) [fill,circle,inner sep=1.2pt] {};

\draw (-1,-1) node {$L$};
\draw (0,-1.5) node {$L_1$};
\draw (1,-1.6) node {$L_2$};
\draw (2.5,-1) node {$L_n$};
\draw (2.25,-2.8) node {$\dots$};
\draw (1,-8.25) node {(b)};

\draw (Gr) node [above] {$G$}
    (Kr) node [left] {$K$}
    (K1r) node [left] {$K_1$}
    (K2r) node [left] {$K_2$}
    (Knr) node [right] {$K_n$}
    (Hr) node [right] {$H$};

\draw[semithick] 
   (Kr) to (Hr) to (K1r)
   (K2r) to (Hr) to (Knr);

\draw [semithick]  
   (Gr) to [out=-140,in=0] (Kr)
   (Kr)  to [out=55,in=185] (Gr)
   (Gr) to [out=-105,in=30] (K1r)
   (K1r) to [out=80,in=-140] (Gr)
   (Gr) to [out=-70,in=60] (K2r)
   (K2r)  to [out=110,in=-110] (Gr)
   (Gr) to [out=-10,in=110] (Knr)
   (Knr)  to [out=170,in=-50] (Gr);
\end{tikzpicture}
\end{center}
\end{figure}

\noindent {\bf Examples.}
As usual, we let $A_n$ and $S_n$ denote the alternating and symmetric groups on
$n$ letters.  In addition, the following notation will be useful:
\begin{itemize}
\item $\G = $ the class of all finite groups;
\item $\solvable = $ the class of all finite solvable groups;
\item $\giant = \bigcup\limits_{n<\omega} \{A_n, S_n\} = $ the alternating or symmetric groups, 
also known as the ``giant'' groups.
\end{itemize}

It is easy to find a lattice $L$ with the property that
$L \cong [H, G]$ implies $G\notin \solvable$.
We will see an example of such a lattice in
Section~\ref{sec:except-seven-elem}.  (For another example, see~\cite{Palfy:1995}.)
\index{Basile, Alberto}%
In his thesis~\cite{Basile:2001}, Alberto Basile proves a result 
which implies that\footnote{Recall, $M_n$ denotes the $(n+2)$-element lattice with $n$ atoms.}%
$M_6 \cong [H, G]$ only if 
$G\notin \giant$. 
Given these examples and Lemma~\ref{lemma-wjd-1}, it is clear that
(B) holds if and only if for each finite lattice $L$ there 
exist finite groups $H \leq G$ such that $L\cong [H,G]$ and $G$ 
is not solvable, not alternating, and not symmetric.

Now, if our goal is to solve the finite lattice
representation problem, Lemma~\ref{lemma-wjd-1} suggests the following path to a
negative solution:
Find examples of lattices $L_i$ which place restrictions on the $G$ for which $L_i
\cong [H,G]$ can hold, say $G\in \sG_i$, and eventually reach $\bigcap_i
\sG_i = \emptyset$ (at which point we are done).  

We would like to generalize Lemma~\ref{lemma-wjd-1} because it is much easier and more
common to find a class of groups ${\sG_i}$ and a lattice $L_i$ with the following
property: 
\[
\text{If $L_i\cong [H,G]$ \emph{with $H$ core-free in $G$}, then $G\in
  {\sG_i}$}. \qquad (\star)
\]
This leads naturally to the following
question: Given a class of groups $\sG$ and a finite lattice $L$ satisfying ($\star$),   
when can we safely drop the caveat ``with $H$ core-free in $G$'' and get back to the
hypothesis of Lemma~\ref{lemma-wjd-1}?
There is a very simple sufficient condition involving
the class $\sG^c := \{ G \in \G \mid G\notin \sG\}$.  (Recall, if 
$\sK$ is a class of algebras, then $\bH(\sK)$ is the class of homomorphic images of members of $\sK$.)
\begin{lemma}
  \label{lemma-wjd-2}
Let $\sG$ be a class of groups and $L$ a finite lattice such that
\begin{equation}
  \label{eq:100}
L \cong [H,G] \text{ with $H$ core-free} \quad \Rightarrow \quad G\in \sG,
\end{equation}
and suppose $\bH(\sG^c) = \sG^c$. Then,  
\begin{equation}
  \label{eq:200}
L \cong [H,G] \quad \Rightarrow \quad G\in \sG.
\end{equation}
\end{lemma}
\begin{proof}
Suppose $L$ satisfies~(\ref{eq:100}) and $\bH(\sG^c) = \sG^c$, that is, $\sG^c$ is closed under homomorphic
images. (For groups this means if $G \in \sG^c$ and $N\subnormal G$, then
$G/N\in \sG^c$.) 
If~(\ref{eq:200}) fails, then there is a
finite group $G\in \sG^c$ with $L\cong [H,G]$.  Let $N = \core_G(H)$.  Then $L \cong
[H/N,G/N]$ and $H/N$ is core-free in $G/N$ so, by hypothesis~(\ref{eq:100}), $G/N \in \sG$.  But
$G/N \in \sG^c$, since $\sG^c$ is closed under homomorphic images.
\end{proof}
~\\[-6pt]
\noindent{\bf Examples.}  
As mentioned above, there is a lattice $L$ with the property that
$L \cong [H, G]$ implies $G$ is not solvable,
so let $\sG = \solvable^c$.  Then $\sG^c = \solvable$ 
is closed under homomorphic images.  
For the second example above, we have $\sG = \giant^c$,
so $\sG^c = \bigcup_{n<\omega}\{A_n, S_n\}$.
This class is also closed under homomorphic images.  It follows from
Lemma~\ref{lemma-wjd-2} that these examples do not require
the core-free hypothesis.   
In contrast, consider the following result of 
\index{K\"ohler}%
K\"ohler~\cite{Kohler:1983}: If 
$n-1$ is not a power of a prime, then\footnote{Recall, for groups, subdirectly
  irreducible is equivalent to having a unique minimal normal subgroup.} 
\[
M_n \cong [H, G] \text{ with $H$ core-free} \quad \Rightarrow \quad G \text{ is
  subdirectly irreducible.}
\]
Lemma~\ref{lemma-wjd-2} does not apply in this case since $\sG^c$, the class of
subdirectly \emph{reducible} groups, is obviously not closed under homomorphic 
images.\footnote{Every algebra, and in particular every group $G$, has a subdirect
  decomposition into subdirectly irreducibles, $G\leq G/N_1 \times \cdots\times
  G/N_n$.  Thus, there will always be homomorphic images, $G/N_i$, which are
  subdirectly irreducible.} 

Though Lemma~\ref{lemma-wjd-2} seems like a useful observation, the last example
above shows that a 
generalized version of Lemma~\ref{lemma-wjd-1} -- a version based on hypothesis ($\star$) --
would be more powerful, as it would allow us to impose greater restrictions on $G$,
such as those implied by the results of K\"ohler and others.  
Fortunately, the ``parachute''
construction used in the proof of Lemma~\ref{lemma-wjd-1} works in the more general case,
with only a trivial modification to the hypotheses -- namely, the lattices $L_i$
should not be two-element chains (which almost goes without saying in the present context).  
(Recall, $\two$ denotes the two-element chain.)
\begin{lemma}
\label{lemma-wjd-3}
Let $\sG_1, \dots, \sG_n$ be classes of groups and  
suppose that for each $i\in \{1, \dots, n\}$ there is a finite lattice $L_i\ncong \two$
which satisfies the following:  
\[
\text{If $L_i\cong [H,G]$ and $H$ is core-free in $G$, then $G\in
  {\sG_i}$}. \qquad (\star)
\]
Then (B) is equivalent to
  \begin{enumerate}
  \item[(C)] For every finite lattice $L$, there is a finite group $G \in
    \bigcap\limits_{i=1}^n \sG_i$ such that $L \cong [H,G]$.
  \end{enumerate}
\end{lemma}
\begin{proof}
Obviously, (C) implies (B).  Assume (B) and let $L$ be any finite lattice.  Suppose 
$\sG_1, \dots, \sG_n$ and $L_1, \dots, L_n$ satisfy ($\star$) and 
$L_i \ncong \two$ for all $i$.
Note that there is no loss of generality in assuming that $n\geq 2$. 
For if $n=1$, just throw in one of the examples above to make $n=2$.
Call this additional class of groups $\sG_2$.  Then, at the end of the argument, we'll
have $G\in \sG_1\cap \sG_2$, and therefore, $G\in \sG_1$, which is the stated
conclusion of the theorem in case $n=1$.

Construct the lattice $\sP = \sP(L, L_1, \dots, L_n)$ as in the proof of
Lemma~\ref{lemma-wjd-1}.  By (B) there exist finite groups $H \leq G$ with $\sP \cong
[H,G]$, and we can assume without loss of generality that $H$ is
core-free\footnote{This is standard.  For, if $\sP \cong [H,G]$ with
  $N:=\core_G(H)\neq 1$, then $\sP \cong [H/N, G/N]$.} in $G$.
Let $K, K_1, \dots, K_n$ be the subgroups of $G$ which cover $H$ and satisfy $L \cong
[K, G]$, and $L_i \cong [K_i, G], \; 1\leq i \leq n$, as in
Figure~\ref{fig:parachute} (b).
Thus, $L$ is an upper interval in the subgroup lattice of $G$, and it remains to show
that $G\in \bigcap\limits_{i=1}^n \sG_i$.  
This will follow from ($\star$) once we prove that 
each $K_i$ is core-free in $G$.
We now give an easy direct proof this fact,
but we note that it also follows from Lemma~\ref{lemma-wjd-5} below, as well as from a more
general result about \emph{L-P lattices}. (See, e.g., B\"orner~\cite{Borner:1999}.)

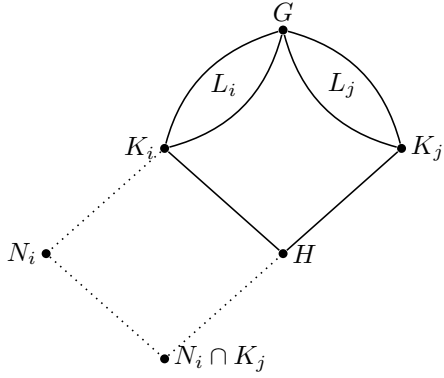
\begin{figure}[!h]
  \centering
\label{fig:isomorphism}  
\begin{tikzpicture}[scale=0.7]
  \node (G) at (0,6.25) [fill,circle,inner sep=1.2pt] {};
  \node (K1) at (-2.25,4) [fill,circle,inner sep=1.2pt] {};
  \node (K2) at (2.25,4) [fill,circle,inner sep=1.2pt] {};
  \node (H) at (0,2) [fill,circle,inner sep=1.2pt] {};
  \node (N1) at (-4.5,2) [fill,circle,inner sep=1.2pt] {};
  \node (N1K2) at (-2.25,0) [fill,circle,inner sep=1.2pt] {};

\draw (-1.125,5.2) node {$L_i$};
\draw (1.125,5.2) node {$L_j$};

 \draw (G) node [above] {$G$}
    (K1) node [left] {$K_i$}
    (K2) node [right] {$K_j$}
    (H) node [right] {$H$}
    (N1) node [left] {$N_i$}
    (N1K2) node [right] {$N_i\cap K_j$};

\draw[semithick, dotted] 
   (K1) to (N1) to (N1K2) to (H);

\draw[semithick] 
   (K1) to (H) to (K2)
   (G) to [out=197,in=75] (K1) 
   (K1) to [out=15,in=-105] (G)
   (G) to [out=-15,in=105] (K2) 
   (K2) to [out=165,in=-75] (G);

\end{tikzpicture}
\caption{The impossibility of a non-trivial core, $N_i = \core_G(K_i)$, in a parachute lattice.}
\end{figure}

For each $i\in\{1, \dots, n\}$, let $N_i = \core_G(K_i)$.  We prove that
$N_i=1$ for all $i$.  Suppose, on the contrary, that $N_i \neq 1$ for some
$i$, and consider any $K_j$ with $j\neq i$.\footnote{This is where we use $n\geq
  2$; though, if $n=1$, we could have used $K$ instead of $K_j$, but then we
  would need to assume $L\ncong \two$.}
A sketch of the part of the subgroup lattice under consideration is shown in
Figure~\ref{fig:isomorphism}. 
Notice that $N_i K_j = G$. For, $N_i$ is not below $H$, since
$H$ is core-free, so $N_i H = K_i$, so
$N_i K_j$ is above both $K_i$ and $K_j$.
Now, clearly,
$N_i\cap K_j\subnormal K_j$, and 
the standard isomorphism theorem implies
\[
K_j/(N_i\cap K_j)
\cong 
N_i K_j/N_i =
G/N_i.
\]
In particular, under this correspondence we have, 
\[
[N_i\cap K_j, K_j] \ni H \mapsto N_i H = K_i \in [N_i, G],
\]
and it follows that the intervals $[K_i, G]$ and $[H, K_j]$ must be isomorphic
as lattices.  However, by construction, 
$H$ is a maximal subgroup of $K_j$, so we have
$[H, K_j]\cong \two \ncong L_i  \cong [K_i, G]$.
This contradiction proves that $\core_G(K_i) = 1$ for all $1\leq i\leq n$, as claimed.
\end{proof}

\section{ISLE properties of groups}
\label{sec:isle-prop-groups}
The previous section motivates the study of what we call
\emph{interval sublattice enforceable} (\ISLE) properties of groups.  In this section we
formalize this concept, as well as some of the concepts introduced above, and we summarize
what we have proved about them.  We conclude with some conjectures that 
will provide the basis for future research.

By a \defn{group theoretical class}, or \defn{class of groups}, we mean a
collection $\sG$ of groups that is closed under isomorphism:
if $G_0\in \sG$ and  $G_1\cong G_0$, then $G_1\in \sG$.
A \defn{group theoretical property}, or simply \defn{property of groups},
is a property $\cP$ such that if a group $G_0$ has property $\cP$ and
$G_1\cong G_0$, then $G_1$ has property $\cP$.\footnote{It seems there
  is no single standard definition of \emph{group theoretical class}.
  While some authors (e.g.,~\cite{Doerk:1992}, \cite{BBE:2006}) use the definition given here,
  others (e.g.~\cite{Robinson:1996}, \cite{Rose:1978}) require that a group
theoretical class contain groups of order~1.}  
Thus if $\sG_{\cP}$ denotes the collection of groups with group theoretical
property $\cP$, then  $\sG_{\cP}$  is a class of groups, and belonging to a
class of groups is a group theoretical property.  Therefore, we need not
distinguish between a property of groups and the class of groups which possess
that property.
A group in the class $\sG$ is called a 
\index{$\mathscr{G}$-group}%
\emph{$\mathscr{G}$-group},
and a group with property $\cP$ is called a 
\index{$\mathscr{P}$-group}%
\emph{$\mathcal{P}$-group}.  Occasionally we write $G \vDash \cP$ to indicate
that $G$ is a $\cP$-group. 

We say that a group theoretical property (or class) $\cP$ is
\index{ISLE}
\emph{interval sublattice enforceable} (\ISLE) 
\index{interval sublattice enforceable (ISLE)}
if there exists a
lattice $L$ such that $L\cong [H,G]$ implies $G$ is a $\cP$-group.
(By the convention agreed upon at the outset of
this chapter, it is implicit in the notation $L \cong [H,G]$ that $G$ is a finite  
group; thus the class $\G$ of all finite groups is trivially an \ISLE\ class.)
We say that the property (or class) $\cP$ is 
\index{cf-ISLE}
\emph{core-free interval sublattice enforceable} (cf-\ISLE)
if there exists a lattice $L$ such that if $L\cong [H,G]$ with $H$ core-free in
$G$, then $G$ is a $\cP$-group.  

Clearly, if $\cP$ is \ISLE, then it is also
cf-\ISLE, and Lemma~\ref{lemma-wjd-2} above gives a sufficient condition for
the converse to hold.  We restate this formally as follows:

\vskip2mm

\noindent {\bf Lemma~\ref{lemma-wjd-2}${}'$.}
If $\cP$ is cf-\ISLE\ and if $\sG_{\cP}^c = \{G\in \G \mid G\nvDash \cP\}$
is closed under homomorphic images, $\bH(\sG_{\cP}^c) = \sG_{\cP}^c$, then $\cP$ is \ISLE.

\vskip2mm

As we noted in the previous section, two examples of \ISLE\ classes are
\begin{itemize}
\item $\sG_0 = \mathfrak{S}^c = $ the finite non-solvable groups;
\item $\sG_1 =(\mathfrak{Gi})^c = $ the finite non-giant groups, 
$\{G\in \G \mid (\forall n<\omega) \; (G \neq A_n \text{ and }  G\neq S_n) \}$;
\end{itemize}
The following classes are at least cf-\ISLE:\footnote{The symbols we use to denote these classes are not standard.}
\begin{itemize}
\item $\sG_2 = $ the finite subdirectly irreducible groups;
\item $\sG_3 = $ the finite groups having no nontrivial abelian normal subgroups.
\item $\sG_4 = \{G\in \G \mid C_G(M) = 1 \text{ for a minimal normal
  subgroup } M\subnormal G\}$
\end{itemize}
Note that $\sG_4 \subset \sG_2\cap \sG_3 \subset \sG_0$.

Given two (group theoretical) properties $\cP_1, \cP_2$, we write
$\cP_1 \rightarrow \cP_2$ to denote that property 
$\cP_1$ implies property $\cP_2$. In other words,
$G\vDash \cP_1$ only if $G\vDash \cP_2$.
Thus $\rightarrow$ provides a natural partial order on any given set of 
properties, as follows:
\[
\cP_1 \leq \cP_2  \iff \cP_1 \rightarrow \cP_2 \iff \sG_{\cP_1}\subseteq
\sG_{\cP_2},
\]
where $\sG_{\cP_i} = \{G\in \G \mid G\vDash \cP_i\}$.
The following is an obvious corollary of the parachute construction. 
\begin{corollary}
\label{cor:isle-prop-groups-1}
  If $P = \{\cP_i \mid i\in \sI\}$ is a collection of (cf-)\ISLE\ properties,
  then $\Meet P$ is (cf-)\ISLE.
\end{corollary}
Note: the conjunction $\Meet \cP$ corresponds to the class $\{G \in \G \mid (\forall i \in \sI) \; G\vDash \cP_i \}$.

It is clear from the foregoing that if solvability were an \ISLE\ property then
we would have a solution to the \FLRP.  But solvability is obviously not \ISLE.
For, if $L\cong [H, G]$ then for any non-solvable group $K$ we have $L\cong
[H\times K, G\times K]$, and of course $G\times K$ is not solvable.  
Notice, however, that $H\times K$ is not core-free, so a more interesting
question to ask might be whether solvability is a cf-\ISLE\ property. 
The following lemma proves that this is not the case.  
\begin{lemma}
  \label{lem:ISLE-must-have-wreaths}
Let $\cP$ be a cf-\ISLE\ property, and let $L$ be a finite lattice such that 
$L\cong [H,G]$ with $H$ core-free implies $G\vDash \cP$.  Also, suppose there
exists a group $G$ witnessing this; that is, $G$ has a core-free subgroup
$H$ with $L\cong [H,G]$.   
Then, for any finite nonabelian simple group $S$, there exists a wreath product group
of the form $W = S\wr \bar{U}$ that is also a $\cP$-group.
\end{lemma}

\begin{proof}
\index{Kurzweil, Hans}
  We apply the idea of Kurzweil twice
  (cf.~Theorem~\ref{thm:duals-interv-subl}).  Fix a finite nonabelian simple
  group $S$, and suppose the index of $H$ in $G$ is $|G:H| = n$.
  Then the action of $G$ on the cosets of $H$ induces an automorphism of the
  group $S^n$ by permutation of coordinates.  Denote this representation by
  $\phi: G \rightarrow \Aut(S^n)$, 
  and let the image of $G$ be $\phi(G) =
  \bar{G} \leq \Aut(S^n)$.  
\index{wreath product}%
  The semidirect product (or wreath product) under this action is the group
  \[
  U:= S\wr_\phi G = S^n \rtimes_\phi G = S^n \rtimes \bar{G} =  S\wr \bar{G},
  \]
  with multiplication given by
  \[
  (s_1, \dots, s_n, x) (t_1, \dots, t_n, y) = 
  (s_1 t_{x(1)}, \dots, s_nt_{x(n)}, x y),
  \]
  for $s_i, t_i \in S$ and $x, y \in \bar{G}$.
  An illustration of the subgroup lattice of such a wreath product appears in Figure~\ref{fig:kurzweil}.
  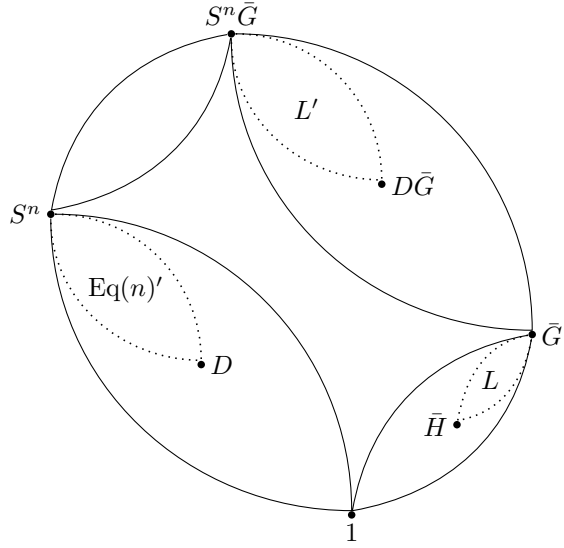
\begin{figure}[!h]
\begin{center}
  \begin{tikzpicture}[scale=.8]
    \node (G) at (3,3) [fill,circle,inner sep=1pt] {}; 
    \draw (G) node [right] {$\bar{G}$};
    \node (H) at (1.75,1.5) [fill,circle,inner sep=1pt] {}; 
    \draw (H) node [left] {$\bar{H}$};
    \node (Sn) at (-5,5) [fill,circle,inner sep=1pt] {}; 
    \draw (Sn) node [left] {$S^n$};
    \node (D) at (-2.5,2.5) [fill,circle,inner sep=1pt] {}; 
    \draw (D) node [right] {$D$};
    \node (DG) at (0.5,5.5) [fill,circle,inner sep=1pt] {}; 
    \draw (DG) node [right] {$D \bar{G}$};
    \node (1) at (0,0) [fill,circle,inner sep=1pt] {}; 
    \draw (1) node [below] {$1$};
    \node (SnG) at (-2,8) [fill,circle,inner sep=1pt] {}; 
    \draw (SnG) node [above] {$S^n \bar{G}$};
   \draw
    (G) to [out=190,in=80] (1) to [out=10,in=-100] (G)
    (Sn) to [out=0,in=90] (1) to [out=180,in=-90] (Sn)
    (SnG) to [out=190,in=80] (Sn) to [out=10,in=-100] (SnG)
    (SnG) to [out=0,in=90] (G) to [out=180,in=-90] (SnG);
    \draw[dotted, semithick]
    (G) to [out=190,in=80] (H) to [out=10,in=-100] (G)
    (SnG) to [out=0,in=90] (DG) to [out=180,in=-90] (SnG)
    (Sn) to [out=0,in=90] (D) to [out=180,in=-90] (Sn);
    \draw 
    (-3.75,3.75) node {$\Eq(n)'$}
    (-.75,6.75) node {$L'$}
    (2.3,2.25) node {$L$};
  \end{tikzpicture}
\end{center}
    \caption{Representation of the dual of a group representable lattice.}
    \label{fig:kurzweil}
  \end{figure}
  The dual lattice $L'$ is an upper interval in the subgroup lattice of this group, namely,
  $L'\cong [D\rtimes \bar{G}, U]$.
  (As usual, $D$ denotes the diagonal subgroup of $S^n$.)  It is important to
  note that if $H$ is core-free in $G$ -- equivalently, if $\ker \phi = 1$
  -- then the foregoing construction results in the subgroup $D\rtimes \bar{G}$ being
  core-free in $U$.   (We postpone the proof of this fact.)

  Now if we repeat the foregoing procedure, with $H_1 := D\rtimes \bar{G}$ denoting the
  (core-free) subgroup of $U$ such that $L' \cong [H_1, U]$, then we find that
  $L = L''\cong [D_1\rtimes \bar{U}, S^m\rtimes
    \bar{U}]$, where $m = |U:H_1|$.\footnote{Here we use $D_1$ to denote the
    diagonal subgroup of $S^m$ to distinguish it from $D$, the diagonal subgroup
    of $S^n$.} 
 Assuming $D_1\rtimes \bar{U}$ is core-free in $W = S^m\rtimes \bar{U}$, then, 
it follows by the original hypothesis that $W$ must be a $\cP$-group.

  To complete the proof, we check that starting with a core-free subgroup
  $H \leq G$ in the Kurzweil construction just described results in a
  core-free subgroup $D\rtimes \bar{G} \leq U$.   Let $N = \core_U(D\rtimes
  \bar{G})$.  Then, for all $n=(d,\dots, d, x) \in N$ and for all 
  $u = (t_1,\dots, t_n, g)\in U$, we have $unu^{-1}\in N$.  In particular, we
  are free to choose
  $t_1 = t_2$, all other $t_k$ distinct, and $g=1$.  Then 
  \[
  u n u^{-1} = (t_1,\dots, t_n, 1) (d, \dots, d, x) (t_1^{-1},\dots, t_n^{-1}, 1)
  =(t_1 d \,t_{x(1)}^{-1},\dots, t_nd \,t_{x(n)}^{-1}, 1) \in N.\]
  Therefore, $t_1 d\, t_{x(1)}^{-1} = \cdots = t_nd \,t_{x(n)}^{-1}$. With $t_1 = t_2$
  and all other $t_k$ distinct, it's clear that $x$ must stabilize the set $\{1,2\}$.
  Of course, the same argument applies in case $t_1 = t_3$ with all other $t_k$
  distinct,\footnote{Note that we can be sure $|G:H| = n > 2$, since $|G:H|=2$
    would imply $H\subnormal G$, which contradicts that $H$ is core-free in $G$.}
  so we conclude that $x$ stabilizes the set $\{1, 3\}$ as well.  Therefore, $x(i)
  = i$, for $i=1, 2, 3$.  Since the same argument works for all $i$, we see that
  $n=(d,\dots,d, x) \in N$ implies $x\in \ker \phi = 1$.  This puts $N$ below
  $D\times 1$, and the only normal subgroup of $U$ that lies 
  below $D\times 1$ is the trivial subgroup.
\end{proof}
The foregoing result enables us to conclude that any class of groups that does
not include wreath products of the form $S\wr G$ for all finite simple groups
$S$ cannot be a cf-\ISLE\ class. 

\vskip2mm

We conclude this section with the following two equivalent conjectures:
\begin{conjecture}
\label{conjecture:isle-prop}
  If $\cP$ is a (cf-)\ISLE\ property, then $\neg \cP$ is not a (cf-)\ISLE\ property.
\end{conjecture}
\begin{conjecture}
\label{conjecture:isle-prop2}
If $\sG$ is a (cf-)\ISLE\ class, then $\sG^c$ is not a (cf-)\ISLE\ class.
\end{conjecture}
\noindent A pair of lattices witnessing the failure of either of these conjectures 
would solve the \FLRP.  More precisely, if $\sG$ is a class and $L_0$ and $L_1$ are
lattices such that 
\[
L_0 \cong [H, G] \; \Rightarrow \; G\in \sG \quad \text{ and } \quad
L_1 \cong [H, G] \; \Rightarrow \; G\in \sG^c
\]
Then the parachute lattice $\sP(L_0, L_1)$ is not an interval in the subgroup
lattice of a finite group.

\section{Dedekind's rule}
\label{sec:dedekinds-rule}

We prove a few more lemmas which lead to additional constraints on any group which has a
non-trivial parachute lattice as an upper interval in its subgroup lattice.  
We will need the following standard theorem\footnote{See, for example, page~122
  of Rose, \emph{A Course on Group Theory}~\cite{Rose:1978}.} which
we refer to as
\defn{Dedekind's rule}:
\index{Dedekind's rule|(} 
\begin{theorem}[Dedekind's rule]
  \label{lemma-dedekind}
Let $G$ be a group and let $A, B$ and $C$ be subgroups of $G$ with $A\leq B$.  Then,
\begin{align}
\label{eq:dedekind1}
A(C\cap B) &= AC \cap B,\qquad \text{ and }\\
\label{eq:dedekind2}
(C\cap B)A &= CA \cap B.
\end{align}
\end{theorem}

Our next lemma (Lemma~\ref{lemma-wjd-4}) is a slight variation on a standard
result that we find very useful. 
The standard result is essentially part (ii) of Lemma~\ref{lemma-wjd-4}.
Surely part (i) of the lemma is also well known, though we have not seen it
elsewhere. We will see that the standard result is powerful enough to answer all
of our questions about parachute lattices, but later, in
Section~\ref{sec:except-seven-elem}, we make use of (i) in a situation where
(ii) does not apply.  

To state Lemma~\ref{lemma-wjd-4}, we need some new notation.  
Let $U$ and $H$ be subgroups of a group,
let $U_0 := U\cap H$, and consider the interval $[U_0, U]:=\{ V \mid U_0 \leq V
\leq U\}$.   
In general, when we write $UH$ we mean the \emph{set}
$\{ u h \mid u\in U, h\in H\}$, and we write $U \join V$ or $\<U,H\>$ to mean
the group generated by $U$ and $H$.  Clearly $UH \subseteq \<U,H\>$. 
Equality holds if and only if $U$ and $H$ permute, that is, $U H = H U$.
In any case, it is often helpful to visualize part of the subgroup lattice of
$\<U,H\>$, as shown below.

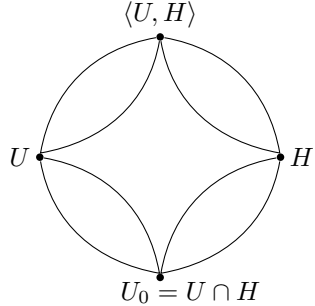
\begin{center}
  \begin{tikzpicture}[scale=.4]
    \node (H) at (4,4) [fill,circle,inner sep=1pt] {}; \draw (H) node [right] {$H$};
    \node (U) at (-4,4) [fill,circle,inner sep=1pt] {}; \draw (U) node [left] {$U$};
    \node (U0) at (0,0) [fill,circle,inner sep=1pt] {}; 
    \draw (1,-.5) node {$U_0 = U\cap H$};
    \node (UH) at (0,8) [fill,circle,inner sep=1pt] {}; \draw (UH) node [above] {$\<U,H\>$};
   \draw
    (H) to [out=190,in=80] (U0) to [out=10,in=-100] (H)
    (U) to [out=-10,in=100] (U0) to [out=170,in=-80] (U)
    (UH) to [out=190,in=80] (U) to [out=10,in=-100] (UH)
    (UH) to [out=-10,in=100] (H) to [out=170,in=-80] (UH);
  \end{tikzpicture}
\end{center}
Recall that the usual isomorphism theorem for groups implies
that 
if $H$ is a normal subgroup of $\<U, H\>$, 
then the interval
$[H, \<U, H\>]$ is isomorphic to the interval $[U\cap H, U]$.  The 
purpose of the next lemma is to relate these two
intervals in cases where we drop the assumption $H\subnormal \<U,H\>$
and add the assumption $UH = \<U,H\>$.

If the two subgroups $U$ and $H$ permute, then we define 
\begin{equation}
  \label{eq:dedekind-1}
[U_0, U]^H := \{ V\in [U_0,U] \mid VH = HV\},
\end{equation}
which consists of those subgroups $V$ in $[U_0, U]$ that permute with
$H$. 

If $H$ normalizes $U$ (which implies $UH=HU$), 
then we 
define
\begin{equation}
  \label{eq:dedekind-2}
[U_0, U]_H := \{ V\in [U_0,U] \mid H\leq N_G(V)\},
\end{equation}
where $G:=UH$.  This is the set consisting of those subgroups $V$ in $[U_0, U]$ that
are normalized by $H$.  The latter are sometimes called 
\index{invariant subgroup} 
\emph{$H$-invariant subgroups}.
Notice that to even
define $[U_0, U]_H$ we must have $H\leq N_G(U)$, and in this case, as we will
see below, the two sublattices coincide: 
$[U_0, U]_H  =  [U_0, U]^H$.  

We are finally ready to state the main result relating the sets defined
in~(\ref{eq:dedekind-1}) and (\ref{eq:dedekind-2}) (when they
exist) to the interval $[H, UH]$.
\begin{lemma}
  \label{lemma-wjd-4}
Suppose $U$ and $H$ are permuting subgroups of a group. 
Let $U_0 := U\cap H$.  Then
\begin{enumerate}[(i)]
\item $[H, UH]  \cong  [U_0, U]^H \leq [U_0, U]$.
\item If $U \subnormal UH$, then  $[U_0, U]_H  = [U_0, U]^H \leq [U_0, U]$.
\item If $H \subnormal UH$,  then  $[U_0, U]_H  = [U_0, U]^H = [U_0, U]$.
\end{enumerate}
\end{lemma}
\begin{remarks}
Since $G=UH$ is a group, the hypothesis of (ii) is equivalent to
$H\leq N_G(U)$, and the hypothesis of (iii) is equivalent to $U\leq N_G(H)$.
Part (i) of the lemma says that when two subgroups permute, we can
identify the interval above either one of them with the sublattice of
subgroups below the other that permute with the first.
Part (ii) is similar except we identify the interval above $H$ with
the  sublattice of $H$-invariant subgroups below $U$.  Once we have proved (i), the
proof of (iii) follows trivially from the standard isomorphism theorem for
groups, so we omit the details.
\end{remarks}

\begin{proof}
To prove (i), we show that the following maps are inverse order isomorphisms:
\begin{align}
\label{eq:inverse-isos}
\phi: \;& [H, UH] \ni X \mapsto U\cap X \in [U_0, U]^H\\
\psi: \;& [U_0, U]^H \ni V \mapsto VH \in [H, UH].\nonumber
\end{align}
Then we show that $[U_0, U]^H$ is a sublattice of $[U_0,U]$, that is, 
$[U_0, U]^H\leq [U_0,U]$.

Fix $X\in [H, UH]$. We claim that $U\cap X \in [U_0, U]^H$. Indeed,
\[
    (U\cap X) H = UH \cap X=
               HU \cap X
                = H(U \cap X).
\]
The first equality holds by~(\ref{eq:dedekind2}) since $H\leq X$, the second holds
by assumption, and the third by~(\ref{eq:dedekind1}).
  This proves $U\cap X \in [U_0, U]^H$.  Moreover, by the first equality,
$\psi \circ \phi (X) = (U\cap X)H =UH \cap X = X$,
so $\psi \circ \phi$ is the identity on $[H, UH]$.

If $V\in [U_0, U]^H$, then $VH = HV$ implies $VH \in [H, UH]$. Also, $\phi \circ
\psi$ is the identity on $[U_0, U]^H$, since $\phi \circ \psi(V)= VH \cap U =
V(H\cap U)= VU_0 = V$, by~(\ref{eq:dedekind1}). 
This proves that $\phi$ and $\psi$ are inverses of each other on the sets indicated, and
it's easy to see that they are order preserving:
$X\leq Y$ implies $U\cap X \leq U\cap Y$, and $V\leq W$ implies $VH \leq WH$.
Therefore, $\phi$ and $\psi$ are inverse order isomorphisms.

To complete the proof of (i), we show that
$[U_0, U]^H$ is a sublattice of $[U_0, U]$.  Suppose $V_1$ and $V_2$ are
subgroups in $[U_0, U]$ which permute with $H$.  
It is easy to see that their join $V_1 \join V_2 = \<V_1, V_2\>$ also permutes
with $H$, so we just check that their intersection permutes with $H$.  Fix
$x \in V_1 \cap V_2$ and $h\in H$.  We show $xh = h'x'$ for some $h'\in H, \, x'
\in V_1\cap V_2$. Since $V_1$ and $V_2$ permute with $H$, we have $xh = h_1 v_1$
and $xh = h_2 v_2$ for some $h_1, h_2\in H, \, v_1 \in V_1, \, v_2 \in V_2$.
Therefore, $h_1 v_1 = h_2 v_2$, which implies $v_1 = h_1^{-1}h_2 v_2 \in HV_2$,
so $v_1$ belongs to $V_1 \cap HV_2$. Note that $V_1 \cap HV_2$ is below both $V_1$ and
$U\cap HV_2 = \phi \psi(V_2) = V_2$.  Therefore, $v_1 \in V_1 \cap HV_2 \leq V_1
\cap V_2$, and we have proved that $xh = h_1 v_1$ for $h_1\in H$ and $v_1 \in
V_1\cap V_2$, as desired. 

To prove (ii), assuming $U\subnormal G$, we show that if $U_0 \leq V \leq U$,
then $VH = HV$ if and only if $H\leq N_G(V)$.
If $H\leq N_G(V)$, then $VH = HV$ (even when $U \notsubnormal G$).
Suppose $VH = HV$.  We must show $(\forall v\in V)\, (\forall h\in H)\; hvh^{-1}\in
V$.  Fix $v\in V, \, h\in H$.  Then, $hv = v'h'$ for some $v'\in V,\, h'\in H$, since
$VH = HV$.  Therefore, $v' h' h^{-1} = hvh^{-1} = u$ for some $u\in U$, since
$H\leq N_G(U)$. This proves that $hvh^{-1}\in VH\cap U = V(H\cap U) = VU_0 = V$, as
desired.
\end{proof}

Next we prove that 
any group which has a nontrivial parachute lattice as an upper interval
in its subgroup lattice must have some rather special properties.  
\begin{lemma}
\label{lemma-wjd-5}
 Let $\sP = \sP(L_1, \dots, L_n)$ with $n\geq 2$ and $|L_i|>2$ for all
$i$, and suppose $\sP \cong [H, G]$, with $H$ core-free in $G$.  
\begin{enumerate}[(i)]
\item If $1\neq N \subnormal G$, then $NH = G$.
\item If $M$ is a minimal normal subgroup of $G$, then $C_G(M)=1$.
\item $G$ is subdirectly irreducible.
\item $G$ is not solvable.
\end{enumerate}
\end{lemma}
\begin{remark}
If a subgroup $M\leq G$ is abelian, then $M \leq C_G(M)$, so (ii) implies
that a minimal normal subgroup (hence, every normal subgroup) of $G$ must be
nonabelian.  
\end{remark}
\begin{proof}
(i) Let $1\neq N \subnormal G$.  Then $N \nleq H$, since $H$ is core-free in $G$.
Therefore, $H < NH$.   As in Section~\ref{sec:parachute-lattices}, we let $K_i$
denote the subgroups of $G$ 
corresponding to the atoms of $\sP$.  
Then $H$ is covered by each $K_i$, so $K_j\leq NH$ for some $1\leq j\leq n$.  
Suppose, by way of contradiction, that $NH < G$.  
By assumption, $n\geq 2$ and $|L_i|>2$.  Thus for any $i\neq j$ we have
$K_i\leq Y < Z < G$ for some subgroups $Y$ and $Z$ which satisfy
$(NH)\cap Z = H$ and $(NH)\join Y = G$.  Also, $(NH)Y = NY$ is a group, so
$(NH)Y=NH\join Y = G$.  But then, by Dedekind's rule, we have
\[
Y = HY = ((NH)\cap Z) Y = (NH)Y \cap Z = G\cap Z = Z,
\]
contrary to $Y<Z$.  This contradiction proves that $NH = G$.
\\[8pt]
\index{Dedekind's rule|)}%
(ii) If $C_G(M)\neq 1$, then (i) implies $C_G(M)H = G$,
since $C_G(M) \subnormal N_G(M) =G$. 
Consider any $H< K < G$. Then $1 < M\cap K < M$ (strictly, by
Lemma~\ref{lemma-wjd-4}). Now $M\cap K$ is normalized by $H$ and centralized
(hence normalized) by $C_G(M)$.  (Indeed, $C_G(M)$ centralizes every subgroup of
$M$.) Therefore, $M\cap K \subnormal C_G(M)H = G$, contradicting the minimality of
$M$. 
\\[8pt]
(iii) We prove that $G$ has a \emph{unique} minimal normal subgroup.  Let $M$ be a
minimal\footnote{If $G$ is simple, then $M = G$; ``minimal'' assumes
  nontrivial.} normal subgroup of $G$ and let $N \subnormal G$ be any normal subgroup not 
containing $M$.  We show that $N = 1$.  Since both subgroups
are normal, the \emph{commutator}\footnote{The \emph{commutator of $M$ and $N$} is the subgroup
generated by the set $\{mnm^{-1}n^{-1} \mid m\in M, n\in N\}.$
The \emph{commutator of $M$} is the subgroup generated by
 $\{a b a^{-1} b^{-1}: a, b \in M\}$.  The \emph{$n^{th}$ degree
commutator of $M$}, denoted $M^{(n)}$, is defined recursively as the commutator of
$M^{(n-1)}$. A group $M$ is \emph{solvable} if $M^{(n)} = 1$ for some $n \in \N$.}
 of $M$ and $N$ 
lies in the intersection $M\cap N$, which is trivial by the minimality of $M$.   
Thus, $M$ and $N$ centralize each other.  In particular,
$N \leq C_G(M) = 1$, by (ii).
\\[8pt]
(iv) Let $M'$ denote the commutator of $M$.   As remarked above, $M$ is
nonabelian, so $M' \neq 1$.  Also, $M' \subnormal M
\subnormal G$, and $M'$ is a \emph{characteristic} subgroup of $M$ (i.e.,
$M'$ invariant under $\Aut(M)$).  Therefore, $M'\subnormal G$, and, 
as $M$ is a \emph{minimal} normal subgroup of $G$, we have $M' = M$.  Thus, $M$ is
not solvable, so $G$ is not solvable.
\end{proof}
\begin{remark}
It follows from (i) that, if $\sP$ is a nontrivial parachute lattice
with $\sP \cong [H, G]$, where $H$ is core-free, then $\core_G(X) = 1$ for every $H
\leq X < G$.  This gives a second way to complete the proof of Lemma~\ref{lemma-wjd-3}.
\end{remark}

To summarize what we have thus far, the lemmas above imply that (B) holds if and only if
every finite lattice is an interval $[H, G]$, with $H$ core-free in $G$, where
\begin{enumerate}[(i)]
\item $G$ is not solvable, not alternating, and not symmetric;
\item $G$ has a unique minimal normal subgroup $M$ which satisfies $MH = G$
and $C_G(M) = 1$; in particular,
$M$ is nonabelian and $\core_G(X) = 1$ for all $H\leq X < G$.
\end{enumerate}

Finally, we note that Theorem 4.3.A of Dixon and Mortimer~\cite{Dixon:1996}
describes the structure of the unique minimal normal subgroup as follows:
\begin{enumerate}[(i)]
\item[(iii)] $M = T_0\times \cdots \times T_{r-1}$, where $T_i$ are simple minimal normal subgroups of
  $M$ which are conjugate (under conjugation by elements of $G$). Thus, $M$ is a
  direct power of a simple group $T$.
\end{enumerate}
  In fact, when $C_G(M)=1$, as in our application,
we can specify these conjugates more precisely. 
Let $T$ be any minimal normal subgroup of $M$. Note that $T$ is simple.
Let $N = N_H(T) = \{h\in H \mid T^h = T\}$ be the normalizer of $T$ in
$H$.  Then the proof of the following lemma is routine, so we omit it.
\begin{lemma}
If $H/N = \{N, h_1N, \dots, h_{k-1}N\}$ is a full set of left cosets of $N$
in $H$, then $k=r$ and $M = T_0\times \cdots \times T_{r-1} = T \times
T^{h_1} \times T^{h_{r-1}}$. 
\end{lemma}

\index{B\"order, Ferdinand}%
\index{Baddeley, Robert}%
\index{Lucchini, Adrea}%
We conclude this chapter by noting that other researchers, such as Baddeley,
B\"orner, and Lucchini, have proved similar results 
for the more general case of \emph{quasiprimitive permutation groups}. 
In particular, our proof of Lemma~\ref{lemma-wjd-5} (i) uses the same argument
as the one in \cite{Borner:1999}, where it is used to prove Lemma 2.4: if $L
\cong [H, G]$ is an \defn{LP-lattice},\footnote{An LP-lattice is one in which
  every element except $0$ and $1$ is a non-modular element.}
then $G$ must be a quasiprimitive permutation 
group.  
We remark that parachute lattices, in which each panel
 $L_i$ has $|L_i|>2$, are LP-lattices, so
Lemma~\ref{lemma-wjd-5} follows from
theorems of Baddeley, B\"orner, Lucchini, et al. 
(cf.~\cite{Lucchini:1997}, ~\cite{Borner:1999}).

However, the main purpose of the parachute construction, besides providing
a quick route to Lemma~\ref{lemma-wjd-5},
is to demonstrate a natural way to insert arbitrary finite lattices $L_i$ as
upper intervals $[K_i, G]$ in $\Sub[G]$, {\it with $K_i$ core-free in $G$}.
Then, once we prove special properties of
groups $G$ for which $L_i = [K_i, G]$ ($K_i$ core-free), 
it follows that \emph{every} finite lattice $L$ must be an upper
interval $L = [K, G]$ for some $G$ satisfying all of these properties,
assuming the \FLRP\ has a positive answer.  This forms the basis and motivation 
for the idea of (cf-)\ISLE\ properties, as discussed in
Section~\ref{sec:isle-prop-groups}. 

\chapter{Lattices with at Most Seven Elements}
\label{cha:lattices-with-at}
\section{Introduction}
In the spring of 2011, our research seminar was fortunate enough to have 
as a visitor 
\index{Jipsen, Peter}%
Peter Jipsen, who initiated the project of cataloging
every small finite lattice $L$ for which there is a known finite algebra 
$\bA$ with $\Con\bA\cong L$. 
It is well known that all lattices with at most six elements are representable.
In fact, these can be found as intervals in subgroup lattices of finite groups,
but this fact was not known until recently.  

\index{Watatani, Yasuo}%
By 1996, Yasuo Watatani had found each six-element lattice, except for the
two lattices appearing below, as intervals in subgroup lattices of finite
groups.  
See~\cite{Watatani:1996}.
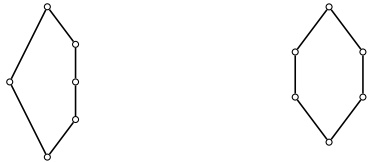
\begin{multicols}{2}
\hskip4cm
  \begin{tikzpicture}[scale=.5]
          \node (41) at (3.25,1)  [draw, circle, inner sep=\dotsize] {};
      \node (45) at (3.25,5)  [draw, circle, inner sep=\dotsize] {};
      \node (42) at (4,2)  [draw, circle, inner sep=\dotsize] {};
      \node (43) at (4,3)  [draw, circle, inner sep=\dotsize] {};
      \node (44) at (4,4)  [draw, circle, inner sep=\dotsize] {};
      \node (33) at (2.25,3)  [draw, circle, inner sep=\dotsize] {};
      \draw[semithick] (41) to (42) to (43) to (44) to (45) to (33) to (41);

  \end{tikzpicture}
  \par \vfill \columnbreak
  \begin{tikzpicture}[scale=.6]
    
      \node (52) at (6.5,1)   [draw, circle, inner sep=\dotsize] {};
      \node (72) at (8,1)   [draw, circle, inner sep=\dotsize] {};
      \node (53) at (6.5,2)   [draw, circle, inner sep=\dotsize] {};
      \node (73) at (8,2)    [draw, circle, inner sep=\dotsize] {};
      \node (61) at (7.25,0)    [draw, circle, inner sep=\dotsize] {};
      \node (64) at (7.25,3)    [draw, circle, inner sep=\dotsize] {};
      \draw[semithick] (61) to (52) to (53) to (64) to (73) to (72) to (61);

  \end{tikzpicture}
\end{multicols}
\index{Aschbacher, Michael}%
Then, in 2008, Michael Aschbacher
showed in~\cite{Aschbacher:2008} how to construct some (very large) twisted wreath product 
groups that have the lattices above as intervals in their subgroup lattices.
Note that, although it was apparently quite difficult to find \emph{group}
representations of the lattices shown above, it is quite easy to
represent them concretely as the lattices of congruences of very small finite
algebras.  Take, for example, the set $X=\{0,1,\dots, 6\}$ and consider the
lattice $L\leq \Eq(X)$ generated by the partitions
  \[
|0,3,4|1,6|2,5| \; \text{ and } \;  |0,6|1,5|2|3|4|  \; \leq \; |0,6|1,4,5|2|3|\; \leq \;|0,6|1,4,5|2,3|.
  \]
This concrete representation of the lattice on the left above happens to be closed:
$\rho \lambda(L) = L$, so it is equal to the congruence lattice $\Con\<X, \lambda(L)\>$.  

We prove two main results in this chapter.  The first is
\begin{theorem}
\label{thm:sevenelementlattices}
  Every finite lattice with at most seven elements, with one possible exception,
  is representable as the congruence lattice of a finite algebra.
\end{theorem}
The second result concerns the one possible exception of this theorem, 
a seven element lattice, which we call $L_7$.  It is the focus of
Section~\ref{sec:except-seven-elem}.
As we explain below, 
if $L_7$ is representable as the congruence lattice of a finite
algebra, then it must appear as an interval in the subgroup lattice of a finite
group.\footnote{Note that the result of \Palfy\ and \Pudlak\ does \emph{not} say
  that every representable lattice is isomorphic to an interval in a subgroup
  lattice of a finite group.  Rather, it is a statement about the whole class of
  representable lattices.  However, for certain lattices, such as the one described
in Section~\ref{sec:except-seven-elem}, we can prove that it belongs to $\sL_3$ if and only if it belongs to $\L_4$.}
 Our main result, Theorem~\ref{thm:except-seven-elem}, places some fairly
strong restrictions on such a group.  Our motivation is to apply this
new theorem, along with some well known theorems classifying finite groups, to
eventually either find such a group or prove that none exists.  This application
will be the focus of future research.

\section{Seven element lattices}
\label{sec:seven-elem-latt}
In this section we show that, with one possible exception (discussed in the next
section), every lattice with at 
most seven elements is representable as a congruence lattice of a finite algebra.
There are 53 lattices with at most seven elements.\footnote{The Hasse diagrams
  of all lattices with at most seven elements are shown
  here~\url{http://db.tt/2qJUkoaG}
or alternatively here~\url{http://math.chapman.edu/~jipsen/mathposters/lattices7.pdf} (courtesy of 
Peter Jipsen).}
Representations for most of these lattices can be found 
quite easily by applying the methods described in previous chapters.
The easiest, of course, are the distributive lattices, which we know are
representable by Theorem~\ref{thm:distr-lattices}.
Some others are found to be representable by searching (with a computer) for closed concrete
representations $L \leq \Eq(X)$ over some small set $X$, say $|X|<8$.
Still others are found by checking that they are obtained by applying
operations under which $\sL_3$ is closed (\S~\ref{sec:clos-prop-class}).  For
\index{closure properties!applied}%
example, the lattice on the left in Figure~\ref{fig:ordinal-and-parallel-ex} is
\index{parallel sum!applied}%
\index{ordinal sum!applied}%
the ordinal sum of two copies of the distributive lattice $\2 \times \2$.  On
the right of the same figure is the parallel sum of the distributive lattices
$\2$ and $\3$. 
\begin{figure}[h!]
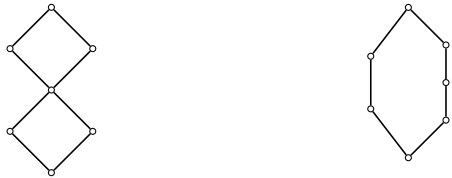

\begin{multicols}{2}
\hskip4cm
  \begin{tikzpicture}[scale=.55]
    \input{OrdinalSumExample.tex}
  \end{tikzpicture}
  \par \vfill \columnbreak
\hskip1cm
  \begin{tikzpicture}[scale=.5]
    \input{ParallelSumExample.tex}
  \end{tikzpicture}
\end{multicols}
  \caption{The ordinal sum of $\2\times \2$ with itself (left) and the parallel
    sum of $\2$ and $\3$ (right).}
  \label{fig:ordinal-and-parallel-ex}
\end{figure}

Using these methods, it was not hard to find, or at least prove the existence of,
congruence lattice representations of all seven element lattices except for the
seven lattices appearing in Figure~\ref{fig:Mysevens}, plus their duals.  Four of these
seven are self-dual, so there are ten lattices in total for which a
representation is not relatively easy to find.\footnote{The names of these lattices
  do not conform to any well established naming convention.} 

\begin{figure}[h!]
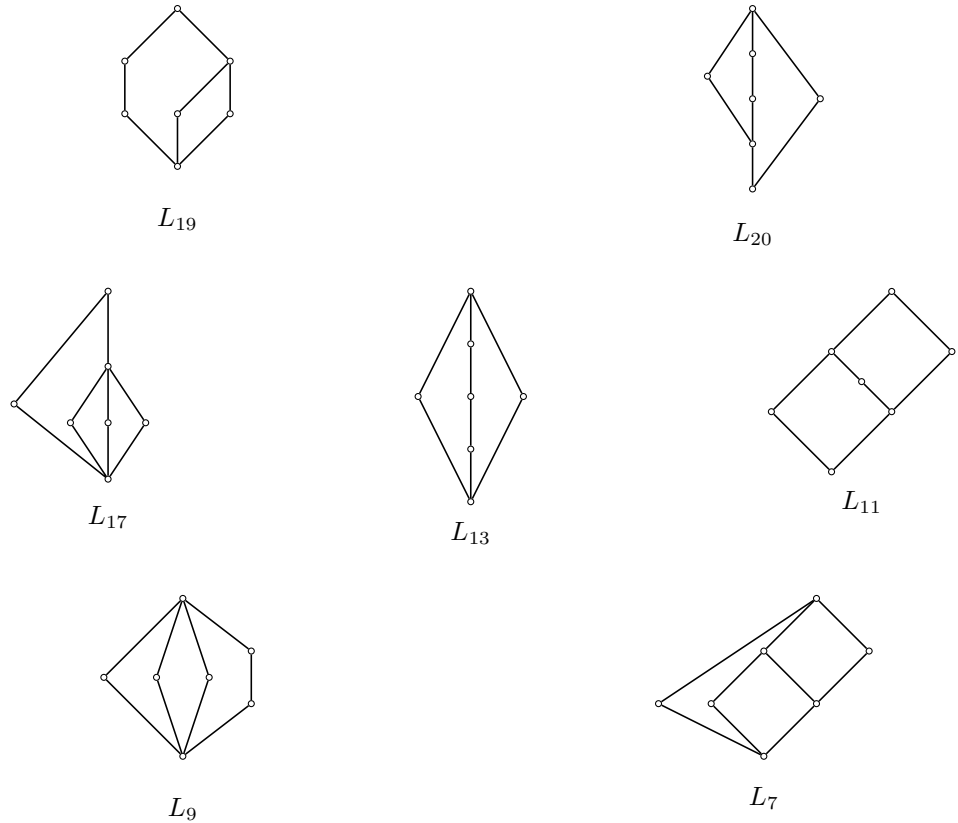

  \begin{multicols}{2}
  \begin{center}
    \begin{tikzpicture}[scale=.7]
      \input{L19.tex}
      \draw (0,-1) node {$L_{19}$};
    \end{tikzpicture}
  \end{center}
  \par \vfill \columnbreak
  \begin{center}
    \begin{tikzpicture}[scale=.6]
      \input{L20.tex}
      \draw (0,-1) node {$L_{20}$};
    \end{tikzpicture}
  \end{center}
\end{multicols}
\begin{multicols}{3}
  \begin{center}
    \begin{tikzpicture}[scale=.5]
      \input{L17.tex}
      \draw (0,-1) node {$L_{17}$};
    \end{tikzpicture}
  \end{center}
  \par \vfill \columnbreak
      \begin{center}
        \begin{tikzpicture}[scale=.7]
          \input{L13.tex}
          \draw (0,-.6) node {$L_{13}$};
        \end{tikzpicture}
      \end{center}
  \par \vfill \columnbreak
      \begin{center}
        \begin{tikzpicture}[scale=.4]
          \input{L3.tex}  
          \draw (1,-1) node {$L_{11}$};
        \end{tikzpicture}
      \end{center}
\end{multicols}
\begin{multicols}{2}
      \begin{center}
        \begin{tikzpicture}[scale=.7]
          \input{L9.tex}
          \draw (6.5,-1) node {$L_{9}$};
        \end{tikzpicture}
      \end{center}
  \par \vfill \columnbreak
      \begin{center}
        \begin{tikzpicture}[scale=.7]
          \input{L10.tex}
          \draw (1,-.8) node {$L_{7}$};
        \end{tikzpicture}
      \end{center}
\end{multicols}
  \caption{Seven element lattices with no obvious congruence lattice representation.}
  \label{fig:Mysevens}
\end{figure}

We now prove the existence of congruence lattice representations for all but the
last of these.  The first two, $L_{19}$ and $L_{20}$ were found using the
closure method with the help of Sage by searching for closed concrete
representations in the partition lattice $\Eq(8)$.
As for $L_{17}$, recall that the lattice $\Sub(A_4)$ of
subgroups of the group $A_4$ (the group of all even permutations of a four
element set) is the lattice shown below. 
        
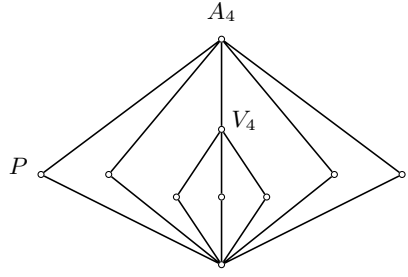
\begin{center}
          \begin{tikzpicture}[scale=.6]
                    \node (bottom) at (0,0)  [draw, circle, inner sep=\dotsize] {};
        \node (top) at (0,5)  [draw, circle, inner sep=\dotsize] {};
        \node (n115) at (-1,1.5)  [draw, circle, inner sep=\dotsize] {};
        \node (015) at (0,1.5)  [draw, circle, inner sep=\dotsize] {};
        \node (115) at (1,1.5)  [draw, circle, inner sep=\dotsize] {};
        \node (n252) at (-2.5,2)  [draw, circle, inner sep=\dotsize] {};
        \node (252) at (2.5,2)  [draw, circle, inner sep=\dotsize] {};
        \node (n42) at (-4,2)  [draw, circle, inner sep=\dotsize] {};
        \node (42) at (4,2)  [draw, circle, inner sep=\dotsize] {};
        \node (03) at (0,3)  [draw, circle, inner sep=\dotsize] {};
        \draw[semithick] 
        (bottom) to (42) to (top) to (n42) to 
        (bottom) to (n252) to (top) to (252) to 
        (bottom) to (n115) to (03) to (115) to
        (bottom) to (015) to (03) to (top);

            \draw[font=\small] (0,5.6) node {$A_4$};
            \draw[font=\small] (.5,3.2) node {$V_4$};
            \draw[font=\small] (-4.5,2.2) node {$P$};
          \end{tikzpicture}
        \end{center}
Here $V_4$ denotes the Klein four subgroup and $P$ marks one of the four Sylow 3
subgroups of $A_4$.
Of course, $\Sub(A_4)$ is the congruence lattice of the permutational algebra
consisting of $A_4$ acting regularly on itself by multiplication.
Now note that $L_{17} \cong P^\uparrow \cup V_4^\downarrow$, the union of a
filter and ideal of a representable lattice. Therefore, $L_{17}$ is representable.

The question of whether the existence of such a ``filter-idea
representation'' implies that the lattice in question is also an interval in a subgroup
lattice seems open.  Although, in the present case, we have found that $L_{17}$
has a group representation.  Indeed, the group 
$G = (A_4 \times A_4) \rtimes C_2$ 
has a subgroup $H\cong S_3$ such that $[H,G]\cong L_{17}$.

\index{dual}%
\index{Kurzweil, Hans}%
\index{Netter, Raimund}%
Now, by the Kurzweil-Netter result, the dual of $L_{17}$ is also representable.
Explicitly, since $L_{17}$ is representable on a 12-element set (the elements of
$A_4$) via the filter-ideal method,\footnote{Note that the filter plus
  ideal method only adds operations to the algebra of which the original lattice
  was the congruence lattice, leaving the universe fixed.  Thus, the
  filter-ideal sublattice is the congruence lattice of an algebra with the same
  number of elements as the original algebra.} 
the dual of $L_{17}$ can be embedded above diagonal subgroup of the 12-th power
of a simple group: $L_{17}' \hookrightarrow [D,S^{12}] \cong (\Eq(12))'$.
Then, adding the operations from the original representation of $L_{17}$ as
described in the proof of Theorem~\ref{thm:duals-interv-subl}, we have an
algebra with universe $S^{12}/D$ and congruence lattice isomorphic to
$L_{17}'$.\footnote{Incidentally, since $L_{17}$ is also representable as an
  interval above a subgroup (of index 48), we could apply the Kurzweil-Netter
  method using this representation instead.  Then we would obtain a
  \emph{group} representation of the dual (namely, an upper interval in a
  group of the form $S^{48} \rtimes G$, where $G = (A_4 \times A_4) \rtimes C_2$).}

The lattice  $L_{13}$ is an interval in a subgroup lattice.  Specifically, 
a \GAP\ search reveals that the group\footnote{In \GAP\ this is {\tt SmallGroup(960,11358)}.}
$G = (C_2 \times C_2 \times C_2 \times C_2) \rtimes A_5$
has a subgroup $H\cong A_4$ such that $[H,G]\cong L_{13}$.
The index is $|G:H|=80$, so the action of $G$ on the cosets $G/H$ is an
algebra on an 80 element universe.

Though we have not found $L_{11}$ as an interval in a subgroup lattice, we
have found that the pentagon $N_5$ is an upper interval 
in the subgroup lattice of the groups
$G = ((C_3 \times C_3) \rtimes Q_8) \rtimes C_3$ and
$G=(A_4 \times A_4) \rtimes C_2$.\footnote{$Q_8$ denotes the eight element quaternion group.}
In each of these groups, there exists a subgroup $H < G$ (of index 36) with
$[H, G] \cong N_5$.  
Let $[H, G] = \{H, \alpha, \beta, \gamma, G\} \cong N_5$.  (See Figure~\ref{fig:L11}.)
Of course, $\Sub(G)$ is a congruence lattice, so
if there exists a subgroup $K \succ 1$, below $\beta$ and not below $\gamma$, 
then $L_{11} \cong K^\downarrow \cup H^\uparrow$.  Indeed, there is such a
subgroup $K$.
 
\begin{figure}
\begin{center}
  {\scalefont{1}
  \begin{tikzpicture}[scale=.6]
    \input{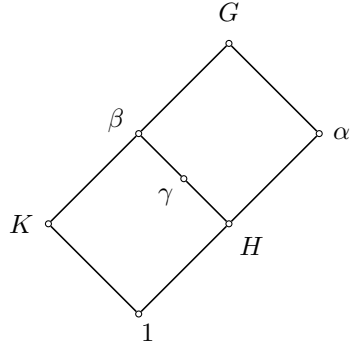};
    \draw
    (2,6.7) node {$G$} 
    (2.5,1.5) node {$H$} 
    (.2,-.4) node {$1$}
    (4.5,4) node {$\alpha$} 
    (-.5,4.3) node {$\beta$} (.6,2.6) node {$\gamma$};
    \draw  (-2.6,2) node {$K$};
  \end{tikzpicture}
}
\end{center}
  \caption{The lattice $L_{11}$ represented as the union of a filter and ideal in the
subgroup lattice of the group  $G$.  Two choices for $G$ that work are
{\tt SmallGroup(216,153)} $=((C_3 \times C_3) \rtimes Q_8) \rtimes C_3$ and
{\tt SmallGroup(288,1025)} $=(A_4 \times A_4) \rtimes C_2$. }
  \label{fig:L11}
\end{figure}

Apart from the easy cases, which we only briefly covered at the start of this
section, there remain just two seven element lattices for which we have not yet
described a representation.  These are the lattices at the bottom of
Figure~\ref{fig:Mysevens}.
Finding a representation of $L_9$, dubbed the ``triple-wing pentagon,'' was
quite challenging.  It sparked the idea of expanding finite algebras, which 
we describe at length in the next chapter (Ch.~\ref{cha:expans-finite-algebr}).
Here we only mention the basic idea as it applies to this particular lattice.  
As the goal is to find an algebra with congruence lattice $L_9$, we start with an
algebra having an $M_4$ congruence lattice -- that is, a six element lattice of
height two with four atoms (which are also coatoms).  Then we expand the algebra
by adding elements to the universe and adding certain operations so that the
newly expanded algebra has almost the same congruence lattice as the original, except
one of the atoms has been doubled.  That is, the resulting congruence lattice is
isomorphic to $L_9$.  This example and the powerful techniques that grew out of it
are described in Chapter~\ref{cha:expans-finite-algebr}.

It is still unknown whether the final lattice appearing in
Figure~\ref{fig:Mysevens} is representable as the congruence lattice of a finite
algebra.  Thus, $L_7$ is the unique smallest lattice for which there is no known
representation.  It is the subject of the next section.

\section{The exceptional seven element lattice}
\label{sec:except-seven-elem}
In this section we consider $L_7$, the last seven element lattice 
appearing in Figure~\ref{fig:Mysevens}.  As yet, we are unable to find a finite
algebra which has a congruence lattice isomorphic to $L_7$, and this is the
smallest lattice for which we have not found such a representation.

Suppose $\bA$ is a finite algebra with $\Con \bA \cong L_7$, and
suppose $\bA$ is of minimal cardinality among those algebras having
a congruence lattice isomorphic to $L_7$.  Then 
$\bA$ must be isomorphic to a transitive \Gset.
(This fact is proved in a forthcoming article,~\cite{gsets}.)  
Therefore, if $L_7$ is representable, we can assume there is a finite group $G$ with a 
core-free\footnote{Recall that the
core of a subgroup $X$ in $G$ is the largest normal subgroup of $G$ contained in
$X$.  This is denoted by $\core_G(X)$.  We say the $X$ is \defn{core-free} in 
$G$ provided $\core_G(X) = 1$.} 
subgroup $H<G$ such that $L_7$ is isomorphic to the interval
sublattice $[H,G] \leq \Sub(G)$.  In this section we present some restrictions
on the possible groups for which this can occur.  

The first restriction, which is
the easiest to observe, is that $G$ must act primitively on the cosets of one of its
maximal subgroups.  This suggests the possibility of describing $G$ in terms of
the
\index{O'Nan-Scott Theorem}%
\index{Aschbacher-O'Nan-Scott Theorem}%
O'Nan-Scott Theorem which characterizes primitive
permutation groups. The goal is to eventually find enough
restrictions on $G$ so as to rule out all finite groups.  As yet, we have not
achieved this goal.  However, the new results in this section reduce the
possibilities to very special subclasses of the O'Nan-Scott classification
theorem.  This paves the way for future studies to focus on these subclasses
when searching for a group representation of $L_7$, or proving that none exists.

The main result of this section is the following:
\begin{theorem}
\label{thm:except-seven-elem}
Suppose $H<G$ are finite groups with $\core_G(H) = 1$ and suppose
$L_7 \cong [H,G]$.  Then the following hold.
\begin{enumerate}[(i)]
\item $G$ is a primitive permutation group.
\item If $N\ssubnormal G$, then $C_G(N) = 1$.
\item $G$ contains no non-trivial abelian normal subgroup.
\item $G$ is not solvable.
\item $G$ is subdirectly irreducible.
\item With the possible exception of at most one maximal subgroup,
  all proper subgroups in the interval $[H,G]$ are core-free. 

\end{enumerate}
\end{theorem}
\begin{remark}
  It is obvious that (ii) $\Rightarrow$ (iii) $\Rightarrow$ (iv), and  (ii) $\Rightarrow$
  (v), but we include these easy consequences in the statement of the result for
  emphasis; for, although the hard work will be in proving (ii) and (vi), our
  main goal is the pair of restrictions (iii) and (v), which allow us to rule
  out a number of the O'Nan-Scott types describing primitive permutation
  groups.  (Section~\ref{sec:onan-scott-theorem} includes a detailed description
  of these types.) 
\end{remark}

Assume the hypotheses of the theorem above.  In particular, throughout this
section \emph{all groups are finite, $H$ is a core-free subgroup of $G$, and $[H,G] \cong
  L_7$}. Label the seven subgroups of $G$ in
the interval $[H,G]$ as in the following diagram:

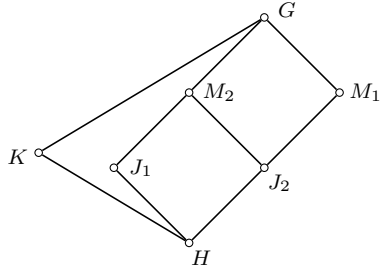
\begin{center}
  {\scalefont{.8}

    \begin{tikzpicture}[scale=1]
      \node (J1) at (0,1)  [draw, circle, inner sep=1pt] {};
      \draw (.36, 1) node {$J_1$};
      \node (H) at (1,0)  [draw, circle, inner sep=1pt] {};
      \draw (1.16, -.2) node {$H$};
      \node (M2) at (1,2)  [draw, circle, inner sep=1pt] {};
      \draw (1.4, 2) node {$M_2$};
      \node (J2) at (2,1)  [draw, circle, inner sep=1pt] {};
      \draw (2.2, .8) node {$J_2$};
      \node (G) at (2,3)  [draw, circle, inner sep=1pt] {};
      \draw (2.3, 3.1) node {$G$};
      \node (M1) at (3,2)  [draw, circle, inner sep=1pt] {};
      \draw (3.35, 2) node {$M_1$};
      \node (K) at (-1,1.2)  [draw, circle, inner sep=1pt] {};
      \draw (-1.28, 1.15) node {$K$};

      \draw[semithick] (H) to (J1) to (M2) to (G) to (M1) to (J2) to (H) (J2) to (M2);
      \draw[semithick] (H) to (K) to (G);

    \end{tikzpicture}
  }
\end{center}
The labels are chosen with the intention of helping us remember to which
subgroups they refer:
the maximal subgroup $M_2$ covers two subgroups in the interval
$[H,G]$, while $J_2$ is covered by two subgroups of $G$.

We now prove the foregoing theorem through a series of claims.
The first thing to notice about the interval $[H,G]$ is that
$K$ is a \defn{non-modular element} of the interval.  This means that
there is a spanning pentagonal ($N_5$) sublattice of the interval with $K$ as the
incomparable proper element. (See the diagram below, for example.)
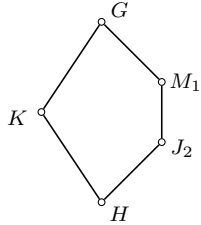
\begin{center}
  {\scalefont{.8}
    \begin{tikzpicture}[scale=.8]
      \node (H) at (0,0)  [draw, circle, inner sep=.9pt] {};
      \draw (.3, -.2) node {$H$};
      \node (K) at (-1,1.5)  [draw, circle, inner sep=.9pt] {};
      \draw (-1.4, 1.4) node {$K$};
      \node (11) at (1,1)  [draw, circle, inner sep=.9pt] {};
      \draw (1.34, .9) node {$J_2$};
      \node (12) at (1,2)  [draw, circle, inner sep=.9pt] {};
      \draw (1.4, 2) node {$M_1$};
      \node (G) at (0,3)  [draw, circle, inner sep=.9pt] {};
      \draw (.3, 3.2) node {$G$};
      \draw[semithick] (H) to (11) to (12) to (G) to (K) to (H);
    \end{tikzpicture}
}
\end{center}
Using this non-modularity property of $K,$ it is easy to 
prove the following
\begin{claim}
\label{claim:K1corefree}
$K$ is a core-free subgroup of $G$.
\end{claim}
\begin{proof}
  Let $N := \core_G(K)$.  If $N \leq X$ for some $X \in \{M_1, M_2, J_1, J_2\}$, 
then $N < X\cap K = H$, so $N = 1$ (since $H$ is core-free).  If
  $N\nleq X$ for all $X \in \{M_1, M_2, J_1, J_2\}$, then $NJ_2 = G$.  But then
Dedekind's rule leads to the following contradiction:
\[
J_2 \leq M_1 \quad \Rightarrow \quad J_2 = J_2(N\cap M_1) = J_2 N \cap M_1 =
G\cap M_1 = M_1.
\]
Therefore, $N = 1$.
\end{proof}
Note that (i) of the theorem follows from Claim~\ref{claim:K1corefree}.  Since
$K$ is core-free, $G$ acts faithfully on the
cosets $G/K$ by right multiplication.  Since $K$ is a maximal subgroup, the
action is primitive.

The next claim is only slightly harder than the previous one as it requires the
more general consequence of Dedekind's rule that we established above in
Lemma~\ref{lemma-wjd-4} (i). 
\begin{claim}
$J_1$ and $J_2$ are core-free subgroups of $G$.
\end{claim}
\begin{proof}
First note that if $N\subnormal G$ then the subgroup $NH$
permutes\footnote{Recall, for subgroups $X$ and $Y$ of a group $G$, we define
  the \emph{sets} $XY = \{xy\mid x\in X, y \in Y\}$, and
  $YX = \{yx\mid x\in X, y \in Y\}$, and 
we say that $X$ and $Y$ are \defn{permuting subgroups} (or that $X$ and $Y$
permute, or that $X$ permutes with $Y$)
 provided the two sets $XY$ and $YX$ coincide, in which case the set forms a group:
$XY = \<X,Y\> = YX$.}
with any subgroup
containing $H$.  To see this, let $H \leq X \leq G$ and note that
\[
  NHX = NX = XN= XHN = XNH,
\]
since $H \leq X$ and $N\subnormal G$.

Suppose $1\neq N\leq J_1$ for some $N \ssubnormal G$. Then $NH = J_1$, so $J_1$ and $K$ are
permuting subgroups.
Since $J_1K = G$ and $J_1\cap K = H$,
Lemma~\ref{lemma-wjd-4} yields
\[
[J_1, G] \cong [H, K]^{J_1} := \{X \in[H, K] \mid J_1X=XJ_1 \}.
\]
But this is impossible since $[H, K]^{J_1} \leq [H,  K] \cong \2$, while $[J_1, G]\cong \3$.
This proves that $\core_G(J_1) = 1$.
The intervals involved in the argument are drawn with bold lines in the
following diagram.
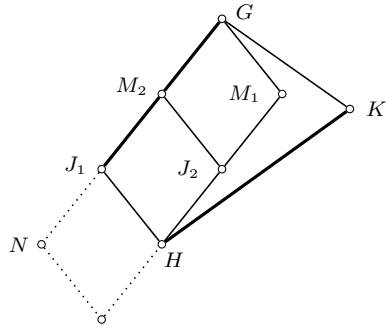
\begin{center}
  {\scalefont{.78}
   \begin{tikzpicture}[scale=1]
      \node (N) at (-.6,0)  [draw, circle, inner sep=1pt] {};
      \draw (-.9, 0) node {$N$};
      \node (NcapH) at (.2,-1)  [draw, circle, inner sep=1pt] {};
      \node (J1) at (0.2,1)  [draw, circle, inner sep=1pt] {};
      \draw (-.15, 1.05) node {$J_1$};
      \node (H) at (1,0)  [draw, circle, inner sep=1pt] {};
      \draw (1.16, -.2) node {$H$};
      \node (M2) at (1,2)  [draw, circle, inner sep=1pt] {};
      \draw (.6, 2.1) node {$M_2$};
      \node (J2) at (1.8,1)  [draw, circle, inner sep=1pt] {};
      \draw (1.35, 1.) node {$J_2$};
      \node (G) at (1.8,3)  [draw, circle, inner sep=1pt] {};
      \draw (2.1, 3.1) node {$G$};
      \node (M1) at (2.6,2)  [draw, circle, inner sep=1pt] {};
      \draw (2.1, 2) node {$M_1$};
      \node (K) at (3.5,1.8)  [draw, circle, inner sep=1pt] {};
      \draw (3.85, 1.8) node {$K$};
      \draw[semithick,dotted] (J1) to (N) to (NcapH) to (H);
      \draw[very thick] (J1) to (M2) to (G) (H) to (K);
      \draw[semithick] (H) to (J1) to (M2) to (G) to (M1) to (J2) to (H) to (K)
      to (G) (J2) to (M2);
    \end{tikzpicture}
 }
\end{center}

The proof 
that $J_2$ is core-free is similar.  Suppose
$1\neq N\leq J_2$ where $N \ssubnormal G$. Then $NH = J_2$ and the subgroups $J_2$ and $K$
permute.
Therefore, $[H, K]^{J_2} \cong [J_2, G]$, 
by Lemma~\ref{lemma-wjd-4},
which is a contradiction since
$[H, K]^{J_2} \leq [H,  K] \cong \2$, while $[J_2, G]\cong \two \times \two$.
\end{proof}

Now that we know $K, J_1, J_2$ are each core-free in $G$, we use this
information to prove that at least one of the other maximal subgroups, 
$M_1$ or $M_2$, is core-free in $G$, thereby establishing (vi) of the theorem.  
We will also see that $G$ is subdirectly irreducible, proving (v).  The proof of
(ii) will then follow from the same argument used to prove 
Lemma~\ref{lemma-wjd-4} (ii), which we repeat below.

\begin{claim}
  Either $M_1$ or $M_2$ is core-free in $G$.  If $M_2$ has non-trivial core
  and $N\ssubnormal G$ is contained in $M_2$, then
  $C_G(N)=1$ and $G$ is subdirectly irreducible.
\end{claim}
\begin{proof}
  Suppose $M_2$ has non-trivial core.  Then there is 
a minimal normal subgroup $1\neq N \ssubnormal G$ 
  contained in $M_2$. 
  Since $H, J_1, J_2$ are core-free, $NH=M_2$.  Consider the centralizer,
  $C_G(N)$, of $N$ in $G$.  Of course, this is a normal subgroup 
  of $G$.\footnote{The centralizer of a normal subgroup $N\subnormal G$ is itself
    normal in $G$.  For, it is the kernel of the conjugation action of $G$ on
    $N$. Thus, $C_G(N) \subnormal N_G(N) = G$.} 
If $C_G(N) = 1$, then, since minimal normal subgroups
  centralize each other, $N$ must be the unique minimal normal subgroup of $G$.
  Furthermore, $M_1$ must be core-free in this case.  Otherwise 
  $N\leq M_1 \cap M_2 = J_2$, contradicting $\core_G(J_2)=1$. 
  Therefore, in case $C_G(N) = 1$ we 
  conclude that $G$ is subdirectly irreducible and $M_1$ is core-free.

  We now prove that the alternative, $C_G(N) \neq 1$, does not occur.
  This case is a bit more challenging and must be split up into further subcases,
  each of which leads to a contradiction.
  Throughout, the assumption $1\neq N \leq M_2$ is in force, and it helps to
  keep in mind the diagram in Figure~\ref{fig:M2-not-core-free}.
\begin{center}
  \begin{figure}
  {\scalefont{.8}
\begin{center}
   \begin{tikzpicture}[scale=1]
      \node (J1) at (0.2,1)  [draw, circle, inner sep=1pt] {};
      \draw (-.05, .85) node {$J_1$};
      \node (H) at (1,0)  [draw, circle, inner sep=1pt] {};
      \draw (1.16, -.2) node {$H$};
      \node (M2) at (1,2)  [draw, circle, inner sep=1pt] {};
      \draw (.6, 2.1) node {$M_2$};
      \node (J2) at (1.8,1)  [draw, circle, inner sep=1pt] {};
      \draw (1.35, 1.) node {$J_2$};
      \node (G) at (1.8,3)  [draw, circle, inner sep=1pt] {};
      \draw (2.1, 3.1) node {$G$};
      \node (M1) at (2.6,2)  [draw, circle, inner sep=1pt] {};
      \draw (2.1, 2) node {$M_1$};
      \node (K) at (3.7,1.8)  [draw, circle, inner sep=1pt] {};
      \draw (4, 1.8) node {$K$};
      \node (NcapH) at (-.65,-1.05)  [draw, circle, inner sep=1pt] {};
      \draw (-.5, -1.36) node {$N\cap H$};
      \node (N) at (-1.5,.3)  [draw, circle, inner sep=1pt] {};
      \draw (-1.8, .35) node {$N$};
      \draw[semithick, dotted]  (H) to (NcapH) (N) to (M2);
      \draw[semithick, dotted]
      (NcapH) to (N);
      \draw[semithick] 
      (H) to (J1) to (M2) to (G) to (M1) to (J2) to (H) to (K) to (G) 
      (J2) to (M2);
    \end{tikzpicture}
\end{center}
  }
    \caption{Hasse diagram illustrating the cases in which $M_2$ has
      non-trivial core: $1\neq N \leq M_2$ for some $N\ssubnormal G$.}
    \label{fig:M2-not-core-free}
  \end{figure}
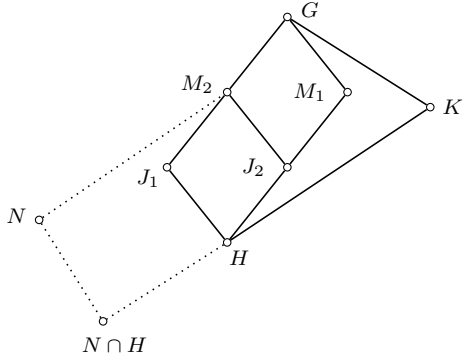
\end{center}

Suppose $C_G(N) \neq 1$.  Then, since $C_G(N)\subnormal G$,
and since $H, J_1, J_2, K$ are core-free, it's clear that $C_G(N)H \in
\{G, M_1, M_2\}$.  We consider each case separately.

\begin{enumerate}
\item[{\it Case 1:}] 
Suppose $C_G(N)H = G$.
Note that $N\cap H < N \cap J_1 < N$ (strictly). The
subgroup $N\cap J_1$ is normalized by $J_1$ and by $C_G(N)$, and so it is normal in
$C_G(N)J_1 \geq C_G(N)H = G$, contradicting the minimality of $N$.  Thus, the case
$C_G(N)H = G$ does not occur.

\item[{\it Case 2:}]
Suppose $C_G(N)H =M_1$.  The subgroup $N\cap J_1$ is
normalized by both $H$ and $C_G(N)$. For, $C_G(N)$ centralizes, hence
normalizes, every subgroup of $N$.  Therefore, $N\cap J_1$ is normalized by 
$C_G(N)H =M_1$.  Of course, it's also normalized by $J_1$, so
$N\cap J_1$ is normalized by the set $M_1J_1$, so it's normalized by the group
generated by that set, which is $\<M_1, J_1\> = G$.\footnote{Actually, the set is
already a group in this case since $M_1J_1 = C_G(N)H J_1 = J_1 C_G(N)H = J_1 M_1$.}
The conclusion is that $N\cap J_1\ssubnormal G$.  
Since $J_1$ is core-free, $N\cap J_1 = 1$.
But this contradicts the (by now familiar) consequence of
Dedekind's rule:  
\[
H < J_1 < M_2  \quad \Rightarrow \quad N\cap H < N\cap J_1 < N\cap M_2.
\]
 Therefore, $C_G(N)H =M_1$ does not occur.

\item[{\it Case 3:}]
Suppose $C_G(N)H = M_2$.
The subgroup $N\cap M_1$ is
normalized by both $H$ and $C_G(N)$.  Therefore, $N\cap M_1$ is normalized by 
$C_G(N)H =M_2$.  Of course, it's also normalized by $M_1$, so
$N\cap M_1$ is normalized by $\<M_1, M_2\> = G$.
The conclusion is that $N\cap M_1\ssubnormal G$.  By 
minimality of the normal subgroup $N$, we must have either $N\cap M_1 = 1$ or 
$N\cap M_1 = N$.  The former equality implies $N\cap J_2=1$, which contradicts 
the strict inequalities of Dedekind's rule,
\begin{equation}
  \label{eq:6}
H < J_2 < M_2  \quad \Rightarrow \quad N\cap H < N\cap J_2 < N\cap M_2,
\end{equation}
while the latter equality ($N\cap M_1 = N$) implies that $N \leq M_1 \cap M_2 = J_2$ which
contradicts 
$\core_G(J_2)=1$. 
\end{enumerate}
\end{proof}

We have proved that either $M_1$ or $M_2$ is core-free in $G$, and
we have shown that, if $M_2$ has non-trivial core, then $G$ is subdirectly
irreducible.  In fact, we proved that $C_G(N)=1$ for the unique minimal normal subgroup
$N$ in this case.  It remains to prove that $G$ is subdirectly irreducible in
case $M_1$ has non-trivial core. The argument is similar to the foregoing, and
we omit some of the details that can be checked exactly as above.

\begin{claim}
If $M_1$ has non-trivial core
  and $N\ssubnormal G$ is contained in $M_1$, then
  $C_G(N)=1$ and $G$ is subdirectly irreducible.
\end{claim}
\begin{proof}
If $M_1$ has non-trivial core, then there is a minimal normal
subgroup $N\ssubnormal G$ contained in $M_1$.  We proved above that
$M_2$ must be core-free in this case, so either $C_G(N)H  = G$, 
$C_G(N)H  = M_1$, 
or $C_G(N)=1$.  The first case is easily ruled out exactly as in Case 1 above. 
The second case is handled by the argument we used in Case 3.  Indeed, if we suppose 
$C_G(N)H = M_1$, 
then 
$N\cap M_2$ 
is normalized by both $H$ and $C_G(N)$, hence by
$M_1$.  It is also normalized by $M_2$, so 
$N\cap M_2\ssubnormal G$.  Thus, by minimality of $N$, 
and since $M_2$ is core-free,
$N\cap M_2 = 1$.  But then $N\cap J_2 = 1$,
leading to a contradiction similar to~(\ref{eq:6}) but with $M_1$ replacing $M_2$.  
Therefore, the case 
$C_G(N)H = M_1$ does not occur, and we have proved $C_G(N)=1$. 
\end{proof}

So far we have proved that all intermediate proper subgroups in the interval $[H, G]$
are core-free except possibly at most one of $M_1$ or $M_2$.  Moreover, we
proved that if one of the maximal subgroups has non-trivial core, then there is
a unique minimal normal subgroup $N\ssubnormal G$ with trivial centralizer,
$C_G(N) = 1$.  As explained above, $G$ is subdirectly irreducible in this case,
since minimal normal subgroups centralize each other.

In order to prove (ii), there remains only one case left to check, and the
argument is by now very familiar.
\begin{claim}
  If each $H\leq X < G$ is core-free and $N$ is a minimal normal subgroup of
  $G$, then $C_G(N) = 1$.
\end{claim}
\begin{proof}
  Let $N$ be a minimal normal subgroup of $G$. Then, by the core-free hypothesis
  we have $NH = G$. Fix a subgroup $H< X < G$.  Then $N\cap H < N\cap X < N$.
  The subgroup $N\cap X$ is normalized by $H$ and by 
  $C_G(N)$.  If $C_G(N) \neq 1$, then $C_G(N)H = G$, by the core-free
  hypothesis, so $N\cap X\ssubnormal G$, contradicting the minimality of $N$.  
  Therefore, $C_G(N) \neq 1$.
\end{proof}
Finally, we note that the claims above taken together prove (ii), and thereby
complete the proof of the theorem.  For if $G$ is subdirectly irreducible with
unique minimal normal subgroup $N$, and if $C_G(N) = 1$, then all normal
subgroups (which necessarily lie above $N$) must have trivial centralizers. 

\section{Conclusion}
We conclude this chapter with a final observation which helps us describe the
O'Nan-Scott type of a group which has $L_7$ as an interval in its subgroup
lattice.  We end with a conjecture that should be the subject of future research.

By what we have proved above, $G$ acts
primitively on the cosets of $K$, and it also acts primitively on the cosets of
at least one of $M_1$ or $M_2$.  Suppose $M_1$ is core-free so that 
$G$ is a primitive permutation group in its action on cosets of $M_1$ and let 
$N$ be the minimal normal subgroup of $G$.  As we have seen, $N$ has trivial
centralizer, so it is nonabelian and is the unique minimal normal subgroup of
$G$.
Now, we have seen that $NH \geq M_2$ in this case, so $H < J_2 < NH$ implies
that $N\cap M_1 \neq 1$.
Similarly, if we had started out by assuming that $M_2$ is core-free, then $NH
\geq M_1$, and $H < J_2 < NH$ would imply
that $N\cap M_2 \neq 1$.  

By the following elementary result (see,
e.g.,~\cite{Isaacs:2008}) we see that the action of $N$ on the cosets of the
core-free maximal subgroup $M_i$ is not regular.\footnote{Recall, 
a transitive permutation group $N$ is
\emph{acts regularly} on a set $\Omega$ provided the stabilizer subgroup of $N$
is trivial.  Equivalently, every non-identity element of $N$ is
fixed-point-free. Equivalently,
$N$ is regular on $\Omega$ if and only if for each $\omega_1, \omega_2 \in
\Omega$ there is a unique $n\in N$ such that $n\omega_1 = \omega_2$.  
In particular, $|N| = |\Omega|$.}
Consequently, $G$ is characterized by case 2 of the version of the O'Nan-Scott
Theorem given in the appendix, Section~\ref{sec:class-perm-groups}. 
\begin{lemma}
If $G$ acts transitively on a set $\Omega$ with stabilizer $G_\omega$, then 
a subgroup $N\leq G$ acts transitively on $\Omega$ if and only if 
$NG_\omega = G$. Also, $N$ is regular if and only if in addition $N \cap
G_\omega = 1$. 
\end{lemma}

\chapter{Expansions of Finite Algebras}
\label{cha:expans-finite-algebr}
\section{Background and motivation}
In this chapter  we present a novel approach to the construction of new finite
algebras and describe the congruence lattices of these algebras. Given a finite
algebra $\<B, \dots\>$, let $B_1, B_2, \dots, B_K$ be sets which intersect $B$
at specific points. We construct an  {\it overalgebra} $\<A, F_A\>$, by which we
mean an expansion of $\<B, \dots\>$ with universe  $A := B \cup B_1 \cup \cdots
\cup B_K$,  and a certain set $F_A$ of unary operations which include idempotent
mappings $e$ and $e_i$ satisfying $e(A) = B$ and $e_i(A) = B_i$.  We explore a
number of such constructions and prove results about the shape of the new
congruence lattices $\Con \<A, F_A\>$ that result. Thus, descriptions of some
new classes of finitely representable lattices is one of our primary
contributions. 
Another, perhaps more significant contribution is the announcement of a
novel approach to the discovery of new classes of representable lattices.

Our main contribution is the description and analysis of a
new procedure for generating finite lattices which are, by
construction, finitely representable.  
Roughly speaking, we start with an arbitrary finite algebra $\bB := \<B,
\dots\>$, with known congruence lattice $\Con\bB$, and we let $B_1, B_2, \dots,
B_K$ be sets which intersect $B$ at certain points.  The choice of intersection
points plays an important r\^ole which we describe in detail later.  We then
construct an  {\it overalgebra} $\bA:=\<A, F_A\>$, by which we mean an expansion of $\bB$
with universe $A = B \cup B_1 \cup \cdots \cup B_K$,  and a certain set $F_A$ of
unary operations which include idempotent mappings $e$ and $e_i$
satisfying $e(A) = B$ and $e_i(A) = B_i$.  

Given our interest in the problem mentioned above, the important consequence of
this procedure is the new (finitely representable) lattice $\Con\bA$ that
it produces.  The shape of this lattice is, of course, determined by
the shape of $\Con\bB$, the choice of intersection points of the $B_i$, and the
unary operations chosen for inclusion in $F_A$.  In this chapter, we
describe a number of constructions of this type and prove some results 
about the shape of the congruence lattices of the resulting overalgebras.  

Before giving an overview of this chapter, we give a bit of background about the
original example which provided the impetus for this work.  In the spring of
2011, our research seminar was fortunate enough to have 
as a visitor 
\index{Jipsen, Peter}%
Peter Jipsen, who initiated the ambitious project of
cataloging every small finite lattice $L$ for which there is a known finite algebra
$\bA$ with $\Con\bA\cong L$.  Before long, we had identified such finite
representations for all lattices of order seven or less, except for the two
lattices appearing in Figure~\ref{fig:sevens}.
(Section~\ref{sec:seven-elem-latt} describes some of the methods we used to find
representations of the other seven-element lattices.)
\begin{figure}[h!]
  \centering
  
  \begin{tikzpicture}[scale=.7]

    \node (01) at (0,1)  [draw, circle, inner sep=1pt] {};
    \foreach \j in {0,2} 
             { \node (1\j) at (1,\j)  [draw, circle, inner sep=1pt] {};}

             \foreach \j in {1,3} 
                      { \node (2\j) at (2,\j)  [draw, circle, inner sep=1pt] {};}
                      { \node (32) at (3,2)  [draw, circle, inner sep=1pt] {};}
                      \draw[semithick] (10) to (01) to (12) to (23) to (32) to (21) to (10) (21) to (12);
                           { \node (m11) at (-1,1)  [draw, circle, inner sep=1pt] {};}
                           \draw[semithick] (10) to (m11) to (23);


                           \foreach \j in {0,3} 
                                    { \node (7\j) at (6.5,\j)  [draw, circle, inner sep=1pt] {};}
                                    \node (71) at (7,1.5)  [draw, circle, inner sep=1pt] {};
                                    \node (61) at (6,1.5)  [draw, circle, inner sep=1pt] {};
                                    \node (51) at (5,1.5)  [draw, circle, inner sep=1pt] {};
                                    \foreach \j in {1,2} 
                                             { \node (8\j) at (7.8,\j)  [draw, circle, inner sep=1pt] {};}
                                             \draw[semithick] (70) to (51) to (73) to (61) to (70) to (71) to (73) to
                                             (82) to (81) to (70);
                                             
  \end{tikzpicture}

  \caption{Lattices of order 7 with no obvious finite algebraic representation.}
  \label{fig:sevens}
\end{figure}
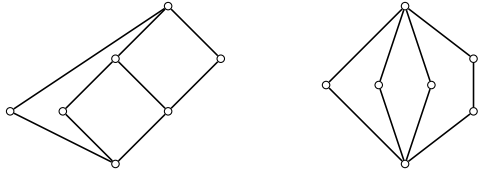
\index{Freese, Ralph}%
\noindent Ralph Freese then discovered a way to construct an algebra 
which has the second of these as its congruence lattice. The idea 
is to start with an algebra $\bB = \<B, \dots\>$ having congruence lattice $\Con \bB
\cong M_4$, expand the universe to the larger set $A = B\cup B_1 \cup B_2$, and
then define the right set $F_A$ of operations on $A$ so that the congruence lattice
of $\bA = \< A, F_F\>$ will be an $M_4$ with one atom ``doubled'' -- that is,
$\Con\bA$ will be the second lattice in figure~\ref{fig:sevens}.

In this chapter we formalize this approach and extend it in four ways.  The first
is a straight-forward generalization of the original overalgebra construction,
and the second is a further expansion of these overalgebras.
The third is a construction based on one suggested by 
\index{Lampe, William}%
Bill Lampe which addresses a basic
limitation of the original procedure.  Finally, we give a generalization of the third
construction.  For each of these constructions we prove
results which allow us to describe the congruence lattices of the resulting
overalgebras. 

Here is a brief outline of the remaining sections of this chapter:
In Section~\ref{sec:residuation-lemma} we prove a lemma which greatly simplifies the
analysis of the structure of the newly enlarged congruence lattice and
its relation to the original congruence lattice.
In Section~\ref{sec:overalgebras} we define {\it overalgebra} and in 
Section~\ref{sec:overalgebras-i} we give a formal description of
the original construction mentioned above.  We then describe the
original example in detail before proving some general results about
the congruence lattices of such overalgebras.
At the end of Section~\ref{sec:overalgebras-i} we describe a further expansion
of the set of operations defined in the first construction, and we conclude the
section with an example demonstrating the utility of these additional operations.
Section~\ref{sec:overalgebras-ii}
presents a second overalgebra construction which overcomes a basic limitation 
of the first.  We then
prove a result about the structure of the congruence lattices of these overalgebras,
and close the section with some further examples which illustrate the procedure and
demonstrate its utility.  In Section~\ref{sec:overalgebras-iii}
we describe a construction that further generalizes the one in
Section~\ref{sec:overalgebras-ii}. 
The last section discusses the impact that our results
have on the main problem -- the finite congruence lattice representation problem
-- as well as the inherent limitations of this approach, and concludes with some
open questions and suggestions for further research.  

\section{A residuation lemma}
\label{sec:residuation-lemma}

Let $e^2 = e \in \Pol_1(\bA)$ be an idempotent unary polynomial, define
$B:=e(A)$ and
$F_B := \{ef\resB \mid f\in \Pol_1(\bA)\}$, and consider the 
unary\footnote{In the definition of  $F_B$, we could have used 
  $\Pol(\bA)$ instead of $\Pol_1(\bA)$, and then our discussion would not be
  limited to \emph{unary} algebras.  However, as we are mainly concerned with
  congruence lattices, we lose nothing by restricting the scope in this way.  Also, 
  later sections of this chapter will be solely concerned with unary algebras, so
  for consistency we define $\bB$ to be unary in this section as well.} 
algebra $\bB:= \<B, F_B\>$.  \PAP\
prove in Lemma 1 of~\cite{Palfy:1980} that
the restriction mapping $\resB$, defined on $\Con\bA$ by 
$\alpha\resB = \alpha \cap B^2$, is a lattice epimorphism of $\Con\bA$ onto $\Con\bB$.
In~\cite{McKenzie:1983}, McKenzie, taking Lemma 1 as a starting point,
develops the foundations of what would become tame congruence theory.  
In reproving the \PP\ congruence lattice epimorphism lemma,
\index{McKenzie, Ralph}%
McKenzie introduces the mapping $\hatmap$ defined on $\Con\bB$ by 
\[
\widehat{\beta} = \{(x,y) \in A^2 \mid (ef(x), ef(y))\in \beta\, \text{ for all }\, f\in \Pol_1(\bA) \}.
\]
It is not hard to see that $\hatmap$ maps $\Con\bB$ into $\Con\bA$.  For
example, if $(x,y) \in \widehat{\beta}$ and $g\in \Pol_1(\bA)$, then for all $f\in \Pol_1(\bA)$ we
have $(efg(x),efg(y)) \in \beta$, so $(g(x),g(y))\in \widehat{\beta}$.

For each $\beta \in \Con\bB$, let $\beta^* = \Cg^\bA(\beta)$.  That is,
$^*: \Con\bB \rightarrow \Con\bA$ is
the congruence generation operator restricted to the set $\Con\bB$.
The following lemma concerns the three mappings, $\resB$, $\hatmap$, and $^*$.
The third statement of the lemma, which follows from the first two, 
will be useful in the later sections of this chapter.

\begin{lemma}
\label{lem:residuation-lemma}
  ~
  \begin{enumerate}[(i)]
  \item $^*: \Con\bB \rightarrow \Con\bA$ is a residuated mapping with
    residual $\resB$.
  \item $\resB : \Con\bA \rightarrow \Con\bB$ is a residuated mapping with
    residual $\hatmap$.
  \item For all $\alpha \in \Con\bA, \, \beta \in \Con\bB$,
    \[
    \beta = \alpha\resB \quad \Leftrightarrow  \quad 
    \beta^* \leq \alpha \leq \widehat{\beta}.
    \]
    In particular, 
    $\beta^*\resB = \beta = \widehat{\beta}\resB$.
  \end{enumerate}
\end{lemma}
\begin{proof}
  We first recall the definition of {\it residuated mapping}.  If $X$ and $Y$
  are partially ordered sets, and if 
  $f: X \rightarrow Y$ and 
  $g: Y \rightarrow X$ are order preserving maps, then the following are
  equivalent:
  \begin{enumerate}[(a)]
  \item $f: X \rightarrow Y$ is a {\it residuated mapping} with {\it residual}
    $g: Y \rightarrow X$;
  \item for all $x\in X,\, y\in Y$,  $f(x) \leq y$ iff $x \leq g(y)$;
  \item $g\circ f \geq \id_X$ and $f\circ g \leq \id_Y$.
  \end{enumerate}
  The definition says that for each $y\in Y$ there is a unique
  $x\in X$ that is maximal with respect to the property $f(x) \leq y$, and the
  maximum $x$ is given by $g(y)$.
  Thus, {\it (i)} is equivalent to 
  \begin{equation}
    \label{eq:OAi}
    \beta^* \leq \alpha \quad \Leftrightarrow \quad \beta \leq \alpha\resB
    \quad (\forall \, \alpha \in \Con\bA,\; \forall \, \beta \in \Con\bB).
  \end{equation}
  This is easily verified, as follows:  If 
  $\beta^* \leq \alpha$ and $(x,y)\in \beta$, then
  $(x,y) \in \beta^* \leq \alpha$ 
  and $(x,y) \in B^2$, so $(x,y)\in
  \alpha\resB$.  If $\beta \leq \alpha\resB$ then 
  $\beta^* \leq (\alpha\resB)^* \leq \Cg^\bA(\alpha) = \alpha$.

  Statement {\it (ii)} is equivalent to 
  \begin{equation}
    \label{eq:OAii}
    \alpha\resB\leq \beta 
    \quad \Leftrightarrow \quad 
    \alpha \leq \widehat{\beta}
    \quad (\forall \, \alpha \in \Con\bA,\; \forall \, \beta \in \Con\bB).
  \end{equation}
  This is also easy to check.  For, suppose
  $\alpha\resB\leq \beta$ and $(x,y)\in \alpha$. Then $(ef(x), ef(y)) \in \alpha$
  for all $f \in \Pol_1(\bA)$ and $(ef(x), ef(y)) \in B^2$, therefore, 
  $(ef(x), ef(y)) \in \alpha\resB \leq \beta$, so $(x,y) \in \widehat{\beta}$.
  Suppose $\alpha \leq \widehat{\beta}$ and $(x,y) \in \alpha\resB$. 
  Then $(x,y) \in \alpha \leq  \widehat{\beta}$, so 
  $(ef(x), ef(y)) \in \beta$ for all $f\in \Pol_1(\bA)$, including $f=\id_A$, so 
  $(e(x), e(y)) \in \beta$. But $(x, y) \in B^2$, so $(x, y) = (e(x), e(y)) \in
  \beta$.

  Combining~(\ref{eq:OAi}) and~(\ref{eq:OAii}), we obtain statement {\it (iii)} of the lemma.
\end{proof}

The lemma above was inspired by the two approaches to proving Lemma 1
of~\cite{Palfy:1980}.  In the original paper $^*$ is used, while McKenzie uses
the $\hatmap$ operator.  Both $\beta^*$ and
$\widehat{\beta}$ are mapped onto $\beta$ by the restriction map $\resB$, so
the restriction map is indeed onto $\Con\bB$.
However, our lemma emphasizes the fact that the interval 
\[
  [\beta^*, \widehat{\beta}] = 
  \{
  \alpha \in \Con\bA \mid \beta^* \leq \alpha \leq \widehat{\beta}
  \}
\]
is precisely the set of congruences for
which $\alpha\resB = \beta$.  In other words, the
inverse image of $\beta$ under $\resB$ is
$\beta \resB^{-1} = [\beta^*, \widehat{\beta}]$.
This fact plays a central r\^{o}le in the
theory developed below.
Nonetheless, for the sake of completeness, we conclude this section by
verifying that Lemma 1 of~\cite{Palfy:1980} can be obtained from the lemma above.
\begin{corollary}
  $\resB : \Con\bA \rightarrow \Con\bB$ is onto and preserves meets and joins.
\end{corollary}
\begin{proof}
  Given $\beta\in \Con\bB$, each $\theta\in \Con\bA$ in the interval $[\beta^*,
    \widehat{\beta}]$ is mapped to $\theta\resB = \beta$, so $\resB$ is clearly
  onto.  That $\resB$ preserves meets is obvious.  To see that $\resB$ is
  join preserving, note that for all $\eta, \theta \in \Con\bA$, we have
  \[
  \eta\resB \join \theta\resB \leq (\eta \join \theta)\resB
  \]
  since $\resB$ is order preserving.  The opposite inequality follows
  from~(\ref{eq:OAii}) above. For,
  \[
  (\eta \join \theta)\resB \leq \eta\resB \join \theta\resB
  \quad \Leftrightarrow \quad 
  \eta \join \theta \leq \widehat{\eta\resB \join \theta\resB},
  \]
  and the second inequality holds since, by~(\ref{eq:OAii}) again,
  \[
  \eta \leq \widehat{\eta\resB \join \theta\resB}
  \quad \Leftrightarrow \quad 
  \eta\resB \leq \eta\resB \join \theta\resB
  \]
  and 
  \[
  \theta \leq \widehat{\eta\resB \join \theta\resB}
  \quad \Leftrightarrow \quad 
  \theta\resB \leq \eta\resB \join \theta\resB.
  \]
\end{proof}
\begin{remark}
  This approach to proving Lemma 1 of \cite{Palfy:1980}, which is similar to the
  proof given in \cite{McKenzie:1983}, does not reveal any information about
  the permutability of the congruences of $\bA$, unlike the more direct proof
  given in~\cite{Palfy:1980}. 
\end{remark}

\section{Overalgebras}
\label{sec:overalgebras}
In the previous section, we started with an algebra $\bA$ and
considered a subreduct $\bB$ with universe $B = e(A)$, the image of an
idempotent unary polynomial of $\bA$.  In this section, we start with a
fixed finite algebra $\bB = \<B, \dots \>$ and consider various ways to
construct an \emph{overalgebra}, that is, an algebra $\bA= \<A, F_A\>$ having 
$\bB$ as a subreduct where $B = e(A)$ for some idempotent $e \in F_A$.  
Beginning with a specific finite algebra $\bB$, our goal is to understand what
(finitely representable) congruence lattices $\Con\bA$ can be built up from
$\Con\bB$ by expanding the algebra $\bB$ in this way.

\subsection{Overalgebras I}
\label{sec:overalgebras-i}
Let $B$ be a finite set, say, $B = \{b_1, b_2\dots, b_n\}$, let $F\subseteq B^B$
be a set of unary maps taking $B$ into itself, and consider the unary algebra
$\bB = \<B, F\>$, with universe $B$ and basic operations $F$. 
When clarity demands it, we call this collection of operations $F_B$.
Let $B_1, B_2, \dots, B_{K}$ be sets of
the same cardinality as $B$, which intersect $B$ at exactly one point, as follows:
\begin{align}
  \label{eq:OABB}
  B &= \{b_1, b_2, b_3, \dots, b_{n}\}\nonumber\\
  B_1 &= \{b_1, b^1_2,  b^1_3, \dots, b^1_{n}\}\nonumber\\
  B_2 &=  \{b^2_1, b_2,  b^2_3, \dots, b^2_{n}\}\nonumber\\
  B_3 &=  \{b^3_1, b^3_2,  b_3, \dots, b^3_{n}\}\nonumber\\
  & \vdots\\
  B_{K} &= \{b^K_1, \dots, b^K_{K-1},b_K, b^K_{K+1},\dots, b^K_{n}\}.\nonumber
\end{align}
That is, for all $1 \leq i < j \leq K$, we have
\[
|B_i| = n \geq K, \qquad B\cap B_i = \{b_i\}, \quad \text{ and } \quad B_i \cap B_j = \emptyset.
\]
Sometimes it is notationally convenient to use the label $B_0:=B$. 

Let $\pi_i: B\rightarrow B_{i}$ be given by $\pi_i(b_j) = b_j^{i}$, for
$i=0, 1, 2, \dots, n$ and $j=1, 2, \dots, K$.  (It is
convenient to include $i=0$ in this definition, in which case we let 
$\pi_0(b_j) = b_j^{0} := b_j$.) 
The map $\pi_i$ and the operations $F$ induce a set $F_{i}$ of
unary operations on $B_i$, as follows:
to each $f\in F$ corresponds the operation $f^{\pi_i} : B_i \rightarrow B_i$
defined by $f^{\pi_i} = \pi_i f \pi_i^{-1}$.
Thus, for each $i$, $\bB_i := \<B_i, F_i\>$ 
and $\bB = \<B, F\>$ are isomorphic  algebras.
That is, for all $i=1,\dots, K$, we have 
\begin{align*}
  \pi_i :   \<B, F\> &\cong \< B_i, F_i\>\\
  B\ni b & \mapsto  b\supi \in B_i\\
  F\ni f &\mapsto f^{\pi_i} \in F_i
\end{align*}
To say that $\pi_i$ is an isomorphism of two non-indexed algebras
is to say that $\pi_i$ is a bijection of the universes which respects the
interpretation of the basic operations; that is, 
$\pi_if(b)= f^{\pi_i}(\pi_i b)$.  In the present case, this holds by
construction:\footnote{
  This generalizes to $k$-ary operations if we 
  adopt the following convention: $f^{\pi_i}(a_1,\dots, a_k) =
  \pi_if(\pi_i^{-1}(a_1), \dots, \pi_i^{-1}(a_k))$.}
$\pi_if(b) = \pi_i f(\pi_i^{-1}\pi_i b) = f^{\pi_i}(\pi_i b)$.

Let $A = \bigcup_{i=0}^K B_i$ and define the following unary maps on $A$:
\begin{itemize}
\item  $e_k: A\rightarrow A$ is $e_k(b_i^j)
  = b_i^k \quad (1 \leq i \leq n;\, 0\leq j, k \leq K)$;
\item $s:A\rightarrow A$ is 
  \[
  s(x) = 
  \begin{cases}
    x, & \text{ if $x\in B_0$,}\\
    b_i, & \text{ if $x\in B_i$.}
  \end{cases}
  \]
\end{itemize}
Let 
\[
F_A := \{f e_0 : f\in F\} \cup \{e_k : 0\leq k \leq K\} \cup \{s\},
\] 
and define the unary algebra $\bA := \< A, F_A\>$.  

Throughout, the map $\hatmap$ is defined
in essentially the same way as it is in McKenzie's paper~\cite{McKenzie:1983}.
That is, given two algebras $\bA = \< A, \dots\>$ 
and $\bB = \< B, \dots\>$ with $B = e(A)$ for some idempotent
$e \in \Pol_1(\bA)$,  we define
$\hatmap : \Con\bB \rightarrow \Con\bA$ by
\[
\widehat{\beta} = \{(x,y) \in A^2 \mid (ef(x), ef(y))\in \beta, \; \forall\,
f\in \Pol_1(\bA) \} \quad (\beta \in \Con\bB).
\]

\begin{example}
  \label{ex:3.1}  
  Before proving some results about the basic structure of the
  congruence lattice of an overalgebra, we 
  present the original example, discovered by 
\index{Freese, Ralph}%
Ralph Freese, of a finite algebra with
  a congruence lattice isomorphic to the second lattice in 
  Figure~\ref{fig:sevens}.
  Consider a finite permutational algebra $\bB = \<B, F\>$
  with congruence lattice $\Con\bB \cong M_4$. (Figure~\ref{fig:ConS3})
  There are only a few small algebras to choose 
  from.\footnote{In fact, there are
    infinitely many, but apart from those involving 
    $S_3$, $C_3 \times C_3$, and $(C_3 \times C_3) \rtimes C_3$, they are quite
    large.  The next smallest G-set with $M_4$
    congruence lattice that we know of comes from the group 
    $G = 
    [ 
      ( (C_3 \times C_3) \rtimes C_2 ) 
      \times 
      ( (C_3 \times C_3) \rtimes C_2 )
    ] \rtimes C_2$
    acting on right cosets of $H = D_8$.  
    The index in this case is $|G:H| = 81$.
    (In \GAP, {\tt G:=SmallGroup(648,725)}, 
    and $H$ is 
    found to be the fourth maximal subgroup class representative 
    of the fourth maximal subgroup class representative of $G$.)}
  We consider the right regular $S_3$-set -- i.e.~the algebra $S_3$ acting on
  itself by right multiplication.  In \GAP,\footnote{
    All of the computational experiments we describe in this chapter rely on
    two open source programs, \GAP~\cite{GAP4} and the Universal Algebra
    Calculator~\cite{uacalc} (\uacalc).  To conduct our experiments,
    we have written a small collection of \GAP\ functions; these are
    available at \url{http://math.hawaii.edu/~williamdemeo/Overalgebras.html}.
  }
  
  {\footnotesize
\begin{verbatim}
gap> G:=Group([(1,2), (1,2,3)]);; 
gap> G:=Action(G,G,OnRight);  
Group([ (1,5)(2,4)(3,6), (1,2,3)(4,5,6) ])
\end{verbatim}
  }

  We prefer to use ``0-offset'' notation, and 
  define the universe of the $S_3$-set described above to be $\{0, 1,\dots,
  5\}$ instead of $\{1, 2, \dots, 6\}$.  
  As such, the nontrivial congruence relations of this algebra are,

  {\footnotesize
\begin{verbatim}
gap> for b in AllBlocks(G) do Print(Orbit(G,b,OnSets)-1, "\n"); od;
[ [ 0, 1, 2 ], [ 3, 4, 5 ] ]
[ [ 0, 3 ], [ 2, 5 ], [ 1, 4 ] ]
[ [ 0, 4 ], [ 2, 3 ], [ 1, 5 ] ]
[ [ 0, 5 ], [ 2, 4 ], [ 1, 3 ] ]
\end{verbatim}
  }

  \noindent Next, we create an algebra
  in \uacalc\ format using the two generators of the
  group as basic operations.\footnote{The GAP
  routine {\tt gap2uacalc.g} is available at \url{www.uacalc.org}.}
  {\footnotesize
\begin{verbatim}
gap> Read("gap2uacalc.g");
gap> gset2uacalc([G,"S3action"]);
\end{verbatim}
}
  \noindent This creates a \uacalc\ file 
  specifying an algebra with universe $B = \{0, 1, \dots, 5\}$ and two
  basic unary operations $g_0 = (4\; 3\;  5\;  1\;  0\;  2)$ and $g_1 = (1\;  2 \; 0\;  4\;  5\;  3)$.
  These operations are the permutations $(0,4)(1,3)(2,5)$ and $(0,1,2)(3,4,5)$, which, in 
  ``1-offset'' notation, are the generators $(1,5)(2,4)(3,6)$ and $(1,2,3)(4,5,6)$
  of the $S_3$-set appearing in the \GAP\ output above.
  Figure~\ref{fig:ConS3} displays the congruence lattice of this algebra.

  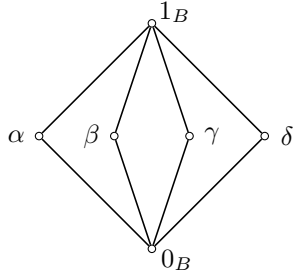
\begin{figure}[h!]
    \centering
    \begin{tikzpicture}[scale=1]
      \node (250) at (2.5,0)  [draw, circle, inner sep=1pt] {};
      \node (253) at (2.5,3)  [draw, circle, inner sep=1pt] {};
      \foreach \i in {1,2, 3, 4} 
               { 
                 \node (\i15) at (\i,1.5)  [draw, circle, inner sep=1pt] {};
                 \draw[semithick] (250) to (\i15) to (253);
               }
               \draw (0.7,1.5) node {$\alpha$};
               \draw (1.7,1.5) node {$\beta$};
               \draw (3.3,1.5) node {$\gamma$};
               \draw (4.3,1.5) node {$\delta$};
               \draw (2.8,3.15) node {$1_B$};
               \draw (2.83,-.15) node {$0_B$};
    \end{tikzpicture}
    \caption{Congruence lattice of the right regular $S_3$-set, where
      $\alpha = | 0, 1, 2 | 3, 4, 5|$,
      $\beta = | 0, 3 | 2, 5 | 1, 4 |$,
      $\gamma = | 0, 4| 2, 3| 1, 5|$, 
      $\delta = | 0, 5| 2, 4| 1, 3|$.
    }
    \label{fig:ConS3}
  \end{figure}

  We now construct an overalgebra which ``doubles'' the congruence $\alpha =
  \Cg^\bB(0,2) = | 0, 1, 2 | 3, 4, 5|$ by choosing intersection points 0 and 2.
  The \GAP\ function {\tt Overalgebra} carries out the construction, and is invoked
  as follows:\footnote{The \GAP\ file 
  {\tt Overalgebras.g} is available at \url{http://dl.dropbox.com/u/17739547/diss/Overalgebras.g}.}

  {\footnotesize
\begin{verbatim}
gap> Read("Overalgebras.g");
gap> Overalgebra([G, [0,2]]);
\end{verbatim}
  }

  \noindent This gives an
  overalgebra with universe $A = B_0 \cup B_1 \cup B_2 = \{ 0, 1, 2, 3, 4, 5\} \cup
  \{0, 6, 7, 8, 9, 10\} \cup\{ 11, 12, 2, 13, 14, 15\}$,
  and the following operations:
  \\[-5pt]
  \begin{center}
    {\small 
      \begin{tabular}{c|r|r|r|r|r|r|r|r|r|r|r|r|r|r|r|r}
        &0&1&2&3&4&5&6&7&8&9&10&11&12&13&14&15\\
        \hline
        $e_0$ & 0& 1& 2& 3& 4& 5& 1& 2& 3& 4& 5& 0& 1& 3& 4& 5\\
        $e_1$ & 0& 6& 7& 8& 9& 10& 6& 7& 8& 9& 10& 0& 6& 8& 9& 10\\
        $e_2$ &11& 12& 2& 13& 14& 15& 12& 2& 13& 14& 15& 11& 12& 13& 14& 15\\
        $s$  & 0& 1& 2& 3& 4& 5& 0& 0& 0& 0& 0& 2& 2& 2& 2& 2\\
        $g_0 e_0$ &4& 3& 5& 1& 0& 2& 3& 5& 1& 0& 2& 4& 3& 1& 0& 2\\
        $g_1 e_0$& 1& 2& 0& 4& 5& 3& 2& 0& 4& 5& 3& 1& 2& 4& 5& 3
    \end{tabular}}
  \end{center}
  ~\\[4pt]
  \noindent If 
  $F_A=\{e_0, e_1, e_2, s, g_0 e_0, g_1 e_0\}$, then the
  algebra $\<A, F_A\>$ has the congruence lattice shown in Figure~\ref{fig:OverAlgebra-S3-0-2}.
  \begin{figure}[h!]
    \centering
    \begin{tikzpicture}[scale=1]
      \node (70) at (6.5,0)  [draw, circle, inner sep=1pt] {};
      \node (71) at (7,1.5)  [draw, circle, inner sep=1pt] {};
      \node (73) at (6.5,3)  [draw, circle, inner sep=1pt] {};
      \node (61) at (6,1.5)  [draw, circle, inner sep=1pt] {};
      \node (51) at (5,1)  [draw, circle, inner sep=1pt] {};
      \node (52) at (5,2)  [draw, circle, inner sep=1pt] {};
      \node (81) at (8,1.5)  [draw, circle, inner sep=1pt] {};
      \draw[semithick] 
      (70) to (51) to (52) to (73)
      (70) to (61) to (73) 
      (70) to (71) to (73) 
      (70) to (81) to (73);
      \draw (4.7,1) node {$\alpha^*$};
      \draw (4.7,2) node {$\widehat{\alpha}$};
      \draw (5.7,1.5) node {$\beta^*$};
      \draw (7.3,1.5) node {$\gamma^*$};
      \draw (8.3,1.5) node {$\delta^*$};
      \draw (6.8,3.15) node {$1_A$};
      \draw (6.83,-.15) node {$0_A$};
    \end{tikzpicture}
    \caption{Congruence lattice of the overalgebra of the $S_3$-set with
      intersection points 0 and 2.}
    \label{fig:OverAlgebra-S3-0-2}
  \end{figure}
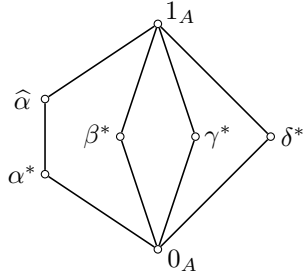

  The congruence relations in Figure~\ref{fig:OverAlgebra-S3-0-2} are as follows:
  \begin{align*}
    \widehat{\alpha} &=|0,1,2,6,7,11,12|3,4,5|8,9,10,13,14,15| \\
    \alpha^* &=|0,1,2,6,7,11,12|3,4,5|8,9,10|13,14,15| \\
    \beta^*&=|0,3,8|1,4|2,5,15|6,9|7,10|11,13|12,14| \\
    \gamma^*&=|0,4,9|1,5|2,3,13|6,10|7,8|11,14|12,15| \\
    \delta^*&=|0,5,10|1,3|2,4,14|6,8|7,9,11,15|12,13|.
  \end{align*}

  It is important to note that the resulting congruence lattice depends
  on our choice of which congruence to ``expand,'' which is controlled by 
  our specification of the intersection points of the overalgebra.
  For example, suppose we want one of the congruences having three
  blocks, say, $\beta = \Cg^\bB(0,3) =| 0, 3 | 2, 5 | 1, 4 |$, to have a non-trivial
  inverse image $\beta\resB^{-1} = 
  [\beta^*, \widehat{\beta}]$.  Then we would select the elements 0
  and 3, (or 2 and 5, or 1 and 4) as the intersection points of the overalgebra.
  To select 0 and 3, we invoke the command 
  {\footnotesize
\begin{verbatim}
gap> Overalgebra([G, [0,3]]);
\end{verbatim}
  }

  \noindent This produces an overalgebra with universe
  $A = B_0 \cup B_1 \cup B_2
  = \{0, 1,  2,  3,  4,  5\} \cup \{ 0, 6,  7,  8,  9, 10\} \cup
  \{11, 12, 13, 3, 14, 15\}$
  and congruence lattice shown in figure~\ref{fig:OverAlgebra-S3-0-3}.

  \begin{figure}[h!]
    \centering
    \begin{tikzpicture}[scale=1.4]
      \node (70) at (6.5,0)  [draw, circle, inner sep=1pt] {};
      \node (71) at (7,1.5)  [draw, circle, inner sep=1pt] {};
      \node (73) at (6.5,3)  [draw, circle, inner sep=1pt] {};
      \node (61) at (6.1,1)  [draw, circle, inner sep=1pt] {};
      \node (62) at (6.1,2)  [draw, circle, inner sep=1pt] {};
      \node (63) at (5.7,1.5)  [draw, circle, inner sep=1pt] {};
      \node (64) at (6.5,1.5)  [draw, circle, inner sep=1pt] {};
      \node (51) at (5,1.5)  [draw, circle, inner sep=1pt] {};
      \node (81) at (8,1.5)  [draw, circle, inner sep=1pt] {};
      \draw[semithick] 
      (70) to (51) to (73)
      (70) to (61) to (63) to (62) to (73) 
      (61) to (64) to (62)
      (70) to (71) to (73) 
      (70) to (81) to (73);
      \draw (4.8,1.5) node {$\alpha^*$};
      \draw (5.98,2.12) node {$\widehat{\beta}$};
      \draw (5.55,1.4) node {$\beta_\eps$};
      \draw (6.65,1.35) node {$\beta_{\eps'}$};
      \draw (6,.8) node {$\beta^*$};
      \draw (7.25,1.5) node {$\gamma^*$};
      \draw (8.2,1.5) node {$\delta^*$};
      \draw (6.7,3.15) node {$1_A$};
      \draw (6.73,-.15) node {$0_A$};
    \end{tikzpicture}
    \caption{Congruence lattice of the overalgebra of the $S_3$-set with
      intersection points 0 and 3.}
    \label{fig:OverAlgebra-S3-0-3}
  \end{figure}
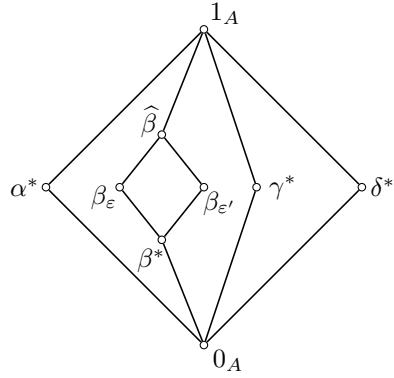
  where
  \begin{align*}
    \alpha^* &=|0,1,2,6,7|3,4,5,14,15|8,9,10|11,12,13| \\
    \widehat{\beta} &=|0,3,8,11|1,4|2,5|6,9,12,14|7,10,13,15| \\
    \beta_{\eps}&=|0,3,8,11|1,4|2,5|6,9,12,14|7,10|13,15| \\
    \beta_{\eps'}&=|0,3,8,11|1,4|2,5|6,9|7,10,13,15|12,14| \\
    \beta^*&=|0,3,8,11|1,4|2,5|6,9|7,10|12,14|13,15| \\
    \gamma^*&=|0,4,9|1,5|2,3,13|6,10|7,8|11,14|12,15| \\
    \delta^*&=|0,5,10|1,3,12|2,4|6,8|7,9|11,15|13,14|.
  \end{align*}
\end{example}

We now prove two theorems which describe the basic structure of the congruence
of an overalgebra constructed as described at the outset of this section.  In
particular, the theorems explain why the interval
$[\alpha^*, \widehat{\alpha}]\cong\two $ appears in the first example above,
while $[\beta^*, \widehat{\beta}]\cong\two\times \two $ appears
in the second.

Given a congruence relation $\beta \in \Con \bB$, let 
$\{b_{\beta(1)}, \dots, b_{\beta(m)}\}$ denote a \emph{transversal} of $\beta$;
i.e.~a full set of $\beta$-class representatives.  
Thus, as a partition of the set $B$, $\beta$ has $m$
classes, or blocks.  (Using the notation $\beta(r)$ for the indices of the
representatives helps us to remember that $b_{\beta(r)}$ is a representative of the
$r$-th block of the congruence $\beta$.)
By the isomorphisms $\pi_i$ defined above, to each
$\beta \in \Con \bB$  there corresponds a congruence relation 
$\beta^{\bBi} \in \Con \bB_i$, and if
$\{b_{\beta(1)}, \dots, b_{\beta(m)}\}$ is a transversal of $\beta$, then 
the map $\pi_i$ also gives a transversal of $\beta^{\bBi}$, namely
$\{\pi_i(b_{\beta(1)}), \dots, \pi_i(b_{\beta(m)})\}
=\{b^i_{\beta(1)}, \dots, b^i_{\beta(m)}\}$.  
Thus, the $r$-th block of $\beta^{\bBi}$ is $b^i_{\beta(r)}/\beta^{\bBi}$.

Let $T = \{b_1, b_2, \dots, b_K\}$ be the set of \emph{tie-points}, that is, the
points at which the sets $B_i\, (1\leq i\leq K)$ intersect the set $B$.
Let $T_r = \{b\in T \mid (b, b_{\beta(r)}) \in \beta\}$ be the set of those
tie-points that are in the $r$-th congruence class of $\beta$.
\begin{theorem}
  \label{OAthm1}
  For each $\beta \in \Con \bB$, 
  \begin{equation}
    \label{eq:OAstar}
    \Cg^\bA(\beta) = \bigcup_{k=0}^K \beta^{\bB_k} \cup \bigcup_{r=1}^m 
    \left(b_{\beta(r)}/\beta \cup \bigcup_{b_j\in T_r} b_j/\beta^{\bB_j}\right)^2.
  \end{equation}
\end{theorem}

\begin{remark}
  Before proceeding to the proof, we advise the reader to consider the small
  example illustrated in Figures~\ref{fig:overalgebra}
  and~\ref{fig:overalgebra1}.  Identifying the objects on the right of
  equation~(\ref{eq:OAstar}) in these figures will make the proof of the theorem
  easier to follow.
  In particular, as the figures make clear, transitivity requires that
  $\beta^{\bB_j}$ classes which are linked together by tie-points must end up in
  the same class of $\Cg^\bA(\beta)$.  This is the purpose of the
  $\bigcup\limits_{r=1}^m (\cdot)^2$ term. 
\end{remark}

\begin{proof} 
  Let $\beta^*$ denote the right-hand side of~(\ref{eq:OAstar}).  
  We first check that $\beta^*\in \Con\bA$. 
  It is easy to see that $\beta^*$ is an equivalence relation, so we need only
  show $f(\beta^*) \subseteq \beta^*$ for all\footnote{Note that $\beta^{\bB_0} = \beta$.} 
  $f\in F_A$, where 
  \[
  F_A := \{f e_0 : f\in F\} \cup \{e_k : 0\leq k \leq K\} \cup \{s\}.
  \] 
  In other words, we prove: if 
  $(x,y)\in \beta^*$ and $f\in F_A$, then $(f(x), f(y))\in \beta^*$.\\[6pt]
  \underline{Case 1}:  $(x,y)\in \beta^{\bB_k}$ for some $0\leq k\leq K$. \\[4pt] Then,
  $(e_i(x),e_i(y))\in \beta^{\bB_i} \subseteq \beta^*$ for all $0\leq i\leq K$,
  and $(f e_0(x),f e_0(y))\in \beta \subseteq \beta^*$ for all $f \in F_B$. Also,
  \[
  (s(x),s(y))= 
  \begin{cases}
    (x,y), & \text{ if $k=0$}\\
    (b_k,b_k), & \text{ if $k\neq 0$}
  \end{cases} 
  \]
  belongs to $\beta^*$.  Thus, $(f(x), f(y))\in \beta^*$ for all $f\in F_A$.\\[6pt]
  \underline{Case 2}: 
  $(x,y) \in \left(b_{\beta(r)}/\beta \cup \bigcup_{b_j\in T_r}
  b_j/\beta^{\bB_j}\right)^2$ for some $1\leq r\leq m$.\\[4pt]
  Assume $x\in b_j/\beta^{\bB_j}$ and $y\in b_k/\beta^{\bB_k}$ for some $b_j, b_k
  \in T_r$.  Then $(e_0(x), b_j) \in \beta$, 
  $(e_0(y), b_k) \in \beta$, and and $b_j\;\beta\; b_{\beta(r)}\; \beta\; b_k$ so 
  \begin{equation}
    \label{eq:OAcase2}
    (e_0(x), e_0(y)) \in \beta.
  \end{equation}
  Thus, for all $0\leq \ell \leq K$ we have 
  $(e_\ell e_0(x), e_\ell e_0(y)) \in \beta^{\bB_\ell}$.  But note that $e_\ell
  e_0 = e_\ell$.  It also follows from~(\ref{eq:OAcase2}) that $(f e_0(x), f
  e_0(y))\in \beta$ for all $f\in F_B$.  Finally, 
  $(s(x),s(y))=(b_j, b_k) \in \beta$.

  The only remaining possibility for case 2 is $x\in b_{\beta(r)}/\beta$ and $y\in
  b_j/\beta^{\bB_j}$ for some $b_j\in T_r$.  Since $b_j\in T_r$, we have
  $(b_j,b_{\beta(r)})\in \beta$, so 
  $(e_0(y), b_j) \in \beta$, so 
  $(e_0(y), b_{\beta(r)}) \in \beta$, so 
  $(e_0(x), x) = (e_0(y), e_0(x)) \in \beta$.
  Therefore, $(e_\ell(y), e_\ell(x)) \in \beta^{\bB_\ell}$ for all $0\leq \ell \leq K$ and 
  $(f e_0(y), f e_0(x)) \in \beta$ for all $f\in F_B$. Finally, $s(x) = x \;\beta\;
  b_{\beta(r)}\; \beta \;b_j = s(y)$, so $(s(x),s(y)) \in \beta$.

  We have established that $f(\beta^*)\subseteq \beta^*$ for all $f\in F_A$.  To
  complete the proof of Theorem~\ref{OAthm1}, we must show that
  $\beta \subseteq \eta \in \Con \bA$ implies
  $\beta^*\leq \eta$.
  If $\beta \subseteq \eta\in \Con \bA$, then $\bigcup \beta^{\bB_k} \subseteq
  \eta$, since $(x,y)\in \beta$ implies $(e_k(x), e_k(y))\in \beta^{\bB_k}$ for all
  $0\leq k\leq K$.  To see that the second term of~(\ref{eq:OAstar}) belongs to
  $\eta$, let $(x,y)$ be an arbitrary element of that term, say, $(x, b_i) \in
  \beta^{\bB_i}$ and $(y, b_j)\in \beta^{\bB_j}$.  As we just observed, $\beta,\,
  \beta^{\bB_i}$, and $\beta^{\bB_j}$ are subsets of $\eta$, and $(b_i, b_j) \in
  \beta$, so $x \; \beta^{\bB_i}\; b_i \; \beta \; b_j \;\beta^{\bB_j} \; y$, so 
  $(x,y)\in \eta$.

\end{proof}

As above, for a given $\beta \in \Con \bB$ with transversal
$\{b_{\beta(1)}, \dots, b_{\beta(m)}\}$, we denote
the set of tie-points contained in the $r$-th block of $\beta$ by $T_r$; that is,
\[
T_r = \{b\in T \mid (b, b_{\beta(r)}) \in \beta\}
=  \bigcup_{k=1}^K B_k \cap b_{\beta(r)}/\beta.
\]
Suppose this set 
is $T_r = \{b_{i_1}, b_{i_2}, \dots, b_{i_{|T_r|}}\}$ and let 
$\sI_r = \{i_1, i_2, \dots, i_{|T_r|}\}$ be the indices of these tie-points.   
Also, we define $\beta^*=\Cg^\bA(\beta)$, for $\beta\in\Con \bB$.

Figures~\ref{fig:overalgebra} and~\ref{fig:overalgebra1} illustrate these
objects for a simple example in which $B_0 = \{b_0, b_1, \dots, b_8\}$, $\beta =
|b_0, b_1, b_2 \,|\,b_3, b_4, b_5\,|\,b_6, b_7, b_8|$, and two blocks of $\beta$ contain
two tie-points each.  In particular,  the set of tie-points in the first block
of $\beta$ is $T_1 = \{b_0, b_2\}$. For the second and third blocks,
$T_2 = \emptyset$ and $T_3 = \{b_6, b_8\}$.

\begin{figure}[h!]
  \centering
      {\scalefont{.8}
        \begin{tikzpicture}[scale=.7]
          \draw[rounded corners] (-1.5,-1.5) rectangle (1.5,1.5);
          \draw[rounded corners] (.5,.5) rectangle (3.5,3.5);
          \draw[rounded corners] (.5,-3.5) rectangle (3.5,-.5);
          \draw[rounded corners] (-3.5,-3.5) rectangle (-.5,-.5); 
          \draw[rounded corners] (-3.5,.5) rectangle (-.5,3.5);
          \draw[rounded corners, dotted] (-1.35,.65) rectangle (1.35,1.35);
          \draw[rounded corners, dotted] (-1.35,-.35) rectangle (1.35,.35);
          \draw[rounded corners, dotted] (-1.35,-1.35) rectangle (1.35,-.65);

          \draw (-2.2, 0) node {$B_0 \rightarrow $};
          \draw (-2, 4) node {$B_1$};
          \draw (-1, 1) node {$b_0$};
          \foreach \i in {0,1,2} {
            \foreach \j in {1,2} {
              \pgfmathtruncatemacro{\x}{3*\i+\j}
              \draw (-\j - 1, \i + 1) node {$b^1_\x$};
            }
          }
          \foreach \i in {1,2} {
            \foreach \j in {0} {
              \pgfmathtruncatemacro{\x}{3*\i+\j}
              \draw (-\j - 1, \i + 1) node {$b^1_\x$};
            }
          }
          \draw (0, 1) node {$b_1$};

          \draw (2, 4) node {$B_2$};
          \draw (1, 1) node {$b_2$};
          \foreach \i in {0,1,2} {
            \foreach \j in {1,2} {
              \pgfmathtruncatemacro{\x}{3*\i+(2-\j)}
              \draw (\j + 1, \i + 1) node {$b^2_\x$};
            }
          }
          \foreach \i in {1,2} {
            \foreach \j in {0} {
              \pgfmathtruncatemacro{\x}{3*\i+(2-\j)}
              \draw (\j + 1, \i + 1) node {$b^2_\x$};
            }
          }
          \foreach \j in {3,4,5} {
            \draw (\j -4, 0) node {$b_\j$};
          }

          \foreach \j in {6,7,8} {
            \draw (\j -7, -1) node {$b_\j$};
          }

          \draw (-2, -4) node {$B_3$};

          \foreach \i in {0,1,2} {
            \foreach \j in {1,2} {
              \pgfmathtruncatemacro{\x}{3*(2-\i)+\j}
              \draw (-\j - 1, -\i - 1) node {$b^3_\x$};
            }
          }
          \foreach \i in {1,2} {
            \foreach \j in {0} {
              \pgfmathtruncatemacro{\x}{3*(2-\i)+\j}
              \draw (-\j - 1, -\i - 1) node {$b^3_\x$};
            }
          }

          \draw (2, -4) node {$B_4$};
          \foreach \i in {1,2} {
            \foreach \j in {0,1,2} {
              \pgfmathtruncatemacro{\x}{3*(2-\i)+\j}
              \draw (3-\j , -\i - 1) node {$b^3_\x$};
            }
          }
          \foreach \i in {0} {
            \foreach \j in {0,1} {
              \pgfmathtruncatemacro{\x}{3*(2-\i)+\j}
              \draw (3-\j, -\i - 1) node {$b^3_\x$};
            }
          }

        \end{tikzpicture}
      }
      \caption{The universe $A = B_0 \cup \cdots \cup B_4$ for a simple example; dotted
        lines surround each congruence class of $\beta$.} 
      \label{fig:overalgebra}
\end{figure}
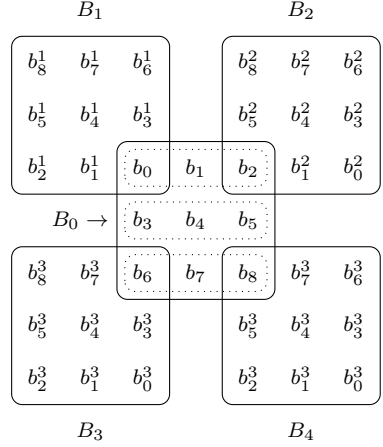

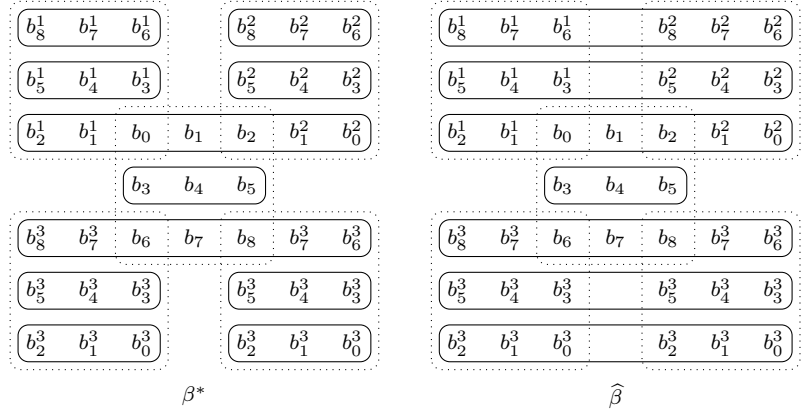
\begin{figure}[h!]
  \centering
      {\scalefont{.8}
        \begin{tikzpicture}[scale=.7]
          \draw[rounded corners,dotted] (-1.5,-1.5) rectangle (1.5,1.5);
          \draw[rounded corners,dotted] (.5,.5) rectangle (3.5,3.5);
          \draw[rounded corners,dotted] (.5,-3.5) rectangle (3.5,-.5);
          \draw[rounded corners,dotted] (-3.5,-3.5) rectangle (-.5,-.5); 
          \draw[rounded corners,dotted] (-3.5,.5) rectangle (-.5,3.5);

          \draw[rounded corners, dotted] (6.5,-1.5) rectangle (9.5,1.5);
          \draw[rounded corners, dotted] (8.5,.5) rectangle (11.5,3.5);
          \draw[rounded corners, dotted] (8.5,-3.5) rectangle (11.5,-.5);
          \draw[rounded corners, dotted] (4.5,-3.5) rectangle (7.5,-.5); 
          \draw[rounded corners, dotted] (4.5,.5) rectangle (7.5,3.5);

          \draw (0, -4) node {$\beta^*$};
          \draw[rounded corners] (-1.35,-.35) rectangle (1.35,.35);
          \draw[rounded corners] (-3.35,.65) rectangle (3.35,1.35);
          \draw[rounded corners] (-3.35,-.65) rectangle (3.35,-1.35);
          \draw[rounded corners] (-3.35,1.65) rectangle (-.65,2.35);
          \draw[rounded corners] (-3.35,2.65) rectangle (-.65,3.35);
          \draw[rounded corners] (-3.35,-1.65) rectangle (-.65,-2.35);
          \draw[rounded corners] (-3.35,-2.65) rectangle (-.65,-3.35);
          \draw[rounded corners] (4-3.35,1.65) rectangle (4-.65,2.35);
          \draw[rounded corners] (4-3.35,2.65) rectangle (4-.65,3.35);
          \draw[rounded corners] (4-3.35,-1.65) rectangle (4-.65,-2.35);
          \draw[rounded corners] (4-3.35,-2.65) rectangle (4-.65,-3.35);

          \draw (8, -4) node {$\hbeta$};
          \draw[rounded corners] (8-1.35,-.35) rectangle (8+1.35,.35); 
          \draw[rounded corners] (8-3.35,.65) rectangle (11.35,1.35); 
          \draw[rounded corners] (8-3.35,-.65) rectangle (11.35,-1.35);
          \draw[rounded corners] (8-3.35,1.65) rectangle (4+8-.65,2.35);
          \draw[rounded corners] (8-3.35,2.65) rectangle (4+8-.65,3.35);
          \draw[rounded corners] (8-3.35,-1.65) rectangle (4+8-.65,-2.35);
          \draw[rounded corners] (8-3.35,-2.65) rectangle (4+8-.65,-3.35);

          \draw (-1, 1) node {$b_0$};
          \foreach \i in {0,1,2} {
            \foreach \j in {1,2} {
              \pgfmathtruncatemacro{\x}{3*\i+\j}
              \draw (-\j - 1, \i + 1) node {$b^1_\x$};
            }
          }
          \foreach \i in {1,2} {
            \foreach \j in {0} {
              \pgfmathtruncatemacro{\x}{3*\i+\j}
              \draw (-\j - 1, \i + 1) node {$b^1_\x$};
            }
          }
          \draw (0, 1) node {$b_1$};

          \draw (1, 1) node {$b_2$};
          \foreach \i in {0,1,2} {
            \foreach \j in {1,2} {
              \pgfmathtruncatemacro{\x}{3*\i+(2-\j)}
              \draw (\j + 1, \i + 1) node {$b^2_\x$};
            }
          }
          \foreach \i in {1,2} {
            \foreach \j in {0} {
              \pgfmathtruncatemacro{\x}{3*\i+(2-\j)}
              \draw (\j + 1, \i + 1) node {$b^2_\x$};
            }
          }
          \foreach \j in {3,4,5} {
            \draw (\j -4, 0) node {$b_\j$};
          }

          \foreach \j in {6,7,8} {
            \draw (\j -7, -1) node {$b_\j$};
          }


          \foreach \i in {0,1,2} {
            \foreach \j in {1,2} {
              \pgfmathtruncatemacro{\x}{3*(2-\i)+\j}
              \draw (-\j - 1, -\i - 1) node {$b^3_\x$};
            }
          }
          \foreach \i in {1,2} {
            \foreach \j in {0} {
              \pgfmathtruncatemacro{\x}{3*(2-\i)+\j}
              \draw (-\j - 1, -\i - 1) node {$b^3_\x$};
            }
          }

          \foreach \i in {1,2} {
            \foreach \j in {0,1,2} {
              \pgfmathtruncatemacro{\x}{3*(2-\i)+\j}
              \draw (3-\j , -\i - 1) node {$b^3_\x$};
            }
          }
          \foreach \i in {0} {
            \foreach \j in {0,1} {
              \pgfmathtruncatemacro{\x}{3*(2-\i)+\j}
              \draw (3-\j, -\i - 1) node {$b^3_\x$};
            }
          }

          \draw (7, 1) node {$b_0$};
          \foreach \i in {0,1,2} {
            \foreach \j in {1,2} {
              \pgfmathtruncatemacro{\x}{3*\i+\j}
              \draw (8-\j - 1, \i + 1) node {$b^1_\x$};
            }
          }
          \foreach \i in {1,2} {
            \foreach \j in {0} {
              \pgfmathtruncatemacro{\x}{3*\i+\j}
              \draw (8-\j - 1, \i + 1) node {$b^1_\x$};
            }
          }
          \draw (8, 1) node {$b_1$};

          \draw (9, 1) node {$b_2$};
          \foreach \i in {0,1,2} {
            \foreach \j in {1,2} {
              \pgfmathtruncatemacro{\x}{3*\i+(2-\j)}
              \draw (8+\j + 1, \i + 1) node {$b^2_\x$};
            }
          }
          \foreach \i in {1,2} {
            \foreach \j in {0} {
              \pgfmathtruncatemacro{\x}{3*\i+(2-\j)}
              \draw (8+\j + 1, \i + 1) node {$b^2_\x$};
            }
          }
          \foreach \j in {3,4,5} {
            \draw (8+\j -4, 0) node {$b_\j$};
          }

          \foreach \j in {6,7,8} {
            \draw (8+\j -7, -1) node {$b_\j$};
          }


          \foreach \i in {0,1,2} {
            \foreach \j in {1,2} {
              \pgfmathtruncatemacro{\x}{3*(2-\i)+\j}
              \draw (8-\j - 1, -\i - 1) node {$b^3_\x$};
            }
          }
          \foreach \i in {1,2} {
            \foreach \j in {0} {
              \pgfmathtruncatemacro{\x}{3*(2-\i)+\j}
              \draw (8-\j - 1, -\i - 1) node {$b^3_\x$};
            }
          }

          \foreach \i in {1,2} {
            \foreach \j in {0,1,2} {
              \pgfmathtruncatemacro{\x}{3*(2-\i)+\j}
              \draw (11-\j , -\i - 1) node {$b^3_\x$};
            }
          }
          \foreach \i in {0} {
            \foreach \j in {0,1} {
              \pgfmathtruncatemacro{\x}{3*(2-\i)+\j}
              \draw (11-\j, -\i - 1) node {$b^3_\x$};
            }
          }
        \end{tikzpicture}
      }
      \caption{Solid lines show the congruence classes of $\beta^*$ (left) and 
        $\hbeta$ (right); dotted lines delineate the sets $B_i$.}
      \label{fig:overalgebra1}
\end{figure}

\begin{theorem} 
  \label{OAthm2}
  For each $\beta \in \Con \bB$, 
  \begin{equation}
    \label{eq:OAbetahat}
    \widehat{\beta} = 
    \beta^* \cup 
    \bigcup_{r=1}^m
    \bigcup^m_{\stackrel{\ell=1}{\ell \neq r}}
    \bigcup_{(j,k) \in \sI_r^2}
    \left(b^j_{\beta(\ell)}/\beta^{\bB_j} \cup b^k_{\beta(\ell)}/\beta^{\bB_k}\right)^2.
  \end{equation}
  Moreover, the interval $[\beta^*, \widehat{\beta}]$ of $\Con\bA$ 
  contains every equivalence relation of $A$ between $\beta^*$ and $\hbeta$, and
  is isomorphic to $\prod (\Eq |T_r|)^{m-1}$; that is,
  \begin{equation}
    \label{eq:OAprop2}
          [\beta^*, \widehat{\beta}] 
          = 
          \{\theta \in \Eq(A) \mid \beta^* \subseteq \theta \subseteq \widehat{\beta} \}
          \cong \prod_{r=1}^m (\Eq |T_r|)^{m-1}.
  \end{equation}
\end{theorem}
\begin{remark}
  Blocks containing only one tie-point, i.e.~those for which $|T_r| = 1$,
  contribute nothing to the direct product in~(\ref{eq:OAprop2}). Also, for some
  $1\leq r \leq m$ we may have $T_r = \emptyset$, in which case we agree 
  to let $\Eq |T_r| = \Eq (0) := \one$.
\end{remark}

\begin{proof}
  Let $\tbeta$ denote the right-hand side of~(\ref{eq:OAbetahat}).  It is
  easy to see that $\tbeta$ is an equivalence relation on $A$.  To see that it
  is also a congruence relation, we will prove $f(\tbeta) \subseteq \tbeta$ for
  all $f\in F_A$.
  Fix $(x,y)\in \tbeta$.  If 
  $(x,y)\in \beta^*$, then $(f(x), f(y)) \in \beta^*$ holds for all $f\in F_A$, as
  in Theorem~\ref{OAthm1}. 
  Suppose $(x,y)\notin \beta^*$, say,  
  $x \in b^j_{\beta(\ell)}/\beta^{\bB_j}$ and 
  $y\in b^k_{\beta(\ell)}/\beta^{\bB_k}$ for some $j, k \in \sI_r$, $1\leq r
  \leq m$, and $\ell\neq r$.  
  Then $x$ and $y$ are in the $\ell$-th blocks of their respective subreduct
  universes, $B_j$ and $B_k$, so 
  for each $0\leq i \leq K$, 
  $(e_i(x), e_i(y)) \in \beta^{\bB_i}$.
  In particular, $(e_0(x), e_0(y)) \in \beta$, so 
  $(g e_0(x), g e_0(y)) \in \beta$ for all $g\in F_B$.
  Also, $(s(x), s(y)) = (b_j, b_k) \in T_r^2 \subseteq \beta$. 
  This proves that for each $f\in F_A$ we have $(f(x), f(y)) \in \tbeta$. (In
  fact, $(f(x), f(y))\in \beta^*$.)  Whence $\tbeta \in \Con\bA$.

  Now notice that $\tbeta\resB = \beta$.  Therefore, by the residuation lemma
  of Section~\ref{sec:residuation-lemma}, we have $\tbeta \leq \widehat{\beta}$.
  To prove the reverse inclusion, we suppose $(x,y)\notin \tbeta$ and show $(x,y) \notin \widehat{\beta}$.
  Without loss of generality, assume $x \in b^j_{\beta(p)}/\beta^{\bB_j}$ and 
  $y\in b^k_{\beta(q)}/\beta^{\bB_k}$, for some $1\leq p, q
  \leq m$ and $1\leq j, k \leq K+1$.  If $p=q$, then $(j,k)\notin \sI_r^2$ for all
  $1\leq r\leq m$ (otherwise $(x,y)\in \tbeta$), so
  $(e_0s(x), e_0s(y)) = (e_0(b_j), e_0(b_k)) = (b_j, b_k) \notin \beta$,
  so $(x,y) \notin \widehat{\beta}$.  If $p\neq q$, then 
  $e_0(x) \in b_{\beta(p)}/\beta$ and $e_0(y)\in b_{\beta(q)}/\beta$ -- distinct $\beta$ classes -- so 
  $(e_0(x),e_0(y))\notin \beta$, so 
  $(x,y) \notin \widehat{\beta}$.

  To prove~(\ref{eq:OAprop2}),
  we first note that every \emph{equivalence} relation $\theta$ on $A$ with
  $\beta^* \subseteq \theta \subseteq \widehat{\beta}$ satisfies 
  $f(\theta)\subseteq \theta$ for all $f\in F_A$, and is therefore a congruence of
  $\bA$. Indeed, in proving $\tbeta= \widehat{\beta}$ above,
  we saw that $f(\tbeta)\subseteq \beta^*$ for all $f\in F_A$, so,
  {\it a fortiori}, $f(\theta)\subseteq \beta^*$ for all equivalence relations  
  $\theta \subseteq \widehat{\beta}$. 
  Therefore, 
  \[
    [\beta^*, \widehat{\beta}] = 
    \{\theta \in \Eq(A) \mid \beta^* \subseteq \theta \subseteq \widehat{\beta} \}.
    \]
    To complete the proof, we must show that this interval is isomorphic to the lattice 
    $\prod_{r=1}^m (\Eq |T_r|)^{m-1}$.
    Consider,
    \[
    \widehat{\beta}/\beta^* = \{(x/\beta^*, y/\beta^*) \in (A/\beta^*)^2 \mid (x,y) \in \widehat{\beta}\}.
    \]
    Let $N$ be the number of blocks of $\widehat{\beta}/\beta^*$ (which, of course, is the
    same as the number of blocks of $\widehat{\beta}$). For $1\leq k \leq N$, let
    $x_k/\beta^*$ be a representative of the $k$-th block of $\widehat{\beta}/\beta^*$.  Let
    $\sB_k = (x_k/\beta^*)/(\widehat{\beta}/\beta^*)$ denote this block; that is,
    \[
    \sB_k = \{y/\beta^* \in A/\beta^* \mid (x_k/\beta^*, y/\beta^*) \in
    \widehat{\beta}/\beta^*\}.
    \]
    Then,
    \[
    \prod_{k=1}^N \Eq(\sB_k) \cong \{ \theta \in \Eq(A) \mid \beta^* \subseteq \theta
    \subseteq \widehat{\beta} \} = [\beta^*, \widehat{\beta}].
    \]
    The isomorphism is given by the maps,
    \begin{align*}
      \prod_{k=1}^N \Eq(\sB_k) \ni \; & \eta \mapsto \; \bigcup_{k=1}^N \eta_k \; \in [\beta^*, \widehat{\beta}]\\
           [\beta^*, \widehat{\beta}] \ni \;  & \theta \mapsto \prod_{k=1}^N \theta \cap \sB_k^2 \in \prod_{k=1}^N \Eq(\sB_k),
    \end{align*}
    where $\eta_k$ denotes the projection of $\eta$ onto its $k$-th coordinate.

    Now, the $r$-th $\beta$-class of $B_0$, denoted $b_{\beta(r)}/\beta$, has $|T_r|$
    tie-points, so there are $|T_r|$ sets, $B_{i_1}, B_{i_2}, \dots, B_{i_{|T_r|}}$,
    each of which intersects $B_0$ at a distinct tie-point in $b_{\beta(r)}/\beta$;
    that is,
    \[
    B_{i_{j}} \cap b_{\beta(r)}/\beta = \{b_{i_j}\} \qquad (b_{i_j} \in T_r).
    \]
    (See Figure \ref{fig:overalgebra1}.) 
    A block $\sB_k$ of $\widehat{\beta}/\beta^*$ has a single element when it contains 
    $b_{\beta(r)}/\beta$.  Otherwise, it has $|T_r|$ elements, namely,
    \[
    b^{i_1}_{\beta_{(\ell)}}/\beta^{\bB_{i_2}}, \,
    b^{i_2}_{\beta_{(\ell)}}/\beta^{\bB_{i_2}}, \dots, \, b^{i_{|T_r|}}_{\beta_{(\ell)}}/\beta^{\bB_{i_{|T_r|}}},
    \]
    for some $1\leq \ell \leq m; \, \ell \neq r$.  Thus, for each $1\leq r \leq m$, we
    have $m-1$ such $|T_r|$-element blocks, so
    \[
    \prod_{k=1}^N \Eq(\sB_k) \cong 
    \prod_{r=1}^m (\Eq |T_r|)^{m-1}.
    \]

\end{proof}

We now describe the situation in which the foregoing construction is most
useful.  Here and in the sequel, instead of $\Eq(2)$, we usually write
$\two$ to denote the two element lattice. 
Given a finite congruence lattice $\Con\bB$ and a pair $(x,y) \in B^2$,
let $\beta\in \Con\bB$ be the unique smallest congruence containing $(x,y)$.
Then $\beta = \Cg^\bB(x,y)$, and if we build an overalgebra as
described above using $\{x,y\}$ as tie-points, then, by
Theorem~\ref{OAthm2}, the interval of all
$\theta \in \Con\bA$ for which $\theta\resB = \beta$ will be 
$[\beta^*,\widehat{\beta}] \cong \Eq(2)^{m-1} = \two^{m-1}$, where $m$ is the
number of congruence classes in $\beta$.  Also, since $\beta$ is the smallest
congruence containing $(x,y)$ we can be sure that, for all $\theta \ngeq \beta$,
the interval $[\theta^*,\widehat{\theta}]$ is trivial; that is,
$\theta^*=\widehat{\theta}$. Finally, for each $\theta > \beta$, we will have 
$[\theta^*,\widehat{\theta}] \cong \two^{r-1}$, where $r$ is the number of
congruence classes of $\theta$.

\begin{example}
  With the theorems above, we can explain the shapes of the congruence
  lattices of Example~\ref{ex:3.1}.  Returning to that example, with base algebra $\bB$
  equal to the right regular $S_3$-set, we now show some other congruence lattices that 
  result by simply changing the set of tie-points, $T$.  
  Recall, the relations in $\Con\bB$ are $\alpha = | 0, 1, 2 | 3, 4, 5|$,
  $\beta = | 0, 3 | 2, 5 | 1, 4 |$,
  $\gamma = | 0, 4 | 2, 3 | 1, 5|$, and
  $\delta = | 0, 5| 2, 4 | 1, 3|$.

  As Theorems~\ref{OAthm1} and~\ref{OAthm2}
  make clear, choosing $T$ to be $\{0,1\}$, $\{0,1,2\}$, or $\{0, 2, 3\}$
  yields the congruence lattices appearing in Figure~\ref{fig:ConOverAlgebras}.
  Figure~\ref{fig:ConOverAlgebras2} shows the congruences lattices resulting
  from the choices $T = \{0,1,2,3\}$ and $T = \{0, 2, 3, 5\}$. 

  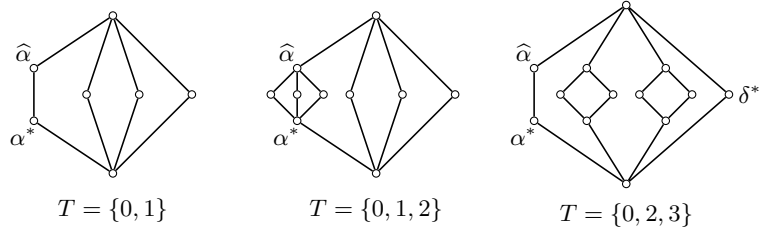
\begin{figure}[h!]
    \centering
    \begin{tikzpicture}[scale=.7]
      \node (150) at (1.5,0)  [draw, circle, inner sep=1.0pt] {};
      \node (01) at (0,1)  [draw, circle, inner sep=1.0pt] {};
      \node (02) at (0,2)  [draw, circle, inner sep=1.0pt] {};
      \node (115) at (1,1.5)  [draw, circle, inner sep=1.0pt] {};
      \node (215) at (2,1.5)  [draw, circle, inner sep=1.0pt] {};
      \node (315) at (3,1.5)  [draw, circle, inner sep=1.0pt] {};
      \node (153) at (1.5,3)  [draw, circle, inner sep=1.0pt] {};
      \draw[semithick] 
      (150) to (01) to (02) to (153) to (115) to (150) to (215) to (153) to (315) to (150);
      \draw[font=\small] (1.5,-.7) node {$T = \{0,1\}$};
      \draw[font=\small] (-.2,.7) node {$\alpha^*$};
      \draw[font=\small] (-.2,2.3) node {$\widehat{\alpha}$};

      \node (650) at (6.5,0)  [draw, circle, inner sep=1.0pt] {};
      \node (51) at (5,1)  [draw, circle, inner sep=1.0pt] {};
      \node (52) at (5,2)  [draw, circle, inner sep=1.0pt] {};
      \node (4515) at (4.5,1.5)  [draw, circle, inner sep=1.0pt] {};
      \node (515) at (5,1.5)  [draw, circle, inner sep=1.0pt] {};
      \node (5515) at (5.5,1.5)  [draw, circle, inner sep=1.0pt] {};
      \node (615) at (6,1.5)  [draw, circle, inner sep=1.0pt] {};
      \node (715) at (7,1.5)  [draw, circle, inner sep=1.0pt] {};
      \node (815) at (8,1.5)  [draw, circle, inner sep=1.0pt] {};
      \node (653) at (6.5,3)  [draw, circle, inner sep=1.0pt] {};
      \draw[semithick] 
      (650) to (51) to (515) to (52) to (653) to (615) to (650) to (715) to (653) to
      (815) to (650)
      (51) to (4515) to (52) to (5515) to (51);
      \draw[font=\small] (6.5,-.7) node {$T = \{0,1,2\}$};
      \draw[font=\small] (4.8,.7) node {$\alpha^*$};
      \draw[font=\small] (4.8,2.3) node {$\widehat{\alpha}$};

      \node (bot) at (11.25,-.2)  [draw, circle, inner sep=1.0pt] {};
      \node (top) at (11.25,3.2)  [draw, circle, inner sep=1.0pt] {};
      \node (a) at (9.5,1)  [draw, circle, inner sep=1.0pt] {};
      \node (A) at (9.5,2)  [draw, circle, inner sep=1.0pt] {};
      \draw[font=\small] (9.3,.7) node {$\alpha^*$};
      \draw[font=\small] (9.3,2.3) node {$\widehat{\alpha}$};

      \node (b) at (10.5,1)  [draw, circle, inner sep=1.0pt] {};
      \node (b1) at (10,1.5)  [draw, circle, inner sep=1.0pt] {};
      \node (b2) at (11,1.5)  [draw, circle, inner sep=1.0pt] {};
      \node (B) at (10.5,2)  [draw, circle, inner sep=1.0pt] {};

      \node (c) at (12,1)  [draw, circle, inner sep=1.0pt] {};
      \node (c1) at (11.5,1.5)  [draw, circle, inner sep=1.0pt] {};
      \node (c2) at (12.5,1.5)  [draw, circle, inner sep=1.0pt] {};
      \node (C) at (12,2)  [draw, circle, inner sep=1.0pt] {};

      \node (d) at (13.2,1.5)  [draw, circle, inner sep=1.0pt] {};
      \draw[font=\small] (13.6,1.5) node {$\delta^*$};
      \draw[semithick] 
      (bot) to (a) to (A) to (top) to (B) to (b1) to (b) to (b2) to (B)
      (b) to (bot) to (c) to (c1) to (C) to (c2) to (c)
      (C) to (top) to (d) to (bot);
      \draw[font=\small] (11.25,-.8) node {$T = \{0, 2, 3\}$};

    \end{tikzpicture}
    \caption{Congruence lattices of overalgebras of the $S_3$-set for various 
      choices of $T$, the set of tie-points.}
    \label{fig:ConOverAlgebras}
  \end{figure}

  \begin{figure}[h!]
    \centering
    \begin{tikzpicture}[scale=.7]
      \node (bot) at (3.25,0.5)  [draw, circle, inner sep=1.0pt] {};
      \node (top) at (3.25,4.5)  [draw, circle, inner sep=1.0pt] {};

      \node (a) at (1,2)  [draw, circle, inner sep=1.0pt] {};
      \node (a1) at (.5,2.5)  [draw, circle, inner sep=1.0pt] {};
      \node (a2) at (1,2.5)  [draw, circle, inner sep=1.0pt] {};
      \node (a3) at (1.5,2.5)  [draw, circle, inner sep=1.0pt] {};
      \node (A) at (1,3)  [draw, circle, inner sep=1.0pt] {};
      \draw[font=\small] (.75,1.7) node {$\alpha^*$};
      \draw[font=\small] (.75,3.3) node {$\widehat{\alpha}$};

      \node (b) at (2.5,2)  [draw, circle, inner sep=1.0pt] {};
      \node (b1) at (2,2.5)  [draw, circle, inner sep=1.0pt] {};
      \node (b2) at (3,2.5)  [draw, circle, inner sep=1.0pt] {};
      \node (B) at (2.5,3)  [draw, circle, inner sep=1.0pt] {};

      \node (c) at (4,2)  [draw, circle, inner sep=1.0pt] {};
      \node (c1) at (3.5,2.5)  [draw, circle, inner sep=1.0pt] {};
      \node (c2) at (4.5,2.5)  [draw, circle, inner sep=1.0pt] {};
      \node (C) at (4,3)  [draw, circle, inner sep=1.0pt] {};

      \node (d) at (5.5,2)  [draw, circle, inner sep=1.0pt] {};
      \node (d1) at (5,2.5)  [draw, circle, inner sep=1.0pt] {};
      \node (d2) at (6,2.5)  [draw, circle, inner sep=1.0pt] {};
      \node (D) at (5.5,3)  [draw, circle, inner sep=1.0pt] {};

      \draw[semithick] 
      (bot) to (a) to (a1) to (A) to (a2) to (a) to (a3) to (A) to (top) to 
      (B) to (b1) to (b) to (b2) to (B)
      (b) to (bot) to (c) to (c1) to (C) to (c2) to (c)
      (C) to (top) to (D) to (d1) to (d) to (d2) to (D)
      (d) to (bot);
      \draw[font=\small] (3.25,-.2) node {$T = \{0,1,2,3\}$};

      \node (Rbot) at (11.25,0.5)  [draw, circle, inner sep=1.0pt] {};
      \node (Rtop) at (11.25,4.5)  [draw, circle, inner sep=1.0pt] {};

      \node (Ra) at (9,2)  [draw, circle, inner sep=1.0pt] {};
      \node (Ra1) at (8.5,2.5)  [draw, circle, inner sep=1.0pt] {};
      \node (Ra2) at (9.5,2.5)  [draw, circle, inner sep=1.0pt] {};
      \node (RA) at (9,3)  [draw, circle, inner sep=1.0pt] {};

      \node (Rb) at (10.5,1.8)  [draw, circle, inner sep=1.0pt] {};
      \node (RB) at (10.5,3.2)  [draw, circle, inner sep=1.0pt] {};
      \draw[font=\small] (10.2,1.5) node {$\beta^*$};
      \draw[font=\small] (10.2,3.4) node {$\widehat{\beta}$};

      \node (Rc) at (12,2)  [draw, circle, inner sep=1.0pt] {};
      \node (Rc1) at (11.5,2.5)  [draw, circle, inner sep=1.0pt] {};
      \node (Rc2) at (12.5,2.5)  [draw, circle, inner sep=1.0pt] {};
      \node (RC) at (12,3)  [draw, circle, inner sep=1.0pt] {};

      \node (Rd) at (13.5,2)  [draw, circle, inner sep=1.0pt] {};
      \node (Rd1) at (13,2.5)  [draw, circle, inner sep=1.0pt] {};
      \node (Rd2) at (14,2.5)  [draw, circle, inner sep=1.0pt] {};
      \node (RD) at (13.5,3)  [draw, circle, inner sep=1.0pt] {};

      \draw[semithick] 
      (Rbot) to (Ra) to (Ra1) to (RA) to (Ra2) to (Ra) 
      (RA) to (Rtop) to (RB)
      (Rb) to (Rbot) to (Rc) to (Rc1) to (RC) to (Rc2) to (Rc)
      (RC) to (Rtop) to (RD) to (Rd1) to (Rd) to (Rd2) to (RD)
      (Rd) to (Rbot);
      \draw [semithick]  
      (Rb) to [out=140,in=-140] (RB)
      (RB) to [out=-40,in=40] (Rb);
      \draw[font=\small] (10.5,2.5) node {$L$};

      \draw[font=\small] (11.25,-.2) node {$T = \{0,2,3, 5\}$};

    \end{tikzpicture}
    \caption{Congruence lattices of overalgebras of the $S_3$-set for various 
      choices of $T$; $L\cong \two^2\times\two^2$.}
    \label{fig:ConOverAlgebras2}
  \end{figure}
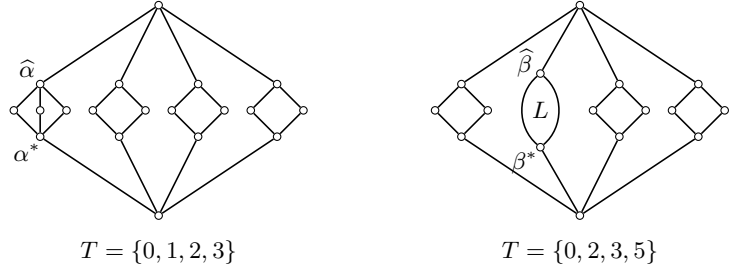

  Since $\beta = | 0, 3 | 2, 5 | 1, 4 |$, when $T = \{0,2,3, 5\}$, the interval
  $[\beta^*,\widehat{\beta}]$ is 
  $\two^2\times\two^2$.  
  In Figure~\ref{fig:ConOverAlgebras2}, we denote this abstractly by $L$,
  instead of drawing all 16 points of this interval.

\end{example}

Next, consider the situation depicted in the last congruence lattice of
Figure~\ref{fig:ConOverAlgebras2}, where 
$L \cong \two^2\times\two^2$, and
suppose we prefer that all the other $\resB$-inverse images be trivial:
$
[\beta^*,\widehat{\beta}]\cong \two^2\times\two^2; \,
\alpha^*=\widehat{\alpha}; \, 
\gamma^*=\widehat{\gamma};\,  
\delta^*=\widehat{\delta}.
$
In other words, we seek a finite algebraic
representation of the lattice in Figure~\ref{fig:ConOverAlgebras3}.
\begin{figure}[h!]
  \centering
  \begin{tikzpicture}[scale=.6]
    \node (Rbot) at (11.25,0.5)  [draw, circle, inner sep=1.0pt] {};
    \node (Rtop) at (11.25,4.5)  [draw, circle, inner sep=1.0pt] {};
    \node (Ra) at (9.25,2.5)  [draw, circle, inner sep=1.0pt] {};
    \node (Rb) at (10.75,1.8)  [draw, circle, inner sep=1.0pt] {};
    \node (RB) at (10.75,3.2)  [draw, circle, inner sep=1.0pt] {};
    \node (Rc) at (12,2.5)  [draw, circle, inner sep=1.0pt] {};
    \node (Rd) at (13.25,2.5)  [draw, circle, inner sep=1.0pt] {};
    \draw[semithick] 
    (Rbot) to (Ra) to (Rtop) to (RB)
    (Rb) to (Rbot) to (Rc) to (Rtop) to (Rd) to (Rbot);
    \draw [semithick]  
    (Rb) to [out=140,in=-140] (RB)
    (RB) to [out=-40,in=40] (Rb);
    \draw[font=\small] (10.75,2.5) node {$L$};
  \end{tikzpicture}
  \caption{A lattice which motivates further expansion of the set of basic 
    operations in the overalgebra.}
  \label{fig:ConOverAlgebras3}
\end{figure}
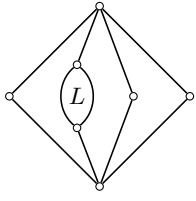
This is easy to achieve by adding more operations in the overalgebra
construction described above.
In fact, it is possible to introduce additional operations
so that, if $\beta = \Cg^\bB(x,y)$, then $\theta^* = \widehat{\theta}$ for all
$\theta \in \Con\bB$ with $\theta \ngeq \beta$.   We now describe these
operations and state this claim more formally as
Proposition~\ref{prop:expansion} below.

We start with the overalgebra construction described above.
Suppose $\beta = \Cg^\bB(x,y)$ has transversal 
$\{b_{\beta(1)}, \dots, b_{\beta(m)}\}$, and for each $1\leq r\leq m$, let 
\[
T_r = \{b\in T \mid (b, b_{\beta(r)}) \in \beta\} = 
\{b_{i_1}, b_{i_2}, \dots, b_{i_{|T_r|}}\}
\]
be the tie-points contained in the $r$-th block of $\beta$, as above. 
Let $\sI_r = \{i_1, i_2, \dots, i_{|T_r|}\}$ be the indices of these
tie-points.  Then $\{B_i : i \in \sI_r\}$ is the collection of subreduct
universes which intersect the $r$-th $\beta$ block of $B$.  
For each $1\leq r\leq m$, define the operation $s_r : A\rightarrow A$ as follows:
\[
s_r(x) =
\begin{cases}
  b_i & \text{ if $x \in B_i$ for some $i \in \sI_r$, }\\
  x & \text{otherwise}.
\end{cases}
\]
Define all other operations as above and let
\[
F_A := \{f e_0 : f\in F\} \cup \{e_k : 0\leq k \leq K\} \cup 
\{s_r : 0\leq r \leq m\},
\]
where $s_0 := s$ was defined earlier.  Finally, let 
$\bA:=\<A, F_A\>$, and define $\theta^*$ and $\widehat{\theta}$ as above.
\begin{prop}
  \label{prop:expansion}
  For each  $\theta \in \Con\bB$,
  \begin{enumerate}
  \item if $\theta \meet \beta = 0_B$, then $\theta^* = \widehat{\theta}$;
  \item if $\theta \geq \beta$, then $[\theta^*, \widehat{\theta}] \cong
    \prod_{r=1}^n (\Eq |T \cap b_{\theta(r)}/\theta|)^{n-1}$, where $n\leq m$  is
    the number of congruence classes of $\theta$.
  \end{enumerate}
\end{prop}
The first part of the proposition is easy to prove, given the additional
operations 
$s_r$, $1\leq r \leq m$.
The second part follows from Theorem~\ref{OAthm2}.

Note that $T_r$ was defined above to be $T \cap b_{\beta(r)}/\beta$, so $T =
\bigcup_{r=1}^m T_r$ is a partition of the tie-points, and it is on this partition
that our definition of the additional operations $s_r$ is based.  A modified version
of the \GAP\ function used above to construct overalgebras allows the user to specify an
arbitrary partition of the tie-points, and the extra operations will be
defined accordingly.  For example, to base the selection and partition of the
tie-points on the congruence $\beta$ in the example above, we invoke the
following command: 

{\footnotesize
\begin{verbatim}
gap> OveralgebraXO([ G, [[0,3], [2,5]] ]);
\end{verbatim}
}

\noindent The resulting overalgebra has congruence lattice isomorphic to the lattice in
Figure~\ref{fig:ConOverAlgebras3}, with
$L \cong \two^2\times\two^2$.  
Similarly, 

{\footnotesize
\begin{verbatim}
gap> OveralgebraXO([ G, [[0,1,2], [3,4,5]] ]);
\end{verbatim}
}

\noindent produces an overalgebra with congruence lattice isomorphic to the one in
Figure~\ref{fig:ConOverAlgebras3}, but with
$L \cong \Eq(3) \times\Eq(3)$.

Incidentally, with the additional operations $s_r$, we
are not limited with respect to how many terms appear in the direct
product.  For example, 

{\footnotesize
\begin{verbatim}
gap> OveralgebraXO([ G, [[0,1,2], [0,1,2], [3,4,5]] ]);
\end{verbatim}
}

\noindent produces an overalgebra with a 130 element congruence lattice 
like the one in Figure~\ref{fig:ConOverAlgebras3}, with 
$L \cong \Eq(3)\times \Eq(3)\times \Eq(3)$, while

{\footnotesize
\begin{verbatim}
gap> OveralgebraXO([ G, [[0,3], [0,3], [0,3], [0,3]] ]);
\end{verbatim}
}

\noindent gives a 261 element congruence lattice 
with $L \cong \two^{16}$. 

We close this subsection with a result which describes one way to add even more
operations to the overalgebra in case we wish to eliminate some of the
congruences in $[\beta^*, \widehat{\beta}]$ without affecting congruences outside that
interval.  In the following claim we assume the base algebra $\bB = \<B, G\>$ is a
transitive $G$-set.
\begin{claim}
Consider the collection of maps $\widehat{g}:A\rightarrow A$ defined 
for each $g\in \Stab_GT: = \{g\in G \mid gb = b \; \forall b \in T\}$ by the rules
\[
\widehat{g}\resBi = e_{g(b_i)} g e_0 \quad (i=1, \dots, n).
\]
Then, for each $\theta \in \Con \bA$,
\begin{equation}
  \label{eq:OA100}
  \widehat{g}(\theta) \nsubseteq \theta \quad \text{ only if } \quad
  \beta^* <\theta <\widehat{\beta}.
\end{equation}
\end{claim}
Of course, these $\widehat{g}$ maps may not be the only
functions in $A^A$ which have the property stated in~(\ref{eq:OA100}).
Also, in general, even with the whole collection of maps $\widehat{g}$ defined
above, we may not be able to eliminate every  $\beta^* <\theta
<\widehat{\beta}$.  
In fact, it's easy to construct examples in which there exist
$\beta^* <\theta <\widehat{\beta}$ such that 
$g(\theta) \subseteq \theta$ for every every $g \in A^A$.

\subsection{Overalgebras II}
\label{sec:overalgebras-ii}
In the previous section we described a procedure for building an
overalgebra $\bA$ of $\bB$ such that for some 
\emph{principal} congruence $\beta\in \Con\bB$ and for all 
$\beta \leq \theta < 1_B$, the inverse image $\theta \resB^{-1} = [\theta^*,
  \widehat{\theta}] \leq \Con\bA$ is non-trivial.
In this section, we start with a non-principal congruence $\beta\in \Con\bB$ and
ask if it is possible to construct an overalgebra $\bA$ such that 
$\theta\resB^{-1}\leq \Con\bA$ is non-trivial if and only if 
$\beta \leq \theta < 1_B$.  
To answer this question, we now describe an overalgebra construction that is
based on a construction proposed by Bill Lampe.

Let $\bB = \<B; F\>$ be a finite algebra, and suppose 
\[
\beta = \Cg^{\bB}((a_1, b_1), \dots, (a_K,b_K))
\]
for some $a_1, \dots, a_K, b_1, \dots, b_K \in B$.
Let $B=B_0, B_1, B_2, \dots, B_{K+1}$ be sets of cardinality $|B| = n$
which intersect as follows: 
\begin{align*}
  B_0\cap B_1 &=\{a_1\}=\{a_1^{1}\},\\
  B_i \cap B_{i+1} &= \{b_i\supi\}=\{a^{i+1}_{i+1}\} \text{ for $1\leq i < K$,}\\
  B_K\cap B_{K+1}&=\{b^{K}_K\}=\{a_1^{K+1}\}.
\end{align*}
All other intersections are empty. (See Figure~\ref{fig:OveralgebrasII}.)

\begin{figure}[h!]
  \centering
      {\scalefont{.9}
        \begin{tikzpicture}[scale=.5]
          \draw (2.3, 3.8) node {$B$};
          \draw (2.3,3.8) ellipse (1cm and 2.2cm);

          \draw (4.5, 2) node {$B_1$};
          \draw (4.5,2) ellipse (2.2cm and 1cm);

          \draw (8, 2.5) node {$B_2$};
          \draw (8,2.5) ellipse (2.2cm and 1cm);

          \draw (11.5, 2) node {$B_3$};
          \draw (11.5,2) ellipse (2.2cm and 1cm);

          \draw[font=\LARGE] (14.8, 2) node {$\cdots$};

          \draw (18,2) node {$B_K$};
          \draw (18,2) ellipse (2.2cm and 1cm);

          \draw (20.2,3.8) node {$B_{K+1}$};
          \draw (20.2,3.8) ellipse (1cm and 2.2cm);

          \node (1) at (2.7,2.25) [fill,circle,inner sep=.8pt] {};
          \draw (1.6, .8) node {$a_1=a_1^1$};
          \draw[->] (1.9, 1.2) to  (2.6,2.1);

          \node (2) at (6.25,2.25) [fill,circle,inner sep=.8pt] {};
          \draw (7.7, .8) node {$b^1_1 = a^2_2$};
          \draw[->] (6.8, 1.2) to  (6.34,2.09);

          \node (3) at (9.75,2.25) [fill,circle,inner sep=.8pt] {};
          \draw (11.3, 3.6) node {$b^2_2 = a^3_3$};
          \draw[->] (10.4,3.15) to (9.84, 2.4);

          \node (4) at (16.25,2.2) [fill,circle,inner sep=.8pt] {};
          \draw (16, 3.5) node {$b^{K-1}_{K-1} = a^K_K$};
          \draw[->] (16, 3.05) to  (16.22,2.35);

          \node (5) at (19.8,2.25) [fill,circle,inner sep=.8pt] {};
          \draw (21.5, .8) node {$b^{K}_{K} = a^{K+1}_1$};
          \draw[->] (20.2, 1.2) to  (19.85,2.07);

        \end{tikzpicture}
      }
      \caption{The universe of the overalgebra.}
      \label{fig:OveralgebrasII}
\end{figure}
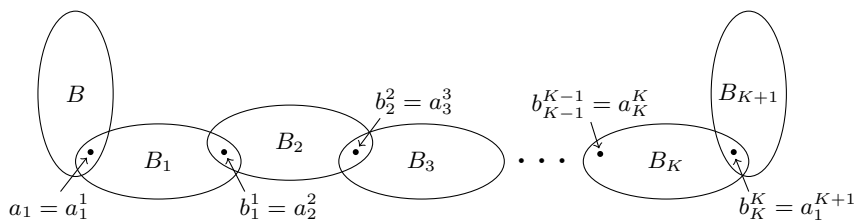

For $0\leq i, j \leq K+1$, let $S_{i,j}:B_i \rightarrow B_j $ be the 
bijection $S_{i,j}(x\supi)=x\supj$.
Put $A:=B_0\cup \dots\cup B_{K+1}$, and define the following functions in $A^A$:
\[
e_0(x)=
\begin{cases}
  x, & x\in B_0,\\
  a_1, &x\in B_j,\; 1\leq j \leq K,\\
  S_{K+1,0}(x), &x\in B_{K+1};\\
\end{cases}
\]

\[
e_i(x)=
\begin{cases}
  a_i^{i}, &x\in B_j,\; j<i,\\
  x, &x\in B_i,\\
  b_i^{i},&x\in B_j, \;j>i;
\end{cases} \qquad (1\leq i\leq K),
\]

\[
e_{K+1}(x)=
\begin{cases}
  S_{0,K+1}(x), &x\in B_0,\\
  a_{1}^{K+1}, &x\in B_j,\; 1\leq j \leq K,\\
  x, &x \in B_{K+1}.\\
\end{cases}
\]
Using these maps we define the set $F_A$ of operations on $A$ as follows: let
$q_{i,j}=S_{i,j}\circ e_i$ 
for $0\leq i, j\leq K+1$ and
define\footnote{If we were to include $q_{i,j}$
  for all $0\leq i, j\leq K+1$, the resulting overalgebra would have the same
  congruence lattice as $\<A, F_A\>$, but using a reduced set of
  operations simplifies our proofs.}  
\[
F_A := \{f e_0 : f\in F\}  \cup \{q_{i,0} : 0\leq i \leq K+1\}\cup \{q_{0,j} : 1\leq j \leq K+1\}.
\]
The overalgebra in this section is defined to be the unary algebra 
$\bA := \< A, F_A\>$.   

\begin{theorem}
  \label{OAthm3}
  Suppose $\bA = \< A, F_A\>$ is the overalgebra
  based on the congruence relation $\beta = \Cg^{\bB}((a_1, b_1), \dots, (a_K,b_K))$, as described above,
  and define
  \[ \beta^* = \bigcup_{j=0}^{K+1} \beta^{\bB_j} \cup 
  (a_1/\beta \cup a_1^1/\beta^{\bB_1} \cup a_2^2/\beta^{\bB_2}    \cup \cdots \cup a_K^K/\beta^{\bB_K}\cup a_1^{K+1}/\beta^{\bB_{K+1}})^2.
  \]
  Then, $\beta^* = \Cg^{\bA}(\beta)$.  

  If $\beta$ has transversal $\{a_1, c_1, c_2, \dots, c_{m-1}\}$, then
  \begin{equation}
    \label{eq:OA312}
    \widehat{\beta} = \beta^* \cup \bigcup_{i=1}^{m-1} (c_i/\beta \cup
    c^{K+1}_i/\beta^{\bB_{K+1}})^2.
  \end{equation}

  Moreover, $[\beta^*, \widehat{\beta}] \cong \two^{m-1}$.
\end{theorem}
\begin{proof}
  It is clear that $\beta^*$ is an equivalence relation on $A$, so we first
  check that $f(\beta^*)\subseteq 
  \beta^*$ for all $f\in F_A$.  This will establish that $\beta^*\in \Con\bA$.
  Thereafter we show that $\beta \subseteq \eta \in \Con\bA$ implies 
  $\beta^*\leq \eta$, which will prove that $\beta^*$ is the smallest congruence
  of $\bA$ containing $\beta$, as claimed in the first part of the theorem.

  Fix $(x,y) \in \beta^*$.  To show $(f(x), f(y)) \in \beta^*$ we consider two
  possible cases.
  \\[6pt]
  \underline{Case 1}: $(x,y)\in \beta^{\bB_j}$ for some $0\leq j \leq K+1$.\\[4pt]
  In this case it is easy to verify that $(q_{i,0}(x), q_{i,0}(y)) \in \beta$ and 
  $(q_{0,i}(x), q_{0,i}(y)) \in \beta^{\bB_i}$  for all $0\leq i \leq
  K+1$.  For example, if $(x,y)\in \beta^{\bB_j}$ with $1\leq j \leq K$, 
  then $(q_{0,i}(x), q_{0,i}(y))  = (a_1^i, a_1^i)$ 
  and $(q_{i,0}(x), q_{i,0}(y))$ is either $(b_i, b_i)$ or $(a_i,
  a_i)$ depending on whether $i$ is below or above $j$, respectively. If $i=j$,
  then $(q_{i,0}(x), q_{i,0}(y))$ is the pair in $B^2$ corresponding to
  $(x,y)\in \beta^{\bB_j}$, so $(q_{i,0}(x), q_{i,0}(y))\in \beta$.
  A special case is $(q_{0,0}(x), q_{0,0}(y)) \in \beta$.  Now, since 
  $q_{0,0} = e_0$, we have $(f e_{0}(x), f e_{0}(y))\in \beta$
  for all $f\in F_B$.
  Altogether, the foregoing implies that $(f(x),f(y))\in \beta^*$
  for all $f\in F_A$.
  \\[6pt]
  \underline{Case 2}: $(x,y)\in \sB^2$ where 
  $\sB := a_1/\beta \cup a_1^1/\beta^{\bB_1} \cup \cdots \cup a_K^K/\beta^{\bB_K}\cup a_1^{K+1}/\beta^{\bB_{K+1}}$.
  \\[4pt]
  Note that  $e_0(\sB) = a_1/\beta$. Therefore, 
  $(e_0(x),e_0(y)) \in  \beta$, so 
  $(fe_0(x),fe_0(y)) \in  \beta$ for all $f\in F_B$.  Also,
  \[
  q_{0,k}(\sB) = S_{0,k} e_0(\sB) = S_{0,k}(a_1/\beta) = 
  a_1^{k}/\beta^{\bB_{k}},
  \]
  which is a single block of $\beta^*$.
  Similarly,
  $e_k(\sB) = a_k^k/\beta^{\bB_k}$, so 
  \[
  q_{k,0}(\sB) = S_{k,0} e_k(\sB) = S_{k,0}(a_k^k/\beta^{\bB_k}) = a_k/\beta.
  \]
  Whence, $(x,y)\in \sB^2$ implies $(f(x), f(y)) \in \beta^*$ for all $f\in F_A$.

  We have thus established that $\beta^*$ is a congruence of $\bA$ which
  contains $\beta$.  We now show that it is the smallest such congruence.  Indeed,
  suppose $\beta \subseteq \eta \in \Con\bA$, and fix $(x,y)\in \beta^*$.
  If $(x,y)\in \beta^{\bB_j}$ for some $0\leq j \leq K+1$, then 
  $(q_{j,0}(x), q_{j,0}(y))\in \beta \subseteq \eta$, so 
  $(x, y) = (q_{0,j}q_{j,0}(x), q_{0,j}q_{j,0}(y))\in \eta$.

  If, instead of $(x,y)\in \beta^{\bB_j}$, we have 
  $(x,y)\in \sB^2$, then without loss of generality $x\in a_i^i/\beta^{\bB_i}$ and 
  $y\in a_j^j/\beta^{\bB_j}$ for some $0\leq i < j \leq K+1$.
  We only discuss the case $1\leq i < j \leq K$, as the other cases can be
  handled similarly.
  Since $x\in a_i^i/\beta^{\bB_i} = b_i^i/\beta^{\bB_i}$, we have 
  $(q_{i,0}(x), b_i) \in \beta$. Similarly, 
  $(a_j, q_{j,0}(y)) \in \beta$.  Therefore, we obtain the following 
  diagram\footnote{
    The diagram illustrates the case $1 \leq i < j \leq K$ where $i+1 < j$.  In
    case $j=i+1$, the diagram is even simpler.  Also, the cases involving $i=0$
    and/or $j=K+1$ can be handled similarly.}
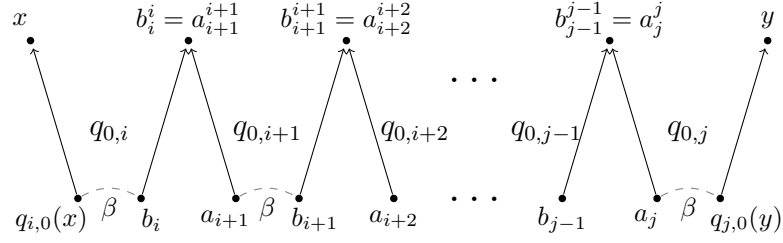
  \begin{center}
    
    \begin{tikzpicture}[scale=.7]
      \node (00) at (-.1,0) [fill,circle,inner sep=1pt] {};
      \node (10) at (1.1,0) [fill,circle,inner sep=1pt] {};
      \draw (-0.6,-.4) node {$q_{i,0}(x)$};
      \draw (.5,-.2) node {$\beta$};
      \draw (1.3,-.4) node {$b_i$};

      \node (30) at (2.9,0) [fill,circle,inner sep=1pt] {};
      \node (40) at (4.1,0) [fill,circle,inner sep=1pt] {};
      \draw (2.7,-.4) node {$a_{i+1}$};
      \draw (3.5,-.2) node {$\beta$};
      \draw (4.4,-.4) node {$b_{i+1}$};

      \node (60) at (5.9,0) [fill,circle,inner sep=1pt] {};
      \draw (5.9,-.4) node {$a_{i+2}$};

      \node (90) at (9.1,0) [fill,circle,inner sep=1pt] {};
      \draw (9.1,-.4) node {$b_{j-1}$};

      \node (110) at (10.9,0) [fill,circle,inner sep=1pt] {};
      \node (120) at (12.1,0) [fill,circle,inner sep=1pt] {};
      \draw (10.7,-.4) node {$a_j$};
      \draw (11.5,-.2) node {$\beta$};
      \draw (12.6,-.4) node {$q_{j,0}(y)$};

      \node (103) at (10,3) [fill,circle,inner sep=1pt] {};
      \node (133) at (13,3) [fill,circle,inner sep=1pt] {};
      \draw (10,3.4) node {$b_{j-1}^{j-1} = a_{j}^{j}$};

      \node (m13) at (-1,3) [fill,circle,inner sep=1pt] {};
      \node (23) at (2,3) [fill,circle,inner sep=1pt] {};
      \node (53) at (5,3) [fill,circle,inner sep=1pt] {};
      \draw (-1.2,3.4) node {$x$};
      \draw (2,3.4) node {$b_i^i = a_{i+1}^{i+1}$};
      \draw (5,3.4) node {$b_{i+1}^{i+1} = a_{i+2}^{i+2}$};
      \draw (13,3.4) node {$y$};

      \path[->] (00) edge (-.95,2.85);       \draw[font=\large] (.5,1.25) node {$q_{0,i}$};
      \path[->] (10) edge (1.95,2.85);       
      \path[->] (30) edge (2.05,2.85);       \draw[font=\large] (3.5,1.25) node {$q_{0,i+1}$};
      \path[->] (40) edge (4.95,2.85);
      \path[->] (60) edge (5.05,2.85);       \draw[font=\large] (6.3,1.25) node {$q_{0,i+2}$};

      \draw [font=\LARGE] (7.5,2.25) node {$\dots$};
      \draw [font=\LARGE] (7.5,0) node {$\dots$};

      \path[->] (90) edge (9.95,2.85);       \draw[font=\large] (8.8,1.25) node {$q_{0,j-1}$};
      \path[->] (110) edge (10.05,2.85);
      \path[->] (120) edge (12.95,2.85);       \draw[font=\large] (11.5,1.25) node {$q_{0,j}$};
      \draw[dashed, gray] 
      (00) to [out=30,in=150] (10)
      (30) to [out=30,in=150] (40)
      (110) to [out=30,in=150] (120);

    \end{tikzpicture}
  \end{center}
  Since $\beta \subseteq \eta \in \Con\bA$, and since $q_{0,k}\in F_A$ for each
  $k$, the diagram makes it clear that $(x,y)$ must belong to $\eta$.

  To prove~(\ref{eq:OA312}), let $\widetilde{\beta}$ denote the right-hand side.  That
  is,
  \[
  \widetilde{\beta}:= \beta^* \cup \bigcup_{i=1}^{m-1} 
  (c_i/\beta \cup c^{K+1}_i/\beta^{\bB_{K+1}})^2.
  \]
  It is clear that $\tbeta \in \Eq(A)$, so we verify $\widetilde{\beta} \in
  \Con \bA$ by proving that $f(\tbeta) \subseteq \tbeta$ for all $f\in F_A$.
  Fix $(x,y) \in \tbeta$.  If 
  $(x,y) \in \beta^*$, then 
  $(f(x),f(y))\in \beta^*$ for all $f\in F_A$, by the first part of the theorem.
  So suppose 
  $(x,y) \in (\CICK)^2$, 
  for some $1\leq i \leq m-1$.
  For ease of notation, define
  \[
  \cick := \CICK.
  \]
  Then, since
  $e_0(\cick) = c_i/\beta$, we have
  $(e_0(x), e_0(y))\in \beta$, so
  $(fe_0(x), fe_0(y))\in \beta$ for all $f\in F_B$.
  Also, for $0\leq k\leq K+1$, we 
  have\footnote{By $c_i^{0}/\beta^{\bB_{0}}$ we mean, of course, $c_i/\beta$.}
  \[
  q_{0,k}(\cick) = S_{0,k}(c_i/\beta) = 
  c_i^{k}/\beta^{\bB_{k}}.
  \]
  Therefore,
  $q_{0,k}(\cick)$ is in a single block of $\beta^*$, so 
  $(q_{0,k}(x), q_{0,k}(y)) \in \beta^*$.
  Also, for $1\leq k \leq K$, we have 
  $e_k(c_i/\beta) = \{a_k^k\}$ and 
  $e_k(c_i^{K+1}/\beta^{\bB_{K+1}})= \{b_k^k\}$, so 
  \[
  q_{k,0}(\cick) = S_{k,0}(\{a_k^k, b_k^k\}) = \{a_k, b_k\} \subseteq a_k/\beta,
  \]
  while, for $k=K+1$, we have 
  $e_{K+1}(\cick) = c_i^{K+1}/\beta^{\bB_{K+1}}$, so 
  \[
  q_{K+1,0}(\cick) = S_{K+1,0}(c_i^{K+1}/\beta^{\bB_{K+1}}) = c_i/\beta.
  \]
  Thus, for all $0\leq k \leq K+1$, we have 
  $(q_{k,0}(x), q_{k,0}(y)) \in \beta^*$.
  This proves that
  $(f(x),f(y))\in \beta^*\subseteq \tbeta$ holds for all $f\in F_A$, so $\tbeta
  \in \Con \bA$.

  Next, note that $\widetilde{\beta}\resB = \beta$, so
  by the residuation lemma of Section~\ref{sec:residuation-lemma}, 
  $\widetilde{\beta} \leq \widehat{\beta}$.  
  Thus, to 
  prove~(\ref{eq:OA312}), it suffices to show that 
  $(x,y)\notin \widetilde{\beta}$ implies 
  $(x,y)\notin \widehat{\beta}$.  This is straight-forward, and 
  similar to the argument we used to check the analogous fact in the proof of
  Theorem~\ref{OAthm2}.  Nonetheless, we verify most of the cases, 
  and omit only a few special cases which are easy to check.

  Suppose $(x,y)\notin \widetilde{\beta}$, and suppose
  $x\in c_p^j/\beta^{\bB_j}$ and
  $y\in c_q^k/\beta^{\bB_k}$ for some $0\leq j \leq k \leq K+1$ and $1\leq p, q
  \leq m-1$.  If $j=0$ and $k=K+1$, then $p\neq q$
  (otherwise, $(x,y) \in \tbeta$).  
  Therefore, $e_0(x) \in c_p/\beta$ and 
  $e_0(y) \in c_q/\beta$, so 
  $(e_{0}(x), e_{0}(y)) \notin \beta$, so
  $(x, y) \notin \widehat{\beta}$.
  If $p=q$, then $j\neq k$
  (otherwise, $(x,y) \in \tbeta$).  Thus,
  \begin{align*}
    (e_j(x), e_j(y)) &= (x, b^j_j) \quad \Rightarrow \quad (q_{j,0}(x), q_{j,0}(y))= (q_{j,0}(x), b_{j});\\
    (e_k(x), e_k(y)) &= (a^k_k,y) \quad \Rightarrow \quad (q_{k,0}(x), q_{k,0}(y))= (a_{k}, q_{k,0}(y)).
  \end{align*}
  One of the pairs on the right is not in $\beta$.  For if both are in $\beta$, then 
  \begin{align*}
    x = q_{0,j}q_{j,0}(x) 
    \; \beta^* \; &
    q_{0,j}(b_{j}) = b^{j}_{j} = a^{j+1}_{j+1} 
    \; \beta^* \; \cdots\\
    &\cdots \; \beta^* \;
    a^{k}_{k}  = q_{0,k}(a_{k})
    \; \beta^* \;
    q_{0,k}q_{k,0}(y) = y,
  \end{align*}
  which contradicts $(x,y)\notin \tbeta$,
  so we must have either $(q_{j,0}(x), q_{j,0}(y))\notin \beta$ or 
  $(q_{k,0}(x), q_{k,0}(y)) \notin \beta$.  Therefore,
  since $e_0 q_{i,0} = q_{i,0}$, we see that
  $(x,y)\notin \widehat{\beta}$.
  The other cases, e.g.~$x\in a_1/\beta$, 
  $y\in c_q^k/\beta^{\bB_k}$, can be checked similarly.

  It remains to prove that $[\beta^*, \widehat{\beta}] \cong \two^{m-1}$, but this
  follows easily from the first part of the proof, where we saw that $(f(x), f(y))\in
  \beta^*$ for all $f\in F_A$ and for all $(x,y)\in \widehat{\beta}$.
  This implies that all equivalence relations on $A$ that are above $\beta^*$ and below
  $\widehat{\beta}$ are, in fact, congruence relations of $\bA$.  The shape of
  this interval of equivalence relations is even simpler than the shape of the
  analogous interval we found in Theorem~\ref{OAthm2}.  In the present case, we
  have 
  \[
    [\beta^*, \widehat{\beta}] = \{\theta \in \Eq(A) \mid \beta^* \subseteq \theta \subseteq \widehat{\beta} \}
    \cong \two^{m-1}.
    \]
\end{proof}

Before stating the next result, we remind the reader that 
$\theta^* = \Cg^\bA(\theta)$ for each $\theta \in \Con\bB$. 
\begin{lemma}
  \label{lem3.1}
  If $\eta \in \Con\bA$ satisfies $\eta\resB =
  \theta$, and if $(x,y) \in \eta \setminus \theta^*$ for some 
  $x\in B_i, \, y\in B_j$, then $i=0, \, j=K+1$, and $\theta \geq \beta$.
\end{lemma}
In other words, unless $i=0$ and $j=K+1$, the congruence $\eta$ doesn't join blocks
of $B_i$ with blocks of $B_j$ (except for those already joined by
$\theta^*$).
\begin{proof}
  We rule out all $0\leq i \leq j \leq K+1$ except for $i=0$ and $j=K+1$ by
  showing that, in each of the following cases, we arrive at the contradiction
  $(x,y)\in \theta^*:= \Cg^\bA(\theta)$.\\[6pt]
  \underline{Case 1}: $i=j$.\\[4pt]
  If $(x,y)\in B_i^2$ for some $0\leq i \leq K+1$, then
  $(q_{i,0}(x),q_{i,0}(y))\in \eta\resB = \theta \leq \theta^*$, so 
  $(x,y) = (q_{0,i} q_{i,0}(x),q_{0,i} q_{i,0}(y))\in \theta^*$.
  \\[6pt]
  \underline{Case 2}: $1\leq i < j \leq K$.\\[4pt]
  In this case,
  \[ 
  (q_{i,0}(x),q_{i,0}(y)) = 
  (q_{i,0}(x), b_{i})\in \theta, \qquad
  (q_{j,0}(x),q_{j,0}(y))= (a_{j},q_{j,0}(y))\in\theta,
  \] 
  When $j= i+1$, we obtain
  \begin{equation}
    \label{eq:OA4}
    x = q_{0,i}q_{i,0}(x)\; \theta^* \; q_{0,i}(b_i) = b_i^i = a^j_j = q_{0,j}(a_j) \;
    \theta^* \; q_{0,j} q_{j,0}(y) = y,
  \end{equation}
  so $(x,y)\in \theta^*$.  
  This can be seen more transparently in a diagram.
  
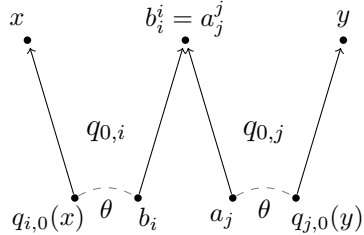
\begin{center}
    
    \begin{tikzpicture}[scale=.7]
      \node (00) at (-.1,0) [fill,circle,inner sep=1pt] {};
      \node (10) at (1.1,0) [fill,circle,inner sep=1pt] {};
      \node (30) at (2.9,0) [fill,circle,inner sep=1pt] {};
      \node (40) at (4.1,0) [fill,circle,inner sep=1pt] {};
      \draw (-0.6,-.4) node {$q_{i,0}(x)$};
      \draw (1.3,-.4) node {$b_i$};
      \draw (.5,-.2) node {$\theta$};
      \draw (2.7,-.4) node {$a_j$};
      \draw (3.5,-.2) node {$\theta$};
      \draw (4.7,-.4) node {$q_{j,0}(y)$};

      \node (m13) at (-1,3) [fill,circle,inner sep=1pt] {};
      \node (23) at (2,3) [fill,circle,inner sep=1pt] {};
      \node (53) at (5,3) [fill,circle,inner sep=1pt] {};
      \draw (-1.2,3.4) node {$x$};
      \draw (2,3.4) node {$b_i^i = a_j^j$};
      \draw (5,3.4) node {$y$};

      \path[->] (00) edge (-.95,2.85);       \draw[font=\large] (.5,1.25) node {$q_{0,i}$};
      \path[->] (10) edge (1.95,2.85);       
      \path[->] (30) edge (2.05,2.85);       \draw[font=\large] (3.5,1.25) node {$q_{0,j}$};
      \path[->] (40) edge (4.95,2.85);
      \draw[dashed, gray] 
      (00) to [out=30,in=150] (10)
      (30) to [out=30,in=150] (40);

    \end{tikzpicture}
  \end{center}
  If $j> i+1$, then
  $(q_{k,0}(x),q_{k,0}(y))= (a_k, b_k) \in \theta$ for all $i<k<j$, and we have
  the following diagram:
  
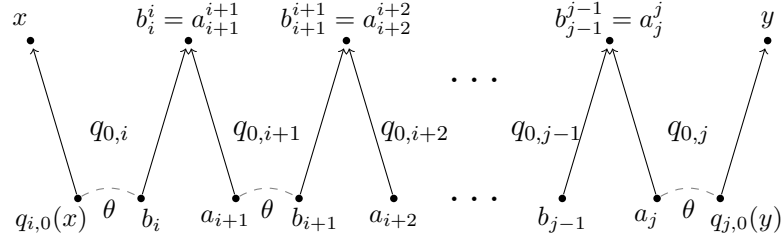
\begin{center}
    
    \begin{tikzpicture}[scale=.7]
      \node (00) at (-.1,0) [fill,circle,inner sep=1pt] {};
      \node (10) at (1.1,0) [fill,circle,inner sep=1pt] {};
      \draw (-0.6,-.4) node {$q_{i,0}(x)$};
      \draw (.5,-.2) node {$\theta$};
      \draw (1.3,-.4) node {$b_i$};

      \node (30) at (2.9,0) [fill,circle,inner sep=1pt] {};
      \node (40) at (4.1,0) [fill,circle,inner sep=1pt] {};
      \draw (2.7,-.4) node {$a_{i+1}$};
      \draw (3.5,-.2) node {$\theta$};
      \draw (4.4,-.4) node {$b_{i+1}$};

      \node (60) at (5.9,0) [fill,circle,inner sep=1pt] {};
      \draw (5.9,-.4) node {$a_{i+2}$};

      \node (90) at (9.1,0) [fill,circle,inner sep=1pt] {};
      \draw (9.1,-.4) node {$b_{j-1}$};

      \node (110) at (10.9,0) [fill,circle,inner sep=1pt] {};
      \node (120) at (12.1,0) [fill,circle,inner sep=1pt] {};
      \draw (10.7,-.4) node {$a_j$};
      \draw (11.5,-.2) node {$\theta$};
      \draw (12.6,-.4) node {$q_{j,0}(y)$};

      \node (103) at (10,3) [fill,circle,inner sep=1pt] {};
      \node (133) at (13,3) [fill,circle,inner sep=1pt] {};
      \draw (10,3.4) node {$b_{j-1}^{j-1} = a_{j}^{j}$};

      \node (m13) at (-1,3) [fill,circle,inner sep=1pt] {};
      \node (23) at (2,3) [fill,circle,inner sep=1pt] {};
      \node (53) at (5,3) [fill,circle,inner sep=1pt] {};
      \draw (-1.2,3.4) node {$x$};
      \draw (2,3.4) node {$b_i^i = a_{i+1}^{i+1}$};
      \draw (5,3.4) node {$b_{i+1}^{i+1} = a_{i+2}^{i+2}$};
      \draw (13,3.4) node {$y$};

      \path[->] (00) edge (-.95,2.85);       \draw[font=\large] (.5,1.25) node {$q_{0,i}$};
      \path[->] (10) edge (1.95,2.85);       
      \path[->] (30) edge (2.05,2.85);       \draw[font=\large] (3.5,1.25) node {$q_{0,i+1}$};
      \path[->] (40) edge (4.95,2.85);
      \path[->] (60) edge (5.05,2.85);       \draw[font=\large] (6.3,1.25) node {$q_{0,i+2}$};

      \draw [font=\LARGE] (7.5,2.25) node {$\dots$};
      \draw [font=\LARGE] (7.5,0) node {$\dots$};

      \path[->] (90) edge (9.95,2.85);       \draw[font=\large] (8.8,1.25) node {$q_{0,j-1}$};
      \path[->] (110) edge (10.05,2.85);
      \path[->] (120) edge (12.95,2.85);       \draw[font=\large] (11.5,1.25) node {$q_{0,j}$};
      \draw[dashed, gray] 
      (00) to [out=30,in=150] (10)
      (30) to [out=30,in=150] (40)
      (110) to [out=30,in=150] (120);

    \end{tikzpicture}
  \end{center}
  Here too we could write out a line analogous to~(\ref{eq:OA4}), but it is
  obvious from the diagram that $(x,y)\in \theta^*$.

  The case $i=0;\; 1 \leq j \leq K$, as well as the case
  $1 \leq i \leq K;\; j = K+1$,
  can be handled with diagrams similar to
  those used above, and the proofs are almost identical, so we omit them.

  The only remaining possibility is $x\in B_0$ and $y\in B_{K+1}$.  In this case
  we have 
  $(q_{k,0}(x),q_{k,0}(y)) = (a_k, b_k) \in \theta$,
  for all $1\leq k \leq K$.
  Therefore, $\theta \geq \beta = \Cg^{\bA}((a_1, b_1), \dots, (a_K, b_K))$.
\end{proof}

\begin{theorem}
  \label{OAthm4}
  Suppose $\bA = \< A, F_A\>$ is the overalgebra
  based on the congruence relation $\beta = \Cg^{\bB}((a_1, b_1), \dots,
  (a_K,b_K))$, as described above. Then,
  $\theta^* < \widehat{\theta}$ if and only if
  $\beta\leq \theta < 1_B$, in which case $[\theta^*, \widehat{\theta}] \cong \two^{r-1}$, where $r$
  is the number of congruence classes of $\theta$. 
\end{theorem}
Consequently, if $\theta \ngeq \beta$, then $\widehat{\theta} = \theta^*$.
\begin{proof}
  Lemma~\ref{lem3.1} implies that $\theta^* < \widehat{\theta}$ only if
  $\beta\leq \theta < 1_B$. On the other hand,  
  if $\beta\leq \theta < 1_B$, then we obtain
  $[\theta^*, \widehat{\theta}] \cong \two^{r-1}$ 
  by the same argument used to prove
  $[\beta^*, \widehat{\beta}] \cong \two^{m-1}$ in 
  Theorem~\ref{OAthm3}.  
\end{proof}

We now consider an example of a congruence lattice having a coatom $\beta$ that is not principal,
and we use the method described in this section to construct an overalgebra $\bA$ for which
$\beta^* < \widehat{\beta}$ in $\Con \bA$, and 
$\theta^* = \widehat{\theta}$ for all $\theta \ngeq \beta$ in $\Con\bB$.

\begin{example}
  
  Let $G$ be the group $C_2 \times A_4$ defined in \GAP\ as
  follows:\footnote{The \GAP\ command {\tt TransitiveGroup(12,7)} also gives a group
    isomorphic to $C_2 \times A_4$, but by defining it explicitly
    in terms of certain generators, we obtain more attractive partitions in the
    congruence lattice.}

  {\footnotesize
\begin{verbatim}
    gap> G:=Group([ (9,10)(11,12)(5,6)(7,8), 
      >               (3,7,12)(9,1,6)(11,4,8)(5,10,2), 
      >               (3,2)(9,11)(5,7)(1,4)(10,12)(6,8) ]);;
\end{verbatim}
  }

  \noindent This is a group of order 24 which acts transitively 
  on the set $\{1, 2, \dots, 12\}$.
  (If we let $H$ denote the stabilizer of a point, say $H:=G_1 \cong C_2$,
  then the group acts transitively by right multiplication on the set $G/H$ of right
  cosets. These two $G$-sets are of course isomorphic.)
  The congruence lattice of this algebra (which is isomorphic
  to the interval from $H$ up to $G$ in the subgroup lattice of $G$) is shown in 
  Figure \ref{fig:OverAlgebra-C2xA4}.
  After relabeling the elements to conform to our
  0-offset notation, the universe is
  $B:=\{0, 1, \dots, 11\}$, and 
  the non-trivial congruences are as follows:
  \begin{align*}
    \alpha &=|0, 1, 4, 5, 8, 9| 2, 3, 6, 7, 10, 11|\\
    \beta &= | 0, 1, 2, 3 | 4, 5, 6, 7 | 8, 9, 10, 11|\\
    \gamma_1 &=| 0, 1 | 2, 3 | 4, 5 | 6, 7 | 8, 9 | 10, 11|\\
    \gamma_2 &=| 0, 2 | 1, 3 | 4, 7 | 5, 6 | 8, 11 | 9, 10 |\\
    \gamma_3 &=| 0, 3 | 1, 2 | 4, 6 | 5, 7 | 8, 10 | 9, 11 |.
  \end{align*}

  \begin{figure}[h!]
    \centering
    \begin{tikzpicture}[scale=1]
      \node (0) at (.75,0) [draw, circle,inner sep=1.0pt] {};
      \node (1) at (-0,1) [draw, circle, inner sep=1.0pt] {};
      \node (2) at (0.75,1) [draw, circle, inner sep=1.0pt] {};
      \node (3) at (1.5,1) [draw, circle, inner sep=1.0pt] {};
      \node (4) at (0.75,2) [draw, circle, inner sep=1.0pt] {};
      \node (5) at (-0.75,2) [draw, circle, inner sep=1.0pt] {};
      \node (6) at (-0,3) [draw, circle, inner sep=1.0pt] {};
      \draw  (1,2) node {$\beta$};
      \draw  (-.98,2) node {$\alpha$};
      \draw  (-.25,.95) node {$\gamma_1$};
      \draw  (1,1) node {$\gamma_2$};
      \draw  (1.75,1) node {$\gamma_3$};

      \draw[semithick]
      (0) to (1)
      (0) to (2)
      (0) to (3)
      (1) to (4)
      (1) to (5)
      (2) to (4)
      (3) to (4)
      (4) to (6)
      (5) to (6);
    \end{tikzpicture}
    \caption{
      The congruence lattice of the permutational algebra $\<B, G\>$, where 
      $B = \{0, 1, \dots, 11\}$ and $G\cong C_2 \times A_4$.}
    \label{fig:OverAlgebra-C2xA4}
  \end{figure}
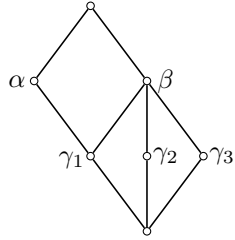

  Clearly, the coatom $\beta$ is not principal.  
  It is generated by $\{(0,3), (8,11)\}$, for example. 
  If our goal is to construct an overalgebra which has $\widehat{\beta} > \beta^*$
  in $\Con \bA$, and $\theta^* = \widehat{\theta}$ for all $\theta \ngeq \beta$ in
  $\Con\bB$, it is clear that the method described in the 
  Section~\ref{sec:overalgebras-i} will not work.
  For, if we base the overalgebra on tie-points $\{0,3\}$, then the universe is $A
  = B \cup B_1 \cup B_2$, where $B\cap B_1 = \{0\}$,  $B\cap B_2 = \{3\}$, and
  $B_1 \cap B_2 = \emptyset$, and the operations are $F_A := \{g e_0 : g\in G\}
  \cup \{e_0, e_1, e_2, s\}$. 
  Since $\beta$ has three congruence classes, by Theorem~\ref{OAthm2}
  the interval of all $\theta \in \Con\bA$ for which 
  $\theta\resB = \beta$ is $[\beta^*, \widehat{\beta}] \cong \two^2$.
  However, we also have $\gamma_3 =\Cg^\bB(0, 3)$, a congruence with 6 classes, so 
  again by Theorem~\ref{OAthm2}, $[\gamma_3^*, \widehat{\gamma_3}]\cong
  \two^5$.
  Thus, using this method it is not possible to obtain a non-trivial
  interval $[\beta^*, \widehat{\beta}]$ while preserving the original congruence
  lattice structure below $\beta$.  This is true no matter which pair $(x,y) \in
  \beta$ we choose as tie-points, since, in every case, the pair will belong to a
  congruence below $\beta$. 

  The procedure described in this subsection 
  does not have the same limitation.
  Indeed, if we set $(a_1, b_1) = (0, 3)$ and $(a_2, b_2) = (8, 11)$ in this 
  construction, then the universe of the overalgebra is $A =
  \bigcup_{i=0}^3 B_i$ where
  $B_0 = \{0, 1, \dots, 11 \}$, 
  $B_1 = \{0, 12, 13, \dots, 22 \}$,
  $B_2 = \{ 23, 24, \dots, 29, 30, 14, 31, 32, 33 \}$, and
  $B_3 = \{ 33, 34, \dots, 44 \}$.  (See Figure~\ref{fig:overalgebra2}.)

  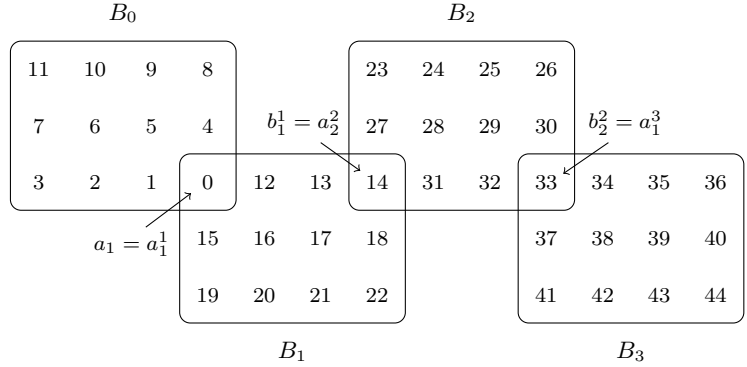
\begin{figure}[h!]
    \centering
        {\scalefont{.8}
          \begin{tikzpicture}[scale=.75]
            \draw[rounded corners] (0,2) rectangle (4,5); 
            \draw[rounded corners] (3,3) rectangle (7,0); 
            \draw[rounded corners] (6,2) rectangle (10,5); 
            \draw[rounded corners] (9,3) rectangle (13,0); 

            \draw[font=\small] (2, 5.5) node {$B_0$};
            \foreach \i in {0,1,2} {
              \foreach \j in {0,1,2,3} {
                \pgfmathtruncatemacro{\x}{4*\i+\j}
                \draw (3.5 -\j, 2.5 +\i) node {\x};
              }
            }

            \draw[font=\small] (5, -.5) node {$B_1$};
            \foreach \i in {0,1,2} {
              \foreach \j in {1,2,3} {
                \pgfmathtruncatemacro{\x}{4*\i+\j+11}
                \draw (3.5 +\j, 2.5 -\i) node {\x};
              }
            }
            \draw (3.5, 1.5) node {15};
            \draw (3.5, .5) node {19};

            \draw[font=\small] (8, 5.5) node {$B_2$};
            \foreach \i in {0,1} {
              \foreach \j in {0,1,2,3} {
                \pgfmathtruncatemacro{\x}{4*\i+\j+23}
                \draw (6.5 +\j, 4.5 -\i) node {\x};
              }
            }
            \draw (7.5, 2.5) node {31};
            \draw (8.5, 2.5) node {32};

            \draw[font=\small] (11, -.5) node {$B_3$};
            \foreach \i in {0,1,2} {
              \foreach \j in {0,1,2,3} {
                \pgfmathtruncatemacro{\x}{4*\i+\j+33}
                \draw (9.5 +\j, 2.5 -\i) node {\x};
              }
            }
            \draw (2.15, 1.4) node {$a_1=a_1^1$};
            \draw[->] (2.4, 1.7) to  (3.2,2.3);
            \draw (5.2, 3.55) node {$b^1_1 = a^2_2$};
            \draw[->] (5.5, 3.2) to  (6.2,2.7);
            \draw (10.9, 3.55) node {$b^2_2 = a^3_1$};
            \draw[->] (10.6,3.2) to (9.8, 2.65);


          \end{tikzpicture}
        }
        \caption{The universe of the overalgebra of the $(C_2\times A_4)$-set,
          arranged to reveal the congruences above $\beta^*$.}
        \label{fig:overalgebra2}
  \end{figure}

  Arranging the subreduct universes as in Figure~\ref{fig:overalgebra2}
  reveals the congruences above $\beta^*$.  
  In fact, the four congruences in the interval $[\beta^*, \widehat{\beta}]$ can be read off
  directly from the diagram.   For example, the congruence classes of $\beta^*$
  are shown in Figure~\ref{fig:overalgebra2cong}, while 
  the congruence $\widehat{\beta}$, in addition to these relations, joins blocks 
  $|4, 5, 6, 7|$ and $|37, 38, 39, 40|$, 
  as well as blocks $|8,9,10,11|$ and $|41,42,43,44|$.
  As for the congruences $\beta_\eps, \, \beta_{\eps'}$, one joins
  $|4, 5, 6, 7|$ and $|37, 38, 39, 40|$, while the other joins
  $|8,9,10,11|$ and $|41,42,43,44|$.  The full congruence lattice, $\Con \bA$,
  appears in Figure~\ref{fig:OverAlgebra-C2xA4-final}.

  \begin{figure}[h!]
    \centering
        {\scalefont{.8}
          \begin{tikzpicture}[scale=.75]
            \draw[rounded corners, dotted] (0,2) rectangle (4,5); 
            \draw[rounded corners, dotted] (3,3) rectangle (7,0); 
            \draw[rounded corners, dotted] (6,2) rectangle (10,5); 
            \draw[rounded corners, dotted] (9,3) rectangle (13,0); 

            \foreach \i in {0,1,2} {
              \foreach \j in {0,1,2,3} {
                \pgfmathtruncatemacro{\x}{4*\i+\j}
                \draw (3.5 -\j, 2.5 +\i) node {\x};
              }
            }

            \foreach \i in {0,1,2} {
              \foreach \j in {1,2,3} {
                \pgfmathtruncatemacro{\x}{4*\i+\j+11}
                \draw (3.5 +\j, 2.5 -\i) node {\x};
              }
            }
            \draw (3.5, 1.5) node {15};
            \draw (3.5, .5) node {19};

            \foreach \i in {0,1} {
              \foreach \j in {0,1,2,3} {
                \pgfmathtruncatemacro{\x}{4*\i+\j+23}
                \draw (6.5 +\j, 4.5 -\i) node {\x};
              }
            }
            \draw (7.5, 2.5) node {31};
            \draw (8.5, 2.5) node {32};

            \foreach \i in {0,1,2} {
              \foreach \j in {0,1,2,3} {
                \pgfmathtruncatemacro{\x}{4*\i+\j+33}
                \draw (9.5 +\j, 2.5 -\i) node {\x};
              }
            }

            \draw[rounded corners]  (.2,4.2) rectangle (3.8,4.8);
            \draw[rounded corners] (6.2,4.2) rectangle (9.8,4.8);

            \draw[rounded corners] (.2,3.2) rectangle (3.8,3.8);
            \draw[rounded corners] (6.2,3.2) rectangle (9.8,3.8);

            \draw[rounded corners] (.2,2.2) rectangle (12.8,2.8);

            \draw[rounded corners] (3.2,1.2) rectangle (6.8,1.8);
            \draw[rounded corners] (9.2,1.2) rectangle (12.8,1.8);

            \draw[rounded corners] (3.2,0.2) rectangle (6.8,0.8);
            \draw[rounded corners] (9.2,0.2) rectangle (12.8,0.8);

          \end{tikzpicture}
        }
        \caption{The universe of the overalgebra;
          solid lines delineate the congruence classes of $\beta^*$.}
        \label{fig:overalgebra2cong}
  \end{figure}
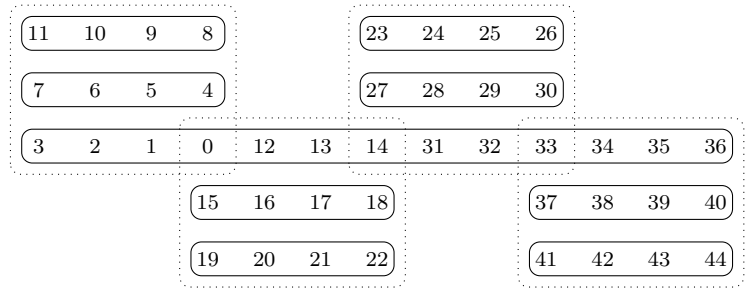

  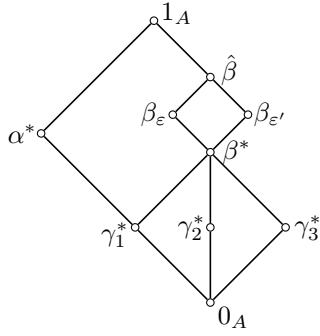
\begin{figure}[h!]
    \centering
        {\scalefont{1}
          \begin{tikzpicture}[scale=1]
            \node (0) at (1,0) [draw, circle,inner sep=1.0pt] {};
            \draw  (1.3,-.2) node {$0_A$};
            \node (6) at (0.25,3.75) [draw, circle, inner sep=1.0pt] {};
            \draw  (.55,3.85) node {$1_A$};

            \node (top) at (1,3) [draw, circle, inner sep=1.0pt] {};
            \draw  (1.25,3.1) node {$\hat{\beta}$};

            \node (b1) at (0.5,2.5) [draw, circle, inner sep=1.0pt] {};
            \draw  (0.25,2.5) node {$\beta_{\eps}$};
            \node (b2) at (1.5,2.5) [draw, circle, inner sep=1.0pt] {};
            \draw  (1.8,2.5) node {$\beta_{\eps'}$};

            \node (bot) at (1,2) [draw, circle, inner sep=1.0pt] {};
            \draw  (1.35,2) node {$\beta^*$};

            \node (5) at (-1.25,2.25) [draw, circle, inner sep=1.0pt] {};
            \draw  (-1.5,2.25) node {$\alpha^*$};

            \node (1) at (-0,1) [draw, circle, inner sep=1.0pt] {};
            \draw (-.27,.9) node {$\gamma_1^*$};
            \node (2) at (1,1) [draw, circle, inner sep=1.0pt] {};
            \draw  (.75,1) node {$\gamma_2^*$};
            \node (3) at (2,1) [draw, circle, inner sep=1.0pt] {};
            \draw  (2.3,1) node {$\gamma_3^*$};

            \draw[semithick]
            (0) to (1)
            (0) to (2)
            (0) to (3)
            (1) to (bot)
            (1) to (5)
            (2) to (bot)
            (3) to (bot)
            (bot) to (b1) to (top) to (b2) to (bot)
            (top) to (6)
            (5) to (6);
          \end{tikzpicture}
          \caption{
            The congruence lattice of the overalgebra $\<A, F_A\>$ of $\<B, G\>$, where
            $B = \{0, 1, \dots, 11\}$ and $G\cong C_2 \times A_4$.} 
          \label{fig:OverAlgebra-C2xA4-final}
        }
  \end{figure}

\end{example}

\subsection{Overalgebras III}
\label{sec:overalgebras-iii}
In Section~\ref{sec:overalgebras-i} we  constructed an algebra $\bA$ with
a congruence lattice $\Con \bA$ having interval sublattices $[\beta^*, \hbeta]$
that are isomorphic to products of powers of partition lattices.  We saw that the
construction has two main limitations.  First, the size of the partition
lattices is limited by the size of the congruence classes of $\beta \in
\Con\bB$.  Second, when $\beta$ is non-principal, it is impossible with this
construction to obtain a nontrivial inverse image $[\beta^*, \hbeta]$ without
also having nontrivial inverse images $[\theta^*, \htheta]$ for some $\theta \ngeq
\beta$.  In Section~\ref{sec:overalgebras-ii}, we presented a construction which
resolves the second limitation. However, the first limitation is even more
severe in that the resulting intervals $[\beta^*, \hbeta]$ are simply powers of
$\two$ -- i.e., Boolean algebras.  In this section, we present a generalization
of the previous constructions which overcomes both of the limitations mentioned above.

Let $\bB = \<B, F\>$ be a finite algebra, and suppose 
\[
\beta = \Cg^{\bB}((a_1, b_1), \dots, (a_{K-1},b_{K-1}))
\]
for some $a_1, \dots, a_{K-1}, b_1, \dots, b_{K-1} \in B$.
Define $B_0=B$ and, for some fixed $Q\geq 0$, let $B_1, B_2, \dots, B_{(2Q+1)K}$ be sets of
cardinality $|B| = n$.  As above, we use the label $x^i$ to denote the element of $B_i$ which
corresponds to $x\in B$ under the bijection.  
For ease of notation, let $M:=(2Q+1)$.  
We arrange the sets so that they intersect as follows:
\begin{align*}
  B_0\cap B_1 &=\{a_1\}=\{a_1^{1}\},\\
  B_1\cap B_2 &=\{b^1_1\}=\{a_2^{2}\},\\
  B_2\cap B_3 &=\{b^2_2\}=\{a_3^{3}\},\\
\vdots\\
B_{K-2}\cap B_{K-1} &= \{b_{K-2}^{K-2}\}=\{a_{K-1}^{K-1}\},\\
  B_{K-1}\cap B_K = B_K\cap B_{K+1}&=\{b^{K-1}_{K-1}\}=\{b^{K}_{K-1}\}=\{b^{K+1}_{K-1}\},\\
  B_{K+1}\cap B_{K+2}&=\{a^{K+1}_{K-1}\} =\{b^{K+2}_{K-2}\},\\
  B_{K+2}\cap B_{K+3}&=\{a^{K+2}_{K-2}\} =\{b^{K+3}_{K-3}\}, \dots
\end{align*}
\begin{align*}
\dots,  B_{2K-2}\cap B_{2K-1} &= \{a_{2}^{2K-2}\}=\{b_{1}^{2K-1}\},\\
  B_{2K-1}\cap B_{2K} = B_{2K}\cap  B_{2K+1}&=\{a^{2K-1}_{1}\}=\{a^{2K}_{1}\}=\{a^{2K+1}_{1}\},\\
  B_{2K+1}\cap B_{2K+2}&=\{b^{2K+1}_{1}\} =\{b^{2K+2}_{2}\},\\
  B_{2K+2}\cap B_{2K+3}&=\{b^{2K+2}_{2}\} =\{b^{2K+3}_{3}\},\\
  \vdots\\
  B_{MK-2}\cap B_{MK-1} &= \{b_{MK-2}^{K-2}\}=\{a_{MK-1}^{K-1}\},\\
  B_{MK-1}\cap B_{MK}&=\{b^{MK-1}_{K-1}\}=\{b^{MK}_{K-1}\}.
\end{align*}
All other intersections are empty. (See Figure~\ref{fig:OveralgebrasIII}.)

\begin{figure}[h!]
  \centering
      {\scalefont{.7}
        \begin{tikzpicture}[scale=.52]
          \draw (0, 3.4) node {$B$};
          \draw (0,3) ellipse (.6cm and 1.2cm);

          \draw (1,2) node {$B_1$};
          \draw (1,2) ellipse (1.2cm and .6cm);

          \draw (3,2) node {$B_2$};
          \draw (3,2) ellipse (1.2cm and .6cm);


          \draw[font=\Large] (5, 2) node {$\cdots$};

          \draw (7,2) node {$B_{K-2}$};
          \draw (7,2) ellipse (1.2cm and .6cm);

          \draw (9,2.2) node {$B_{K-1}$};
          \draw (9,2) ellipse (1.2cm and .6cm);

          \draw (10,.7) node {$B_{K}$};
          \draw (10,1) ellipse (.6cm and 1.2cm);

          \draw (11,2.2) node {$B_{K+1}$};
          \draw (11,2) ellipse (1.2cm and .6cm);

          \draw (13,2.2) node {$B_{K+2}$};
          \draw (13,2) ellipse (1.2cm and .6cm);

          \draw[font=\Large] (15, 2) node {$\cdots$};

          \draw (16.9,1.8) node {$B_{2K-1}$};
          \draw (17,2) ellipse (1.2cm and .6cm);

          \draw (18,3.3) node {$B_{2K}$};
          \draw (18,3) ellipse (.6cm and 1.2cm);

          \draw (19.1,1.8) node {$B_{2K+1}$};
          \draw (19,2) ellipse (1.2cm and .6cm);

          \draw[font=\Large] (21, 2) node {$\cdots$};

          \node (1) at (0,2) [fill,circle,inner sep=.6pt] {};
          \draw (0, .6) node {$a_1{=}a_1^1$};
          \draw (0, 1) to  (0,1.9);

          \node (2) at (2,2) [fill,circle,inner sep=.6pt] {};
          \draw (2, 3.3) node {$b^1_1 {=} a^2_2$};
          \draw (2, 3) to  (2,2.1);

          \node (3) at (4,2) [fill,circle,inner sep=.6pt] {};
          \draw (4, .6) node {$b^2_2 {=} a^3_3$};
          \draw (4, 1) to  (4,1.9);

          \node (4) at (6,2) [fill,circle,inner sep=.6pt] {};

          \node (5) at (8,2) [fill,circle,inner sep=.6pt] {};
          \draw (7.5, .6) node {$b^{K-2}_{K-2} {=} a^{K-1}_{K-1}$};
          \draw (8, 1) to  (8,1.9);

          \node (6) at (10,2) [fill,circle,inner sep=.6pt] {};
          \draw (10, 3.4) node {$b^{K-1}_{K-1} {=} b^{K}_{K-1}{=} b^{K+1}_{K-1}$};
          \draw (10, 3) to  (10,2.1);

          \node (7) at (12,2) [fill,circle,inner sep=.6pt] {};
          \draw (12.5, .6) node {$a^{K+1}_{K-1} {=} b^{K+1}_{K-2}$};
          \draw (12, 1.2) to  (12,1.9);

          \node (8) at (14,2) [fill,circle,inner sep=.6pt] {};

          \node (9) at (16,2) [fill,circle,inner sep=.6pt] {};
          \draw (16, 3.4) node {$b^{2K-1}_{1}$};
          \draw (16, 3) to  (16,2.1);

          \node (10) at (18,2) [fill,circle,inner sep=.6pt] {};
          \draw (18, .6) node {$a^{2K-1}_{1} {=} a^{2K}_{1}{=} a^{2K+1}_{1}$};
          \draw (18, 1) to  (18,1.9);

          \node (11) at (20,2) [fill,circle,inner sep=.6pt] {};
          \draw (20, 3.4) node {$b^{2K+1}_{1}$};
          \draw (20, 3) to  (20,2.1);
          
        \end{tikzpicture}
      }
      \caption{The universe of the overalgebra.}
      \label{fig:OveralgebrasIII}
\end{figure}
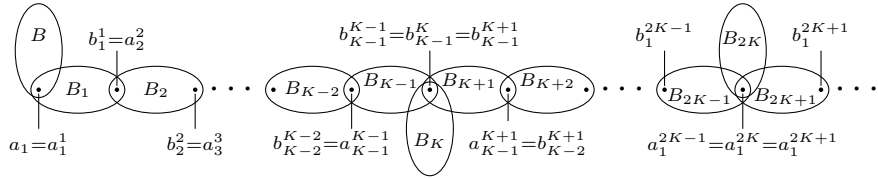
As usual, we put $A:=B_0\cup \dots\cup B_{MK}$, and we proceed to define
some unary operations on $A$.  

First, for $0\leq i, j \leq MK$, let $S_{i,j}:B_i \rightarrow B_j $ be the 
bijection $S_{i,j}(x\supi)=x\supj$, and note that $S_{i,i} = \id_{B_i}$.
Define the following subsets of even and odd multiples of $K$, respectively: 
$\sE = \{2qK : q = 0, 1, \dots, Q\}$ and 
$\sO= \{(2q+1)K : q = 0, 1, \dots, Q\}$. 
For each $\ell\in \sE$, let\\
\[
e_{\ell}(x)=
\begin{cases}
  S_{j,\ell}(x), &\text{ if $x\in B_{j}$ for some $j \in \sE$,}\\
  a^{\ell}_1, &\text{ otherwise.}
\end{cases}
\]
and, for $0 < i < K$,
\[
e_{\ell+i}(x)=
\begin{cases}
  a_i^{\ell+i}, &\text{ if $x\in B_{j}$ for some $j < \ell+i$,}\\ 
  x, &\text{ if $x\in B_{\ell+i}$,}\\
  b_i^{\ell+i}, &\text{ if $x\in B_{j}$ for some $j > \ell+i$.}
\end{cases}
\]
For each $\ell\in \sO$, let \\
\[
e_{\ell}(x)=
\begin{cases}
  S_{j,\ell}(x), &\text{ if $x\in B_{j}$ for some $j \in \sO$,}\\
  b^{\ell}_{K-1}, &\text{ otherwise.}
\end{cases}
\]
and, for $0 < i < K$,
\[
e_{\ell+i}(x)=
\begin{cases}
  b_{K-i}^{\ell+i}, &\text{ if $x\in B_{j}$ for some $j < \ell+i$,}\\ 
  x, &\text{ if $x\in B_{\ell+i}$,}\\
  a_{K-i}^{\ell+i}, &\text{ if $x\in B_{j}$ for some $j > \ell+i$,}
\end{cases}
\]
In other words, if $\ell\in \sE$, then $e_{\ell}$ maps each up-pointing set in
Figure~\ref{fig:OveralgebrasIII} bijectively onto the up-pointing set 
$B_{\ell}$, and maps all other points of $A$ to the tie-point
$a^{\ell}_1\in B_{\ell}$; 
if $\ell\in \sO$, then $e_{\ell}$ maps each down-pointing set in the figure onto the
down-pointing set $B_{\ell}$, and maps all other points to the tie-point
$b^{\ell}_{K-1}$.  For each set $B_{\ell+i}$ in between -- represented in the
figure by an ellipse with horizontal major axis -- there corresponds a map
$e_{\ell+i}$ which act as the identity on $B_{\ell+i}$ and maps all points in
$A$ left of $B_{\ell+i}$ to the left tie-point of $B_{\ell+i}$ and all points to
the right of $B_{\ell+i}$ to the right tie-point of $B_{\ell+i}$.

Finally, for $0\leq i, j\leq MK$, we define
$q_{i,j}=S_{i,j}\circ e_i$ and take the set of basic operations on $A$ to be
\[
F_A := \{f e_0 : f\in F\}  \cup \{q_{i,0} : 0\leq i \leq MK\}\cup \{q_{0,j} : 1\leq j \leq MK\}.
\]
We then consider the overalgebra $\bA := \< A, F_A\>$.   This overalgebra is,
once again, based on the specific congruence
$\beta = \Cg^{\bB}((a_1, b_1), \dots, (a_{K-1},b_{K-1})) \in \Con \bB$, and the following
theorem describes the inverse image of $\beta$ under $\resB$ -- that is, 
the interval $[\beta^*,\hbeta]$ in $\Con \bA$.
\begin{theorem}
  \label{thm-overalgebras-iii}
  Let $\bA = \< A, F_A\>$ be the overalgebra described above,
  and, for each $0\leq i \leq MK$, let $t_i$ denote a tie-point of the set 
  $B_i$.  Define
  \[ \beta^* = \bigcup_{j=0}^{MK} \beta^{\bB_j} \cup 
  \left(\bigcup_{i=0}^{MK}t_i/\beta^{\bB_i}\right)^2.
  \]
  Then, $\beta^* = \Cg^{\bA}(\beta)$.

  If $\beta$ has transversal $\{a_1, c_1, c_2, \dots, c_{m-1}\}$, then
  \begin{equation}
    \label{eq:OA5}
    \widehat{\beta} = \beta^* 
    \cup 
    \bigcup_{i=1}^{m-1}\left(\bigcup_{\ell \in \sE} c_i^\ell/\beta^{\bB_{\ell}}\right)^2
    \cup 
    \bigcup_{i=1}^{m-1}\left(\bigcup_{\ell \in \sO}
    c_i^\ell/\beta^{\bB_{\ell}}\right)^2.
  \end{equation}

  Moreover, $[\beta^*, \widehat{\beta}] \cong (\Eq|\sE|)^{m-1} \times (\Eq|\sO|)^{m-1}$. 
\end{theorem}
\begin{remark}
  Recall that $m$ is the number of congruence classes in $\beta$.
  The number of up-pointing sets in Figure~\ref{fig:OveralgebrasIII}
  is $|\sE|$, while $|\sO|$ counts the number of
  down-pointing sets.  In our construction, we took 
  $|\sE| = |\sO| = Q+1$, but, apart from being notationally convenient, this
  choice was arbitrary; in fact, there's no reason $\sE$ and $\sO$ should be equal
  in number, and they could even be empty.  Choosing $\sO = \emptyset$, for
  example, would result in the interval 
  $[\beta^*, \widehat{\beta}] \cong (\Eq|\sE|)^{m-1}$.  Thus, for any $N$, we can
  construct an algebra $\bA$ that has
  $(\Eq N)^{m-1} \cong [\beta^*, \widehat{\beta}] < \Con \bA$.  

\end{remark}

\noindent {\it Proof of Theorem~\ref{thm-overalgebras-iii}.}
It is easy to check that $\beta^*$ is an equivalence relation on $A$, so we first
check that $f(\beta^*)\subseteq 
\beta^*$ for all $f\in F_A$.  This will establish that $\beta^*\in \Con\bA$.
Thereafter we show that $\beta \subseteq \eta \in \Con\bA$ implies 
$\beta^*\leq \eta$, which will prove that $\beta^*$ is the smallest congruence
of $\bA$ containing $\beta$, as claimed in the first part of the theorem.

Fix $(x,y) \in \beta^*$.  To show $(f(x), f(y)) \in \beta^*$ we consider two
possible cases.
\\[6pt]
\underline{Case 1}: $(x,y)\in \beta^{\bB_j}$ for some $0\leq j \leq (2q+1)K$.\\[4pt]
In this case it is easy to verify that $(q_{i,0}(x), q_{i,0}(y)) \in \beta$ and 
$(q_{0,i}(x), q_{0,i}(y)) \in \beta^{\bB_i}$  for all $0\leq i \leq
K+1$.  For example, if $(x,y)\in \beta^{\bB_j}$ with $1\leq j \leq K$, 
then $(q_{0,i}(x), q_{0,i}(y))  = (a_1^i, a_1^i)$ 
and $(q_{i,0}(x), q_{i,0}(y))$ is either $(b_i, b_i)$ or $(a_i,
a_i)$ depending on whether $i$ is below or above $j$, respectively. If $i=j$,
then $(q_{i,0}(x), q_{i,0}(y))$ is the pair in $B^2$ corresponding to
$(x,y)\in \beta^{\bB_j}$, so $(q_{i,0}(x), q_{i,0}(y))\in \beta$.
A special case is $(q_{0,0}(x), q_{0,0}(y)) \in \beta$.  Therefore,
$q_{0,0} = e_0$, implies $(f e_{0}(x), f e_{0}(y))\in \beta$
for all $f\in F_B$.
Altogether, we have proved that $(f(x),f(y))\in \beta^*$
for all $f\in F_A$.
\\[6pt]
\underline{Case 2}: $(x,y)\in \sB^2$ where $\sB := \bigcup_{i=0}^{MK}t_i/\beta^{\bB_i}$.
\\[4pt]
Note that  $e_0(\sB) = a_1/\beta$. Therefore, 
$(e_0(x),e_0(y)) \in  \beta$, so 
$(fe_0(x),fe_0(y)) \in  \beta$ for all $f\in F_B$.  Also,
\[
q_{0,k}(\sB) = S_{0,k} e_0(\sB) = S_{0,k}(a_1/\beta) = 
a_1^{k}/\beta^{\bB_{k}},
\]
which is a single block of $\beta^*$.
Similarly,
$e_k(\sB) = t_k/\beta^{\bB_k}$, so 
\[
q_{k,0}(\sB) = S_{k,0} e_k(\sB) = S_{k,0}(t_k/\beta^{\bB_k}) = S_{k,0}(t_k)/\beta,
\]
a single block of $\beta^*$.
Whence, $(x,y)\in \sB^2$ implies $(f(x), f(y)) \in \beta^*$ for all $f\in F_A$.

We have thus established that $\beta^*$ is a congruence of $\bA$ which
contains $\beta$.  We now show that $\beta^*$ is the smallest such congruence.  Indeed,
suppose $\beta \subseteq \eta \in \Con\bA$, and fix $(x,y)\in \beta^*$.
If $(x,y)\in \beta^{\bB_j}$ for some $0\leq j \leq MK$, then 
$(q_{j,0}(x), q_{j,0}(y)) = (S_{j,0}e_j(x), S_{j,0}e_j(y))= (S_{j,0}(x), S_{j,0}(y))\in \beta \subseteq \eta$, so 
$(x, y) = (q_{0,j}q_{j,0}(x), q_{0,j}q_{j,0}(y))\in \eta$.

If, instead of $(x,y)\in \beta^{\bB_j}$, we have 
$(x,y)\in \sB^2$, then without loss of generality $x\in a_i^i/\beta^{\bB_i}$ and 
$y\in a_j^j/\beta^{\bB_j}$ for some $0\leq i < j \leq K+1$.
Then, 
$(q_{i,0}(x),q_{i,0}(t_i)) \in \beta$ and 
$(q_{j,0}(t_j),q_{j,0}(y)) \in \beta$ and, since $i<j$, there is a sequence of tie points
$c_i^i, d_{i+1}^{i+1}, c_{i+1}^{i+1}, d_{i+2}^{i+2}, c_{i+2}^{i+2}, \dots,
c_j^j$ (where $\{c, d\} = \{a, b\}$) such that
\begin{equation}
  \label{eq:OA3}
  t_i \, \beta^{\bB_i}\, c_i^i = 
  d_{i+1}^{i+1} \, \beta^{\bB_{i+1}}\, c_{i+1}^{i+1} =
  d_{i+2}^{i+2} \, \beta^{\bB_{i+2}}\, c_{i+2}^{i+2}= \dots
  = c_j^j \, \beta^{\bB_{j}} \, t_j.
\end{equation}
We could sketch a diagram similar to the one given in the proof of
Theorem~\ref{OAthm3}, but it should be obvious by now that the relations~(\ref{eq:OA3})
imply $(t_i, t_j) \in \eta$.  Therefore, $\beta^* = \Cg^{\bA}(\beta)$.

Next we prove equation~(\ref{eq:OA5}).  Let $\tbeta$ denote the right hand side
of~(\ref{eq:OA5}).  We first show $\tbeta\in \Con\bA$.

Let 
\[
\CE:= \bigcup_{\ell \in \sE} c_i^\ell/\beta^{\bB_{\ell}} \quad \text{ and }
\quad 
\CO:=\bigcup_{\ell \in \sO} c_i^\ell/\beta^{\bB_{\ell}}.
\]
Note that $\CE$ is the join of the corresponding ($i$-th) $\beta$ blocks in the
up-pointing sets in Figure~\ref{fig:OveralgebrasIII}.  Thus, $\CE$ can be
visualized as a single slice through all of the up-pointing sets.  Similarly,
$\CO$ is the join of corresponding blocks in the down-pointing sets 
in Figure~\ref{fig:OveralgebrasIII}.  
If $0< i< K$ and $\ell \in \sE$, then
$e_{\ell+i}(\CE) = e_{\ell+i}(\CO) = \{a_i^{\ell+i}, b_i^{\ell+i}\}$.
Thus, for each such $k=\ell+i$ we have 
\[
q_{k0}(\CE) = S_{k0}\,e_k(\CE) = S_{k0}\,e_k(\CO) = q_{k0}(\CO)
= \{a_i, b_i\},
\]
a single block of $\beta$.
Similarly, if $0< i< K$ and $\ell \in \sO$, then
$e_{\ell+i}(\CE) = 
e_{\ell+i}(\CO) = \{a_{K-i}^{\ell+i}, b_{K-i}^{\ell+i}\}$.
Thus, for each such $k=\ell+i$ we have 
\[
q_{k0}(\CE) = S_{k0}e_k(\CE) = S_{k0}e_k(\CO) = q_{k0}(\CO)
= \{a_{K-i}, b_{K-i}\},
\]
which is also a single block of $\beta$.  
It follows that $q_{k0}(\tbeta)\subseteq \tbeta$ for all $k\notin \sE \cup \sO$.
If $k\in \sE$, then $e_k(\CE) = c_i^k/\beta^{\bB_k}$
and $e_k(\CO) = a_1^k$, so $q_{k0}(\CE) = c_i/\beta$ and 
$q_{k0}(\CO) = a_1$.  
Thus, $q_{k0}(\tbeta)\subseteq \tbeta$.
If $k\in \sO$, then 
$e_k(\CE) = b_{K-1}^k$ and $e_k(\CO) = c_i^k/\beta^{\bB_k}$,
so $q_{k0}(\CE) = b_{K-1} $ and 
$q_{k0}(\CO)= c_i/\beta$.  Thus, 
$q_{k0}(\tbeta)\subseteq \tbeta$.
Finally, $e_{0}(\CE) = c_i/\beta$ and 
$e_{0}(\CO) = a_1$, so, for each $f\in F_B$, the operation
$fe_{0}$ takes all of $\CE$ to a single $\beta$ class, and all of $\CO$ to a
single beta class.  That is, $fe_0(\tbeta)\subseteq \tbeta$ for all $f\in F_B$.
This completes the proof that $f(\tbeta)\subseteq \tbeta$ for all $f\in F_A$.

Since the
restriction of $\tbeta$ to $B$ is clearly $\tbeta \resB = \beta$, the
residuation lemma yields $\tbeta \leq \hbeta$, and we now prove
$\tbeta \geq \hbeta$.  Indeed, it is easy to see that, for each $(x,y)\notin
\tbeta$, there is an operation $f\in \Pol_1(\bA)$ such that $(e_0f(x),
e_0f(y))\notin \beta$, and thus $(x,y)\notin \hbeta$.  Verification of this
statement is trivial.  For example, if 
$x\in c_i^\ell/\beta^{\bB_\ell}$ 
for some $1\leq i < m, \, \ell \in \sE$
and $y\notin \CE$, then $e_0(x) \in c_i/\beta$ and $e_0(y)\notin c_i/\beta$,
so 
$(e_0(x), e_0(y))\notin \beta$.  To take a slightly less trivial case, suppose 
$x\in c_i^\ell/\beta^{\bB_\ell}$ 
for some $1\leq i < m, \, \ell \in \sO$
and $y\notin \CO$.  Then $(e_{\ell}(x), e_{\ell}(y))\notin \beta^{\bB_\ell}$,
so $(e_0q_{\ell 0}(x), e_0q_{\ell 0}(y)) = (q_{\ell 0}(x), q_{\ell 0}(y)) \notin
\beta$.  The few remaining cases are even easier to verify, so we omit them.
This completes the proof of~(\ref{eq:OA5}).

It remains to prove
$[\beta^*, \widehat{\beta}] \cong (\Eq|\sE|)^{m-1} \times (\Eq|\sO|)^{m-1}$. 
This follows trivially from what we have proved above.  For,
in proving that $\tbeta$ is a congruence, we showed that, in fact, each
operation $f\in F_A$ maps blocks of $\tbeta \, (= \hbeta)$ into blocks of $\beta^*$.  That is,
each operation collapses the interval $[\beta^*, \hbeta]$.  Therefore, every
equivalence relation on the set $A$ that lies between $\beta^*$ and $\hbeta$ is
respected by every operation of $\bA$. In other words,
\[
  [\beta^*, \hbeta] = \{\theta \in \Eq(A): \beta^* \leq \theta \leq \hbeta\}.
  \]
  In view of the configuration of the universe of $\bA$, 
  as shown in Figure~\ref{fig:OveralgebrasIII},
  it is clear that the interval sublattice 
  $\{\theta \in \Eq(A): \beta^* \leq \theta \leq \hbeta\}$
  is isomorphic to $(\Eq|\sE|)^{m-1} \times (\Eq|\sO|)^{m-1}$. 
  \hfill \qedsymbol

  \section{Conclusions}
  We have described an approach to building new finite algebras out of old
  which is useful in the following situation: given an algebra $\bB$ with a
  congruence lattice $\Con \bB$ of a particular shape, we seek an algebra
  $\bA$ with congruence lattice $\Con \bA$ which has $\Con \bB$ as a
  (non-trivial) homomorphic image; specifically, we construct
  $\bA$ so that $\resB : \Con \bA \rightarrow \Con \bB$ is a lattice epimorphism.
  We described the original example -- the ``triple-winged pentagon'' shown on
  the right of Figure~\ref{fig:sevens} -- found by 
\index{Freese, Ralph}%
Ralph Freese, which 
  motivated us to develop a general procedure for finding such finite algebraic
  representations.

  We mainly focused on a few specific overalgebra constructions.
  In each case, the congruence lattice that results has the same basic shape as
  the one with which we started, except that some congruences are replaced with
  intervals that are direct products of powers of partition lattices.  
  Thus we have identified a broad new class of finitely representable lattices.  
  However, the fact that the new intervals in these lattices must be products of
  partition lattices seems quite limiting, and this is the first limitation that we
  think future research might aim to overcome.  

  We envision potential variations on the constructions described herein,
  which might bring us closer toward the goal of replacing certain congruences
  $\beta\in \Con \bB$ with an more general finite lattices,
  $L\cong [\beta^*, \widehat{\beta}] \leq \Con \bA$.
  Using the constructions described above, we have found examples of overalgebras
  for which it is not possible to simply add operations 
  in order to eliminate \emph{all} relations strictly contained in the interval
  $(\beta^*,  \widehat{\beta})$.
  Nonetheless, we remain encouraged by the success of a very modest example
  in this direction, which we now describe.

  \begin{example}
    \label{ex:conclusion}
    Suppose $\<C, \dots\>$ is an arbitrary finite algebra with congruence lattice 
    $L_C := \Con \<C, \dots\>$. Relabel the elements so that $C = \{1, 2, \dots, N\}$.
    We show how to use the overalgebra construction described in
    Section~\ref{sec:overalgebras-i} to obtain a finite algebra with congruence
    lattice appearing in Figure~\ref{fig:conclusion}.\footnote{John
      Snow 
      \index{Snow, John}%
      has already proved that ``parallel sums'' of finitely representable
      lattices are finitely representable (See Lemmas 3.9 and 3.10
      of~\cite{Snow:2000}).
    }

    \begin{figure}[h!]
      \centering
      \begin{tikzpicture}[scale=.6]
        \node (bot) at (11,0.5)  [draw, circle, inner sep=1.0pt] {};
        \node (top) at (11,4.5)  [draw, circle, inner sep=1.0pt] {};
        \node (a) at (9.5,2.5)  [draw, circle, inner sep=1.0pt] {};
        \node (b) at (12,1.5)  [draw, circle, inner sep=1.0pt] {};
        \node (B) at (12,3.5)  [draw, circle, inner sep=1.0pt] {};
        \draw[semithick] 
        (bot) to (a) to (top) to (B)
        (b) to (bot);
        \draw [semithick]  
        (b) to [out=140,in=-140] (B)
        (B) to [out=-40,in=40] (b);
        \draw[font=\small] (12,2.5) node {$L_C$};
      \end{tikzpicture}
      \caption{
        $L_C$ an arbitrary finitely 
        representable lattice.}
      \label{fig:conclusion}
    \end{figure}
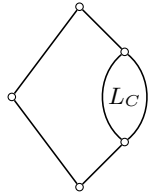

    Let $\bB = \<B, F_B\>$  be a unary algebra with universe 
    \[
    B = \{a_1, a_2, \dots, a_N, b_1, b_2, \dots, b_N \},
    \]
    and congruence lattice
    $\Con\bB = \{0_B, \alpha, \beta, 1_B\} \cong \two \times \two$, where
    \[
    \alpha = |a_1, b_1 | a_2, b_2 | \cdots | a_N, b_N | \quad \text{ and } \quad
    \beta = |a_1, a_2, \dots, a_N | b_1, b_2, \dots, b_N |.
    \]
    Such an algebra exists by the theorem of Berman\cite{Berman:1970}, and Quackenbush and
    Wolk~\cite{Quack:1971}.
    Let $B_1, B_2, \dots, B_N$ be sets of size $2N$ which intersect $B$ as
    follows:
    for all $1 \leq i < j \leq K$,
    \[
    B_0\cap B_i = \{b_i\}, \quad \text{ and } \quad B_i \cap B_j = \emptyset.
    \]
    If $\bA = \< A, F_A\>$ is the overalgebra constructed as in
    Section~\ref{sec:overalgebras-i}, then $\Con \bA$ is isomorphic to the lattice
    in Figure~\ref{fig:conclusion}, but with $L_C$ replaced with $\Eq(C)$.  Now
    expand the set $F_A$ of operations on $A$ as follows: for each $f\in F_C$,
    define $f_0: B\rightarrow B$ by $f_0(a_i) = a_{f(i)}$ and 
    $f_0(b_i) = b_{f(i)}$, and define $\hat{f}: A\rightarrow A$ by $\hat{f}(x) =
    f_0(s(x))$. Defining $F^+_A = F_A \cup \{\hat{f} : f\in F_C\}$, we claim that
    the congruence lattice of the algebra $\< A, F^+_A\>$ is (isomorphic to) the
    lattice appearing in Figure~\ref{fig:conclusion}.
  \end{example}

  As a final remark, we call attention to another obvious limitation of
  the methods describe in this chapter -- they cannot be used to find an
  algebra with congruence lattice isomorphic to the lattice $L_7$, which is the
  subject of Section~\ref{sec:except-seven-elem}.  
  This lattice is simple, so it is certainly not the inverse image under $\resB$ of some
  smaller lattice. 

\chapter{Open Questions}
\label{sec:open-questions}
We conclude this thesis by listing some open questions, the answers to which
will help us better understand finite algebras in general and finite groups in
particular.  It is the author's view that such progress will undoubtedly lead to
a solution to the \FLRP\ in the very near future.

Let $\bH(\sK)$ denote the class of homomorphic images of a class $\sK$
of algebras.
Let $\sL_3$ denote the class of 
\index{representable lattice}
\emph{representable} lattices; that is,
$L \in \sL_3$ if and only if $L \cong \Con \bA$ for some finite algebra $\bA$.
Let $\sL_{4}$ denote the class of 
\index{group representable lattice}
\emph{group representable} lattices; that is, $L \in \sL_{4}$ iff $L\cong [H, G]$ for some
finite groups $H\leq G$.
As we know, $\sL_3 \supseteq \sL_4$.
\begin{enumerate}
\item  Is $\sL_4$ is closed under homomorphic images, $\bH(\sL_4) = \sL_4$?  

\item Is $\bH(\sL_4) \subseteq \sL_3$ true?

\item Is $\bH(\sL_3) = \sL_3$ true?

\item Is $\sL_3 = \sL_4$ true?
In other words, if $L$ is 
the congruence lattice of a finite
algebra, is $L$ (isomorphic to) the congruence lattice of a transitive G-set?
Equivalently, is every congruence lattice of a finite algebra (isomorphic to) an
interval in the subgroup lattice of a finite group?  

\item Suppose $L \in \sL_4$. 
It is true that, 
$L_0 = \{ x\in L \mid x \leq \alpha \text{ or } \beta \leq x \} \in \sL_4$
for all $\alpha, \beta \in L$?  Note that, by the result of 
\index{Snow, John}%
John Snow (Lemma~\ref{lemma:union-filter-ideal}) 
this is true if we replace $\sL_4$ with $\sL_3$.

\item What other properties of groups, in addition to those described in
  Chapter~\ref{cha:subl-interv-enforc}, are interval sublattice enforceable
  (\ISLE) properties?

\item If a group property is \ISLE, is it true that the negation of that property cannot be \ISLE?
(This is~Conjecture~\ref{conjecture:isle-prop}.)

\item Is the lattice $M_7$ the congruence lattice of an algebra of cardinality
  less than $30!/10$?\\
(In~\cite{Feit:1983},  
\index{Feit, Walter}%
Walter Feit finds $M_7 \cong [H,A_{31}]$, where
  $|H|=31\cdot 5$, so $M_7$ is the congruence lattice of a transitive \Gset\ on
  $|A_{31}:H| = 30!/10$ elements.)

\item Is there a general characterization of the class of finite lattices that
  occur as congruence lattices of overalgebras?  As we pointed out in
  Section~\ref{ex:conclusion}, a simple lattice is not the congruence lattice
  of a (non-trivial) expansion of the type described in
  Chapter~\ref{cha:expans-finite-algebr}.  Are there other such
  properties, besides simplicity, describing lattices that cannot be
  the congruence lattice of an overalgebra?

\item Is the seven element lattice $L_{11}$ group representable? \\
  (Recall, we proved that $L_{11}$ is representable in
  Section~\ref{sec:seven-elem-latt} using the filter+ideal method which
  necessarily results in a non-permutational algebra.) 

\item Is every lattice with at most seven elements group representable?\\
  (In Section~\ref{sec:seven-elem-latt} we described the
  seven element lattices which are the most challenging to represent.  These
  appear in Figure~\ref{fig:sevens}.
  We saw that both $L_{13}$ and $L_{17}$ are group representable.
  Though we did not mention it above, we have also found the lattice $L_9$ 
  (which motivated the invention of overalgebras) as an interval in the subgroup
  lattice of $A_{10}$.  At the bottom of this interval is a subgroup of index
  25,400.  So the smallest \Gset\ we have found with congruence lattice
  isomorphic to $L_9$ is on 25,400 elements.  
  Clearly this is not the minimal representation of $L_9$. Indeed, in
  Example~\ref{ex:3.1} we constructed an overalgebra with 16
  elements that has a congruence lattice isomorphic to $L_9$.  We suspect it
  will not be very difficult to prove that the lattices $L_{19}$ and $L_{20}$
  are group representable.  Of the lattices appearing in
  Figure~\ref{fig:sevens} then, $L_7$ may not be representable,
  and $L_{11}$, though representable, seems difficult
  to find as an interval in a subgroup lattice of a finite group.)
\end{enumerate}

\appendix

\part{Appendix}
\chapter{Group Theory Background}
\label{cha:group-theory-backgr}
In this section we review some aspects of group theory that
are relevant to our problem of representing a finite lattice as the congruence
lattice of a finite algebra.

\section{Group actions and permutation groups}
\label{sec:group-acti-perm}
Let $G$ be a group, $\bA= \<A, \barG\>$ a \Gset, and let $\Sym(A)$ denote the group of
permutations of $A$.
For $a\in A$, the one-generated subalgebra $\<a\>\in \Sub(\bA)$ is
called the \defn{orbit} of $a$ in $\bA$. 
It is easily verified that $\<a\>$ is the set
$ \barG a :=  \{\barg a \mid g\in G\}$, and we often use the more suggestive 
$\barG a$ when referring to this orbit.

The orbits of the \Gset\ $\bA$ partition the set $A$ into disjoint
equivalence classes.  The equivalence relation $\sim$ is defined on $A^2$ as follows: 
$x \sim y$ if and only if $\barg x = y$ for some $g\in G$.
In fact, $\sim$ is a congruence relation of the algebra $\bA$ since,
$x \sim y$ implies $\barg x \sim \barg y$.
Thus, as mentioned above, each orbit is indeed a \emph{subalgebra} of $\bA$.

Keep in mind that $A$ is the disjoint union of the
orbits.  That is, if $\{a_1, \dots, a_r\}$ is a full
set of $\sim$-class representatives, then $A = \bigcup_{i=1}^r \barG a_i$ is a disjoint union.

A \Gset\ with only one orbit is called 
\defn{transitive}.  Equivalently, $\<A, \barG\>$ is a transitive \Gset\ if and
only if $(\forall a, b \in A)(\exists g\in G)(\barg a = b)$. 
In this case, we say that $G$ \emph{acts transitively} on $A$,
and occasionally we refer to the group $G$ itself as a \emph{transitive group} of \emph{degree} $|A|$.  

For $a\in A$, the set $\Stab_G(a) := \{g \in G \mid  \barg a  = a\}$ is called the 
\defn{stabilizer} of $a$.  It is easy to verify  that $\Stab_G(a)$ is a
subgroup of $G$.  An alternative notation for the stabilizer 
is $G_a := \Stab_G(a)$.

Let $\lambda : G \rightarrow \barG \leq \Sym(A)$ denote the permutation representation
of $G$; that is, $\lambda(g) = \barg$. Then 
\begin{equation}
  \label{eq:111}
\ker \lambda = \{g\in G  \mid  \barg a = a \text{ for all $a\in A$}\} = \bigcap_{a\in A}
\Stab_G(a)
= \bigcap_{a\in A} G_a.
\end{equation}
Therefore, $G/\ker\lambda \cong \lambda[G]\leq \Sym(A)$.
We say that the representation $\lambda$ of $G$ is \defn{faithful}, or 
that $G$ \emph{acts faithfully} on $A$, just in case $\ker \lambda = 1$. In
this case $\lambda : G \hookrightarrow \Sym(A)$, so $G$ itself is isomorphic to a subgroup of 
$\Sym(A)$, and we call $G$ a \defn{permutation group}.

If $H \leq G$ are groups, the \defn{core} of $H$ in $G$, denoted $\core_G(H)$,
is the largest normal subgroup of $G$ that is contained in $H$.  It is easy to see
that
\[
\core_G(H) = \bigcap_{g\in G}g H g^{-1}.
\]
A subgroup $H$ is called \emph{core-free} provided $\core_G(H)=1$.

Elements in the same orbit of a \Gset\
have conjugate stabilizers.  Specifically, if $a, b\in A$  and $g\in G$ are
such that $\barg a = b$, then 
$\stab{b} = \stab{\barg a} = g \,\stab{a}\, g^{-1}$.
If the \Gset\ happens to be transitive, then it is faithful if and only if the
stabilizer $\stab{a}$ is core-free in $G$. For,
\[
\ker \lambda = \bigcap_{a\in A} \stab{a}
= \bigcap_{g\in G} \stab{\barg a}
= \bigcap_{g\in G} g \,\stab{a}  \,g^{-1}.
\]
Thus $\stab{a}$ is core-free if and only if $\ker \lambda = 1$ if and only if
$G$ acts faithfully on $A$. 

In case $G$ is a transitive permutation group, we say that $G$ is 
\defn{regular} (or that $G$ \emph{acts regularly} on $A$, or that $\lambda
: G \rightarrow \barG$ is a \emph{regular representation})
provided $\stab{a} = 1$ for each $a\in A$; i.e.,
every non-identity element of $G$ is fixed-point-free.\footnote{The action of a
  regular permutation group is sometimes called a ``free'' action.} Equivalently,
$G$ is regular on $A$ if and only if for each $a, b \in A$ there is a unique $g\in G$ such
that $\barg a = b$.  In particular, $|G| = |A|$.

A \defn{block system} for $G$ is a partition of $A$
  that is preserved by the action of $G$.  In other words, a block system is a
  congruence relation of the algebra $\bA = \<A,\barG\>$.
The \defn{trivial block systems} are $0_A = |a_1|a_2|\cdots|a_i|\cdots$ and 
$1_A = |a_1 a_2 \cdots a_i \cdots|$.  The non-trivial block systems are called 
\defn{systems of imprimitivity}.

A nonempty subset $B\subseteq A$ is a \defn{block} for $\bA$ 
if for each $g \in G$ either $\barg B = B$ or $\barg B \cap B = \emptyset$.

Let $\bA = \<A, \barG\>$ be a transitive \Gset.  
In most group theory textbooks one finds the following definition: a group $G$
is called \defn{primitive} if $\bA$ has no systems of imprimitivity;
otherwise $G$ is called \defn{imprimitive}.  In other words, $G$ is
primitive if and only if the transitive \Gset\ $\<A, \barG\>$ is a 
\defn{simple algebra} -- that is, $\Con \<A, \barG\> \cong \2$.
In the author's view, this definition of primitive is meaningless and is the
source of unnecessary confusion.  Clearly \emph{every} finite group acts
transitively on the cosets of a maximal subgroup $H$ and the resulting \Gset\ has 
$\Con \<G/H, \barG\> \cong [H, G] \cong \2$.  This means that, according to the
usual definition, every finite group is primitive.
To make the definition more meaningful, we should require that a primitive group
be isomorphic to a permutation group.  That is, we call a transitive
permutation group \defn{primitive} if the induced algebra is simple.  
To see the distinction, take an arbitrary group $G$ acting on the cosets of a
subgroup $H$.  This action 
is faithful, and $G$ is a permutation group, if and only if 
$H$ is core-free.  If, in addition, $H$ is a maximal subgroup, then the
induced algebra $\<G/H, \barG\>$ is simple.  For these reasons, we will call a
group \defn{primitive} if and only if it has a core-free maximal subgroup.  
(Note that the terms ``primitive'' and ``imprimitive'' are used 
only with reference to \emph{transitive} \Gsets.)

\section{Classifying permutation groups}
\label{sec:class-perm-groups}
A permutation group is either transitive
or is a subdirect product of transitive groups, while a transitive group is
either primitive or is a subgroup of an iterated wreath product of primitive
groups. (See, e.g., Praeger~\cite{Praeger:2006}.)
Hence primitive groups can be viewed as the building blocks of all
permutations groups and their classification helps us to better understand
the structure of permutation groups in general.

The \defn{socle} of a group $G$ is the subgroup generated by the minimal normal
subgroups of $G$ and is denoted by $\Soc(G)$. By~\cite{Dixon:1996}, Corollary
4.3B, the socle 
of a finite primitive group is isomorphic to the direct product of one or more
copies of a simple group $T$.  The O'Nan-Scott Theorem classifies the primitive
permutation groups according to the structure of their socles.  The following
version of the theorem seems to be among the most useful, and it appears for
example in the Ph.D. thesis of 
\index{Coutts, Hannah}%
Hannah Coutts~\cite{Coutts:2010}.

\subsection{The O'Nan-Scott Theorem}
\label{sec:onan-scott-theorem}
\index{O'Nan-Scott Theorem}%
\index{Aschbacher-O'Nan-Scott Theorem}%
\begin{theorem}[O'Nan-Scott Theorem] 
Let $G$ be a primitive permutation
group of degree $d$, and let $N := \Soc(G) \cong T^m$ with $m \geq 1$. 
Then one of the following holds.
\end{theorem}
\begin{enumerate}
\item 
$N$ is regular and
  \begin{enumerate}
  \item 
  \defn{Affine type} $T$ is cyclic of order $p$, so $|N| = p^m$ . Then 
$d = p^m$ and $G$ is permutation isomorphic to a subgroup of the affine
general linear group $\AGL(m,p)$. We call $G$ a group of \emph{affine type}.
\item \defn{Twisted wreath product type} $m \geq 6$, the group $T$ is 
  nonabelian and $G$ is a group of \emph{twisted wreath product type}, with
  $d = |T|^m$.
  \end{enumerate}
\item $N$ is non-regular and non-abelian and
  \begin{enumerate}
  \item 
\defn{Almost simple} $m = 1$ and $T \leq G \leq \Aut(T)$.
\item \defn{Product action} $m \geq 2$ and $G$ is permutation isomorphic to a
subgroup of the product action wreath product $P \wr S_{m/l}$ of degree
$d = nm/l$. The group $P$ is primitive of type 2.(a) or 2.(c), $P$ has
degree $n$ and $\Soc(P) \cong T^l$, where $l \geq 1$ divides $m$.
\item 
\defn{Diagonal type} $m \geq 2$ and $T^m \leq G \leq T^m . (\Out(T ) \times S_m)$, with
the diagonal action. The degree $d = |T|^{m-1}$.
  \end{enumerate}
\end{enumerate}

We can see immediately that there are no twisted wreath product type
groups of degree less than $60^6$ ($=46.656$ billion).
Note that this definition of product action groups is more restrictive
than that given by some authors. This is in order to make the O'Nan-Scott
classes disjoint.

\bibliography{wjd}

\def\cprime{$'$} \def\ocirc#1{\ifmmode\setbox0=\hbox{$#1$}\dimen0=\ht0
  \advance\dimen0 by1pt\rlap{\hbox to\wd0{\hss\raise\dimen0
  \hbox{\hskip.2em$\scriptscriptstyle\circ$}\hss}}#1\else {\accent"17 #1}\fi}
\begin{thebibliography}{10}

\bibitem{Aschbacher:2008}
Michael Aschbacher.
\newblock On intervals in subgroup lattices of finite groups.
\newblock {\em J. Amer. Math. Soc.}, 21(3):809--830, 2008.
\newblock \href {http://dx.doi.org/10.1090/S0894-0347-08-00602-4}
  {\path{doi:10.1090/S0894-0347-08-00602-4}}.

\bibitem{Lucchini:1997}
Robert Baddeley and Andrea Lucchini.
\newblock On representing finite lattices as intervals in subgroup lattices of
  finite groups.
\newblock {\em J. Algebra}, 196(1):1--100, 1997.
\newblock \href {http://dx.doi.org/10.1006/jabr.1997.7069}
  {\path{doi:10.1006/jabr.1997.7069}}.

\bibitem{BBE:2006}
Adolfo Ballester-Bolinches and Luis~M. Ezquerro.
\newblock {\em Classes of finite groups}, volume 584 of {\em Mathematics and
  Its Applications (Springer)}.
\newblock Springer, Dordrecht, 2006.

\bibitem{Basile:2001}
Alberto Basile.
\newblock {\em Second maximal subgroups of the finite alternating and symmetric
  groups.}
\newblock PhD thesis, Australian National University, Canberra, April 2001.

\bibitem{Berman:1970}
Joel Berman.
\newblock {\em Congruence lattices of finite universal algebras}.
\newblock PhD thesis, University of Washington, 1970.
\newblock Available from: \url{http://db.tt/mXUVTzSr}.

\bibitem{Birkhoff:1935}
Garrett Birkhoff.
\newblock On the structure of abstract algebras.
\newblock {\em Proc. Cmabridge Phil. Soc.}, 31:433--454, 1935.

\bibitem{Birkhoff:1940}
Garrett Birkhoff.
\newblock {\em Lattice {T}heory}.
\newblock American Mathematical Society, New York, 1940.

\bibitem{Borner:1999}
Ferdinand B{\"o}rner.
\newblock A remark on the finite lattice representation problem.
\newblock In {\em Contributions to general algebra, 11 ({O}lomouc/{V}elk\'e
  {K}arlovice, 1998)}, pages 5--38, Klagenfurt, 1999. Heyn.

\bibitem{Coutts:2010}
Hannah Coutts.
\newblock {\em Topics in Computational Group Theory: Primitive permutation
  groups and matrix group normalisers}.
\newblock PhD thesis, University of St.~Andrews, 2010.
\newblock Available from:
  \url{http://www-circa.mcs.st-and.ac.uk/Theses/HCoutts_thesis.pdf}.

\bibitem{Dedekind:1877}
Richard Dedekind.
\newblock {\"U}ber die {A}nzahl der {I}deal-classen in den verschiedenen
  {O}rdnungen eines endlichen {K}\"orpers.
\newblock In {\em {F}estschrift zur {S}aecularfeier des {G}eburtstages von
  {C}.~{F}.~{G}auss}, pages 1--55. {V}ieweg, {B}raunschweig, 1877.
\newblock see {G}es. {W}erke, {B}and {I}, 1930, 105--157.

\bibitem{gsets}
William DeMeo and Ralph Freese.
\newblock Congruence lattices of intransitive {G}-sets.
\newblock {\it preprint}, 2012.
\newblock Available from: \url{http://db.tt/tNzQsZl9}.

\bibitem{Dixon:1996}
John~D. Dixon and Brian Mortimer.
\newblock {\em Permutation groups}, volume 163 of {\em Graduate Texts in
  Mathematics}.
\newblock Springer-Verlag, New York, 1996.

\bibitem{Doerk:1992}
Klaus Doerk and Trevor Hawkes.
\newblock {\em Finite soluble groups}, volume~4 of {\em de Gruyter Expositions
  in Mathematics}.
\newblock Walter de Gruyter \& Co., Berlin, 1992.

\bibitem{Feit:1983}
Walter Feit.
\newblock An interval in the subgroup lattice of a finite group which is
  isomorphic to {$M_{7}$}.
\newblock {\em Algebra Universalis}, 17(2):220--221, 1983.
\newblock \href {http://dx.doi.org/10.1007/BF01194532}
  {\path{doi:10.1007/BF01194532}}.

\bibitem{Freese:1975}
Ralph Freese.
\newblock Congruence lattices of finitely generated modular lattices.
\newblock In {\em Proceedings of the {L}attice {T}heory {C}onference ({U}lm,
  1975)}, pages 62--70, Ulm, 1975. Univ. Ulm.

\bibitem{uacalc}
Ralph Freese, Emil Kiss, and Matthew Valeriote.
\newblock Universal {A}lgebra {C}alculator, 2008.
\newblock Available from: \url{http://www.uacalc.org}.

\bibitem{GAP4}
The GAP~Group.
\newblock {\em {GAP -- Groups, Algorithms, and Programming, Ver.~4.4.12}},
  2008.
\newblock Available from: \url{http://www.gap-system.org}.

\bibitem{GratzerSchmidt:1963}
G.~Gr{\"a}tzer and E.~T. Schmidt.
\newblock Characterizations of congruence lattices of abstract algebras.
\newblock {\em Acta Sci. Math. (Szeged)}, 24:34--59, 1963.

\bibitem{Gratzer:1968}
George Gr{\"a}tzer.
\newblock {\em Universal algebra}.
\newblock D. Van Nostrand Co., Inc., Princeton, N.J.-Toronto, Ont.-London,
  1968.

\bibitem{Isaacs:2008}
I.~Martin Isaacs.
\newblock {\em Finite group theory}, volume~92 of {\em Graduate Studies in
  Mathematics}.
\newblock American Mathematical Society, Providence, RI, 2008.

\bibitem{Jonsson:1972}
Bjarni J{\'o}nsson.
\newblock {\em Topics in universal algebra}.
\newblock Lecture Notes in Mathematics, Vol. 250. Springer-Verlag, Berlin,
  1972.

\bibitem{Kohler:1983}
Peter K{\"o}hler.
\newblock {$M_{7}$} as an interval in a subgroup lattice.
\newblock {\em Algebra Universalis}, 17(3):263--266, 1983.
\newblock \href {http://dx.doi.org/10.1007/BF01194535}
  {\path{doi:10.1007/BF01194535}}.

\bibitem{Kurzweil:1985}
Hans Kurzweil.
\newblock Endliche {G}ruppen mit vielen {U}ntergruppen.
\newblock {\em J. Reine Angew. Math.}, 356:140--160, 1985.
\newblock \href {http://dx.doi.org/10.1515/crll.1985.356.140}
  {\path{doi:10.1515/crll.1985.356.140}}.

\bibitem{McKenzie:1983}
Ralph McKenzie.
\newblock Finite forbidden lattices.
\newblock In {\em Universal Algebra and Lattice Theory ({P}uebla, 1982)},
  volume 1004 of {\em Lecture Notes in Math.}, pages 176--205, Berlin, 1983.
  Springer.

\bibitem{McKenzie:1984}
Ralph McKenzie.
\newblock A new product of algebras and a type reduction theorem.
\newblock {\em Algebra Universalis}, 18(1):29--69, 1984.
\newblock \href {http://dx.doi.org/10.1007/BF01182247}
  {\path{doi:10.1007/BF01182247}}.

\bibitem{alvi:1987}
Ralph~N. McKenzie, George~F. McNulty, and Walter~F. Taylor.
\newblock {\em Algebras, lattices, varieties. {V}ol. {I}}.
\newblock Wadsworth \& Brooks/Cole, Monterey, CA, 1987.

\bibitem{Netter:1986}
R.~Netter.
\newblock Eine bemerkung zu kongruenzverbanden.
\newblock preprint, 1986.

\bibitem{Palfy:1987}
P.~P. P{\'a}lfy.
\newblock Distributive congruence lattices of finite algebras.
\newblock {\em Acta Sci. Math. (Szeged)}, 51(1-2):153--162, 1987.

\bibitem{Palfy:1995}
P{\'e}ter~P{\'a}l P{\'a}lfy.
\newblock Intervals in subgroup lattices of finite groups.
\newblock In {\em Groups '93 {G}alway/{S}t.\ {A}ndrews, {V}ol.\ 2}, volume 212
  of {\em London Math. Soc. Lecture Note Ser.}, pages 482--494. Cambridge Univ.
  Press, Cambridge, 1995.
\newblock \href {http://dx.doi.org/10.1017/CBO9780511629297.014}
  {\path{doi:10.1017/CBO9780511629297.014}}.

\bibitem{Palfy:2001}
P{\'e}ter~P{\'a}l P{\'a}lfy.
\newblock Groups and lattices.
\newblock In {\em Groups {S}t. {A}ndrews 2001 in {O}xford. {V}ol. {II}}, volume
  305 of {\em London Math. Soc. Lecture Note Ser.}, pages 428--454, Cambridge,
  2003. Cambridge Univ. Press.
\newblock \href {http://dx.doi.org/10.1017/CBO9780511542787.014}
  {\path{doi:10.1017/CBO9780511542787.014}}.

\bibitem{Palfy:2009}
P{\'e}ter~P{\'a}l P{\'a}lfy.
\newblock The finite congruence lattice problem, September 2009.
\newblock Summer School on General Algebra and Ordered Sets Star{\'a}
  Lesn{\'a}, 6, 2009.
\newblock Available from: \url{http://db.tt/DydVmisY}.

\bibitem{Palfy:1980}
P{\'e}ter~P{\'a}l P{\'a}lfy and Pavel Pudl{\'a}k.
\newblock Congruence lattices of finite algebras and intervals in subgroup
  lattices of finite groups.
\newblock {\em Algebra Universalis}, 11(1):22--27, 1980.
\newblock \href {http://dx.doi.org/10.1007/BF02483080}
  {\path{doi:10.1007/BF02483080}}.

\bibitem{Praeger:2006}
Cheryl~E. Praeger.
\newblock Seminormal and subnormal subgroup lattices for transitive permutation
  groups.
\newblock {\em J. Aust. Math. Soc.}, 80(1):45--63, 2006.
\newblock \href {http://dx.doi.org/10.1017/S144678870001137X}
  {\path{doi:10.1017/S144678870001137X}}.

\bibitem{Pudlak:1976}
P.~Pudl{\'a}k and J.~T{\.u}ma.
\newblock Yeast graphs and fermentation of algebraic lattices.
\newblock In {\em Lattice theory ({P}roc. {C}olloq., {S}zeged, 1974)}, pages
  301--341. Colloq. Math. Soc. J\'anos Bolyai, Vol. 14. North-Holland,
  Amsterdam, 1976.

\bibitem{Pudlak:1980}
Pavel Pudl{\'a}k and Ji{\v{r}}{\'{\i}} T{\.u}ma.
\newblock Every finite lattice can be embedded in a finite partition lattice.
\newblock {\em Algebra Universalis}, 10(1):74--95, 1980.
\newblock \href {http://dx.doi.org/10.1007/BF02482893}
  {\path{doi:10.1007/BF02482893}}.

\bibitem{Quack:1971}
R.~Quackenbush and B.~Wolk.
\newblock Strong representation of congruence lattices.
\newblock {\em Algebra Universalis}, 1:165--166, 1971/72.

\bibitem{Robinson:1996}
Derek J.~S. Robinson.
\newblock {\em A course in the theory of groups}, volume~80 of {\em Graduate
  Texts in Mathematics}.
\newblock Springer-Verlag, New York, second edition, 1996.

\bibitem{Rose:1978}
John~S. Rose.
\newblock {\em A course on group theory}.
\newblock Dover Publications Inc., New York, 1994.
\newblock Reprint of the 1978 original [Cambridge University Press; MR0498810
  (58 \#16847)].

\bibitem{Rottlaender:1928}
Ada Rottlaender.
\newblock Nachweis der {E}xistenz nicht-isomorpher {G}ruppen von gleicher
  {S}ituation der {U}ntergruppen.
\newblock {\em Math. Z.}, 28(1):641--653, 1928.
\newblock \href {http://dx.doi.org/10.1007/BF01181188}
  {\path{doi:10.1007/BF01181188}}.

\bibitem{Schmidt:1982}
E.~Tam{\'a}s Schmidt.
\newblock {\em A survey on congruence lattice representations}, volume~42 of
  {\em Teubner-Texte zur Mathematik}.
\newblock BSB B. G. Teubner Verlagsgesellschaft, Leipzig, 1982.

\bibitem{Schmidt:1994}
Roland Schmidt.
\newblock {\em Subgroup lattices of groups}, volume~14 of {\em de Gruyter
  Expositions in Mathematics}.
\newblock Walter de Gruyter \& Co., Berlin, 1994.

\bibitem{Silcock:1977}
Howard~L. Silcock.
\newblock Generalized wreath products and the lattice of normal subgroups of a
  group.
\newblock {\em Algebra Universalis}, 7(3):361--372, 1977.

\bibitem{Snow:2000}
John~W. Snow.
\newblock A constructive approach to the finite congruence lattice
  representation problem.
\newblock {\em Algebra Universalis}, 43(2-3):279--293, 2000.
\newblock \href {http://dx.doi.org/10.1007/s000120050159}
  {\path{doi:10.1007/s000120050159}}.

\bibitem{Suzuki:1982}
Michio Suzuki.
\newblock {\em Group theory. {I}}, volume 247 of {\em Grundlehren der
  Mathematischen Wissenschaften [Fundamental Principles of Mathematical
  Sciences]}.
\newblock Springer-Verlag, Berlin, 1982.
\newblock Translated from the Japanese by the author.

\bibitem{Tuma:1986}
Ji{\v{r}}{\'{\i}} T{\ocirc{u}}ma.
\newblock Some finite congruence lattices. {I}.
\newblock {\em Czechoslovak Math. J.}, 36(111)(2):298--330, 1986.

\bibitem{Watatani:1996}
Yasuo Watatani.
\newblock Lattices of intermediate subfactors.
\newblock {\em J. Funct. Anal.}, 140(2):312--334, 1996.
\newblock \href {http://dx.doi.org/10.1006/jfan.1996.0110}
  {\path{doi:10.1006/jfan.1996.0110}}.

\bibitem{Whitman:1946}
Philip~M. Whitman.
\newblock Lattices, equivalence relations, and subgroups.
\newblock {\em Bull. Amer. Math. Soc.}, 52:507--522, 1946.

\end{thebibliography}
\bibliographystyle{plainurl}

\printindex

\end{document}